\documentclass[twoside,a4paper,12pt,centertags]{amsart}
\usepackage{amsmath,amssymb,verbatim,vmargin, amscd}
\usepackage[colorlinks,citecolor=red,pagebackref,hypertexnames=false]{hyperref}

\theoremstyle{plain}
\newtheorem{thm}{Theorem}[section]
\newtheorem{lem}[thm]{Lemma}
\newtheorem{cor}[thm]{Corollary}
\newtheorem{prop}[thm]{Proposition}

\theoremstyle{definition}
\newtheorem{defn}[thm]{Definition}
\newtheorem{rem}[thm]{Remark}

\mathchardef\semic="303B
\newcommand{\dirac}{{D}}

\newcommand{\R}{{\mathbb R}}
\newcommand{\Q}{{\mathbb Q}}
\newcommand{\N}{{\mathbb N}}
\newcommand{\C}{{\mathbb C}}
\newcommand{\Z}{{\mathbb Z}}
\newcommand{\IL}{{\mathbb L}}
\newcommand{\IP}{{\mathbb P}}
\newcommand{\IH}{\mathbb H}
\newcommand{\bS}{{\mathbb S}}
\newcommand{\BBMO}{\mathbb {BMO}}
\newcommand{\BMO}{\mathrm{BMO}}

\newcommand{\mH}{{\mathcal H}}

\newcommand{\mX}{{\mathcal X}}

\newcommand{\mL}{{\mathcal L}}

\newcommand{\mA}{{\mathcal A}}
\newcommand{\mN}{{\mathcal N}}
\newcommand{\mE}{{\mathcal E}}
\newcommand{\mD}{{\mathcal D}}

\newcommand{\mR}{{\mathcal R}}
\newcommand{\mT}{{\mathcal T}}
\newcommand{\mS}{{\mathcal S}}

\newcommand{\mI}{{\mathcal I}}

\DeclareMathOperator{\re}{Re}
\DeclareMathOperator{\im}{Im}

\newcommand{\brac}[1]{\langle #1 \rangle}
\newcommand{\supp}{\text{{\rm supp}}\,}
\newcommand{\dist}{\text{{\rm dist}}\,}

\newcommand{\nul}{\textsf{N}}
\newcommand{\ran}{\textsf{R}}
\newcommand{\dom}{\textsf{D}}

\newcommand{\clos}[1]{\overline{#1}}
\newcommand{\closran}[1]{\overline{\ran(#1)}}
\newcommand{\conj}[1]{\overline{#1}}

\newcommand{\sgn}{\text{{\rm sgn}}}
\newcommand{\barint}{\mbox{$ave \int$}}
\newcommand{\divv}{{\text{{\rm div}}}}

\newcommand{\tdd}[2]{\tfrac{\partial #1}{\partial #2}}

\newcommand{\wt}{\widetilde}

\newcommand{\ta}{{\scriptscriptstyle \parallel}}
\newcommand{\no}{{\scriptscriptstyle\perp}}
\newcommand{\pd}{\partial}

\newcommand{\oA}{{\overline A}}

\newcommand{\uT}{{\underline T}}
\newcommand{\bet}{{\mathcal B}}

\newcommand{\loc}{\text{{\rm loc}}}
\newcommand{\tN}{\widetilde N_*}
\newcommand{\tNs}{\widetilde N_\sharp}
 \newcommand{\tNsa}{\widetilde N_{\sharp,\alpha}}

\newcommand{\bx}{{\bf x}}
\newcommand{\by}{{\bf y}}

\newcommand{\reu}{\mathbb{R}^{1+n}_+}
\newcommand{\ree}{\mathbb{R}^{1+n}}
\newcommand{\abs}[1]{|#1|}

\newcommand{\Norm}[2]{\|#1\|_{#2}}

\newcommand{\pair}[2]{\langle #1,#2 \rangle}
\newcommand{\Bpair}[2]{\Big\langle #1,#2 \Big\rangle}

\newcommand{\modz}{[z]}
\newcommand{\MM}{{\mathbf M}}
\newcommand{\SF}{S}
\newcommand{\Qpsi}[2]{\Q_{#1,#2}}
\newcommand{\Tpsi}[2]{\bS_{#1,#2}}

\def\barint_#1{\mathchoice
            {\mathop{\vrule width 6pt
height 3 pt depth -2.5pt
                    \kern -8.8pt
\intop \kern -4pt}\nolimits_{#1}}%
            {\mathop{\vrule width 5pt height
3 pt depth -2.6pt
                    \kern -6.5pt
\intop \kern -4pt}\nolimits_{#1}}%
            {\mathop{\vrule width 5pt height
3 pt depth -2.6pt
                    \kern -6pt
\intop \kern -4pt}\nolimits_{#1}}%
            {\mathop{\vrule width 5pt height
3 pt depth -2.6pt
          \kern -6pt \intop \kern -4pt}\nolimits_{#1}}}
          
          \def\bariint_#1{\mathchoice
            {\mathop{\vrule width 10pt
height 3 pt depth -2.5pt
                    \kern -12.8pt
\intop \kern -10pt\intop \kern -4pt}\nolimits_{#1}}%
            {\mathop{\vrule width 9pt height
3 pt depth -2.6pt
                    \kern -10.5pt
\intop \kern -10pt\intop \kern -4pt}\nolimits_{#1}}%
            {\mathop{\vrule width 9pt height
3 pt depth -2.6pt
                    \kern -10pt
\intop \kern -10pt\intop \kern -4pt}\nolimits_{#1}}%
            {\mathop{\vrule width 9pt height
3 pt depth -2.6pt
          \kern -10pt \intop \kern -10pt\intop \kern -4pt}
      \nolimits_{  #1}}}
          
\renewcommand{\iint}{\int \kern -10pt\int}       

\usepackage{color}

\definecolor{gr}{rgb}   {0.,   0.8,   0. } 
\definecolor{bl}{rgb}   {0.,   0.5,   1. } 
\definecolor{mg}{rgb}   {0.7,  0.,    0.7}

\makeatletter
\@namedef{subjclassname@2010}{
  \textup{2010} Mathematics Subject Classification}
\makeatother

\title[\textit{A priori} estimates via first order systems]{\textit{A priori} estimates for boundary value elliptic problems via first order systems}
\author{Pascal Auscher}
\address{Univ. Paris-Sud, laboratoire de Math\'ematiques, UMR 8628 du CNRS, F-91405 {\sc Orsay}} 
\email{pascal.auscher@math.u-psud.fr}

\author{Sebastian Stahlhut}
\address{Univ. Paris-Sud, laboratoire de Math\'ematiques, UMR 8628 du CNRS, F-91405 {\sc Orsay}} 
\email{sebastian.stahlhut@math.u-psud.fr}

\date{24 juin 2014}

\begin{document}



\begin{abstract} We prove a number of \textit{a priori} estimates for weak solutions of elliptic equations or systems with vertically independent coefficients in the upper-half space. These estimates are designed towards applications to boundary value problems of Dirichlet and Neumann type in various topologies. We work in classes of solutions which include the energy solutions. For those solutions, we use a description using the   first order systems satisfied by   their conormal gradients and the theory of Hardy spaces associated with such systems but the method also allows us to design  solutions which are not necessarily energy solutions. We obtain precise comparisons between square functions, non-tangential maximal functions and norms of boundary trace. The main thesis is that the range of exponents for such results is related to when those Hardy spaces (which could be abstract spaces) are identified to concrete spaces of tempered distributions. We  consider  some adapted  non-tangential sharp functions and prove comparisons with square functions. We obtain  boundedness results for layer potentials, boundary behavior, in particular  strong limits, which is new, and jump relations.  One application is an extrapolation for solvability  ``\`a la  {\v{S}}ne{\u\i}berg''. Another one is stability of solvability in perturbing the coefficients in $L^\infty$ without further assumptions. We stress that our results do not require De Giorgi-Nash assumptions, and we   improve the available ones when we do so. 
\end{abstract}

\maketitle

\tableofcontents

\section{Introduction}

Our main goal in this work is to provide \textit{a priori} estimates for boundary value problems for $t$-independent systems in the upper half space. We will apply this to perturbation theory for  solvability. Of course, this topic has been much studied,  but our methods and results are original in this context.  We obtain new estimates and also design  solutions in many different classes.  A remarkable feature is that we do not require any kind of existence or uniqueness to build such solutions. In fact, the point of the reduction of second order PDEs to  first order systems is that for such systems the aim is to understand the initial value problem, and solving the PDE means inverting a boundary operator to create the initial data for the first order system. The initial value problem looks easier. However, the system has now a big null space and this creates other types of difficulties as we shall see. 

If $E(\Omega)$ is a normed space of $\C$-valued functions on a set  $\Omega$ and $F$ a normed space, then $E(\Omega;F)$ is the space of $F$-valued functions  with $\| |f|_{F}\|_{E(\Omega)}<\infty$. More often, we forget about the underlying $F$ if the context is clear.

We denote points in $\ree$ by boldface letters $\bx,\by,\ldots$ and in coordinates in $\R \times \R^n$ by $(t,x)$ etc. We set $\R^{1+n}_+=(0,\infty)\times \R^n$. 
Consider the system of $m$  equations  given by
\begin{equation}  \label{eq:divform}
  \sum_{i,j=0}^n\sum_{\beta= 1}^m \pd_i\big( A_{i,j}^{\alpha, \beta}(x) \pd_j u^{\beta}(\bx)\big) =0,\qquad \alpha=1,\ldots, m
\end{equation}
in $\R^{1+n}_+$,
where $\pd_0= \tdd{}{t}$ and $\pd_i= \tdd{}{x_{i}}$ if $i=1,\ldots,n$.  For short, we write $Lu=-\divv A \nabla u=0$ to mean \eqref{eq:divform}, where we always assume that the matrix   \begin{equation}   \label{eq:boundedmatrix}
  A(x)=(A_{i,j}^{\alpha,\beta}(x))_{i,j=0,\ldots, n}^{\alpha,\beta= 1,\ldots,m}\in L^\infty(\R^n;\mL(\C^{m(1+n)})),
\end{equation} is bounded and measurable, independent of $t$, and satisfies 
 the strict accretivity condition on  a subspace $\mH$  of $ L^2(\R^n;\C^{m(1+n)})$,  that is,
for some $\lambda>0$ 
\begin{equation}   \label{eq:accrassumption}
   \int_{\R^n} \re (A(x)f(x)\cdot  \conj{f(x)}) \,  dx\ge \lambda 
   \sum_{i=0}^n\sum_{\alpha=1}^m \int_{\R^n} |f_i^\alpha(x)|^2dx, \ \forall  f\in \mH.
\end{equation}
The subspace $\mH$ is formed of those functions $(f_{j}^\alpha)_{j=0, \ldots, n}^{\alpha=1, \ldots, m}$  such that $(f_{j}^\alpha)_{j=1,\ldots,n}$ is curl-free  in $\R^n$ for all $\alpha$.  
The system \eqref{eq:divform} is always considered in the sense of distributions with weak solutions, that is  $H^1_{loc}(\R^{1+n}_{+};\C^m)=W^{1,2}_{loc}(\R^{1+n}_{+};\C^m)$ solutions.

It was proved in \cite{AA1} that weak solutions of $Lu=0$  in the classes $$\mE_{0}=\{u\in \mD';\|\tN(\nabla u )\|_{2}<\infty\}$$ or $$\mE_{-1}=\{u\in \mD';\|\SF(t\nabla u)\|_{2}<\infty\}$$ (where $\tN(f)$ and
$S(f)$ stand for a non-tangential maximal function  and square function: definitions will be given later) have certain semigroup representation in their conormal gradient
$$
\nabla_A u(t,x):=  \begin{bmatrix} \pd_{\nu_A}u(t,x)\\ \nabla_x u(t,x) \end{bmatrix}.$$
More precisely, one has
\begin{equation}
\label{eq:sgrep}
\nabla_A u(t,\, .\,)= S(t) ( \nabla_A u|_{t=0})
\end{equation}
for a certain semigroup $S(t)$ acting on the subspace $\mH$ of $L^2$ in the first case and in the corresponding subspace in $\dot H^{-1}$, where $\dot H^s$ is  the homogeneous Sobolev space of order $s$, in the second case. Actually, the second representation was only expli\-citly  derived in  subsequent works (\cite{AMcM, R2}) provided one defines the conormal gradient at the boundary in this subspace of $\dot H^{-1}$. In \cite{R2}, the semigroup representation was extended to 
weak solutions in the intermediate classes   defined by $\mE_{s}=\{u\in \mD'; \|\SF(t^{-s}\nabla u)\|_{2}<\infty\}$ for $-1<s<0$ and the semigroup representation holds in $\dot H^{s}$. In particular, for $s=-1/2$, the class of weak solutions in $\mE_{{-1/2}}$ is exactly the class of energy solutions used in \cite{AMcM,AM} (other classes were defined in \cite{KR} and used in \cite{HKMP2}). And  the boundary value problems associated to $L$ can always be solved in the energy class.  However, we shall not use this solvability property nor any other one until Section \ref{sec:perturbation}. 

Here,  we intend to study the following problems: \\

\textbf{Problem 1:} For which $p\in (0,\infty)$ do we have
\begin{equation}
\label{eq:NeuReg}
\|\tN(\nabla u )\|_{p}\sim \|\nabla_A u|_{t=0}\|_{X_{p}}\sim \|S(t\partial_{t}\nabla u)\|_{p}
\end{equation}
for solutions of $Lu=0$ such that  $u\in \mE=\cup_{-1\le s\le 0}\, \mE_{s}$?\\

\textbf{Problem 2:} For which $p\in (0,\infty)$, do we have 
\begin{equation}
\label{eq:Dir}
 \|S(t\nabla u)\|_{p}\sim \|\nabla_A u|_{t=0}\|_{\dot W^{-1,p}}
\end{equation}
for solutions of $Lu=0$ such that  $u\in \mE=\cup_{-1\le s\le 0}\, \mE_{s}$? Here, $\dot W^{-1,p}$ is the usual  homogeneous Sobolev space of order -1 on $L^p$: an estimate for  partial derivatives   in $\dot W^{-1,p}$ amounts to a usual $L^p$ estimate.  Moreover, do we have an analog when $p=\infty$, in which case we look for a Carleson measure  estimate of $|t\nabla u|^2$ to the left and $BMO^{-1}$ to the right, and a weighted Carleson measure estimate of $|t\nabla u|^2$ to the left and H\"older spaces $\dot \Lambda^{\alpha-1}$ to the right?\\

Let us comment on Problem 1: here for the problem to make sense, we take $X_{p}=L^p$ if $p>1$ and $X_{p}=H^p$, the Hardy space, for $p\le 1$ and soon  discover the constraint $p>\frac{n}{n+1}$. The equivalence between non-tangential maximal estimates and $X_{p}$ norms is known in the following case: the inequality $\gtrsim$ is a very general fact proved for all weak solutions and $1<p<\infty$   in \cite{KP}  and when $\frac{n}{n+1}< p \le 1$ in \cite{HMiMo},  and their arguments carry over to our situation. The inequality $\lesssim$ was proved in \cite{HKMP2} for $1<p<2+\varepsilon$ and in  \cite{AM} for  $1-\varepsilon <p\le  1$ (and also $1<p<2$ by interpolation) assuming some interior regularity of solutions (the De Giorgi-Nash condition) of $Lu=0$. To our knowledge, \textit{a priori} comparability with  the square function $S(t\partial_{t}\nabla u)$ has not been studied so far, but this is a key feature of our analysis, roughly because the square function norms in \eqref{eq:NeuReg} define spaces that  interpolate while it is not clear for the spaces corresponding to  non-tangential maximal norms in \eqref{eq:NeuReg}. The range of $p$ in Problem 1 allows one to formulate Neumann and regularity problems with $L^p/H^p$ data, originally introduced in \cite{KP}, in a meaningful way.  By this, we mean that the conormal derivative and the tangential gradient at the boundary are in the natural spaces for those problems to have a chance to be solved with such solutions. Outside this range of $p$, there will be no solutions in our classes. 

Let us turn to Problem 2:  that such comparability holds for a range of $p$ containing $[2,\infty]$ and beyond under the De Giorgi-Nash condition on $L$ was already used in \cite{AM}. We provide here the  proof. The inequalities obtained in \cite{HKMP2} contain extra terms and are less precise. The comparability in Problem 2 allows one to formulate the Dirichlet problem with $L^p$ data and even $BMO$ or $\dot \Lambda^\alpha$ data and also a Neumann problem with $\dot W^{-1,p}$ or $BMO^{-1}$ or $\dot \Lambda^{\alpha-1}$ data. Note that we are talking about square functions without mentioning non-tangential maximal estimates on the solutions $u$ which are usually smaller  in $L^p$ sense. A beautiful result in \cite{HKMP1} is the converse inequality for solutions of real elliptic equations. 

We shall study comparability with appropriate non-tangential sharp  functions, namely study when does
$$\|\tN(u-u_{0})\|_{p}\lesssim \|S(t\nabla u)\|_{p}$$
hold. The advantage of this inequality compared to the one with the non-tangential maximal function (which will be studied as well)  is that we may allow $p=\infty$, in which case the right hand side should be replaced by the Carleson measure estimate of $|t\nabla u|^2$, and beyond using adapted versions for H\"older estimates.

The  boundary spaces obtained in Problem 1 for $L$ and in Problem 2 for $L^*$ are usually in duality. This was used in \cite{AM} to give new lights, with sharper results, on the duality principles for elliptic boundary value problems studied first in \cite{KP} and then \cite{DK}, \cite{HKMP2},  and to apply this to extrapolation. 

Our main results are the following (here in dimension $1+n\ge 2$).

\begin{thm} \label{thm:main1}
The range of $p$ in Problem 1 for solutions $u\in \mE$ of $Lu=0$  is an interval $I_{L}$   contained in $(\frac{n}{n+1},\infty)$ and containing $(\frac{2n}{n+2}-\varepsilon, 2+\varepsilon')$  for some $\varepsilon,\varepsilon'>0$. 
Moreover, if $n=1$ then $I_{L}=(\frac{1}{2},\infty)$, if $L$ has constant coefficients then $I_{L}=(\frac{n}{n+1},\infty)$ and if $n\ge 2$ and $L_{\ta}^*$ has the De Giorgi condition then $I_{L}=(1-\varepsilon,2+\varepsilon')$ where $\varepsilon$ is related to the regularity exponent in the De Giorgi condition. 
\end{thm}

Here, $L_{\ta}$ is the tangential part of  operator in $L$,  obtained by deleting in $L$ any term with a $\partial_{t}=\partial_{0}$ derivative in it. As $L$ has $t$-independent coefficients,  $L_{\ta}$ is seen as an operator on $\R^n$ and the De Giorgi condition for $L_{\ta}^*$  is about  the regularity of weak solutions of $L_{\ta}^*u=0$ in $\R^n$.  For example, this holds when $L_{\ta}$ is a scalar real operator, but also when  $1+n=2$ (this is due to Morrey). 
In that case,   the other coefficients of $L$ are arbitrary. 

\begin{thm}\label{thm:main2}
The range of exponents in Problem 2 for solutions $u\in \mE$ of $L^*u=0$ is ``dual'' to the one in Theorem \ref{thm:main1}. That is,  for $p\in I_{L}$, we obtain \eqref{eq:Dir} for $p'$ if $p>1$, the modification for $\BMO$ if $p=1$ and the modification for $\dot \Lambda^\alpha$ with $\alpha=n(\frac{1}{p}-1)$ if $p<1$.  
\end{thm}

\

\emph{Although we can not define the objects in the context of this introduction,  the main thesis of this work is as follows.  The  exponents  $p$ in the  first theorem  are the exponents for which the   Hardy space  $\IH^p_{DB}$  for the  first order operator $DB$ associated to $L$ (as discovered in   \cite{AAMc}) is identified to $\IH^p_{D}$.  The semigroup $S(t)$ mentioned above coincides with $e^{-t|DB|}$  seen as some kind of Poisson semigroup or Cauchy extension depending on the point of view. Hence, a large part of this work is devoted to say when $\IH^p_{DB}$ and $\IH^p_{D}$ are the same.}

\

A word on the \textit{a priori} class $\mE$  is in order: in fact, we want to work with  a class for which the semigroup representation for the conormal gradient \eqref{eq:sgrep} is valid and this is the only reason for restricting to this class of solutions at this time. To make a parallel (and this case corresponds to the $L=-\Delta$ here), this is like proving such estimates for an harmonic function assuming it is the Poisson integral of an $L^2$ function: such estimates are in the fundamental work of  Fefferman-Stein \cite{FS}.  Removing this \textit{a priori} information  uses specific arguments on harmonic functions (also found in \cite{FS}). Removing that $u\in \mE$ \textit{a priori} will also require specific arguments. This will be the purpose of a forthcoming work by the first author with M. Mourgoglou \cite{AM2}. It will be proved semigroup representation:  every solution of $L$ with  $\|\tN (\nabla u)\|_{p}<\infty$ in the range of $p$ for Theorem \ref{thm:main1}  has the semigroup representation \eqref{eq:sgrep} in an appropriate functional setting; and every solution of $L^*$ with $\|S(t\nabla u)\|_{p}$ or even weighted Carleson control in the ``dual'' range of Theorem \ref{thm:main2} and a weak control at infinity  has the semigroup representation in an appropriate functional setting.

\

We remark that the results obtained here  impact on the boundary layer potentials.  A.~Ros\'en \cite{R1} proposed an abstract definition of boundary layer potentials $\mD_{t}$ and $\mS_{t}$ which turned out to coincide with the ones constructed in \cite{AAAHK} for real equations of their perturbations via the fundamental solutions. These abstract definitions use the first order semi-group $S(t)$ mentioned above, instantly proving the $L^2$ boundedness of $\mD_{t}$ and $\nabla \mS_{t}$, which was a question raised by S. Hofmann \cite{H}. Thus in the interval of $p$ and its dual arising in the two theorems above, we obtain boundedness,  jump relations, non-tangential maximal estimates and square functions estimates. In particular, we obtain strong limits as $t\to 0$, which is new for $p\ne 2$, the case $p=2$ following from a combination of  \cite{AA1} and \cite{R1}.
 It goes without saying that these results are obtained without any kernel information nor fundamental solution: this is far beyond Calder\'on-Zygmund theory and subsumes the results in \cite{HMiMo}.  
\

In the context of Theorem \ref{thm:main2}, we also prove $\|\tN(u-u_{0})\|_{p'}\lesssim \|S(t\nabla u)\|_{p'}$ in the same range  and with modification for $p=1$ and below.  Our non-tangential sharp functions above  can be seen as  a part of non-tangential sharp functions adapted to the first order operators $BD$ 
for which we have the equivalence $\|\tNs(\phi(tBD)h)\|_{p'}\sim \|S(t\nabla u)\|_{p'}$ for an appropriate $h$, where $\tNs(\phi(tBD)h)= \tN(\phi(tBD)h-h)$.  Modified sharp functions, where averages are replaced by the action of more general operators, were introduced by Martell \cite{Mar} and then used by \cite{DY} in developing their $BMO$ theory associated with operators. Some versions were also used by \cite{HM} and \cite{HMMc} in the context of second order operators under divergence form on $\R^n$.  All these versions used $\phi$ such that $\phi(tBD)$ have enough decay in some pointwise or averaged sense. Here we have to consider the Poisson type semigroup  $e^{-t|BD|}$ to get back to solutions of $L$. The difficulty lies in the fact that these operators have small decay  and we overcome this using the depth of Hardy space theory because these operators are bounded there while they may not be  bounded on $L^p$.
 
\

Let us turn to   boundary value problems  for solutions of $Lu=0$ or $L^*u=0$ and formulate four such problems:
\begin{enumerate}
  \item $(D)_{Y}^{L^*}=(R)_{Y^{-1}}^{L^*}$:  $L^*u=0$,  $u|_{t=0}\in Y$, $t\nabla u\in \mT$. 
  \item $(R)_{X}^L$:  $Lu=0$, $\nabla_{x}u|_{t=0}\in X$, $\tN(\nabla u)\in \mN$.
  \item $(N)_{Y^{-1}}^{L^*}$:  $L^*u=0$,  $\partial_{\nu_{A^*}}u|_{t=0}\in \dot Y^{-1}$, $t\nabla u\in \mT$.
   \item $(N)_{X}^L$: $Lu=0$, $\partial_{\nu_{A}}u|_{t=0}\in X$, $\tN(\nabla u)\in \mN$.
\end{enumerate}

Here $X$ is a space $X_{p}$ with $p\in I_{L}$,  $Y$ is the dual space of such an  $X$ (we are ignoring whether functions are scalar or vector-valued; context is imposing it)  and $\dot Y^{-1}= \divv_{x}(Y^n)$ with the quotient topology. Then $\mN=L^p$  for $p\in I_{L} $ and $\mT$ is a tent space $T^{p'}_{2}$ if $p\ge 1 $ and a weighted Carleson measure space $T^\infty_{2,n(\frac{1}{p}-1)}$
if $p<1$.  In each case, we want to solve, possibly uniquely, with control from the data. For example, for $(D)_{Y}^{L^*}$ we want $\|t\nabla u\|_{\mT} \lesssim \|u|_{t=0}\|_{Y}$, etc. The behavior at the boundary is continuity (strong or weak-$*$)  at $t=0$; non-tangential  convergence can occur in some cases but is not part of the convergence at the boundary. 

As in \cite{KP, AM}, we say that a boundary value problem is solvable for the energy class if  the energy solution corresponding to the data (assumed in the proper trace space as well) satisfies the  required control by the data. For energy solutions, we have semi-group representation or, equivalently,  boundary layer representation.  By solvability, we mean existence of a solution for \emph{any} boundary datum, with control. Precise definitions will be recalled in Section \ref{sec:perturbation} where we describe a method to construct solutions and show the following extrapolation theorem. 

\begin{thm}\label{thm:extra}
Consider any of the four boundary value problems with a given  space of boundary data in the list above. If it is solvable for the energy class then it is solvable  in nearby  spaces of  boundary data. 
\end{thm}

For example, if $X=X_{p}$, one can take $X_{q}$ for $q$ in a neighborhood of $p$. For Neumann and regularity problems, this seems to be new for $p\le 1$ in this generality. See \cite{KalMit} for the case of the Laplacian on Lipschitz domains when $p=1$. Also for $p=1$ and duality,  we get extrapolation for $\BMO$ solvability of the Dirichlet problem.  We note that we only get solvability  in the conclusion. In  the case of real equations  as in \cite{DKP} where such an extrapolation of proved, harmonic measure techniques naturally lead to solvability for the energy class after perturbation. 

We shall also prove a stability result for each boundary value problem with respect to 
perturbations in $L^\infty$ with $t$-independent coefficients of the operator $L$. Such results when $ p\le 1$ are known assuming invertibility of the single layer potential  with De Giorgi-Nash conditions in \cite{HKMP2}, which is not the case here.

 Before we end this introduction, let us mention that most of the work to prove Theorems \ref{thm:main1} and \ref{thm:main2} has not much to do with the elliptic system given by $L$ and their solutions (except under  De Giorgi-Nash conditions). In fact, this is mainly a consequence of inequalities for Hardy spaces associated to first order systems $DB$ or $BD$ on the boundary and the operators $D$ can be much more general than the one arising from the boundary value problems. These type of operators  were introduced in the topic by McIntosh and led to one proof of the $L^2$ boundeness of the Cauchy integral from the solution of the Kato square root problem in one dimension although the original article \cite{CMcM} does not present it this way (see also \cite{KM}, \cite{AMcN1}). An extended higher dimensional setting was introduced in  \cite{AKMc}, and further studied in \cite{HMP1, HMP2, HMc, AS},  where $D$ is a differential first order operator with constant coefficients having some coercivity and $B$ is the operator of pointwise multiplication by an accretive matrix function. But the relation between elliptic  systems \eqref{eq:divform} and boundary operators of the form $DB$  was only established recently in \cite{AAMc}, paving the way  to the representations in \cite{AA1} mentioned above.  The Hardy space theory we need is the one associated to operators with Gaffney-Davies type estimates developed in \cite{AMcR, HM} and followers.  We just mention that our operators are non injective, hence it makes the theory a little more delicate. 
 
  For the first part of the article, we shall review the needed material. Then we turn to the proof of estimates  which will imply Theorems \ref{thm:main1} and \ref{thm:main2} when specializing to solutions of $Lu=0$. A large part of the end of the article is to study the case of operators with De Giorgi-Nash conditions.  The application of our theory to perturbations for solvability of the boundary value problems is  given in the last section.
  
  We shall not attempt to treat intermediate situations for the boundary value problems, that is assuming some fractional order of regularity for the data. This has been recently done in  \cite{BM}
  with data in Besov spaces for elliptic equations $Lu=0$ assuming De Giorgi type conditions for $ L$ and $L^*$. 

\

Acknowledgements: Part of this work appears in  the second author's Doctoral dissertation at Orsay and was presented at the IWOTA conference in Sydney in July 2011. While preparing this manuscript,  D. Frey, A. McIntosh and P. Portal informed us of a different approach to Hardy space coincidence with tent space estimates (Section \ref{sec:Hardy}) for the operator $DB$, obtaining the same range as in Theorem \ref{thm:hpdb}.   We also thank S. Hofmann for letting us know preliminary versions of \cite{HMiMo}.  The  authors were partially supported by the ANR project ``Harmonic analysis at its boundaries'' ANR-12-BS01-0013-01 and they thank the ICMAT for hospitality during the  preparation of this article.

\section{Setup}\label{sec:setup}

\subsection{Boundary function spaces} 
In this memoir, we work on the upper-half space $\reu$ and its boundary identified to $\R^n$. We consider a variety of function spaces defined on the boundary $\R^n$ and  valued in $\C^N$ for some integer $N$. Function or distribution spaces $X(\R^n;\C^N)$ will often be written $X$ if this is not confusing. 
For example, $L^q:=L^q(\R^n; \C^N)$ is the standard Lebesgue space. For $0<q\le 1$,   $H^q$ denotes the Hardy space  in its real version.  It will be sometimes convenient to set $H^q=L^q$ even when $q>1$. 

The dual of $H^q$ for a duality extending the $L^2$ sesquilinear pairing when $q>1$ is thus $H^{q'}$ and is the space $\dot \Lambda^{n(\frac 1 q -1)}$ when $q\le 1$. 
Here, $\dot \Lambda^0$ denotes $\BMO$ for convenience; for $0<s<1$,   $\dot \Lambda^s$ is the H\"older space of those continuous functions with    $|f(x)-f(y)|\le C |x-y|^s$ (equipped with a semi-norm); for $s\ge 1$, we say $f\in \dot \Lambda^s$ if the distributional partial derivatives of $f$ belong to $\dot \Lambda^{s-1}$.

For $q>1$, $W^{1,q}$ is the standard Sobolev space of order 1 on $L^q$ and $\dot W^{1,q}$ denotes its homogeneous version: the space of Schwartz distributions with $\|\nabla f\|_{q}<\infty$ or, equivalently, the closure of $ W^{1,q}$ for $\|\nabla f\|_{q}$. It becomes a Banach space when moding out the constants. For $\frac{n}{n+1}<q\le 1$, we also set  $\dot H^{1,q}$, the space of Schwartz distributions with $\nabla f\in H^q$ (componentwise). 
Again, we sometimes use the notation $\dot H^{1,q}=\dot W^{1,q}$ also when $q>1$ for convenience. 

The dual of $\dot W^{1,q}$ is $\dot W^{-1,q'}:=\divv (L^{q'})^n$  with quotient topology. 
The dual of $\dot H^{1,q}$, $q\le 1$, is $\dot \Lambda^{s-1}:= \divv (\dot \Lambda^{s})^n$ when $s=n(\frac 1 q -1)\in [0,1)$, equipped with the quotient topology.

We shall also use the homogeneous Sobolev spaces $\dot \mH^s$ for $s\in \R$. We mention that for $s\ge 0$, they can be realized within $L^2_{loc}$ and equipped with  a semi-norm. For $s<0$, the homogeneous Sobolev spaces embed in the Schwartz distributions. 

\subsection{Bisectorial operators}
The space of  continuous linear operators between normed  vector spaces  $E,F$ is denoted by $\mL(E,F)$ or $\mL(E)$ if $E=F$. For an unbounded linear  operator $\mA$, its domain is denoted by $\dom(\mA)$, its null space $\nul(\mA)$ and its range $\ran(\mA)$.  The spectrum is denoted by $\sigma(\mA)$. 

An unbounded linear operator $\mathcal{A}$ on a Banach space $\mX$ is called \emph{bisectorial} of angle $\omega\in[0,\pi/2)$ if  it is closed, its spectrum is contained in the closure of $S_{\omega}:=S_{\omega+}\cup S_{\omega-}$,  where $S_{\omega+}:=\{z\in\C;\abs{\arg z}<\omega\}$ and $S_{\omega-}:=-S_{\omega+}$, and one has the resolvent estimate
\begin{equation}\label{eq:bisectorial}
  \Norm{(I+\lambda\mathcal{A})^{-1}}{\mL(\mX)}\leq C_{\mu}\qquad\forall\ \lambda\notin S_{\mu},\quad\forall\ \mu>\omega.
\end{equation}
Assuming  $\mX$ is  reflexive,  this implies that the domain is dense and also the  fact that the null space and the closure of the range split. More precisely, we say that the operator $\mA$  \emph{kernel/range decomposes} if $\mX=\nul(\mA) \oplus \closran \mA$ ($\oplus$ means that the sum is topological).   Bisectoriality in a reflexive space is  stable under taking adjoints. 
  
For any bisectorial operator in a reflexive Banach space, one can define a calculus of bounded operators by  the Cauchy integral formula,
\begin{equation}
\begin{split}
  \psi(\mathcal{A}) &:=\frac{1}{2\pi i}\int_{\partial S_{\nu}}\psi(\lambda)(I-\frac{1}{\lambda}\mathcal{A})^{-1}\frac{d\lambda}{\lambda}, \label{eq:cauchyformula} \\
  \psi &\in \Psi(S_{ \mu}):=\{\phi\in H^{\infty}(S_{\mu}):\phi(z)= O\big(\inf({\abs{z},\abs{z^{-1}}})^{\alpha}\big),\alpha>0\}, 
\end{split}
\end{equation}
with $\mu>\nu>\omega$ and where $H^{\infty}(S_{\mu})$ is the space of bounded holomorphic functions in  $S_{\mu}$. If one can show the estimate $\|\psi(\mA)\|\lesssim C_{\mu}\|\psi\|_{\infty}$
for all $\psi\in \Psi(S_{\mu})$ and all $\mu$ with $\omega<\mu<\frac{\pi}{2}$, then this  allows to  extend the calculus on $\clos{\ran(\mA)}$ to   all $\psi\in H^{\infty}(S_{\mu})$ and  all $\mu$ with $\omega<\mu<\frac{\pi}{2}$ in a consistent way for different values of $\mu$. In that case,  $\mathcal{A}$ is said to have an \emph{$H^{\infty}$-calculus} of angle $\omega$ on $\clos{\ran(\mA)}$, and $b(\mA)$ is defined by a limiting procedure for any $b\in H^\infty(S_{\mu})$. For those $b$ which are also defined at $0$, one extends the $H^\infty$-calculus to $\mX$ by setting $b(\mA)=b(0)I$ on $\nul(\mA)$.  For later use, we shall say that a holomorphic function on $S_{\mu}$ is non-degenerate if it is non identically 0 on each connected component of $S_{\mu}$.

\subsection{The first order operator D}\label{sec:D} We assume  that $D$ is  a first order differential operator on $\R^n$ acting on Schwartz distributions  valued in $\C^N$, whose symbol satisfies the conditions (D0), (D1) and (D2) in \cite{HMc}. Later,  we shall assume that $D$ is self-adjoint on $L^2$ but for what follows in this section, this is not necessary by observing that the three conditions can be shown to be stable under taking the adjoint symbol and operator.  For completeness, we recall the three conditions here although what we will be using are the consequences below. 

First $D$ has the form
\begin{equation*}\tag{D0}\label{eq:D0}
  D=-i\sum_{j=1}^n\hat{D}_j\partial_j,\qquad\hat{D}_j\in\mL(\C^N).
\end{equation*}
 It can also be viewed as the Fourier multiplier operator with symbol $\hat{D}(\xi)=\sum_{j=1}^n\hat{D}_j\xi_j$. The symbol is required to satisfy the following properties:
\begin{equation*}\tag{D1}\label{eq:D1}
  \kappa\abs{\xi}\abs{e}\leq\abs{\hat{D}(\xi)e}\qquad\forall\ \xi\in\R^n,\quad\forall\ e\in\ran(\hat{D}(\xi)),
\end{equation*}
where $\ran(\hat{D}(\xi))$ stands for the range of $\hat{D}(\xi)$, and
\begin{equation*}\tag{D2}\label{eq:D2}
  \sigma(\hat{D}(\xi))\subseteq \clos{S_{\omega}},
\end{equation*}
where $\kappa>0$ and $\omega\in[0,\pi/2)$ are some constants. 

 For  $1<q<\infty$,  this induces the  unbounded operator $D_{q}$ on each $L^q$ with domain $\dom_q(D):= \dom_{L^q}(D)=\{u\in L^q;Du\in L^q\}$ and $D_{q}=D$ on $\dom_{q}(D)$.  We keep using the notation $D$ instead of $D_{q}$ for simplicity.   The following properties have been  shown in \cite{HMP2}, except for the last one shown in \cite{AS}.
\begin{enumerate}
  \item $D$ is a bisectorial operator with $H^\infty$-calculus in $L^q$.
\item $L^q=N_{q}(D)\oplus \clos{\ran_{q}(D)}$, the closure being in the $L^q$ topology.
  \item $N_{q}(D)$ and $\clos{\ran_{q}(D)}$, $1<q<\infty$, are complex interpolation families.
  \item $D$ has the coercivity condition
  $$
  \|\nabla u \|_{q} \lesssim \| Du\|_{q}\quad \mathrm{for\ all\ } u \in \dom_{q}(D)\cap \clos{\ran_{q}(D)} \subset W^{1,q}.$$ Here, we use the notation $\nabla u$ for $\nabla \otimes u$. 
  \item  $\dom_{q}(D)$, $1<q<\infty$, is a  complex interpolation family.
  \end{enumerate}

The results in \cite{HMP2} are obtained by  applying  the Mikhlin multiplier theorem to  the resolvent and also to the projection from $L^q$ on $\clos{\ran_{q}(D)}$ along $N_{q}(D)$ by checking  the   symbol is $C^\infty$ away from 0 and has the appropriate estimates for all its partial derivatives. This projection, which we denote by $\IP$, will play an important role  (it does not depend on $q$) and we have 
$$
\IP(L^q)=\clos{\ran_{q}(D)}.
$$
 This theorem can be shown to apply to the operators $b(D)$ of the bounded holomorphic functional calculus. Moreover, (4) is a consequence of the $L^q$ boundedness of $\nabla D^{-1} \IP$, which again follows from Mikhlin multiplier  theorem.   Even if this is not  done this way in \cite{AS}, one can show the property (5) using  the Mikhlin multiplier theorem as in \cite{HMP2}. 

By standard singular integral theory,    
all operators to which the ($C^\infty$ case of the) Mikhlin multiplier theorem applies extend boundedly  to the Hardy spaces $H^q$, $0<q\le 1$.  In particular, $\IP$ is a bounded projection on $H^q$ so $\IP(H^q)$ is a closed complemented subspace  of $H^q$.

Set  $X_{p}=L^p$ when $1<p<\infty$,  $X_{p}=H^p$ when $p\le 1$ and also $X_{\infty}=BMO$  the space of bounded mean oscillations functions. 

We mention the following consequence:  For  $0<q<\infty$,  each  $\IP(X_{q})$ contains $\IP(\mD_{0})$ as a dense subspace  where  $\mD_{0}$ is the space of $C^\infty$ functions with compact support and all vanishing moments.  Note that  a Fourier transform argument shows $\IP(\mD_{0})\subset \mS$, where $\mS$ is the Schwartz space. Similarly, the same statement holds if $\mD_{0}$ is replaced by the subspace $\mS_{0}$ of $\mS$ of those functions with  compactly supported Fourier transform away from the origin.

As said, all this applies to the adjoint of $D$ (we shall assume $D$ self-adjoint subsequently). Hence, the resolvent of $D$ is bounded on  $X_{p}^*$, the dual space to $X_{p}$ with the estimate \eqref{eq:bisectorial},  and    $\IP$ is a  bounded projection on  $X_{p}^*$.   In particular, $\IP(X_{p}^*)$ is complemented in $X_{p}^*$. Also, using that the $X_{p}$, $0<p\le\infty$, spaces   form a complex interpolation scale, the same holds for  the spaces
$\IP(X_{p})$, $0<p\le\infty$.

\subsection{The operators DB and BD}

We let $D$ as defined above and we assume from now on that $D$ is self-adjoint on $L^2$. We consider an operator $B$ of multiplication by a matrix $B(x)\in \mL(\C^N)$. We assume that as a function, $B\in L^\infty$ and note $\|B\|_{\infty}$ its norm. Thus as a multiplication operator, $B$ is bounded on all $L^q$ spaces with norm equal to $\|B\|_{\infty}$ when $1<q<\infty$. We also assume that 
$B$ is strictly accretive in $\clos{\ran_{2}(D)}$, that is for some $\kappa>0$, 
\begin{equation}
\label{eq:accretivity}
\re \langle u, Bu \rangle \ge \kappa \|u\|_{2}^2, \quad \forall\ u \in \clos{\ran_{2}(D)}. 
\end{equation}
In this case, let 
\begin{equation}
\label{eq:accretivityangle}
   \omega := \sup_{u\in \clos{\ran_{2}(D)}, u\not = 0} |\arg(\langle u, Bu \rangle)|  <\frac \pi 2
\end{equation}
denote the  {\em angle of accretivity} of $B$ on $\clos{\ran_{2}(D)}$.
Note that $B$ may not be invertible on $L^2$. Still for $X$ a subspace of $L^2$, we set
$B^{-1}X= \{u \in L^2\, ; \,  Bu \in X\}$. Note that   $B^*$ is also strictly accretive on $\clos{\ran_{2}(D)}$   with the same lower bound and angle of accretivity. 
\begin{prop}   \label{prop:typeomega} With the above assumptions, we have the following facts.
\begin{itemize}
\item[{\rm (i)}]
The operator $DB$, with domain $B^{-1}\dom_{2}(D)$,  is  bisectorial with angle $\omega$, \textit{i.e.} $\sigma(DB)\subseteq \clos{S_{\omega}}$ and there are resolvent bounds 
$\|(\lambda I - DB)^{-1}\| \lesssim 1/ \dist(\lambda, S_\mu)$ when $\lambda\notin S_\mu$, $\omega <\mu<\pi/2$.
\item[{\rm (ii)}]
The operator $DB$ has range $\ran_{2}(DB)=\ran_{2}(D)$ and null space $\nul_{2}(DB)=B^{-1}\nul_{2}(D)$ such that topologically (but  not necessarily orthogonally) one has
$$
L^2 = \clos{\ran_{2}(DB)} \oplus \nul_{2}(DB).
$$
\item[{\rm (iii)}] 
The restriction of $DB$ to $\clos{\ran_{2}(DB)}$ 
is a closed, injective operator with dense range in
$\clos{\ran_{2}(D)}$.  Moreover, the same statements  on spectrum and resolvents as in (i) hold.

\item[{\rm (iv)}]   Statements similar to (i), (ii) and (iii) hold for $BD$ with $\dom_{2}(BD)=\dom_{2}(D)$, defined as the adjoint of $DB^*$ or equivalently by $BD= B(DB)B^{-1}$ on $\clos{\ran_{2}(BD)}\cap \dom_{2}(D)$ with $\ran_{2}(BD):=B\ran_{2}(D)$, and  $BD=0$ on the null space $\nul_{2}(BD):=\nul_{2}(D)$. 
\end{itemize}
\end{prop}
 
For a proof, see \cite{ADMc}. Note that the accretivity of $B$ is only needed on $\ran_{2}(D)$.  The fact that $D$ is self-adjoint is used  in this statement. In fact, for a self-adjoint operator $D$ on a separable Hilbert space instead of $L^2$ and a bounded operator $B$ which is accretive on     $\clos{\ran_{2}(D)}$, the statement above is valid. 

We come back to the concrete $D$ and $B$ above. We isolate this result as it will play a special role throughout.

\begin{prop}\label{prop:projection}
Consider the orthogonal projection $\IP$ from $L^2$ onto $\clos{\ran_{2}(D)}$. Then 
$\IP$ is an isomorphism between $\clos{\ran_{2}(BD)}$ and $\clos{\ran_{2}(D)}$.
\end{prop} 

\begin{proof}
Using $\nul_{2}(BD)=\nul_{2}(D)$, we have the splittings
$$
L^2 = \clos{\ran_{2}(BD)} \oplus \nul_{2}(D)= \clos{\ran_{2}(D)} \oplus \nul_{2}(D).
$$
It is then a classical fact from operator theory that $\IP: \clos{\ran_{2}(BD)} \to \clos{\ran_{2}(D)}$ is invertible with inverse being $\IP_{BD}:\clos{\ran_{2}(D)} \to \clos{\ran_{2}(BD)}$, where $\IP_{BD}$ is the projection onto $\clos{\ran_{2}(BD)}$ along $\nul(D)$ associated to the first splitting. Indeed, if $h \in \clos{\ran_{2}(D)}$, then $h-\IP_{BD} h\in \nul_{2}(D)$, thus $\IP (h-\IP_{BD} h)=0$. It follows that $h=\IP h= (\IP \circ\IP_{BD})h$. Similarly, we obtain  $h= (\IP_{BD}\circ \IP) h$ for $h\in \clos{\ran_{2}(BD)}$. 
\end{proof}

We also state the following decay estimates. See \cite{elAAM}.

\begin{lem} [$L^2$ off-diagonal decay] \label{lem:odd} Let $T=BD$ or $DB$.
 For every integer $N$ there exists $C_N>0$
such that
\begin{equation} \label{odn}
\|1_{E}\,(I+itT)^{-1} u\|_{2} \le C_N \brac{\dist (E,F)/|t|}^{-N}\|u\|_{2}
\end{equation}
for all $t\ne 0$, 
whenever $E,F \subseteq \R^n$ are closed sets,  $u \in L^2$
is such that $\supp u\subseteq F$.  We have set $\brac x:=1+|x|$ and
$\dist(E,F) :=\inf\{|x-y|\, ; \, x\in E,y\in F\}$.
\end{lem}

\begin{rem}\label{rem:extension} Any operator satisfying such estimates with $N>\frac{n}{2}$ has an extension from $L^\infty$ into $L^2_{loc}$.
\end{rem}

\section{Holomorphic functional calculus} 

\subsection{$L^2$ results}
We begin with recalling the following result  due to \cite{AKMc}. A direct proof is in \cite{elAAM}.

\begin{prop}\label{prop:equi}  If $T=DB$ or $T=BD$, then one has the equivalence
\begin{equation}\label{eq:sfT}
  \int_0^\infty\|tT(I+t^2T^2)^{-1} u \|_{2}^2 \, \frac{dt}t \sim \|u\|_{2}^2, \qquad  \text{for all }\ u\in \clos{\ran_{2}(T)}.
\end{equation}
\end{prop}

Note that if $u\in \nul_{2}(T)$ then $tT(I+t^2T^2)^{-1} u =0$. Thus by the kernel/range decomposition, we have the inequality $\lesssim$ for all $u\in L^2$.

The next result summarizes the needed consequences of this quadratic estimate. This statement, contrarily to the previous one, is abstract and applies to $T=BD$ or $DB$ on $L^2$.   

\begin{prop}\label{prop:SFimpliesFC} Let  $T$ be an $\omega$-bisectorial operator on  a separable Hilbert space $\mH$ with $0\le \omega<\pi/2$.  Assume that the quadratic estimate
\begin{equation} 
   \int_0^\infty\|tT(I+t^2T^2)^{-1} u \|^2 \, \frac{dt}t \sim \|u\|^2 \   \text{holds for all }\ u\in \clos{\ran(T)}.
\end{equation}Then, the following statements hold.
\begin{itemize}
  \item $T$ has an $H^\infty$-calculus on $\clos{\ran(T)}$,  which can be extended to  $\mH$ by setting $b(T)=b(0)I$ on $\nul(T)$ whenever $b$ is also defined at 0. 
\item For any $\omega<\mu<\pi/2$ and  any  non-degenerate $\psi\in \Psi(S_{\mu})$, the comparison  \begin{equation} \label{eq:psiT}
  \int_0^\infty\|\psi(tT) u \|^2 \, \frac{dt}t \sim \|u\|^2 \  \text{holds for all }\ u\in \clos{\ran(T)}.
\end{equation}  
 \item  $\clos{\ran(T)}$ splits topologically into two spectral subspaces 
\begin{equation}     \label{eq:hardysplit}\clos{\ran(T)}=\mH^+_{T}\oplus \mH^-_{T}
\end{equation}
  with $\mH^\pm_{T}=\chi^\pm(T)(\clos{\ran(T)})$ and $\chi^\pm(T)$ are projections with $\chi^\pm(z)=1$ if $\pm \re z>0$ and $\chi^\pm(z)=0$ if $\pm \re z<0$.  
   \item The operator $\sgn(T)= \chi^+(T)- \chi^-(T)$ is a bounded involution on $\clos{\ran(T)}$.
\item The operator $|T|=sgn(T)T = \sqrt {T^2}$ with $\dom(|T|)=\dom(T)$ is an $\omega$-sectorial operator  with $H^\infty$-calculus on $\mH$ and 
$-|T|$ is the infinitesimal generator of  a bounded analytic semigroup of operators $(e^{-z|T|})_{z\in S_{\frac{\pi}{2}-\omega+}}$ on $\mH$. 
\item For $h\in \dom(T)$, $h\in \mH^\pm_{T}$ if and only if $|T|h=\pm Th.$ As a consequence 
$e^{\mp zT}$ are well-defined operators on $\mH^\pm_{T}$ respectively,  and $e^{-zT}\chi^+(T)$ and $e^{+zT}\chi^-(T)$ are well-defined operators on $\mH$   for $z\in S_{\frac{\pi}{2}-\omega+}$. 
\end{itemize}

Finally, all these properties hold for the adjoint $T^*$ of $T$.  \end{prop}

This result is for later use.

\begin{prop}\label{prop:bdproj}If $b\in H^\infty(S_{\mu})$ and $b$ is defined at 0, then, for all $h\in L^2$, 
$\IP b(BD)\IP h= \IP b(BD)h$. If $\psi\in \Psi(S_{\mu})$, then for all $h\in L^2$,  $\psi(BD)\IP h= \psi(BD)h$.
\end{prop}

\begin{proof} Remark that $h-\IP h\in \nul_{2}(D)=\nul_{2}(BD)$. Thus, $b(BD)(h-\IP h)=b(0) (h-\IP h)$. Hence $\IP b(BD)(h-\IP h)= 0$. 
If $b=\psi$ then $\psi(BD)$ annihilates the null space of $BD$, hence $\psi(BD)(h-\IP h)=0$
(This is consistent with the fact that one can set $\psi(0)=0$ by continuity).

\end{proof}

\subsection{$L^p$ results}\label{sec:lpresults}

There has been  a series of works \cite{HMP1,HMP2, Ajiev, HMc, AS}  concerning extension to $L^p$ of the $L^2$ theory. We summarize here the results described in \cite{AS}.  

Let $D$ and $B$ be as before and $1<q<\infty$. Then we have a meaning of $D$ and $B$ as operators on $L^q$, thus of $BD$ and $DB$ as unbounded operators with natural domains $\dom_{q}(D)$ and $B^{-1}\dom_{q}(D)$ respectively.  Introduce the set
of coercivity of $B$ (it also depends on $D$) as $$
\mI(BD)=\{q\in (1,\infty)\, ;\, \|Bu\|_{q} \gtrsim \|u\|_{q} \ \mathrm{for\ all\ } u\in \ran_{q}(D)\}.
$$
By density, we may replace $\ran_{q}(D)$ by its closure. The following observation will be frequently used.

\begin{lem}\label{lem:biso} If $q\in \mI(BD)$ then $B|_{\clos{\ran_{q}(D)}}: \clos{\ran_{q}(D)} \to
\clos{\ran_{q}(BD)}$ is an isomorphism and ${\ran_{q}(BD)}=  B{\ran_{q}(D)}$.  Moreover, 
 $\nul_{q}(BD)=\nul_{q}(D)$. 
\end{lem}

\begin{proof} See Proposition 2.1, (2) and (3),  in \cite{AS}.
\end{proof}

\begin{rem}\label{rem:setofcoerc}
It is shown in \cite{HMc, AS} that  the set of coercivity of $B$ is  open. As it  contains $q=2$, let $\mI_{2}$ be the connected component of $\mI(BD)\cap \mI(B^*D)$ that contains 2. Remark that if $B(x)$ is invertible in $L^\infty$ then $B$ is invertible in $\mL(L^q)$ for all $1<q<\infty$ and $\mI_{2}=(1,\infty)$. Otherwise, we do not even know if the set of coercivity of $B$ is connected. 
\end{rem}  

For an interval $I\subset (1,\infty)$, its dual interval is $I'=\{p'; p\in I\}$ where $p'$ is the conjugate exponent to $p$. The following result is taken from \cite{AS} with a cosmetic modification in the statement.

\begin{thm}\label{thm:Lpfc}
There exists  an open interval $I(BD)=(p_{-}(BD),p_{+}(BD))$, maximal in $\mI_{2}$, containing $2$,  with the following dichotomy:          bisectoriality of $BD$ with angle $\omega$, $H^\infty$-calculus with angle $\omega$ in $L^p$,   and    kernel/range decomposition hold for $BD$  in $L^p$ if $p\in I(BD)$ and all fail if $p=p_{\pm}(BD)$ and $p\in \mI_{2}$.  The same property holds for $DB$ with  $I(DB)=I(BD)$.  The same property holds for $B^*D=(DB)^*$ and $DB^*=(BD)^*$ in the dual interval $I(DB^*)=I(B^*D)=(I(BD))'$. Thus we have the relations, 
\begin{equation}
\label{eq:limits}
p_{\pm}(BD)=p_{\pm}(DB), \qquad p_{\pm}(BD)= p_{\mp}(B^*D)'.
\end{equation}
\end{thm} 

If $p_{\pm}(BD)$ is  an endpoint of $\mI_{2}$, then we do not know what happens for $p=p_{\pm}(BD)$ from this theory.  

We remark that the calculi in $L^p$ are consistent for all $p\in I(BD)$. For example, if $T_{p}=BD$ with domain $\dom_{p}(D)$ then $(I+iT_{p})^{-1}u=  (I+iT_{q})^{-1}u$ whenever $u\in L^p\cap L^q$ and $p,q\in I(BD)$. Thus, we do not distinguish them from now on. 

\begin{cor}\label{cor:idb} If $q\in I(BD)=I(DB)$, then ${\ran_{q}(DB)}=  {\ran_{q}(D)}$.    
\end{cor}

The inclusion ${\ran_{q}(DB)}\subset  {\ran_{q}(D)}$ is always true. The converse is not clear when $q\notin I(BD)$, so we shall use this equality only for $q$ in this range. 

\begin{proof} The above theorem and  Corollary 2.3 in \cite{AS}  give us the assumptions of Proposition 2.1, (4) in \cite{AS}, of which ${\ran_{q}(DB)}=  {\ran_{q}(D)}$ is a consequence. 
\end{proof}

\begin{prop}\label{prop:projectionlp}
Consider the orthogonal projection $\IP$ from $L^2$ onto $\clos{\ran_{2}(D)}$. For $p\in I(BD)$, 
$\IP$ extends to an isomorphism between $\clos{\ran_{p}(BD)}$ and $\clos{\ran_{p}(D)}$ with $\|h\|_{p}\sim \|\IP h\|_{p}$ for all $ h\in \clos{\ran_{p}(BD)}$.
\end{prop}

\begin{proof}
Using $\nul_{p}(BD)=\nul_{p}(D)$ from Lemma \ref{lem:biso}, and  the kernel/range decomposition for $D$ and for  $BD$ if $p\in I(BD)$, 
$$
L^p = \clos{\ran_{p}(BD)} \oplus \nul_{p}(D)= \clos{\ran_{p}(D)} \oplus \nul_{p}(D).
$$
The projection onto $\clos{\ran_{p}(D)}$ along $\nul_{p}(D)$ is the extension of $\IP$ to $L^p$. 
The projection from $L^p$ onto $\clos{\ran_{p}(BD)}$ along $\nul_{p}(D)$ is the extension $\IP_{BD}$  defined on $L^2$ in the proof of Proposition \ref{prop:projection}.   Using the same notation for the extensions, it follows  that
 $\IP: \clos{\ran_{p}(BD)} \to \clos{\ran_{p}(D)}$ and $\IP_{BD}:\clos{\ran_{p}(D)} \to \clos{\ran_{p}(BD)}$ are inverses of each other.  \end{proof}

\begin{cor} For all $p\in I(BD)$, the conclusions of Proposition \ref{prop:SFimpliesFC} hold for  $T=BD$ and $DB$ on $L^p$  in place of $\mH$  with the exception of \eqref{eq:psiT} which reads
\begin{equation} \label{eq:psiTLp}
 \bigg\|\bigg( \int_0^\infty |\psi(tT) u|^2  \, \frac{dt}t  \bigg)^{1/2}\bigg\|_{p}\sim \|u\|_{p} \  \text{holds for all }\ u\in \clos{\ran_{p}(T)}
\end{equation}  
for any $\omega<\mu<\pi/2$ and   any non-degenerate   $\psi\in \Psi(S_{\mu})$.   Furthermore, one has $\lesssim$ in general for all $u\in L^p$. 
\end{cor}

The last part of the corollary follows from extension of an abstract theorem of Le Merdy \cite[Corollary 2.3]{LeM}, saying that for an injective sectorial operator $T$ on $L^p$, the $H^\infty$-calculus on  $L^p$  is equivalent to the square function estimate \eqref{eq:psiTLp}.  This uses the notion of $R$-sectoriality which we have not defined here but follows from the $H^\infty$-calculus. The extension to injective bisectorial operators is straightforward with the notion of  $R$-bisectoriality. If $T$ is not injective but one has the kernel/range decomposition, then its restriction to $\clos{\ran_{p}(T)}$ is injective, and  the proof of Le Merdy's theorem extends easily also in this case. In our situation, 
for $p\in I(BD)$, $T=BD$ or $DB$ may  not be  injective on $L^p$ but its restriction to $\clos{\ran_{p}(T)}$ is injective as one has the kernel/range decomposition.  One can apply Le Merdy's extended theorem to $T$ on $\clos{\ran_{p}(T)}$ and obtain $H^\infty$-calculus   on $\clos{\ran_{p}(T)}$ (which, for this particular $T$, is  equivalent to the $R$-bisectoriality on $L^p$, see \cite{HMc, AS}), and then extend it to all of $L^p$ as described before.

Note also that by interpolation between Lemma \ref{lem:odd} and the boundedness on $L^p$ of the resolvent for $p\in I(BD)$, one has 

 \begin{lem} [$L^p$ off-diagonal decay] \label{lem:oddLp} Let $T=BD$ or $DB$ and $p\in I(BD)$. 
 For every integer $N$ there exists $C_N>0$
such that
\begin{equation} \label{odnp}
\|1_{E}\,(I+itT)^{-1}1_{F} u\|_{p} \le C_N \brac{\dist (E,F)/|t|}^{-N}\|u\|_{p}
\end{equation}
for all $t\ne 0$, 
whenever $E,F \subseteq \R^n$ are closed sets,  $u \in L^p$
is such that $\supp u\subseteq F$.  
\end{lem}

Actually, it is observed in  \cite{HMc} that the proof for $p=2$ (Lemma \ref{lem:odd}) goes through, which gives another argument. 

\subsection{The one dimensional case}

\begin{prop}\label{prop:n=1} Assume $D$ and $B$ are as above and  $n=1$.  Assume that $\hat D(\xi)$ is invertible for all $\xi\ne 0$. Then $p_{-}(DB)=1$ and $p_{+}(DB)=\infty$. In particular, $DB$ and $BD$ have bounded holomorphic functional calculi on  $L^p$ spaces for $1<p<\infty$. 
\end{prop}

\begin{proof}  We fix $1<p<\infty$. By Theorem \ref{thm:Lpfc}, it suffices to show that the kernel/range decomposition  holds on $L^p$ for $BD$.  

First,  as $\hat D(\xi)$ is invertible for all $\xi\ne 0$,  \eqref{eq:D0} implies that  for all $u\in L^p(\R^n;\C^N)$, $Du=-i\hat D_{1} u'$ with $\hat D_{1}$ being an invertible matrix on $\C^N$. Thus, we have that $N_{p}(D)=0$  and $\clos{\ran_{p}(D)}=L^p$, the closure being taken in $L^p$. 
As a consequence, if $B$ is accretive on $\clos{\ran_{2}(D)}=L^2$, it is  invertible in $L^\infty$ by Lebesgue differentiation theorem,  and one has $\mI_{2}=(1,\infty)$.  By Lemma \ref{lem:biso},  we have, since $p\in \mI_{2}$, $\nul_{p}(BD)= \nul_{p}(D)=\{0\}$ and $\clos{\ran_{p}(BD)}=\clos{B\ran_{p}(D)}= B\clos{\ran_{p}(D)}=L^p$.  Thus the kernel/range decomposition holds trivially. 
\end{proof}

\begin{rem}
If one does not assume $\hat D(\xi)$  invertible for all $\xi\ne 0$,  it is not clear whether one has the kernel/range decomposition, even assuming  $B$ invertible  on $L^\infty$. Assume $B$ invertible  on $L^\infty$.  By the results in \cite{HMc} (see \cite{AS}, Lemma 5.2, for the explicit statement),  $BD$ is ($R$-)bisectorial on $\clos{\ran_{p}(BD)}$ when $p\in (1,\infty)\cap (\frac{2}{3}, \infty)=(1,\infty)$. It is trivially bisectorial on $\nul_{p}(BD)$. The only thing missing might be the kernel/range decomposition. 
\end{rem}

\subsection{Constant coefficients}\label{sec:constant}

We come back to arbitrary dimensions. A simple example is when $B$ is a constant  and strictly accretive matrix on $\clos{\ran_{2}(D)}$ with $D$ being still self-adjoint. Then it follows from  \cite[Proposition A.8]{HMc} that the interval of coercivity is all $(1,\infty)$. 

Now $BD$ is  another first order differential operator which satisfies   (D0), (D1) and (D2) of Section  \ref{sec:D} with $\omega$ being the angle of accretivity of $B$. Thus the conclusion is that $BD$ is a bisectorial operator with $H^\infty$-calculus in $L^q$ for all $q\in (1,\infty)$. 

Therefore the theory above tells that $p_{-}(BD)=1$ and $p_{+}(BD)=\infty$.

\subsection{$L^p-L^q$ estimates}\label{sec:lplq}

We summarize here estimates that we will use later. Proofs can be found in \cite{Sta1}. They concern only  the exponents in the interval $I(BD)=I(DB)=(p_{-}, p_{+})$. 

First,  we introduce   subclasses of $H^\infty(S_{\mu})$. For  $\sigma,\tau\ge 0$, let
$$\Psi_{\sigma}^\tau(S_{\mu})= \{\psi\in H^{\infty}(S_{\mu}):\psi(z)=O\big(\inf({\abs{z}^\sigma,\abs{z}^{-\tau}})\big)\},
$$
with convention that $\abs{z}^0=1$. For $\sigma,\tau>0$, $\Psi_{\sigma}^\tau(S_{\mu})\subset \Psi(S_{\mu})$. For $\sigma=0$, we have no vanishing at 0, for $\tau=0$, no decay at $\infty$, and 
$\Psi_{0}^0(S_{\mu})= H^{\infty}(S_{\mu})$. 

\begin{prop} Let $T=BD$ or $DB$. Let $p,q\in I(T)$ with $p\le q$. Let $\psi \in \Psi_{\sigma}^\tau(S_{\mu})$ with $\sigma>0, \tau> \frac{n}{p}-\frac{n}{q}$ and $g\in H^\infty(S_{\mu})$. 
Then for all $t>0$, closed sets $E,F \subset \R^n$ and $u\in L^p$ with support in $F$:
\begin{equation} \label{eq:odnpsipq}
\|1_{E}g(T)\psi(tT)1_{F} u\|_{q} \lesssim \|g\|_{\infty} t^{\frac{n}{q}-\frac{n}{p}}\brac{\dist (E,F)/t}^{-\sigma c}\|u\|_{p}.
\end{equation}
If, furthermore, $g(z)= \varphi(rz)$ with $|\varphi(z)|\le \inf(|z|^M, 1)$ for some $M>0$, then  for all $t\ge r >0$,  closed sets $E,F \subset \R^n$ and $u\in L^p$ with support in $F$
\begin{equation} \label{eq:odnpsiphipq}
\|1_{E}\varphi(rT)\psi(tT)1_{F} u\|_{q} \lesssim t^{\frac{n}{q}-\frac{n}{p}}\brac{\dist (E,F)/r}^{-M c}\|u\|_{p}.
\end{equation}
Here, $c$ is any positive number smaller than $1- (\frac{1}{p}-\frac{1}{q})(\frac{1}{p_{-}}-\frac{1}{p_{+}})^{-1}$ and can be taken equal to 1 when $p=q$. The implicit constants are independent of $t, E, F, r$ and $u$. 
\end{prop}

Besides the precise values, it is important to notice that  the exponent expressing the decay grows linearly with the order of decay of $\psi$ at 0 in the first estimate and with the order of decay of $\varphi$ at 0 in the second one.  Notice that the first estimate contains in particular global $L^p-L^q$ estimates 
\begin{equation} \label{eq:odnpsipqglobal}
\|\psi(tT) u\|_{q} \lesssim t^{\frac{n}{q}-\frac{n}{p}}\|u\|_{p}
\end{equation}
for all $\psi$ as above. Such an estimate is not true for the resolvent if $p<q$ unless  $T$ has a trivial null space. See \cite{Sta1} for more. 

Here is an extension of Remark \ref{rem:extension}.

\begin{cor}\label{cor:extension} If $\tau>0$, $\sigma>\frac{n}{p}$, $2<p<p_{+}$ 
and $\psi\in \Psi_{\sigma}^\tau(S_{\mu})$, then $\psi(tT)$ has a bounded extension from $L^\infty$ to $L^p_{loc}$. 
\end{cor}

\begin{proof}
We take $h\in L^\infty$ and $B$ a ball of radius $t$.   Write $h=\sum h_{j}$ where $h_{0}=h1_{2B}$ and 
$h_{j}= h 1_{2^{j+1}B\setminus 2^jB}$. Then $\|\psi(tT)h_{j}\|_{L^p(B)} \lesssim 2^{-j\sigma} \|h_{j}\|_{p}\lesssim 2^{-j(\sigma-\frac{n}{p})} \|h\|_{\infty}$. It remains to sum. 
\end{proof}

\section{Hardy spaces}  

The theory of Hardy spaces   associated to operators allows us  to introduce a  scale of abstract spaces. One goal will be to identify ranges of $p$ for which they agree with subspaces of $L^p$ or $H^p$. 

\subsection{Tent spaces: notation and some review}

 For $0<q< \infty$, $T^{q}_{2}$ is the  tent space   of \cite{CMS}. This is the space of $L^2_{\loc}(\reu)$ functions  $F$ such that 
 $$
 \|F\|_{T^q_{2}} = \|\SF F\|_{q} <\infty
 $$
 with for all $x\in \R^n$,
 \begin{equation}
\label{eq:sfdef}
 (\SF F)(x): = \left( \iint_{t>0, |x-y|<at}  |F(t,y)|^2\, \frac{dtdy}{t^{n+1}}\right)^{1/2}, 
\end{equation}
 where $a>0$ is a fixed number. Two different values $a$ give equivalent $T^q_{2}$ norms.

  For $q=\infty$, $T^\infty_{2}(\reu)$ is defined via Carleson measures  by  $\|F\|_{T^\infty_{2}}<\infty$, where $\|F\|_{T^\infty_{2}}$ is the smallest positive constant $C$ in 
$$
\iint_{T_{x,r}}   |F(t,y)|^2\, \frac{dtdy}{t} \le C^2 |B(x,r)|
$$
for all   open balls $B(x,r)$  in $\R^n$ and $T_{x,r}=(0,r)\times B(x,r)$. 
For $0<\alpha<\infty$,   $T^\infty_{2,\alpha}(\reu)$ is defined by $\|F\|_{T^\infty_{2,\alpha}}<\infty$ where 
$\|F\|_{T^\infty_{2,\alpha}}$ is the smallest positive constant $C$ in 
$$
\iint_{T_{x,r}}   |F(t,y)|^2\, \frac{dtdy}{t} \le C^2 |B(x,r)|^{1+\frac{2\alpha}{n}}
$$
for all   open balls $B(x,r)$  in $\R^n$. For convenience, we set $T^\infty_{2,0}=T^\infty_{2}$.

For $1\le q<\infty$ and $p$ the conjugate exponent to $q$,  $T^{p}_{2}$ is the dual of $ T^q_{2}$  for the duality
$$
(F,G):=\iint_{\reu} F(t,y) \overline{G(t,y)} \, \frac{dtdy}{t}.
$$
For $0<q\le 1$ and $\alpha= n (\frac{1}{q}-1)$, $T^\infty_{2,\alpha}$ is the dual of $T^q_{2}$ for the same duality form. Although not done explicitly there, it suffices to adapt the proof of \cite[Theorem 1]{CMS}.

\subsection{General theory}

We summarize here the theory pionnered in \cite{AMcR, HM} for operators $T$ satisfying $L^2$ off-diagonal estimates of any polynomial order  \eqref{odn} and further developed in \cite{JY, DY, HMMc, HNP, HLMMY, DL, AL}, etc. Here, there is an issue about  homogeneity of the operator and notice that both $DB$ and $BD$ are of order 1.  We stick to this homogeneity. The needed assumptions on  $T$ for what follows is  bisectoriality on $L^2$  with $H^\infty$-calculus on $\clos{\ran_{2}(T)}$ and $L^2$ off-diagonal estimates  \eqref{odn}.
Let $\omega\in [0,\pi/2)$ be the angle of the $H^\infty$-calculus. In what follows, $\mu$ is an arbitrary real number with $\omega<\mu<\pi/2$.

For $\psi\in H^{\infty}(S_{\mu})$, let
$$
\Qpsi \psi T f= (\psi(tT)f)_{t>0}, \quad f\in L^2
$$
and
$$
\Tpsi \psi T F= \int_{0}^\infty \psi(tT)F(t,\cdot)\, \frac{dt}{t}, \quad F\in T^2_{2}.
$$
The second definition is provided one can make sense of the integral. Precisely, for $\psi\in \Psi(S_{\mu})$ (the class is defined in \eqref{eq:cauchyformula}),
the operators $\Qpsi \psi T: L^2\to T^2_{2}$ and $\Tpsi \psi T: T^2_{2}\to L^2$ are bounded as follows from the square function estimates \eqref{eq:psiT}   for $T$ and its adjoint $T^*$.  Indeed, 
$\Tpsi \psi T$ is the  adjoint to $\Qpsi {\psi^*} {T^*}$ where   $\psi^*(z)=\overline{\psi(\bar z)}$.

Recall that for  $\sigma,\tau\ge 0$, 
$$\Psi_{\sigma}^\tau(S_{\mu})= \{\psi\in H^{\infty}(S_{\mu}):\psi(z)=O\big(\inf({\abs{z}^\sigma,\abs{z}^{-\tau}})\big)\}.
$$
So $$\Psi(S_{\mu}) =\bigcup_{\sigma>0, \tau>0}\Psi_{\sigma}^\tau(S_{\mu}).$$
For $0<\gamma$, let 
$$\Psi^{\gamma}(S_{\mu})=\bigcup_{\sigma>0, \tau> \gamma}\Psi_{\sigma}^\tau(S_{\mu}),
$$
$$\Psi_{\gamma}(S_{\mu})=\bigcup_{\sigma>\gamma, \tau>0}\Psi_{\sigma}^\tau(S_{\mu}).
$$
Set $\gamma (p)=| \frac{n}{p}-\frac{n}{2}|$ for $0<p\le \infty$.  If $p\le 1$ and $\alpha=n( \frac{1}{p}-1)$, then $\gamma(p)=\frac{n}{2}+\alpha$.

Consider the  table 
\begin{center}
\begin{table}[htdp]

\begin{tabular}{|c|c|c|c|}
\hline
exponents= & $\mT=$ & $\Psi_{\mT}(S_{\mu})=$ &$\Psi^{\mT}(S_{\mu})=$   \\
\hline
$0<p\le2$ & $T^p_{2}$ &$ \Psi^{\gamma(p)}(S_{\mu})$  & $\Psi_{\gamma(p)}(S_{\mu})$  \\
\hline
$2\le p\le \infty$ &  $T^p_{2}$ & $\Psi_{\gamma(p)}(S_{\mu})$ & $ \Psi^{\gamma(p)}(S_{\mu})$  \\
\hline
$0\le \alpha=n( \frac{1}{p}-1) <\infty$ & $T^\infty_{2,\alpha}$ & $\Psi_{\gamma(p)}(S_{\mu})$ & $ \Psi^{\gamma(p)}(S_{\mu})$ \\
\hline
\end{tabular}
\label{default}
\end{table}
\end{center}

Note that $\Psi^{\gamma(2)}(S_{\mu})=\Psi_{\gamma(2)}(S_{\mu})=\Psi(S_{\mu})$ so the next result is consistent with the $L^2$ theory.

\begin{prop}\label{prop:table}
For any space $\mT$ in the table, $\psi\in \Psi_{\mT}(S_{\mu}), \varphi\in \Psi^{\mT}(S_{\mu})$ and $b\in H^\infty(S_{\mu})$, then $\Qpsi \psi T b(T) \Tpsi \varphi T$ initially defined on $T^2_{2}$,  extends to a bounded operator on $\mT$ by density if $\mT=T^p_{2}$ and by duality if $\mT=T^\infty_{2,\alpha}$.
\end{prop}

\begin{proof} One can extract these classes in the range $1<p<\infty$ from \cite{HNP} and in  the other ranges from \cite{HMMc} (replacing $\frac{n}{4}$ adapted to second order operators to $\frac{n}{2}$ here). Actually, there is a possible interpolation method to reobtain directly the results in \cite{HNP} without recoursing to UMD technology, once one knows the results for $0<p\le 1$. See \cite{Sta}. 
\end{proof}

We also recall the Calder\'on reproducing formula in this context (See \cite[Remark 2.1]{AMcR}).  As the proof is not given there, we sketch one possible argument.  

\begin{prop}\label{prop:calderon}
For  any  $\sigma_{1},\tau_{1}\ge 0$ and non-degenerate $\psi \in \Psi_{\sigma_{1}}^{\tau_{1}}(S_{\mu})$ and any $\sigma,\tau>0$, there exists $\varphi\in \Psi_{\sigma}^\tau(S_{\mu})$ such that 
\begin{equation}
\label{eq:Calderon}
\int_{0}^\infty \varphi( t z) \psi( t z) \, \frac{dt}{t}=1   \quad \forall z\in S_{\mu}.
\end{equation}
As a consequence, 
\begin{equation}
\label{eq:CalderonT}
\Tpsi\varphi T \Qpsi \psi T f = f,  \quad \forall f \in \clos{\ran_{2}(T)}.  
\end{equation}
\end{prop}

\begin{proof} Assume $\psi \in \Psi_{\sigma_{1}}^{\tau_{1}}(S_{\mu})$ with $\sigma_{1},\tau_{1}\ge 0$  and  $\psi$ is non-degenerate.  Let $\theta(z)=e^{-\modz - \modz^{-1}}$ with $\modz=z$ if $\re z>0$ and $\modz=-z$ if $\re z<0$.  Clearly $\theta\in  \cap_{\sigma>0,\tau>0} \Psi_{\sigma}^\tau(S_{\mu})$ and so does
\begin{equation*}
\varphi(z)= \begin{cases}
 c_{+} \overline{\psi(\bar z)}\theta(z)      &\text{for }\  z\in S_{\mu+},\\
  c_{-} \overline{\psi(\bar z)}\theta(-z)    & \text{for }\  z\in S_{\mu-}.
\end{cases}
\end{equation*}
The constants $c_{\pm}$  are chosen such that $\int_{0}^\infty \psi(\pm t) \varphi(\pm t) \, \frac{dt}{t}=1$ (note that the integrals are positive numbers because $\psi$ is non-degenerate, hence $|\psi(\pm t)|>0$ almost everywhere, so that there is such a choice for $c_{\pm}$). 
Next,  \eqref{eq:Calderon} follows by analytic continuation. 
\end{proof}  

\begin{rem} The function $\psi$ can be taken without any decay at 0 and $\infty$: it is enough that the product $\psi\varphi$ has both decay.  

\end{rem}

Let $\mT$ be any of the spaces in the table above and $\psi\in \Psi(S_{\mu})$.  Set
$$\IH^\mT_{\Qpsi \psi T}= \{f\in \clos{\ran_{2}(T)};  \Qpsi \psi T f \in \mT\}
$$
equipped with the (quasi-)norm $\|f\|_{\IH^\mT_{\Qpsi \psi T}}=\|\Qpsi \psi T f\|_{\mT}$
and 
$$\IH^\mT_{\Tpsi \psi T}= \{\Tpsi \psi T F ;   F\in \mT \cap T^2_{2}\}
$$
equipped with the (quasi-)norm $\|f\|_{\IH^\mT_{\Tpsi \psi T}}= \inf\{ \|F\|_{\mT}; f=\Tpsi \psi T F, \, \,F \in \mT \cap T^2_{2}\}$.
We do not need to introduce completions at this point. 

\begin{cor}\label{cor:pre-hardy}  For any $\mT$ in the above table, non-degenerate $\psi\in \Psi_{\mT}(S_{\mu})$ and $ \varphi\in \Psi^{\mT}(S_{\mu})$ , we have 
$$
\IH^\mT_{\Qpsi \psi T}= \IH^\mT_{\Tpsi \varphi T}
$$
with equivalent \text{(quasi-)}\-norms. We set $\IH^\mT_{T}$ this space and call it the pre-Hardy space associated to $(T,\mT)$.  For any $b\in H^{\infty}(S_{\mu})$, this space is preserved by $b(T)$ and $b(T)$ is bounded on it. For $\mT=T^p_{2}$, we simply set $\IH^p_{T}= \IH^{T^p_{2}}_{T}$ and for $\alpha>0$, we set $\IL^{\alpha}_{T}=\IH^{T^\infty_{2,\alpha}}_{T}.$
\end{cor} 

Of course, the pre-Hardy space associated to $(T,\mT)$ is not complete as defined. The issue of finding a completion within  a classical space is not an easy one. 

We shall say that $\psi$ is \emph{allowable} for $\IH^\mT_{T}$ if we have the equality  $\IH^\mT_{\Qpsi \psi T}=\IH^\mT_{T}$ with equivalent  (quasi-)norms.  The set of allowable $\psi$ contains  the non-degenerate functions in $\Psi_{\mT}(S_{\mu})$ but could be larger in some cases.
 
As the  $H^\infty$-calculus extends to $\IH_{T}^\mT$, the operators $e^{-s|T|}$ extend to bounded operators on $\IH_{T}^\mT$ with uniform bound in $s>0$ and have the semigroup property. A question is the continuity on $s\ge 0$, which as is well-known reduces to continuity at $s=0$. In the reflexive Banach space case, this can be solved by abstract methods for bisectorial operators (see below). However, this excludes the quasi-Banach case we are also interested in. The following result seems new in the theory (this is not an abstract one as it uses the fact that we work with operators defined on $L^2$ and measure theory)   and includes the reflexive range $p>1$ as well.

\begin{prop}\label{prop:stronglimit} For all $0<p<\infty$ and $h\in \IH^p_{T}$, we have the strong limit
$$
\lim_{s\to 0} \|e^{-s|T|}h-h\|_{\IH^p_{T}}=0.
$$
\end{prop}

\begin{proof} 
We choose  $\psi(z)=\modz^Ne^{-\modz}$, with $N>\frac{n+1}{2}$ and $N>|\frac{n}{p}-\frac{n}{2}|$. 
Set $\Gamma(x)$ the cone of $(t,y)$ with $0\le |x-y|<t$, and for $0< \delta \le R<\infty$, $\Gamma_{\delta}(x)$ its truncation for $t\le\delta $, $\Gamma^{R}(x)$ its truncation for $t\ge R $ and $\Gamma^{R}_{\delta }(x)= \Gamma^{R}(x) \setminus \Gamma_{\delta }(x)$. Set  $\Sigma h = S(\psi(tT) h)$ and
$\Sigma h(x)= \big(\int\!\!\!\int_{\Gamma(x)} |\psi(tT)h(y)|^2 \, \frac{dtdy}{t^{n+1}}\big)^{1/2}$  so that 
$\|\Sigma h\|_{L^p}\sim \|h\|_{\IH^p_{T}}$ as $\psi$ is allowable for $\IH^p_{T}$. Let $\Sigma^R h(x)$,  $\Sigma_{\delta }h(x)$ and $\Sigma_{\delta }^Rh(x)$   be defined as $\Sigma h(x)$ with integral on  $\Gamma^R(x)$, $\Gamma_{\delta }(x)$  and 
$\Gamma_{\delta }^R(x)$ respectively. Remark that by the choice of $\psi$, we have 
$$
(\Sigma h)^2(x)= \iint_{\Gamma(x)} t^{2N-n-1}||T|^Ne^{-t|T|}h(y)|^2 \, {dtdy}.
$$
 It easy to see that 
$
\Sigma (e^{-s|T|}h)(x) \le Sh(x)$ for all $s>0$ by using  $2N-n-1>0$ and  observing that the translated cone $\Gamma(x)+(s,0)$ is contained in $\Gamma(x)$.  Thus we have $\Sigma (e^{-s|T|}h-h)(x)\le 2 \Sigma h(x)$ so that by the Lebesgue convergence theorem, it suffices to show that $\Sigma (e^{-s|T|}h-h)(x)$ converges to 0 almost everywhere. Using the same idea, we have
\begin{align*}
   \Sigma (e^{-s|T|}h-h)(x)& \le \Sigma _{\delta }(e^{-s|T|}h-h)(x)+\Sigma_{\delta }^R(e^{-s|T|}h-h)(x)+\Sigma^{R}(e^{-s|T|}h-h)(x)   \\
    &  \le 2 \Sigma _{\delta +s}h(x)+ \Sigma_{\delta }^R(e^{-s|T|}h-h)(x)+ 2\Sigma^{R}(h)(x).  
\end{align*}
Pick $x\in \R^n$ so that $\Sigma h(x)<\infty$ and let $\varepsilon>0$. Then pick $R$ large and $\delta $ small so that $\Sigma^{R}h(x) <\varepsilon$ and $\Sigma_{2\delta }h(x)<\varepsilon$. Hence, for 
$s<\delta $, we have
$$
\Sigma (e^{-s|T|}h-h)(x) \le 4\varepsilon+ \Sigma_{\delta }^R(e^{-s|T|}h-h)(x).
$$
Now, a rough estimate using the $L^2$ boundedness of $\psi(tT)$ yields
$$
\Sigma_{\delta }^R(e^{-s|T|}h-h)^2(x) \le \int_{\delta }^R \frac{dt}{t^{n+1}} \|e^{-s|T|}h-h\|_{2}^2
$$
and, as $h\in  \clos{\ran_{2}(T)}$ and the semigroup is continuous on $L^2$, the proof is complete. 
\end{proof}

For later use, we also have behavior at $\infty$. 

\begin{prop}\label{prop:stronglimitinfty} For all $0<p<\infty$ and $h\in \IH^p_{T}$, we have the strong limit
$$
\lim_{s\to \infty} \|e^{-s|T|}h\|_{\IH^p_{T}}=0.
$$
\end{prop}

\begin{proof} With the same square function $\Sigma$ as above, we have $\Sigma(e^{-s|T|}h) \le  \Sigma h \in L^p$ and $\Sigma(e^{-s|T|}h)\to 0$ almost everywhere when $s\to \infty$. We conclude from the Lebesgue dominated convergence. 
\end{proof}

Let us turn to some duality statements. 

\begin{prop}\label{prop:dualspaces} Let  $\mT=T^p_{2}$, $0<p<\infty$ and $\mT^*$ be its dual space. Let $\psi$ be allowable for $\IH^\mT_{T}$ and $\IH^{\mT^*}_{T^*}$,  (for example, $\psi \in  \Psi_{\gamma(p)}(S_{\mu})\cap \Psi^{\gamma(p)}(S_{\mu})$). For any $G\in \mT^*$, then $J(G):f\mapsto  (  \Qpsi \psi {T}f, G)\in (\IH^\mT_{T})^*$. Conversely,  to any  $\ell\in (\IH^\mT_{T})^*$, there corresponds a $G\in \mT^*$, such that $\ell(f)=J(G)(f)$ for any $f\in \IH^\mT_{T}$. 
\end{prop}

\begin{proof} The proof is quite standard. That $J(G) \in (\IH^\mT_{T})^*$ follows using that $\psi$ is allowable for $\IH^\mT_{T}$, hence $\| \Qpsi \psi {T}f\|_{\mT}\sim \|f\|_{\IH^\mT_{T}}$ and the duality of tent spaces. Conversely, let $\varphi$ associated to $\psi$ as in  Proposition \ref{prop:calderon}. Let $\ell\in (\IH^\mT_{T})^*$, then $\ell\circ \Tpsi{\varphi}{T}$ is defined on $\mT\cap T^2_{2}$ and $|\ell\circ \Tpsi{\varphi}{T} (F)| \lesssim \|F\|_{\mT}$. By density in $\mT$ and duality,  there exists $G\in \mT^*$ such that $|\ell\circ \Tpsi{\varphi}{T} (F)|= (F,G)$ for all $F\in \mT\cap T^2_{2}.$ Inserting $F=\Qpsi \psi {T}f$, we obtain the $(\Qpsi \psi {T}f,G)= \ell\circ \Tpsi{\varphi}{T} (\Qpsi \psi {T}f)=\ell(f)$. 
\end{proof}

It will be easier to work within $\IH^2_{T}=\clos{\ran_{2}(T)} $. This is why we  systematically use pre-Hardy spaces.

\begin{prop}\label{prop:duality}
Let  $\mT=T^p_{2}$, $0<p<\infty$ and $\mT^*$ be its dual space.  Denote by $\pair {\, }{}$ the $L^2$ sesquilinear inner product. Then for any  $f\in \IH^\mT_{T}$,  $g\in \IH^{\mT^*}_{T^*}$ $|\pair f g| \lesssim \|f\|_{\IH^\mT_{T}}\|g\|_{\IH^{\mT^*}_{T^*}}$.  More generally,  for any $f\in \clos{\ran_{2}(T)},  g \in \clos{\ran_{2}(T^*)}$ and any $\psi,\varphi\in \Psi(S_{\mu})$ for which the Calder\'on reproducing formula \eqref{eq:Calderon}  holds, one has
$|\pair f g | \le \| \Qpsi \psi {T}f\|_{\mT} \|\Qpsi {\varphi^*} {T^*} g\|_{\mT^*}.$
Next, for any  $g\in \IH^{\mT^*}_{T^*}$, $ \|g\|_{\IH^{\mT^*}_{T^*}}\sim \sup  \{|\pair f g |;  f\in \mT, \|f\|_{\IH^\mT_{T}}=1\}$. 
When $1<p<\infty$, we can revert the roles of $\mT$ and $\mT^*$, that is,  $\pair {\, }{}$ is a duality for  the pair of spaces $ ( \IH^p_{T}, \IH^{p'}_{T^*})$.  \end{prop}

We  mention as a corollary the usual principle that upper bounds in square functions for allowable $\psi$ imply lower bounds for all $\varphi$  with  the dual operator.

\begin{prop}\label{prop:lowerbounds} Let  $\mT=T^p_{2}$, $1<p<\infty$ so that $\mT^*=T^{p'}_{2}$. Assume that $\pair {\, }{}$ is a duality for  the pair of normed spaces $ (X ,Y )$ with $X\subset \clos{\ran_{2}(T)}$ and  $Y\subset  \clos{\ran_{2}(T^*)}$ and that for any allowable $\psi\in \Psi(S_{\mu})$ for $\IH^\mT_{T}$, we have
$ \| \Qpsi \psi {T}f\|_{\mT} \lesssim \|f\|_{X}$ for all $f\in \clos{\ran_{2}(T)}$. Then for any non-degenerate $\varphi\in \Psi(S_{\mu})$, we have  $\|g\|_{Y}\le \|\Qpsi {\varphi^*} {T^*} g\|_{\mT^*}$ for all $g \in \clos{\ran_{2}(T^*)}$.
\end{prop}

\begin{proof}
This a consequence of the previous result with the fact that given  a non-degenerate $\varphi$, one can find $\psi$ in any class $\Psi_{\sigma}^\tau(S_{\mu})$, thus one allowable $\psi$  for  $\IH^\mT_{T}$,  for which the Calder\'on reproducing formula \eqref{eq:Calderon} holds. 
\end{proof}

For $0<p\le1$,  we can take advantage of the notion of molecules.  We follow \cite{HMMc}. For a cube (or a ball) $Q\subset \mathbb{R}^{n}$ denote the dyadic annuli by $S_{i}\left(Q\right)$, which is defined by $S_{i}\left(Q\right):=2^{i}Q\backslash 2^{i-1}Q$ for $i=1,2,3,...$ and $S_{0}\left(Q\right):=Q$. Here $\lambda Q$ is the cube with same center as $Q$ and sidelength $\lambda\ell\left(Q\right)$. Let $0<p\leq 1$, $\epsilon>0$ and $M\in \N$. We say that a function $m\in L^2$ is a $\left(\IH^{p}_{T},\epsilon,M\right)$-molecule if there exists a cube $Q\subset \R^{n}$ and a function $b\in  \dom_{2}(T^M)$ such that $T^{M}b=m$ and  
\begin{align}\label{eq:mol}
||\left(\ell\left(Q\right)T\right)^{-k}m||_{L^{2}\left(S_{i}\left(Q\right)\right)}\leq \left(2^{i}\ell\left(Q\right)\right)^{\frac{n}{2}-\frac{n}{p}}2^{-i\epsilon} && i=0,1,2,...; k=0,1,2,...,M.
\end{align}
Remark that  $m\in \ran_{2}(T)$ and also that $m\in L^p$ with $\|m\|_{p}\lesssim 1$ independently of $Q$. 

\begin{defn}
Let  $0<p\leq1$, $\epsilon >0$ and $M\in \N$. For  $f\in \clos{\ran_{2}(T)}$,  $f=\sum_{j}\lambda_{j}m_{j}$  is a molecular $\left(\IH^{p}_{T},\epsilon,M\right)$-representation of $f$ if each $m_{j}$ is an $\left(\IH^{p}_{T},\epsilon,M\right)$-molecule,  $(\lambda_{j})\in \ell^{p}$ and the series converges in $L^2$. We define
\begin{align*}
\IH_{T,mol,M}^{p}:=\left\{f\in\clos{\ran_{2}(T)} ; f \text{ has a molecular $\left(\IH^{p}_{T},\epsilon,M\right)$-representation }\right\}
\end{align*}
with the quasi-norm (it is a  norm only when $p=1$) 
\begin{align*}
||f||_{\IH_{T,mol,M}^{p}}:=\inf\left\{\|(\lambda_{j})\|_{\ell^p}\right\},
\end{align*}
taken over all   molecular  $\left(\IH^{p}_{T},\epsilon,M\right)$-representations $f=\sum_{j=0}^{\infty}\lambda_{j}m_{j}$,  where  $\|(\lambda_{j})\|_{\ell^p}:=\left(\sum_{j=0}^{\infty}|\lambda_{j}|^{p}\right)^{\frac{1}{p}}$. 
\end{defn}

\begin{rem}
Note the continuous inclusion $\IH_{T,mol,M_{1}}^{p} \subset \IH_{T,mol,M_{2}}^{p}$ if $M_{2}\ge M_{1}$. In particular,  $\IH_{T,mol,M}^{p} \subset \IH_{T,mol,1}^{p}$. \end{rem}

\begin{prop}  Let  $0<p\leq1$,  $M\in \N$ with $M> \frac{n}{p}-\frac{n}{2}$.  Then $\IH_{T,mol,M}^{p}= \IH^p_{T}$ with equivalence of quasi-norms. 
\end{prop}

\begin{proof} Adapt \cite{HM, HMMc}. 
\end{proof}

\begin{rem}
It would also make sense to consider the atomic versions but at this level of generality, we do not know whether $ \IH^p_{T}$ has  an atomic decomposition. 
\end{rem}

The following corollary is a useful consequence. 

\begin{cor}\label{lem:lowerp<2} For $0<p<2$, then $\IH^p_{T} \subset L^p$ with $\|f\|_{p}\lesssim \|f\|_{\IH^p_{T}}$. 
\end{cor}

\begin{proof}
For $0<p\le 1$, this is a consequence of the fact that any  $\left(\IH^{p}_{T},\epsilon,M\right)$-molecule satisfies $\|m\|_{p}\lesssim 1$ and of the previous proposition. For $1\le p\le 2$, we proceed by interpolation as follows. Fix one $\varphi\in \Psi_{\frac{n}{2}}(S_{\mu})$ and consider the map
$\Tpsi \varphi T$. By Proposition \ref{prop:table}, it is bounded from $T^2_{2}$ to $L^2$ and from $T^1_{2}\cap T^2_{2}$ to $\IH^1_{T}$ with $\|\Tpsi \varphi T F\|_{\IH^1_{T}} \lesssim \|F\|_{T^1_{2}}$ so that  it maps  $T^2_{1}\cap T^2_{2}$ to $L^1$  with
$\|\Tpsi \varphi T F\|_{1} \lesssim \|F\|_{T^2_{1}}$. By interpolation, the bounded extension on  
$T^1_{2}$ is bounded from $T^p_{2}$ into $L^p$. It is a standard duality argument to show that  this extension  agrees with $\Tpsi \varphi T$ on $T^p_{2}\cap T^2_{2}$.
\end{proof}

We finish with non-tangential maximal estimates. Recall the Kenig-Pipher functional
 \begin{equation}
\label{eq:KP}
 \tN(g)(x):= \sup_{t>0}  \bigg(\bariint_{W(t,x)} |g|^2\bigg)^{1/2}, \qquad x\in \R^n,
\end{equation}
with $W(t,x):= (c_0^{-1}t,c_0t)\times B(x,c_1t)$, for some fixed constants $c_0>1$, $c_{1}>0$. 

\begin{lem}\label{lem:upperp<2} For all  $0<p\le 1$,  one has the estimate
\begin{equation}
\label{eq:upperp<2}
  \|\tN(e^{-t|T|}h) \|_{p} \lesssim  \|h\|_{\IH^p_{T}}, \ \forall h\in \clos{\ran_{2}(T)}.
 \end{equation}
 Furthermore, it also holds for $1<p<2$ if it holds at $p=2$. 
\end{lem}

\begin{proof} For $0<p\le 1$, this comes from  $\|\tN(e^{-t|T|}m) \|_{p} \lesssim 1$  for any $\left(\IH^{p}_{T},\epsilon,M\right)$-molecule with $M\in \N$ for $M$ large  enough depending on $n,p$ (Adapt the proofs in \cite{HM}, \cite{JY} or \cite{DL}. See also \cite{Sta} for an explicit argument.  It is likely that one can prove that the lower bound $M> \frac{n}{p}-\frac{n}{2}$ works but we don't need such a precision). 
This implies the inequality for any $f\in \IH_{T,mol,M}^{p}= \IH^p_{T}$. 

We use interpolation for $1<p<2$ as follows.  Fix $\varphi\in \Psi_{\frac{n}{2}}(S_{\mu})$ and $\psi \in \Psi^{\frac{n}{2}}(S_{\mu})$ such that $\Tpsi \varphi T \Qpsi\psi T =I$ on $\clos{\ran_{2}(T)}$.  From the assumption, the sublinear operator
$V: F\mapsto \tN(e^{-t|T|}\Tpsi \varphi T F)$ is  bounded from $T^2_{2}$ into $L^2$ with $\|VF\|_{2}\lesssim \|F\|_{T^2_{2}}$  and we just proved it is bounded from $T^2_{2}\cap T^1_{2}$ into $L^1$ with $\|VF\|_{1}\lesssim \|F\|_{T^1_{1}}$. Using real interpolation (\cite{CMS}, corrected in \cite{Be}), and density (of $T^1_{2}\cap T^2_{2}$ into $T^p_{2}\cap T^2_{2}$ for the $T^p_{2}$ topology) this implies that  $V$ maps $T^2_{2}\cap T^p_{2}$ into $L^p$ with $\|VF\|_{p}\lesssim \|F\|_{T^p_2}$. Applying this to $F=\Qpsi\psi T h$ when $h\in \clos{\ran_{2}(T)}$, this implies \eqref{eq:upperp<2}.  
\end{proof}

\subsection{Spaces associated to $D$}

We specialize the general theory to the situation where $T=D$. Because $D$ is self-adjoint, we can also consider the $\left(\IH^{p}_{D},M\right)$-atoms, which are those $\left(\IH^{p}_{D},\epsilon, M\right)$-molecules associated to a cube $Q$ and supported in $Q$. The atomic space $\IH^p_{D,ato, M}$ is defined similarly to the molecular one and one has $\IH^p_{D,ato, M}=\IH^p_{D}$ when $0<p\le 1$ and $M>\frac{n}{p}- \frac{n}{2}$. The proof is explicitly done in \cite{AMcMo} for $p=1$ and applies \textit{in extenso} to $p<1$. See also \cite{HLMMY} for the case of second order operators and $p\le 1$. 

We also remark that $\IH^p_{D,ato, 1} = \IH^p_{D,mol, 1}$ with equivalence of norms (This argument is due to A. McIntosh). The inclusion $\subset$ is obvious. In the opposite direction, if $m$ is an $\left(\IH^{p}_{D},\epsilon, 1\right)$-molecule, then one can write $m=Db$ with estimate   \eqref{eq:mol} and  $M=1$. Then we write $b=\sum \chi_{i}b$ with $(\chi_{i})$ a smooth partition of unity associated to the annular set $S_{i}(Q)$: they satisfy $0\le \chi_{i}\le 1$, $\|\nabla \chi_{i}\|_{\infty}\lesssim (2^{i}\ell(Q))^{-1}$ and $\chi_{i}$ supported on $S_{i}(2Q)$. Then, it is easy to show that $a_{i}=2^{i\epsilon} D(\chi_{i}b)$ is up to a dimensional constant an $\left(\IH^{p}_{D},1\right)$-atom  and that the series $m=\sum 2^{-i\epsilon}a_{i}$ converges in $L^2$. 

\begin{thm} \label{thm:hpd}
Let $\frac{n}{n+1}<p\le 1$. Then 
\begin{equation}
\label{eq:hpDp<1}
\IH^p_{D}=\IH^p_{D,mol, 1}=\IH^p_{D,ato, 1}= H^p \cap \clos{\ran_{2}(D)}= H^p\cap \IP (L^2)= \IP(H^p\cap L^2)
\end{equation}
with  
$$\|f\|_{\IH^p_{D}}\sim \|f\|_{\IH^p_{D,mol,1}}\sim \|f\|_{\IH^p_{D,ato,1}}\sim \|f\|_{H^p}, \quad \forall f\in \clos{\ran_{2}(D)}.
$$ 
 Let $1<p<\infty$. Then
 \begin{equation}
\label{eq:hpDp>1 }
\IH^p_{D}= \clos{\ran_{p}(D)} \cap \clos{\ran_{2}(D)}= L^p\cap \clos{\ran_{2}(D)}= L^p \cap \IP(L^2)= \IP (L^p\cap L^2) 
\end{equation} 
with 
$$\|f\|_{\IH^p_{D}}\sim \|f\|_{L^p}, \quad \forall f\in \clos{\ran_{2}(D)}.
$$
 Let $p=\infty$. Then 
  \begin{equation}
\label{eq:hpDBMO}
\BBMO_{D}= \BMO \cap \clos{\ran_{2}(D)}= \BMO \cap \IP(L^2)= \IP (\BMO\cap L^2) 
\end{equation} 
with 
$$\|f\|_{\BBMO_{D}}\sim \|f\|_{\BMO}, \quad \forall f\in \clos{\ran_{2}(D)}.$$
Let $0\le \alpha<1$. Then 
  \begin{equation}
\label{eq:hpDLambda}
\IL^\alpha_{D}=\dot \Lambda^\alpha \cap \clos{\ran_{2}(D)}= \dot \Lambda^\alpha \cap \IP(L^2)= \IP (\dot \Lambda^\alpha \cap L^2) 
\end{equation} 
with 
$$\|f\|_{\IL^\alpha_{D}}\sim \|f\|_{\dot \Lambda^\alpha}, \quad \forall f\in \clos{\ran_{2}(D)}.$$

 \end{thm}

\begin{proof} 
Let us assume first $p\le 1$. As $\clos{\ran_{2}(D)}=\IP(L^2)$ the fourth equality is a trivial. The inclusion $\supset$ of the fifth equality comes from the fact that $\IP$ is bounded on $H^p$, and for the converse, if $h\in H^p\cap \IP(L^2)$ then $h=\IP h\in \IP(H^p\cap L^2)$.
By general theory and the discussion above, $ \IH^p_{D}= \IH^p_{D,mol, n}\subset \IH^p_{D,mol, 1}=\IH^p_{D,ato, 1}$. Now a $\left(\IH^{p}_{D},1\right)$-atom $a=Db$ belongs to $\ran_{2}(D)$ and also to $H^p$ as $p>\frac{n}{n+1}$ and $\int a=\int Db =0$. As convergence of atomic decompositions is in $L^2$, so also in tempered distributions, it follows that  
$\IH^p_{D,ato, 1} \subset \clos{\ran_{2}(D)} \cap H^p$.  It remains to show $\IP(H^p\cap L^2)\subset \IH^p_{D}$. Let $L=D^2\IP -\Delta (I-\IP)$ where $\Delta$ is the ordinary negative self-adjoint Laplacian on $L^2$. Clearly $L$ is self-adjoint on $L^2$, positive, it has a homogeneous of order 2 symbol,   $C^\infty$  away from 0,  with $\hat L(\xi)\sim |\xi|^2$ (in the sense of self-adjoint matrices). 
One can estimate the kernel of the convolution operator $t^2Le^{-t^2L}$  and  find pointwise decay in 
$t^{-n}(1+\frac{|x|}{t})^{-n-2}$.   Similarly for  all its partial derivatives with $-n-3$ replacing $-n-2$.
  By standard theory for the Hardy space as in \cite{CMS},  for $h\in \IP(H^p\cap L^2)$,
$F(t,\cdot)= t^2L e^{-t^2L}h \in T^p_{2}\cap T^2_{2} $, thus for any $\varphi$ such that $\IH^p_{D}=\IH^{T^p_{2}}_{\Tpsi \varphi D}$ we have $\Tpsi \varphi D F\in \IH^p_{D}$. 
Now $ L\IP= D^2\IP=D^2$. Thus, as $h=\IP h$,  $ F(t,\cdot)= t^2D^2 e^{-t^2D^2}h=\psi(tD)h$ with $\psi(z)= z^2e^{-z^2}$. If one chooses $\varphi\in \Psi_{\gamma(p)}(S_{\mu})$ such that \eqref{eq:Calderon} holds then $\Tpsi \varphi D F= \Tpsi \varphi D \Qpsi \psi D h= h$ so that $h\in \IH^p_{D}$.

If $1<p<\infty$, the third equality is trivial, the fourth and the inclusion $\IP(L^p\cap L^2)\subset \IH^p_{D}$ are obtained as above.  By using truncation in $t$  for $T^p_{2}$ functions in a Calder\'on reproducing formula $\Tpsi \varphi D \Qpsi \psi D=I$, we see easily that $\IH^p_{D}\subset \clos{\ran_{p}(D)} \cap \clos{\ran_{2}(D)}$ and obviously $\clos{\ran_{p}(D)} \cap \clos{\ran_{2}(D)} \subset L^p  \cap \clos{\ran_{2}(D)}$. 

The proof for $\BMO$ type spaces is obtained by duality from $p=1$, noticing that the duality form is the same for $\IH^1_{D},\BBMO_{D}$ and $H^1,BMO$, and that $\IP=\IP^*$ is bounded on $H^1$ and $BMO$. 

The proof for $\dot \Lambda^\alpha$ type space is also obtained by duality from the case $p<1$. We omit further details. 
\end{proof}

\subsection{General facts about comparison of $\IH^p_{DB}$ and $\IH^p_{D}$}

Of course, by definition $\IH^2_{DB}=\IH^2_{D}$, thus we look at other values of $p$.

\begin{prop}\label{cor:lowerbounddbp<2} For $\frac{n}{n+1}<p<2$,  we have $ \IH^p_{DB} \subset \IH^p_{D}$ with 
continuous inclusion. More precisely,  the inequality
 $$
 \|h\|_{H^p} \lesssim \|\Qpsi \psi {DB} h\|_{T^p_{2}}, \quad \forall h\in \clos{\ran_{2}(D)},
$$
holds when  $\psi \in \Psi^{\gamma(p)}(S_{\mu})$, where $H^p=L^p$  when $p>1$. 
\end{prop}

 \begin{proof} Indeed, if $p\le 1$ it is clear that an $\left(\IH^{p}_{DB},\epsilon,M\right)$-molecule $a=(DB)^Mb$ writes $a= D (B(DB)^{M-1}b)$, hence is an $\left(\IH^{p}_{D},\epsilon,1\right)$-molecule.  We conclude using Theorem \ref{thm:hpd}. For $1<p<2$, we use the interpolation argument of  Lemma \ref{lem:upperp<2}.
\end{proof}

\begin{prop} \label{cor:upperbounddbp>2} For $2<p<\infty$,  we have $ \IH^p_{D} \subset \IH^p_{DB}$ with 
continuous inclusion. More precisely,  the inequality
 $$
 \|\Qpsi \psi {DB} h\|_{T^p_{2}} \lesssim  \|h\|_{p} , \quad \forall h\in \clos{\ran_{2}(D)},
$$
holds when  $\psi \in \Psi_{\gamma(p)}(S_{\mu})$. 
\end{prop}


\begin{proof} It suffices to prove this for $\psi\in \Psi_{\frac{n}{2}}(S_{\mu})$. The method of Lemma \ref{lem:sfpinfty} below proves in particular for such $\psi$ that 
$$
 \|\Qpsi \psi {DB} h\|_{T^p_{2}} \lesssim  \|h\|_{p} , \quad \forall h\in L^p\cap L^2.
$$
If $h\in \IH^p_{D}$, then $h=\IP h$ and $  \|\IP h\|_{p}\sim  \|h\|_{\IH^p_{D}}$ by Theorem \ref{thm:hpd} and  as $\Qpsi \psi {DB} h= \Qpsi \psi {DB} \IP h$, we obtain
$$
 \|\Qpsi \psi {DB} \IP h\|_{T^p_{2}} \lesssim  \|\IP h\|_{p}\sim  \|h\|_{\IH^p_{D}}.
$$
\end{proof}

We are interested in the equality $\IH^p_{DB}=\IH^p_{D}$. 

\begin{thm}\label{thm:hpdb=hpd}  Let $\frac{n}{n+1}<p< \infty$. Assume that  $\IH^p_{DB}= \IH^p_{D}$ with 
equivalent norms. Then for any $b\in H^\infty(S_{\mu})$, $b(DB)$ is bounded on $\IH^p_{D}$.
Thus, $\|b(DB)h\|_{H^p}\lesssim \|b\|_{\infty}\|h\|_{H^p}$ for any $h\in \IH^p_{D}$ where $H^p=L^p$ if $p>1$. Furthermore, $(e^{-t|DB|})_{t>0}$ is a strongly continuous semigroup  on $\IH^p_{D}$. 
 \end{thm}

\begin{proof} From Corollary \ref{cor:pre-hardy}, we know that $b(DB)$ is bounded on $\IH^p_{DB}$ for any $0<p<\infty$.  The strong continuity of the semigroup on $\IH^p_{DB}$ is also shown in Proposition \ref{prop:stronglimit}.  The same properties hold  for any equivalent topology. 
\end{proof}

We turn to dual statements. The result which will guide our discussion is the following one. Recall that $\IP$ is the orthogonal projection from $L^2$ onto $ \clos{\ran_{2}(D)}$.

\begin{thm}\label{thm:duality}  Let $\frac{n}{n+1}<q< \infty$. Assume that  $\IH^q_{DB^*}= \IH^q_{D}$ with 
equivalent norms. Then    if $q>1$ and $p=q'$, 
$\IP: \IH^p_{BD} \to \IH^p_{D}$ is an isomorphism and if  $q\le 1$ and $\alpha=n(\frac{1}{q}-1)$, 
 $\IP: \IL^{\alpha}_{BD} \to \IL^{{\alpha}}_{D}$ is an isomorphism. In the   range $q>1$ and $p=q'$, the converse holds: if $\IP: \IH^p_{BD} \to \IH^p_{D}$ is an isomorphism then $\IH^q_{DB^*}= \IH^q_{D}$ with 
equivalent norms.
 \end{thm}

\begin{proof} This is in fact a simple functional analytic statement. Let us prove the direct part.  We have $\frac{n}{n+1}<q<\infty$ and    $\IH^{q}_{DB^*}=\IH^{q}_{D}$, with equivalence of norms.  We want to show the isomorphism property of $\IP$.  We know that $\IP: \clos{\ran_{2}(BD)}=\IH^2_{BD} \to \clos{\ran_{2}(D)}=\IH^2_{D}$ is  isomorphic, thus bijective. It suffices to prove the norm comparison. 
Assume first $1<q$. Set  $p=q'$. Let $g\in \clos{\ran_{2}(BD)}$. Then using Proposition \ref{prop:duality} for $T=DB^*$ and also for $D$, one has
\begin{align} \label{eq:420}
  \|g\|_{\IH^p_{BD}}  
    &
     \sim \sup\{ |\pair g f|\, ; \, \|f\|_{\IH^q_{DB^*}}\le 1\} \\
    &\nonumber  \sim \sup\{ |\pair g f|\, ; \, \|f\|_{\IH^q_{D}}\le 1\} \\
    &\nonumber  = \sup\{ |\pair {\IP g} f|\, ; \, \|f\|_{\IH^q_{D}}\le 1\} \\
    &\nonumber  \sim \|\IP g\|_{\IH^p_{D}}.
\end{align}  
Next,  if we assume $q\le 1$, then  we work with $\alpha=n(\frac{1}{q}-1)$ and  $\IL^\alpha_{BD}$,  and exactly the same argument applies. 

For the converse in the case $q>1$, it suffices to reverse the role of the spaces.  Let $f\in \clos{\ran_{2}(D)}$.  Then,
\begin{align*}
  \|f\|_{\IH^q_{DB^*}}  
    &  \sim \sup\{ |\pair g f|\, ; \, \|g\|_{\IH^p_{BD}}\le 1\} \\
    &  = \sup\{ |\pair {\IP g} f|\, ; \, \|g\|_{\IH^p_{BD}}\le 1\} \\
    & \sim \sup\{ |\pair h f|\, ; \, \|h\|_{\IH^p_{D}}\le 1\} \\
& \sim \|f\|_{\IH^q_{D}}.
\end{align*}  
\end{proof}

\begin{cor}\label{cor:hpbd0} Let $\frac{n}{n+1}<q< \infty$. Assume that  $\IH^q_{DB^*}= \IH^q_{D}$ with 
equivalent norms. Let $b\in H^\infty(S_{\mu})$.     If $q>1$ and $p=q'$,  then
 $$\|\IP b(BD)\|_{p} \lesssim \|\IP h\|_{p}, \quad \forall h\in \IH^2_{BD}.$$ If $q\le 1$, then for $ \alpha= n(\frac{1}{q}-1)$, 
 $$\|\IP b(BD)h\|_{\dot\Lambda^\alpha} \lesssim \|\IP h\|_{\dot\Lambda^\alpha}, \quad \forall h\in \IH^2_{BD}.$$
 \end{cor}
 
 \begin{proof} This is just using the similarity induced by $\IP$ from the previous theorem and the $H^\infty$-calculus on $\IH^p_{BD}$ or $\IL^\alpha_{BD}$ from Corollary \ref{cor:pre-hardy}. 
\end{proof}

Another version is also useful.

\begin{cor}\label{cor:hpbd} Let $\frac{n}{n+1}<q< \infty$. Assume that  $\IH^q_{DB^*}= \IH^q_{D}$ with 
equivalent norms. Then    if $q>1$ and $p=q'$, for any $b\in H^\infty(S_{\mu})$ which is defined at $0$, 
$\IP b(BD)$ is bounded on $\IH^p_{D}$ with 
$
\|\IP b(BD)h\|_{p}\lesssim \|h\|_{p}$ for all $h\in \IH^p_{D}$. Also    $(\IP e^{-t|BD|})_{t>0}$ is  a strongly continuous semigroup  on $\IH^p_{D}$.
 If $q\le 1$, then for $ \alpha= n(\frac{1}{q}-1)$, 
  $\IP b(BD)$ is bounded on $\IL^\alpha_{D}$ with $
\|\IP b(BD)h\|_{\IL^\alpha_{D}}\lesssim \|h\|_{\IL^\alpha_{D}}$ for all $h \in\IL^\alpha_{D}$.
 Furthermore,   $(\IP e^{-t|BD|})_{t>0}$ is a  weakly-$*$ continuous semigroup 
 on $\IL^\alpha_{D}$. 
 \end{cor}
 
 \begin{proof} Let us begin with the case $q>1$. Let $h\in \IH^p_{D}$. From the previous theorem, 
 there exists a unique $h'\in \IH^p_{BD}$ such that $h=\IP h'$.  By Proposition \ref{prop:bdproj}, since $b$ is defined at 0, 
 $$
 \IP b(BD)h= \IP b(BD) \IP h'= \IP b(BD) h'.
 $$
 Thus, by the previous corollary,  
 $$\|\IP b(BD) h\|_{p}= \| \IP b(BD) h'\|_{p }\lesssim \|\IP h'\|_{p}.
 $$
 The proof when $q\le 1$ is similar and we skip it. 
\end{proof}
 
 \begin{rem}\label{rem:weaker} The assumption $\IH^q_{DB^*}= \IH^q_{D}$ with 
equivalent norms can be weakened in Theorem  \ref{thm:duality} by the following one. This will be useful when $q\le 2$.  It suffices to assume that  $\IH^q_{DB^*}\subset \IH^q_{D}$, that  
$\|h\|_{\IH^q_{D}}\sim \|h\|_{\IH^q_{DB^*}}$ for all $h\in \IH^q_{DB^*}$ and that the inclusion is dense for the $\IH^q_{D}$ topology. Indeed, when $q>1$, \eqref{eq:420} rewrites 
\begin{align*} \label{eq:420}
  \|g\|_{\IH^p_{BD}}  
    &
     \sim \sup\{ |\pair g f|\, ; \, f\in\IH^q_{DB^*}, \|f\|_{\IH^q_{DB^*}}\le 1\} \\
    &\nonumber  \sim \sup\{ |\pair g f|\, ; \, f\in\IH^q_{DB^*},   \|f\|_{\IH^q_{D}}\le 1\} \\
    &\nonumber  \sim \sup\{ |\pair {g} f|\, ; \,  f\in\IH^q_{D}, \|f\|_{\IH^q_{D}}\le 1\} \\
     &\nonumber  = \sup\{ |\pair {\IP g} f|\, ; \,   f\in\IH^q_{D}, \|f\|_{\IH^q_{D}}\le 1\} \\
    &\nonumber  \sim \|\IP g\|_{\IH^p_{D}}.
\end{align*}  
The second line comes from the equivalence of norms, and the third from the density.  The same reasoning holds when $q\le 1$.  
The same weaker assumption can be taken in Corollary \ref{cor:hpbd0} as well. 
\end{rem}

\subsection{The spectral subspaces} 

The pre-Hardy spaces split in two spectral subspaces. This will become useful when relating this to boundary value problems as these spectral subspaces will identify to trace spaces for  elliptic systems. 

Because $T=DB$ or $BD$ is bisectorial with $H^\infty$-calculus on $L^2$, we have two spectral subspaces of $H^2_{T}= \clos{\ran_{2}(T)}$, called $H^{2,\pm}_{T}$ as defined in Proposition \ref{prop:SFimpliesFC} by 
$H^{2,\pm}_{T}= \chi^\pm(T)(H^2_{T})$. 

This can be extended to the pre-Hardy spaces $\IH^\mT_{T}$ by setting $$\IH^{\mT, \pm}_{T}:= \chi^\pm(T)(\IH^\mT_{T})= \IH^\mT_{T}\cap H^{2,\pm}_{T}.$$ This leads to the spaces $\IH^{p,\pm}_{T}$ for $0<p\le\infty$ and $\IL^{\alpha,\pm}_{T}$ for $\alpha\ge 0$. 

We have the following properties. 

(1) $\IH^{\mT}_{T}= \IH^{\mT, +}_{T} \oplus \IH^{\mT, -}_{T}$ where the sum is topological for the topology of $\IH^{\mT}_{T}$. 

(2) $(e^{\mp tT}\chi^\pm(T))_{t>0}$ are  semigroups on $ \IH^{\mT, \pm}_{T}$, which coincide with $(e^{-t|T|})_{t>0}$. Thus, they are strongly continuous if $\mT=T^p_{2}$ and weakly-$*$ continuous if $\mT=T^{\infty}_{2,\alpha}$.

\section{Pre-Hardy spaces identification}\label{sec:Hardy}

 The range of $q$ for which $\IH^q_{DB}= \IH^q_{D}$ will be our goal, together with the determination  the classes of allowable $\psi$, as we will need something more precise than what the general theory predicts. This is the most important section of this article and we give full details.

 Recall that $I(BD)=I(DB)= (p_{-}(BD), p_{+}(BD))$. We sometimes set $p_{-}=p_{-}(BD)=p_{-}(DB)$ and $p_{+}=p_{+}(BD)=p_{+}(DB)$ to simplify notation.  The situation for the operator $DB$ is simple to state and meets our needs for applications to elliptic PDEs.  We use the notation $q_{*}= \frac{nq}{n+q}$ for the lower Sobolev exponent of $q$ and $q^*=\frac{nq}{n-q}$ for the Sobolev exponent of $q$ when $q<n$. 
At this level of generality, we have the following range for comparison of Hardy spaces.  Later (Section \ref{sec:DGN}) we obtain a much bigger range under some De Giorgi assumptions when $DB$ arises  from a  second order equation or system. 

\begin{thm}\label{thm:hpdb} For $(p_{-}(DB))_{*}<p<p_{+}(DB)$,  we have $\IH^p_{DB}= \IH^p_{D}$ with 
equivalent norms. More precisely,  the comparison 
 $$
\|\Qpsi \psi {DB} h\|_{T^p_{2}} \sim \|h\|_{H^p}, \quad \forall h\in \clos{\ran_{2}(DB)}=\clos{\ran_{2}(D)},
$$
holds when  $\psi \in \Psi^{\gamma(p)}(S_{\mu})$ if $(p_{-})_{*}<p<2$  and $\psi \in \Psi(S_{\mu})$ if $2\le p <p_{+}$.  In particular,  we have the square function estimates
$$
\|\SF(tDB{e^{-t|DB|}}h)\|_{p} \sim \|\SF(t\partial_{t}{e^{-t|DB|}}h)\|_{p}\sim  \|h\|_{H^p}, \quad \forall h\in \clos{\ran_{2}(D)}.
$$
\end{thm}

\begin{rem} In the case of constant $B$ as in Section \ref{sec:constant}, or under the assumption of Proposition \ref{prop:n=1} if $n=1$,  we have $p_{-}(DB)=1$ and $p_{+}(DB)=\infty$, hence the interval is the largest possible one $(\frac{n}{n+1},\infty)$. 
\end{rem}

The situation for $BD$ is a little more complicated, since we want $\psi(z)=O(z)$ for applications and also since the functions of $BD$ do not give all the information we need for the elliptic PDEs. We state this in three different results. 

\begin{thm}\label{thm:hpbd} For $p_{-}(DB)<p<(p_{+}(DB))^*$,   we have $\IP: \IH^p_{BD} \to \IH^p_{D}$ is an isomorphism. More precisely,      the comparison 
 $$
\|\Qpsi \psi {BD} h\|_{T^p_{2}} \sim \|\IP h\|_{\IH^p_{D}}, \quad \forall h\in \clos{\ran_{2}(BD)},
$$
holds when  $\psi \in \Psi^{\gamma(p)}(S_{\mu})$ if $p_{-}<p<2 $,  $\psi \in \Psi(S_{\mu})$ if $2<p<p_{+}$ and  
$\psi \in \Psi_{ \frac{n}{p}-\frac{n}{p_{+}}}(S_{\mu})$ if $p_{+}\le p <(p_{+})^*$. 

Moreover, if $p_{+}>n$, then for $0<\alpha<1-\frac{n}{p_{+}}$,  we have that $\IP: \IL^{{\alpha}}_{BD} \to \IL^{{\alpha}}_{D}$ is an isomorphism with 
$$
\|\Qpsi \psi {BD} h\|_{T^\infty_{2,\alpha}} \sim \|\IP h\|_{\dot \Lambda^\alpha}, \quad \forall h\in \clos{\ran_{2}(BD)},
$$
when $\psi \in \Psi_{\alpha+\frac{n}{p_{+}}}(S_{\mu})$. 

\end{thm}

\begin{cor} 
For $p\in I(BD)$, 
$ \IH^p_{BD}= \clos{\ran_{2}(BD)}\cap \clos{\ran_{p}(BD)}= \clos{\ran_{2}(BD)}\cap L^p$.
\end{cor}

\begin{proof} 
Remark that for $p\in I(BD)$, we have $\|\IP h\|_{p}\sim \|h\|_{p}$ for all $h\in  \clos{\ran_{2}(BD)}$ 
by Proposition \ref{prop:projectionlp}. So $\IH^p_{BD}$ is the set of $h\in \clos{\ran_{2}(BD)}$ for which $\|h\|_{p}<\infty$, which is $\clos{\ran_{2}(BD)}\cap \clos{\ran_{p}(BD)}= \clos{\ran_{2}(BD)}\cap L^p$. 
\end{proof}

\begin{rem}
 If,  furthermore, $B$ is invertible in $L^\infty$, then $ \IH^p_{BD}= \clos{\ran_{2}(BD)}\cap L^p$ also for $\max (1,(p_{-})_{*})\le p<p_{-}$. See \cite{Sta}. In particular, $\IH^p_{BD}$ is also a subspace of $L^p$ but this is not so useful in practice.
\end{rem}

\begin{rem} 
In each theorem, the classes of $\psi$ for the upper bounds are what is expected  from the general theory when $p<2$ and a little better when $p>2$. In particular, the  classes for the upper bounds of $\|\Qpsi \varphi T h\|_{\mT}$  obtained for $p>2$ will require a specific statement  (Corollary \ref{cor:psi1} and Proposition \ref{prop:p+>n}). Note that all these classes allow the behavior $\psi(z)=O(z)$ at 0.  This will be important for applications to elliptic equations.   However, it could be that we want to use square functions with some $\psi(z)=O(z)$ at 0 for $p$ beyond the exponent $(p_{{+}}(BD))^*$.  Indeed,  the value of $p_{+}(BD)$ is  usually close to 2 while one needs to consider $p=\infty$. This is the object of the next  result where the failure of good vanishing order at 0 is compensated by being approximable to higher order at 0 on each component of $S_{\mu}$.
\end{rem}

We introduce  specific classes in $H^\infty(S_{\mu})$.  We let  $ \mR^{k}(S_{\mu})$, $k=1,2$,  be the subclasses of $H^\infty(S_{\mu})$ of those $\phi$ of the form
\begin{equation}
\phi(z)= \sum_{m=1}^M c_{m}(1+imz)^{-k}
\end{equation}
for some integer $M\ge 1$ and $c_{m}\in \C$. Next, for $\sigma>0$,  we define $\mR^{k}_{\sigma}(S_{\mu})$
as the subset  of those $\psi\in H^\infty(S_{\mu})$ for which  there exist  $\phi_{\pm}\in \mR^{k}(S_{\mu})$ with 
\begin{equation}
|\psi(z)-\phi_{\pm}(z)|=O(|z|^\sigma), \quad \forall z\in S_{\mu\pm}.
\end{equation}
We mean here that we may use different approximations of $\psi$ in each sector $S_{\mu+}$ and $S_{\mu-}$. The main example is for us $\psi(z)= \modz e^{-\modz}$. For $z\in {S_{\mu+}}$, $\psi(z)=ze^{-z}$, so this is the restriction of an analytic function on $\C$ and for any given $\sigma>1$,  it is easy (by solving a finite dimensional linear Vandermonde  system)  to find $\phi_{+}\in \mR^{1}(S_{\mu})$ such that $|\psi(z)-\phi_{+}(z)|=O(|z|^\sigma)$ for $z\in S_{\mu+}$. The same thing can be done in $S_{\mu-}$ but  $\phi_{-}$ must be different. 

\begin{thm}\label{thm:goodhpbd}  Assume that for some $q$ with $\frac{n}{n+1}<q< p_{+}(DB^*)$   we have $\IH^q_{DB^*}= \IH^q_{D}$ with 
equivalent norms. Let    $\psi\in \mR^{1}_{\sigma}(S_{\mu}) \cap  \Psi_{1}^\tau(S_{\mu})$ with $\sigma>\gamma(q)$ and $\tau>0$ if $q<2$,  and $\psi\in   \Psi^\tau(S_{\mu})$ with $\tau>\gamma(q)$ if $q>2$.

 If $q>1$ and  $p=q'$, 
we have
$$
\|\Qpsi \psi {BD} h\|_{T^p_{2}} \sim \|\IP h\|_{L^p}, \quad \forall h\in \clos{\ran_{2}(BD)},
$$
 and  if $q\le 1$ and $\alpha=n(\frac{1}{q}-1)$,
$$
\|\Qpsi \psi {BD} h\|_{T^\infty_{2,\alpha}} \sim \|\IP h\|_{\dot \Lambda^\alpha}, \quad \forall h\in \clos{\ran_{2}(BD)}.
$$
 In particular, if $q>1$, we have the square function estimates
$$
\|\SF(tBD{e^{-t|BD|}}h)\|_{p}\sim \|\SF(t\partial_{t}{e^{-t|BD|}}h)\|_{p}\sim  \|\IP h\|_{L^p}, \quad \forall h\in \clos{\ran_{2}(BD)},
$$
and, if $q\le 1$, the weighted Carleson measure  estimates
$$
\|tBD{e^{-t|BD|}}h\|_{T^\infty_{2,\alpha}} \sim \|t\partial_{t}{e^{-t|BD|}}h\|_{T^\infty_{2,\alpha}} \sim \|\IP h\|_{\dot \Lambda^\alpha}, \quad \forall h\in \clos{\ran_{2}(BD)}.
$$
\end{thm}

 When $(p_{-}(DB^*))_{*}<q<p_{+}(DB^*)$, Theorem \ref{thm:hpdb} already takes care of the conclusions without the condition $\mR^{1}_{\sigma}(S_{\mu})$.  We state it this way for later use.  In fact, for the boundary value problems later, we also need  
tent space estimates for $tD{e^{-t|BD|}}h$. 
When $B^{-1}$ exists in $L^\infty$,  in particular, this covers the case of second order equations, these results are enough for our needs. But for systems with $B^{-1}\in L^\infty$  not granted, one still has to work a little bit.  This result covers  both situations.

\begin{thm}\label{thm:supergoodhpbd}  Assume that for some $q$ with $\frac{n}{n+1}<q< p_{+}(DB^*)$   we have $\IH^q_{DB^*}= \IH^q_{D}$ with 
equivalent norms. Let   $\phi\in\mR^2_{\sigma}(S_{\mu})\cap \Psi^\tau_{0}(S_{\mu})$ with $\sigma>\gamma(q), \tau>1 $   if $q<2$ and $\phi \in \Psi^\tau_{0}(S_{\mu})$ with $\tau>1+\gamma(q)$ if $q>2$. 

If $q>1$ and   $p=q'$, 
 we have
$$
\|tD \phi ( {tBD}) h\|_{T^p_{2}} \sim \|\IP h\|_{L^p}, \quad \forall h\in \clos{\ran_{2}(BD)},
$$
 and, if $q\le 1$, and $\alpha=n(\frac{1}{q}-1)$,
 $$
\|tD \phi ({tBD}) h\|_{T^\infty_{2,\alpha}} \sim \|\IP h\|_{\dot \Lambda^\alpha}, \quad \forall h\in \clos{\ran_{2}(BD)}.
$$

In particular, if $q>1$, we have the square function estimate
$$
\|\SF(tD{e^{-t|BD|}}h)\|_{p}\sim  \|\IP h\|_{L^p}, \quad \forall h\in \clos{\ran_{2}(BD)}, 
$$
and, if $q\le 1$, the weighted Carleson measure  estimate
$$
\|tD{e^{-t|BD|}}h\|_{T^\infty_{2,\alpha}} \sim \|\IP h\|_{\dot \Lambda^\alpha}, \quad \forall h\in \clos{\ran_{2}(BD)}.
$$
\end{thm}

Compare the conclusions of the last two theorems: in one case we have $t|BD|e^{-t|BD|}$ and $tBDe^{-t|BD|}$; in the other we have $tDe^{-t|BD|}$. So we have cancelled $B$. Other conditions on $\phi$ suffice for this theorem to hold. We shall stop here the search on such conditions.

\subsection{Proof of Theorem \ref{thm:hpdb}}

\subsubsection{Upper bounds} 

We begin with upper bounds separating the cases $p>2$ and $p<2$.  The case $p=2$ is, of course, contained in Proposition \ref{prop:SFimpliesFC}.

\begin{prop}\label{prop:upperbdp>2} For $T=DB$ or $BD$,  $2<p<p_{+}(DB)$ and $\psi \in \Psi(S_{\mu})$, it holds
$$
\|\Qpsi \psi T h\|_{T^p_{2}} \lesssim \|h\|_{p}, \quad \forall h\in \clos{\ran_{2}(T)}.
$$
\end{prop}

\begin{proof} It is well known (see \cite{Stein}) that  for $p>2$
$$
\|\Qpsi \psi T h\|_{T^p_{2}}  \lesssim  \bigg\|\bigg( \int_0^\infty |\psi(tT) h|^2  \, \frac{dt}t  \bigg)^{1/2}\bigg\|_{p}.
$$
Then we use \eqref{eq:psiTLp}. 
\end{proof}

\begin{rem} Observe that the inequality $\|\Qpsi \psi T h\|_{T^p_{2}} \lesssim \|h\|_{p}$ holds for $h\in L^p\cap L^2$ for $p$ in the above range. Indeed, $h=h_{N}+h_{R}$ where $h_{N}$ is in the null part of $T$ and $h_{R}$ in the closure of the range of $T$. We have $\Qpsi \psi T h_{N}=0 $ and the inequality applies to $h_{R}$. As $h_{R}=\IP_{T}h$ and the projection is bounded on $L^p$ by  the kernel/range decomposition, $\|h_{R}\|_{p}\lesssim \|h\|_{p}$.\end{rem}

Now the main estimate is the following:

\begin{thm}\label{thm:upperqpsip<2} For $(p_{-}(DB))_{*}<p<2$ and $\psi \in \Psi^{\gamma(p)}(S_{\mu})$,  it holds
\begin{equation}
\label{eq:p<2}
\|\Qpsi \psi {DB} h\|_{T^p_{2}} \lesssim \|h\|_{p}, \quad \forall h\in \clos{\ran_{2}(D)}.
\end{equation}

\end{thm}

The proof is quite long and will be divided in two cases: $(p_{-}(DB))_{*}>1$ and $(p_{-}(DB))_{*} \le 1$. In the first case, we go via weak type estimates and extend an argument of \cite{HMc} 
to square functions. In the second case, we use atomic theory. 

We remark that, thanks to the equivalence of norms, it is enough to show the inequality for $\psi\in \Psi_{\sigma}^\tau(S_{\mu})$ for $\sigma,\tau$ as large as one needs. We shall do this and we will not try to track their precise values. 

To treat the first case and, in fact, exponents $1<p<2$, we show the following extrapolation lemma. It is  convenient to use the notation 
$\SF_{\psi, DB}h= \SF(\Qpsi \psi {DB}h)$ where $\SF$ is the square function defined in \eqref{eq:sfdef} with $a=1$ so that $\|\SF_{\psi, DB}h\|_{p}\sim \|\Qpsi \psi {DB} h\|_{T^p_{2}}$.  Recall also that the homogeneous Sobolev space $\dot W^{1,q}$ is  the closure of the inhomogeneous Sobolev space $W^{1,q}$ for the semi-norm  $\|u\|_{\dot W^{1,q}}=\|\nabla u\|_{q}<\infty$.

The following is implicit in \cite{HMc}. 

\begin{lem}\label{lem:approx}Let $1<q<\infty$. Then $h\in \clos{\ran_{2}(D)} \cap \clos{\ran_{q}(D)}$ if and only if $h=Du$ for some 
$u\in \dot W^{1,2} \cap \dot W^{1,q}$ with $\|h\|_{q}\sim \|\nabla u \|_{q}$ and $\|h\|_{2}\sim \|\nabla u\|_{2}$.
\end{lem}

\begin{proof}     Let $h\in\clos{\ran_{2}(D)} \cap \clos{\ran_{q}(D)}$. Let $h_{k}= (I+\frac{i}{k}D)^{-1}h - (I+ikD)^{-1}h = Du_{k}$, $k\ge 1$, where $u_{k}= i(k-\frac{1}{k})(I+ikD)^{-1}(I+\frac{i}{k}D)^{-1}h$. We have 
$u_{k}\in{\dom_{2}(D)} \cap {\dom_{q}(D)}$,  $u_{k}\in\clos{\ran_{2}(D)} \cap \clos{\ran_{q}(D)}$ as resolvents preserve the closure of the range.  Also  $h_{k}  \in \ran_{2}(D) \cap \ran_{q}(D)$ and $h_{k}$ converges to $h$ is both $L^2$ and $L^q$ topologies (see Section \ref{sec:D}).  Using the coercivity property of $D$, we have $\|\nabla (u_{k}-u_{\ell})\|_{q}\lesssim \|D(u_{k}-u_{\ell})\|_{q}=\|h_{k}-h_{\ell}\|_{q}$ and similarly in $L^2$.   Taking limits of the Cauchy sequences, we obtain $u\in \dot W^{1,2} \cap \dot W^{1,q}$ with the required property. 
Conversely, let $u\in \dot W^{1,2} \cap \dot W^{1,q}$ such that $\|Du\|_{q}\sim \|\nabla u\|_{q}$ and $\|Du\|_{2}\sim \|\nabla u\|_{2}$. Then, one can find $u_{k}\in W^{1,2} \cap  W^{1,q}$ such that $\nabla u_{k}$ converges to $\nabla u$ in both $L^2$ and $L^q$ topologies.  Thus, 
 $Du_{k}\in \ran_{2}(D) \cap \ran_{q}(D)$ converges to  $Du$ in both $L^2$ and $L^q$ topologies, so that $Du\in \clos{\ran_{2}(D)} \cap \clos{\ran_{q}(D)}$. 
\end{proof}

Armed with this  lemma, the inequality \eqref{eq:p<2} is equivalent to 
$$
\|\SF_{\psi, DB}Du\|_{q} \lesssim  \|u\|_{\dot W^{1,q}}, \quad \forall u\in \dot W^{1,2} \cap \dot W^{1,q}.
$$

\begin{lem}\label{lem:extrapolation} Let $ p_{-}(DB) <q<2$. Fix $\psi\in \Psi_{\sigma}^\tau(S_{\mu})$ with $\sigma,\tau\gg1$ as needed. 
If 
$$
\|\SF_{\psi, DB}Du\|_{q} \lesssim  \|u\|_{\dot W^{1,q}}, \quad \forall u\in \dot W^{1,2} \cap \dot W^{1,q}.
$$
then for  $\max ( 1, q_{*}) <p<q$, one has 
$$
\|\SF_{\psi, DB}Du\|_{p} \lesssim  \|u\|_{\dot W^{1,p}}, \quad \forall u\in \dot W^{1,2} \cap \dot W^{1,p}.
$$
\end{lem}

Let us conclude \eqref{eq:p<2} from this.    By Lemma \ref{lem:approx}, if  $u \in \dot W^{1,2}$, then $h=Du\in \clos{\ran_{2}(D)}$, so that the inequality holds for $q=2$ by $H^\infty$-calculus. Then one can iterate Lemma \ref{lem:extrapolation} at most a finite number of times to obtain the inequality  when $\max(1, (p_{-}(DB))_{*})<p<2$. Applying Lemma \ref{lem:approx} yields the inequality  \eqref{eq:p<2} for all such $p$.

\begin{proof}[Proof of Lemma \ref{lem:extrapolation}] It is enough to show the weak type estimate
\begin{equation}
\label{eq:weaktype}
\|\SF_{\psi, DB}Du\|_{p,\infty} \lesssim  \|u\|_{\dot W^{1,p}}
\end{equation}
for $u\in \dot W^{1,2} \cap \dot W^{1,p}$. 
Indeed, one can use N. Badr's theorem \cite{badr}, which says that the homogeneous Sobolev spaces have  the real interpolation property, and interpolate with the inequality at $p=q$ for  the sublinear operator $u\mapsto \SF_{\psi, DB}Du$.   

To prove \eqref{eq:weaktype}, we use  the Calder\'on-Zygmund decomposition of Sobolev functions in \cite{Auscher}, extended straight forwardly to  $C^N$-valued functions.  

 Fix $\lambda>0$ and $u\in \dot W^{1,2} \cap \dot W^{1,p}$. Choose a  collection of cubes $\left(Q_{j}\right)$, (vector-valued) functions $g$ and $b_{j}$ such that 
$
u=g+\sum_{j}b_{j} 
$
and the following properties hold:
\begin{align}
||\nabla g||_{L^{\infty}}&\leq C\lambda, \label{eq:CZ1} \\
b_{j}\in W_{0}^{1,p}\left(Q_{j},\mathbb{C}^{N}\right) &\text{ and } \int_{Q_{j}}|\nabla b_{j}|^{p}\leq C\lambda^{p}|Q_{j}|,\label{eq:CZ2}\\
\sum_{j}|Q_{j}|&\leq C\lambda^{-p}\int_{\mathbb{R}^{n}}|\nabla u|^{p},\label{eq:CZ3}\\
\sum_{j} 1_{Q_{j}}&\leq C',\label{eq:CZ4}
\end{align}
where $C$ and $C'$ depend only on dimension and $p$. Remark that \eqref{eq:CZ2},  Sobolev-Poincar\'e inequality  with a real $r$ such that $p \le r\le p^*$, and in particular $r=q$, gives us 
\begin{equation}
\label{eq:CZ5}
||b_{j}||_{r}\lesssim |Q_{j}|^{\frac{1}{r}-\frac{1}{p}+\frac{1}{n}}||\nabla b_{j}||_{p}\lesssim \lambda |Q_{j}|^{\frac{1}{r}+\frac{1}{n}}.
\end{equation}
Also, we note that the bounded overlap \eqref{eq:CZ4} implies that $\|\sum_{j} \nabla b_{j}\|_{p}+ \|\nabla g\|_{p}\lesssim  \|\nabla u\|_{p}$,  hence for all $r\ge p$, 
\begin{equation}
\label{eq:CZ6}
\lambda^{-r}\|\nabla g\|_{r}^r\lesssim  \lambda^{-p}\|u\|_{\dot W^{1,p}}^p.
\end{equation}
In particular, this holds for $r=2$ so that we also have the qualitative bound $\|\sum_{j} \nabla b_{j}\|_{2}+ \|\nabla g\|_{2} <\infty$ and the decomposition is also in $\dot W^{1,2}$ (It follows from the construction that  $b_{j}\in  W^{
1,2}$ for each $j$).

Introduce for some integer $M>1$, chosen large enough in the course of the argument, 
\begin{align*}
\varphi\left(z\right):=\sum_{m=0}^{M}\tbinom{M}{m} \left(-1\right)^{m}\left(1+imz\right)^{-1}\in H^\infty(S_{\mu})
\end{align*}
as in \cite[Section 4]{HMc}.  This function satisfies $|\varphi\left(z\right)|\lesssim \inf(|z|^M,1)$.
We decompose $u=g+\tilde g + b$ where  $\tilde g:= \sum_{j} (I-\varphi(\ell_{j}BD))b_{j}$ and $b=\sum_{j} \varphi(\ell_{j}BD)b_{j}$  with $\ell_{j}:= \ell(Q_{j})$. 
As usual, the set $\{ \SF_{\psi,DB}Du > 3\lambda\}$ is contained in the union of 
$A_{1}=\{ \SF_{\psi,DB}Dg > \lambda\}$, $A_{2}=\{ \SF_{\psi,DB}D\tilde g > \lambda\}$ and
$A_{3}=\{ \SF_{\psi,DB}Db > \lambda\}$. 
For $A_{1}$ we use the hypothesis and \eqref{eq:CZ6}, and
$$
|A_{1}| \lesssim \lambda^{-q}\|\nabla g\|_{q}^q\lesssim  \lambda^{-p}\|u\|_{\dot W^{1,p}}^p.
$$
For $A_{2}$, we use also the hypothesis but in the form \eqref{eq:p<2} to get
$$
|A_{2}| \lesssim \lambda^{-q}\|D \tilde g\|_{q}^q\lesssim  \lambda^{-p}\|u\|_{\dot W^{1,p}}^p.
$$
For the last inequality, notice that $\|D \tilde g\|_{q} \lesssim \|BD \tilde g\|_{q}$ as $q\in I(BD)$ and
$BD\tilde g = \sum_{m=1}^M c_{m}\sum_{j} (I+im\ell_{j}BD)^{-1}BDb_{j}$, so that the inequality is shown in \cite[Section 4.1]{HMc}. 

The main new part compared to \cite{HMc} is the treatment of the set $A_{3}$, for which we follow, in part, \cite{AHM}. As usual, since $|\cup 4Q_{j} |$ is under control from \eqref{eq:CZ3}, it is enough to control the measure of $\wt A_{3}= \{ \SF_{\psi,DB}Db > \lambda\} \cap F$ where  $F= \R^n\setminus \cup 4Q_{j}$. We use then the $L^2$ Markov inequality
$$
|\wt A_{3} |\le \lambda^{-2} \int_{F}  |\SF_{\psi,DB}Db|^2  = \lambda^{-2} \iint |\psi(tDB) Db(y)|^2\, \frac{|B(y,t)\cap F|}{t^n}\, \frac{dydt}{t}.
$$
We  decompose $\psi(tDB) Db(y)= f_{loc}(t,y)+ f_{glob}(t,y)$
with $$f_{loc}(t,y)= \sum_{j } 1_{2Q_{j}}(y)\psi(tDB) D\varphi(\ell_{j}BD)b_{j}(y) $$
and $$f_{glob}(t,y)=  \sum_{j } 1_{(2Q_{j})^c}(y)\psi(tDB) D\varphi(\ell_{j}BD)b_{j}(y).$$
Let us call $I_{loc}$ and $I_{glob}$ the integrals obtained. We begin with the estimate of $I_{loc}$. If $y \in 2Q_{j}$ and $t\le 2\ell_{j}$ then 
$B(y,t) \subset 4Q_{j}$, hence $B(y,t)\cap F=\emptyset$. Thus, in $I_{loc}$ we may replace 
$f_{\loc}$ by
$$\tilde f_{loc}(t,y)= \sum_{j } 1_{2Q_{j}}(y)1_{(2\ell_{j},\infty)}(t)\psi(tDB) D\varphi(\ell_{j}BD)b_{j}(y). $$
At this point we dualize against $H$ with $\int\!\!\!\int |H(t,y)|^2\,  \frac{dydt}{t}=1$, so that using Fubini's theorem and Cauchy-Schwarz inequality
\begin{align*}
  I_{loc}^{1/2}  &  \lesssim \iint  \tilde f_{loc}(t,y) \overline {H(t,y)} \, \frac{dydt}{t}  \\
    &  \le \sum_{j} I_{j} |Q_{j}|^{1/2} \inf_{x\in Q_{j}}\MM_{2}\wt H(x), \end{align*}
  where  $$
    I_{j}^2:=\int_{2\ell_{j}}^\infty\int_{\R^n} |\psi(tDB) D\varphi(\ell_{j}BD)b_{j}(y)|^2\, \frac{dydt}{t},
$$
$\wt H(y)^2:= \int_{0}^\infty |H(t,y)|^2\,  \frac{dt}{t}$, $\MM_{2}\wt H:= (\MM |\wt H|^2)^{1/2}$ and $\MM$ is the Hardy-Littlewood maximal operator. 
We have $\psi(tDB) D\varphi(\ell_{j}BD)b_{j}= D\psi(tBD) \varphi(\ell_{j}BD)b_{j} $ (remark that $b_{j}\in L^2$, so we can use functional calculus and  the commutation holds) and using the accretivity of $B$ and \eqref{eq:odnpsipq}
$$
\|D\psi(tBD) \varphi(\ell_{j}BD)b_{j}\|_{2}\lesssim \|BD\psi(tBD) \varphi(\ell_{j}BD)b_{j}\|_{2}
\lesssim t^{-1} t^{\frac{n}{2}-\frac{n}{q}}\|b_{j}\|_{q}.
$$
It follows easily using \eqref{eq:CZ2} that 
$$
I_{j}\lesssim \lambda |Q_{j}|^{1/2}.
$$
It is classical from Kolmogorov's inequality and the weak type (1,1) of $\MM$ that 
\begin{equation}
\label{eq:kolmo}
\sum_{j}  |Q_{j}| \inf_{x\in Q_{j}}\MM_{2}\wt H(x) \lesssim  \|\wt H\|_{2}^{1/2} |\cup Q_{j}|^{1/2}=  |\cup Q_{j}|^{1/2}.
\end{equation}
Altogether, we conclude using \eqref{eq:CZ3} that 
$$\lambda^{-2}I_{loc} \lesssim   |\cup Q_{j}| \lesssim \lambda^{-p} \|u \|_{\dot W^{1,p}}^p.
$$
We next turn to $I_{glob}$. Using the same dualization argument, we have for some $H$ as above, 
\begin{align*}
  I_{glob}^{1/2}    \lesssim  \sum_{j,r\ge 1} I_{j,r} 2^{r{n}/{2}} |Q_{j}|^{1/2} \inf_{x\in Q_{j}}\MM_{2}\wt H(x), \end{align*}
  where  $$
    I_{j,r}^2:=\int_{0}^\infty\int_{S_{r}(2Q_{j})} |\psi(tDB) D\varphi(\ell_{j}BD)b_{j}(y)|^2\, \frac{dydt}{t},
$$
and we use the notation $S_{r}(Q)$ introduced for molecules. Since the integrals are localized we cannot use the same argument as before by using the accretivity of $B$ on the range. Nevertheless, we prove a  local version in the following lemma, which will be used many times later on.

\begin{lem}\label{lem:localcoerc}[Local coercivity inequality] For any $u\in L^2_{\loc}$ with $Du\in L^2_{loc}$,  any ball $B(x,r)$ in $\R^n$ and $c>1$, 
\begin{equation} 
\label{eq:localcoerc}
\int_{B(x,r)} |Du|^2  \lesssim \int_{B(x,cr)} |BDu|^2  + r^{-2} \int_{B(x,cr)} | u|^2,
\end{equation}
with the implicit constant depending only on the ellipticity constants of $B$, dimension,  $N$ and $c$. 
\end{lem}

We postpone the proof of  the lemma.  As
$\psi(tDB) D\varphi(\ell_{j}BD)b_{j}= D\psi(tBD) \varphi(\ell_{j}BD)b_{j}$, we can apply it to $u_{j}=\psi(tBD) \varphi(\ell_{j}BD)b_{j}$, which leads to bound $I_{j,r}^2$ by two integrals  with slightly larger regions $\wt S_{r}(2Q_{j})$ of the same type as $S_{r}(2Q_{j})$ and with integrands $|BDu_{j}|^2$ and $|(2^{-r}\ell_{j })^{-1}u_{j}|^2$ respectively.  We then truncate both integrals at $\ell_{j}$. For  $t\le \ell_{j}$, using the $L^q-L^2$ off-diagonal estimate \eqref{eq:odnpsipq} (which requires $\tau$ large enough), 
$$
\int_{\wt S_{r}(2Q_{j})} |BDu_{j}(y)|^2\, {dy} \lesssim  t^{-2}t^{\frac{2n}{2}-\frac{2n}{q}}\brac{2^r \ell_{j}/t}^{-2\sigma c}\|b_{j}\|_{q}^2
$$
which, using \eqref{eq:CZ5}, leads to 
$$
\int_{0}^{\ell_{j}}\int_{\wt S_{r}(2Q_{j})} |BDu_{j}(y)|^2\, \frac{dydt}{t} \lesssim  \ell_{j}^{-2+ \frac{2n}{2}-\frac{2n}{q}} 2^{-2r\sigma c} \|b_{j}\|_{q}^2 \lesssim 2^{-2r\sigma c} \lambda^2 |Q_{j}|.
$$
The argument for  $(2^{-r}\ell_{j })^{-1}u_{j}$ replacing $BDu_{j}$ is the same if $q<2$ and leads to a similar  estimate with $1+\sigma c $ in place of $\sigma c$. If $q=2$, we may  use  an $L^{s}-L^2$ estimate for some  $s<2$ instead and \eqref{eq:CZ5} for $\|b_{j}\|_{s}$. 

When $t\ge \ell_{j}$, we deduce from \eqref{eq:odnpsiphipq} (provided $\tau $ is large enough)
$$
\int_{\wt S_{r}(2Q_{j})} |BDu_{j}(y)|^2\, {dy} \lesssim  t^{-2}t^{\frac{2n}{2}-\frac{2n}{q}}\brac{2^r \ell_{j}/\ell_{j}}^{-2Mc}\|b_{j}\|_{q}^2
$$
and then
$$
\int_{\ell_{j}}^{\infty}\int_{\wt S_{r}(2Q_{j})} |BDu_{j}(y)|^2\, \frac{dydt}{t} \lesssim  2^{-2rMc} \lambda^2 |Q_{j}|.
$$
The argument for $(2^{-r}\ell_{j })^{-1}u_{j}$ replacing $BDu_{j}$ is the same if $q<2$ and leads to a similar  estimate with $1+ Mc$ in place of $Mc$. If $q=2$, we may use an $L^{s}-L^2$ estimate for some  $s<2$ instead and \eqref{eq:CZ5} for $\|b_{j}\|_{s}$. 

In total, we obtain an estimate 
$$
I_{j,r}\lesssim \sum_{j,r\ge 1} 2^{-rK} \lambda |Q_{j}|^{1/2},
$$
where $K$ can be arbitrary large (upon choosing $\sigma,M$ large) so that using \eqref{eq:kolmo}
$$
 I_{glob}^{1/2}    \lesssim  \sum_{j,r\ge 1} \lambda 2^{r(\frac{n}{2}-K)} |Q_{j}| \inf_{x\in Q_{j}}\MM_{2}\wt H(x) \lesssim \lambda |\cup Q_{j}|^{1/2} 
 $$
 and the desired conclusion follows.   \end{proof} 
 
 \begin{proof}[Proof of lemma \ref{lem:localcoerc}]
For this inequality, we let $\chi$ be a  scalar-valued cut-off function with $\chi=1$ on $B(x,r)$, supported in $B(x,cr)$  and with $\|\nabla \chi\|_{\infty}\lesssim r^{-1}$. As $\chi u \in \dom(D)$ and using that the  commutator  between $\chi$ and $D$ is the pointwise multiplication by a  matrix with bound controlled by $|\nabla \chi|$, 
$$
\int_{B(x,r)} |Du|^2\,   \leq \int_{\R^n} |\chi Du|^2\,    \lesssim  \int_{\R^n}  |D(\chi u)|^2\,   + \int_{\R^n}  |\nabla\chi |^2| u|^2\,  .
$$
Since $B$ is accretive on $\ran_{2}(D)$, we have $ \int_{\R^n}  |D(\chi u)|^2 \lesssim  \int_{\R^n}  |BD(\chi u)|^2$. Now, we use again the commutation between $\chi$ and $D$ together with $\|B\|_{\infty}$. This proves \eqref{eq:localcoerc}.
\end{proof}

To continue the proof of Theorem \ref{thm:upperqpsip<2}, we have to consider the case $p\le 1$, which occurs only when $(p_{-})_{*} <1$. In this case, it is enough to consider a $\left(\IH^{p}_{D},1\right)$-atom $a=Db$ with $a,b$ supported in a cube $Q$ and show that $\|\SF_{\psi,DB}a\|_p
\lesssim 1$ uniformly for some $\psi\in \Psi_{\sigma}^\tau(S_{\mu})$ with $\sigma,\tau$ as large as one needs. 

As usual the local term is handled by the $L^2$ bound
$$
\|\SF_{\psi,DB}a\|_{L^p(4Q)}\le |4Q|^{\frac{1}{p}-\frac{1}{2}}\|S_{\psi,DB}a\|_{L^2(4Q)} \lesssim |4Q|^{\frac{1}{p}-\frac{1}{2}} \|a\|_{2}\lesssim 1.
$$
Next, for the non-local term, we remark that if $x\notin 4Q$ and $t\in (0,\infty)$, then $\brac{\dist (B(x,2t),Q)/t}\geq  C \brac{\dist (x,Q)/t}$. Using $\psi(tDB)a= \psi(tDB)Db= D\psi(tBD) b$,  the local coercivity inequality \eqref{eq:localcoerc}    and $L^q-L^2$ off-diagonal estimates \eqref{eq:odnpsipq} (provided $\tau$ is large enough), we have 
\begin{align*}
\|\psi(tDB)a\|_{L^2(B(x,t))}  
  & \lesssim \|BD\psi(tBD)b\|_{L^2(B(x,2t))} + t^{-1}\|\psi(tBD)b\|_{L^2(B(x,2t))}  \\
    &  \lesssim  t^{-1} t^{\frac{n}{2}-\frac{n}{q}}\brac{\dist (x,Q)/t}^{-K}\|b\|_{q},
\end{align*}
 where $K$ can taken as large as one wants upon taking $\sigma$ large, and one chooses $q$ with $p_{-}<q<p^*$ and $q\le 2$, which is possible as $(p_{-})_{*}<p\le 1$. Thus, for  $x\notin 4Q$
 $$
\SF_{\psi,DB}a(x) \lesssim   (d(x,Q))^{-1-\frac{n}{q}} \|b\|_{q}.
$$
As $q<p^*$, it follows that  $1+\frac{n}{q}>\frac{n}{p}$, so one can integrate the $p$th power and get
$$
\|\SF_{\psi,DB}a\|_{L^p((4Q)^c)} \lesssim \ell(Q)^{-1-\frac{n}{q}+\frac{n}{p}} \|b\|_{q} \lesssim 1,
$$
where the last inequality is merely  H\"older's inequality and $\|b\|_{2}\lesssim \ell(Q)^{1+\frac{n}{2}-\frac{n}{p}}.
$

We have obtained all the upper bounds in Theorem \ref{thm:hpdb}. We complete the proof by proving the lower bounds. 

\subsubsection{Lower bounds}

Those have already been obtained in Proposition \ref{cor:lowerbounddbp<2} for $\frac{n}{n+1}<p<2$ and we remark that $
(p_{-})_{*}>\frac{n}{n+1}$. It remains to see them for $2<p<p_{+}$. We have seen in Proposition  
\ref{prop:duality} that 
for all $h\in \clos{\ran_{2}(D)}$ and $g\in \clos{\ran_{2}(B^*D)}$ and any $\psi,\varphi\in \Psi(S_{\mu})$ for which the Calder\'on reproducing formula \eqref{eq:Calderon}  holds, one has
$|\pair h g | \le \| \Qpsi \psi {DB}h\|_{T^p_{2}}\|\Qpsi {\varphi^*} {B^*D} g\|_{T^{p'}_{2}}.$
Now, we have ${\varphi^*}(tB^*D)g=B^*{\varphi^*}(tDB^*)(B^*)^{-1}g$. Using that 
$B^*$ is bounded, $p'\in I(DB^*)=I(BD)'$ since $p\in I(BD)$ and $B^*$ is an isomorphism from $\clos{\ran_{p'}(D)}$ onto $\clos{\ran_{p'}(B^*D)}$,
\begin{equation}
\label{eq:upperboundb*d}
\|\Qpsi {\varphi^*} {B^*D} g\|_{T^{p'}_{2}} \lesssim \|\Qpsi {\varphi^*} {DB^*} (B^*)^{-1}g\|_{T^{p'}_{2}}
\lesssim \|(B^*)^{-1}g\|_{p'}\sim \|g\|_{p'},
\end{equation}
provided $\varphi$ is allowable for $\IH^{p'}_{DB^*}$ which is the case if we choose, as we may,   $\varphi\in \Psi^{\gamma(p')}(S_{\mu})$. Thus 
$$|\pair h g | \le \| \Qpsi \psi {DB}h\|_{T^p_{2}} \|g\|_{p'}.
$$
Now, from \cite{AS}, Proposition 2.1, (5), $\clos{\ran_{p}(D)}$ and $\clos{\ran_{p'}(B^*D)}$ are dual spaces for the $L^2$ pairing: this and a density argument yield $\|h\|_{p}\lesssim \| \Qpsi \psi {DB}h\|_{T^p_{2}}$. This completes the proof of Theorem \ref{thm:hpdb}.

\subsection{Proof of Theorem \ref{thm:hpbd}}

Before we move to the proof, let us explain the ranges of $p$ and $\alpha$.  In Theorem \ref{thm:hpdb}, the range for $q$ for $\IH^q_{DB^*}=\IH^q_{D}$  is   $(p_{-}(DB^*))_{*} <q<p_{+}(DB^*)$. But $p_{+}(DB^*)'=p_{-}(BD)$ and $p_{-}(DB^*)'=p_{+}(BD)$, so this is the range 
$(p_{+}(BD)')_{*} <q<p_{-}(BD)'$. If $p_{+}(BD)\le n$, we have $(p_{+}(BD)')_{*}= (p_{+}(DB^*)^*)'$
(with $n^*=\infty$ by convention).  If $p_{+}(BD)> n$, then we obtain the range $[0, \alpha(BD))$ with $\alpha(BD)=n(\frac{1}{(p'_{+}(BD))_{*}}-1)= 1- \frac{n}{p_{+}(DB)}$. In all,  we obtain the ranges for $p$ and $\alpha$ specified in the statement.

\subsubsection{Lower bounds}\label{sssec:lowerboundshpbd}

The lower bounds of the tent space norms $\|\Qpsi \varphi {BD} h\|_{\mT} $ by norms on $\IP h$ is a modification of the arguments in Proposition \ref{prop:lowerbounds}.   For example, for $p=q'$ and $q>1$, take $\psi,\varphi\in \Psi(S_{\mu})$ for which the Calder\'on reproducing formula \eqref{eq:Calderon}  holds. then 
\begin{align*}
\|\IP g\|_{p} & = \sup\{ |\pair {\IP g} f|\, ; \, \|f\|_{\IH^q_{D}}\lesssim 1\}\\
 & \sim \sup\{ |\pair g f|\, ; \, \|f\|_{\IH^q_{D}}\lesssim 1\} \\
 & \le \sup\{ \| \Qpsi \psi {BD}f\|_{T^p_{2}} \|\Qpsi {\varphi^*} {DB^*} g\|_{T^{q}_{2}}\, ; \, \|f\|_{\IH^q_{D}}\lesssim 1\} \\
    &  \lesssim  \sup\{\| \Qpsi \psi {BD}g\|_{T^p_{2}}  \|f\|_{q}\,  ; \, \|f\|_{\IH^q_{D}}\lesssim 1\}  \\
    & \lesssim \| \Qpsi \psi {BD}g\|_{T^p_{2}}.
\end{align*}  
The fourth line holds provided we also choose $\varphi$ allowable for $\IH^q_{DB^*}$ while $\psi$ can be arbitrary.

The same argument holds when $q\le 1$, working in the H\"older spaces $\IL^\alpha_{BD}$ and $\IL^\alpha_{D}$ and corresponding tent space $T^\infty_{2,\alpha}$.

\subsubsection{Upper bounds}

For $p_{-}<p<2$, we have just seen the desired upper bound in \eqref{eq:upperboundb*d} up to changing $p'$ to $p$ and $B^*$ to $B$. 

Proposition \ref{prop:upperbdp>2} takes care of the case $2<p<p_{+}$. 

Next, we consider the case $p_{+}\le p<(p_{+})^*$.
We adapt an argument of \cite{AHM}  which works for both $BD$ or $DB$.  Let  $\psi\in \Psi_{\sigma}^\tau(S_{\mu})$ with $\sigma>0$ and $\tau>0$. Recall that $\modz= \sgn(z) z$. Consider for $\alpha\in \C$ with $\re \alpha>0$, 
$$\psi_{\alpha}(z)= \frac{\modz^{\alpha-\sigma}}{(1+\modz)^{\alpha-\sigma}}\ \psi(z).$$ Remark that
$$
\frac{\modz^\alpha}{(1+\modz)^\alpha}= (1+\modz^{-1})^{-\alpha}
$$
and since $z\in S_{\mu}$ implies $\modz, \modz^{-1} , 1+\modz^{-1} \in S_{\mu+}$, we have that
$$
\sup_{z\in S_{\mu}} \left|\frac{\modz^\alpha}{(1+\modz)^\alpha}\right| \le e^{\mu |\im \alpha|}.
$$
It follows that $\psi_{\alpha}\in \Psi_{\re \alpha}^\tau(S_{\mu})$ with 
$$
|\psi_{\alpha}(z)| \le Ce^{\mu |\im \alpha|} \inf( |z|^{\re \alpha}, |z|^{-\tau}).
$$
Clearly, the map $\alpha\mapsto \psi_{\alpha}$ is analytic from $\re \alpha>0$ to $\Psi(S_{\mu})$ with $\psi=\psi_{\sigma}$.  

For $T=DB$ of $BD$, set 
$$
Q_{\alpha}f = \Qpsi {\psi_{\alpha}} T f = (\psi_{\alpha}(tT)f)_{t>0}, \quad f\in L^2.
$$
Thus $Q_{\alpha}$ is an analytic family of bounded operators from $L^2$ to $T^2_{2}$ with 
$$
\|Q_{\alpha}f\|_{T^p_{2}}\lesssim  e^{\mu |\im \alpha|} \|f\|_{2}.
$$
 In the statements below, implicit or explicit constants $C$  are allowed to depend on the real part of $\alpha$ but not on its imaginary part.

\begin{lem} For $\re \alpha>0$, $Q_{\alpha}$ maps $L^p\cap L^2$ to $T^p_{2}$ when $2\le p<p_{+}$ with $\|Q_{\alpha}f\|_{T^{p}_{2}}\lesssim  e^{\mu |\im \alpha|} \|f\|_{p}$.
\end{lem}

\begin{proof} This is a reformulation of  Proposition \ref{prop:upperbdp>2} together with the remark that follows it. We note that the control of 
the norm with $e^{\mu |\im \alpha|}$ comes from examination of the proof of  Le Merdy's theorem \cite{LeM} to get \eqref{eq:psiTLp}. 
\end{proof}

\begin{lem}\label{lem:sfpinfty} For $\re \alpha> \frac{n}{p_{+}(T)}$, $Q_{\alpha}$ maps $L^p\cap L^2$ to $T^p_{2}$ when $2\le p\le \infty$ with $\|Q_{\alpha}f\|_{T^{p}_{2}}\lesssim  e^{\mu |\im \alpha|} \|f\|_{p}$.

\end{lem}

\begin{proof} For fixed $\alpha$ it is enough to consider the case $p=\infty$ as one can then complex interpolate from \cite{CMS} between $T^2_{2}$ and $T^\infty_{2}$.
  We claim that for any $2< q<p_{+}$,  and any ball $B_{r}$ of $\R^n$, with radius $r$,  setting $\Omega=(0,r)\times     B_{r}$,
\begin{multline}\label{varphi}
\left(
\frac1{|B_{r}|}\,\iint_{\Omega} |\psi_{\alpha}(tT) f(x)|^2\,\frac{dtdx}{t}
\right)^{1/2}
\\
\le
Ce^{\mu|\im \alpha|}\,\sum_{j=1}^\infty 2^{-j\,(\re\alpha-\frac{n}{q})}\,\left(\barint_{2^jB_{r}}|f|^q\right)^{1/q}.
\end{multline}
  Admitting this claim,  the right hand side is dominated by the $L^\infty$ norm of $f$ by using $\re \alpha> \frac{n}{p_{+}}$ and choosing $q<p_{+}$ appropriately. Then the supremum over all $B_{r}$ of the left hand side is precisely the $T^\infty_{2}$ norm of $Q_{\alpha}f$.
  
To  prove the claim, we write $f=f_{\rm loc}+f_{\rm glob}$ where $f_{\rm loc}=f\,1_{4B_{r}}$. Then, using the $L^2-T^2_{2}$ boundedness of $Q_{\alpha}$,
\begin{multline*}
\frac1{|B_{r}|}\,\iint_{\Omega} |\psi_{\alpha}(tT) f_{\rm loc}(x)|^2\,\frac{dtdx}{t}
\le
\frac1{|B_{r}|}\int_{\R^n} \Big(\int_0^\infty |\psi_{\alpha}(tT) f_{\rm loc}(x)|^2\,\frac{dt}{t}\Big)\,dx
\\
\lesssim
\frac1{|B_{r}|}\int_{\R^n} |f_{\rm loc}|^2
\lesssim
\barint_{4B_{r}} |f|^2.
\end{multline*}

It is then enough to show
$$
\bigg(\barint_{B_{r}} |\psi_{\alpha}(tT)f_{\rm glob}|^2\bigg)^{1/2}
\le
 Ce^{\mu|\im \alpha|}\,\frac{t^{\re\alpha}}{r^{\re \alpha}}\,
\sum_{j=2}^\infty 2^{-j\,(\re \alpha-\frac{n}{q})}\,\left(\barint_{2^{j+1}B_{r}}|f|^q\right)^{1/q}.
$$
Indeed, plugging this estimate in the integral on the Carleson region $\Omega$, we obtain the claim. 

To this end,  we set $f_j=f\,1_{S_j(B_{r})}$,  so that $f_{\rm glob} = \sum_{j\ge 3} f_{j}$ and by Minkowski's and H\"older's inequalities
$$
\bigg(\barint_{B_{r}} |\psi_{\alpha}(tT) f_{\rm glob}|^2\bigg)^{1/2} \le \sum_{j\ge 3} \bigg(\barint_{B_{r}} |\psi_{\alpha}(tT) f_{j}|^q\bigg)^{1/q}.$$
 Fix $j\ge 3$ and use \eqref{eq:odnpsipq}  with $p=q$ to obtain
$$
\bigg(\barint_{B_{r}} |\psi_{\alpha}(tT) f_{j}|^q \bigg)^{1/q} \le
Ce^{\mu|\im \alpha|}\,\frac{t^{\re\alpha}}{r^{\re \alpha}}\,
2^{-j\,(\re \alpha-\frac{n}{q})}\,\left(\barint_{2^{j+1}B_{r}}|f|^q\right)^{1/q}.
$$
The claim is proved. 
\end{proof}

\begin{lem} For $0<\re \alpha \le  \frac{n}{p_{+}}$, $Q_{\alpha}$ maps $L^p\cap L^2$ to $T^p_{2}$ when $2\le p< \frac{np_{+}}{n-p_{+}\re \alpha} $ with $\|Q_{\alpha}f\|_{T^{p}_{2}}\lesssim  e^{\mu |\im \alpha|} \|f\|_{p}$. 
\end{lem}

\begin{proof}  This is \textit{verbatim} the interpolation argument in \cite{AHM}.
\end{proof}

\begin{cor}\label{cor:psi1} For $p_{+}\le p$ and  $\psi\in \Psi_{\frac{n}{p}-\frac{n}{p_{+}}}(S_{\mu})$, then  $\Qpsi \psi T $ maps $L^p\cap L^2$ to $T^p_{2}$  with $\|\Qpsi \psi T f\|_{T^{p}_{2}}\lesssim   \|f\|_{p}$. In particular, if $\psi\in \Psi_{1}(S_{\mu})$,  then $\Qpsi \psi T $ maps $L^p\cap L^2$ to $T^p_{2}$  when $2\le p< (p_{+}(T))^* $.
\end{cor}

\begin{proof}
This is an easy consequence of the previous construction when $\sigma$ is any number larger than $\frac{n}{p}-\frac{n}{p_{+}}$ to start with. We leave details to the reader. 
\end{proof}

This corollary proves  the part of Theorem \ref{thm:hpbd} that concerns upper bounds for $T=BD$ and $2<p<(p_{+})^*$. 

To finish the proof of Theorem \ref{thm:hpbd}, it suffices to prove the following stronger result. 

\begin{prop}\label{prop:p+>n} If $p_{+}=p_{+}(BD)>n$, then for $0\le \alpha<1-\frac{n}{p_{+}}$,  $$
\|\Qpsi \psi {BD} h\|_{T^\infty_{2,\alpha}} \lesssim \|h\|_{\dot \Lambda^s}, \quad \forall h\in \dot \Lambda^\alpha\cap L^2,
$$
when $\psi \in \bigcup_{ \sigma> \alpha+\frac{n}{p_{+}}, \tau>0}\Psi_{\sigma}^\tau(S_{\mu})$ and in particular for $\psi\in  \Psi_{1}(S_{\mu})$. 
\end{prop}

\begin{proof}
We observe that \eqref{varphi} applies to $\psi$ replacing $\re \alpha$ by $\sigma$ and in the right hand side $h$ by $h-c$ where $c$ is any constant. Indeed, constants are annihilated by  $BD$, or more concretely $\psi(tBD)c=0$. The action of $\psi(tBD)$ on $L^\infty$ is guaranteed by 
Corollary \ref{cor:extension} applied with $q$ close to $p_{+}$ and $\sigma>\frac{n}{p_{+}}$.
Thus the left hand side of  \eqref{varphi} remains the same. 
Now, we choose $c$ to be the mean value of $f$ on $B_{r}$. When $h\in \dot \Lambda^\alpha$, a telescoping argument yields
$\left(\barint_{\hspace{2pt} 2^{j+1}B_{r}}|h - c|^q\right)^{1/q} \lesssim \|h\|_{\dot \Lambda^\alpha} 2^{j\alpha}r^\alpha.
$
Thus the series in $j$ converges as long as $\sigma-\frac{n}{q}+\alpha>0$, which is possible  since 
$\sigma> \alpha+\frac{n}{p_{+}}$ and choosing $q<p_{+}$ close to $p_{+}$,  and we obtain the desired conclusion when $\alpha>0$.

The same argument works for $h\in \BMO=\dot \Lambda^0$, and $2^{j\alpha}$ is replaced by $\ln (j+1)$. 
\end{proof} 

\subsection{Proof of Theorem \ref{thm:goodhpbd}} 
 
\subsubsection{Lower bounds}  The argument is the same as for Theorem \ref{thm:hpbd} in Section \ref{sssec:lowerboundshpbd}.

\subsubsection{Upper bounds}   
We begin with the case $2<q<p_{+}(DB^*)$, that is $p_{-}(BD)<p<2$. Then $\psi \in \Psi^\tau(S_{\mu})$ is allowable for $\IH^p_{BD}$ when $\tau>\gamma(p)$, which is the case as $\gamma(q)=\gamma(p)$. 

We turn to $q<2$.  We proceed with the following lemma. 

\begin{lem}\label{lem:r1} Let $\phi\in \mR^{k}(S_{\mu})$, $k=1,2$, with $\phi(0)=0$. Then for all $2<p<\infty$
$$
\|\Qpsi \phi {BD} h\|_{T^p_{2}} \lesssim \|\IP h\|_{L^p}, \quad \forall h\in L^2,
$$
and for all $0\le \alpha<1$, 
$$
\|\Qpsi \phi {BD} h\|_{T^\infty_{2,\alpha}} \lesssim \|\IP h\|_{\dot \Lambda^\alpha}, \quad \forall h\in L^2.
$$
\end{lem}

\begin{proof} The proof is basically  the same as for Lemma \ref{lem:sfpinfty}. Let $h\in L^2$. Fix a ball  $B_{r}$, with radius $r$ and set $\Omega=(0,r)\times     B_{r}$. Using that we have $L^2$ off-diagonal decay of any order $N\ge 1$ for the resolvent and its iterates, and $\phi(0)=0$ so that we have a square function estimate with $\phi(tBD)$,  we obtain as in \eqref{varphi}
\begin{equation}\label{varphi2}
\left(
\frac1{|B_{r}|}\,\iint_{\Omega} |\phi(tBD) h(x)|^2\,\frac{dtdx}{t}
\right)^{1/2}
\lesssim \sum_{j=1}^\infty 2^{-j\,(N-\frac{n}{2})}\,\left(\barint_{2^jB_{r}}|h|^2\right)^{1/2}.
\end{equation}
Taking $N>\frac{n}{2}$, this shows that $\|\Qpsi \phi {BD}h\|_{T^\infty_{2}}\lesssim \|h\|_{\infty}$ for all $h\in L^\infty\cap L^2$. 
Interpolating with the $L^2 \to T^2_{2}$ estimate, we obtain the $T^p_{2}$ estimate for all $h\in  L^2\cap L^p$.  
 Since $\phi(0)=0$, $\phi(tBD)h= \phi(tBD)\IP h$ and replacing $h$ by $\IP h$,  the $T^p_{2}$  estimate
 $\|\Qpsi \phi {BD} h\|_{T^p_{2}} \lesssim \|\IP h\|_{L^p}$ holds for $2<p<\infty$ and $h\in L^2$. 
 
Now, letting $f=\IP h  -  \barint_{B_{r}} \IP h$, we have $\phi(tBD)h= \phi(tBD)f$. Here we used that $\phi(tBD)$ maps $L^\infty$ to $L^2_{loc}$ and annihilates constants. Applying \eqref{varphi2} with $f$ replacing $h$, and using that 
$   \big(\barint_{\, \, 2^jB_{r}}|f |^2\big)^{1/2} \lesssim  2^{j\alpha}r^{\alpha} \|\IP h \|_{\dot \Lambda^\alpha}$ if $\alpha>0$ and $\lesssim \ln (1+j) \|\IP h\|_{\dot \Lambda^0}$ (with convention $\dot \Lambda^0=\BMO$ if $\alpha=0$) we obtain 
$$
\left(
\frac1{|B_{r}|}\,\iint_{\Omega} |\phi(tBD) h(x)|^2\,\frac{dtdx}{t}
\right)^{1/2} \lesssim \sum_{j=1}^\infty 2^{-j\,(N-1-\frac{n}{2})} r^\alpha \|\IP h \|_{\dot \Lambda^\alpha}
$$
 and we are done provided we choose $N>1+\frac{n}{2}$.
 \end{proof}

We turn to prove  the upper bounds in Theorem \ref{thm:goodhpbd}. As we assume    
$\IH^q_{DB^*}= \IH^q_{D}$,  Corollary \ref{cor:hpbd0}  implies that  for $h\in \clos{\ran_{2}(BD)}$,  $\|\IP \chi^\pm(BD)h\|_{p}\lesssim \|\IP h\|_{p}$ for $p=q'$  if $q>1$ or $\|\IP \chi^\pm(BD)h\|_{\dot \Lambda^\alpha}\lesssim \|\IP h\|_{\dot \Lambda^\alpha}$ for $\alpha=n(\frac{1}{q}-1)$ if $q\le 1$. 

Next, let $\psi\in \mR^1_{\sigma}(S_{\mu})\cap \Psi_{1}^\tau(S_{\mu})$ with $\sigma>\gamma(q)$  and construct $\phi_{\pm}\in \mR^{1}(S_{\mu})$ such that 
\begin{equation*}
|\psi(z)-\phi_{\pm}(z)|=O(|z|^\sigma), \quad \forall z\in S_{\mu\pm}.
\end{equation*}
Remark that necessarily, $\phi_{\pm}(0)=0$. 
The key point is the following observation: the functions $\psi_{\pm}=(\psi-\phi_{\pm})\chi^\pm
\in \Psi_{\sigma}^{\inf (1,\tau)}(S_{\mu})$ and,  for  $h\in \clos{\ran_{2}(BD)}$, using $h=\chi^{+}(BD)h+ \chi^{-}(BD)h$, we have the decomposition
$$
\psi(tBD) h = \psi_{+}(tBD) h + \psi_{-}(tBD)h+ \phi_{+}(BD)(\chi^+(BD)h)+  \phi_{-}(BD)(\chi^-(BD)h).
$$
Now, the condition $\sigma>\gamma(q)$ implies that $\psi_{\pm}$ are allowable for $\IH^\mT_{BD}$ where $\mT=T^p_{2}$ if $p=q'$ and for $\mT=T^\infty_{2,\alpha}$ if $\alpha=n(\frac{1}{q}-1)$.
 In the case $q=p'$ we deduce from this and Lemma \ref{lem:r1}  
$$ 
\|\Qpsi \psi {BD} h \|_{T^p_{2}} \lesssim \|\IP h \|_{p}+  \|\IP h \|_{p}+ \|\IP \chi^+(BD)h\|_{p} + \|\IP \chi^-(BD)h\|_{p} \lesssim \|\IP h \|_{p}.
$$
The argument when $q\le 1$ and $\alpha= n(\frac{1}{q}-1)$ is similar. This completes the proof of 
the upper bounds in Theorem \ref{thm:goodhpbd}.

\subsection{Proof of Theorem \ref{thm:supergoodhpbd}}

\subsubsection{Lower bounds}

The lower bounds of the tent space norms $\|tD\varphi (tBD) h\|_{\mT} $ by norms on $\IP h$ is again a  modification of the arguments in Proposition \ref{prop:lowerbounds}.  
Take $\psi,\varphi$ for which the Calder\'on reproducing formula \eqref{eq:Calderon}  holds.  Here we take $\varphi \in H^\infty(S_{\mu})$ and $\psi(z)= z\tilde \psi(z)$ where $\tilde \psi$ is allowable  for $H^q_{DB^*}$. 
We observe that for 
$g\in \clos{\ran_{2}(BD)}$ and  $f\in \clos{\ran_{2}(D)}$,
\begin{align*}
\pair g f &= \int_{0}^\infty \pair {\varphi(tBD) g}{tDB^*\tilde\psi^*(tDB^*) f} \, \frac{dt}{t }  \\
    &  =   \int_{0}^\infty \pair {tD\varphi(tBD) g}{B^*\tilde\psi^*(tDB^*) f} \, \frac{dt}{t } 
\end{align*}
using the self-adjointness of $D$. 

Now we may proceed as in the proof of Theorem \ref{thm:goodhpbd}. For $p=q'$ and $q>1$, 
\begin{align*}
\|\IP g\|_{p} & \sim \sup\{ |\pair {\IP g} f|\, ; \, \|f\|_{\IH^q_{D}}\lesssim 1\}\\
 & =  \sup\{ |\pair g f|\, ; \, \|f\|_{\IH^q_{D}}\lesssim 1\} \\
 & \le \sup\{ \| tD \varphi (tBD)g\|_{T^p_{2}} \|B^*\|_{\infty} \|\Qpsi {\tilde\psi^*} {DB^*} f\|_{T^{q}_{2}}\, ; \, \|f\|_{\IH^q_{D}}\lesssim 1\} \\
    &  \lesssim  \sup\{\| tD \varphi (tBD)g\|_{T^p_{2}}  \|f\|_{q}\,  ; \, \|f\|_{\IH^q_{D}}\lesssim 1\}  \\
    & \lesssim \|tD \varphi (tBD)g\|_{T^p_{2}}.
\end{align*}  
The fourth line holds since we chose $\tilde\psi$ allowable for $\IH^q_{DB^*}$. 

The same argument holds when $q\le 1$, working in the H\"older space $\IL^\alpha_{BD}$ and corresponding tent space $T^\infty_{2,\alpha}$. 

\subsubsection{Upper bounds} 

We begin with the case $2<q<p_{+}(DB^*)$, that is $p_{-}(BD)<p<2$ and $\phi \in \Psi_{0}^\tau(S_{\mu})$. Now, for $h\in \clos{\ran_{2}(BD)}$, 
$$
tD \phi ({tBD}) h= tD \phi ({tBD}) BB^{-1} h= tDB\phi(tDB) (B^{-1}h).
$$
As $B^{-1}h \in \clos{\ran_{2}(D)}$, $z\phi\in \Psi_{1}^{\tau-1}(S_{\mu})$ with $\tau-1>\gamma(p)$, we can use  Theorem \ref{thm:hpdb} and then the invertibility of $B:\clos{\ran_{p}(D)} \to \clos{\ran_{p}(BD)}$ to obtain
$$
\|tD \phi ({tBD}) h\|_{T^p_{2}} = \|tDB\phi(tDB) (B^{-1}h)\|_{T^p_{2}}\lesssim \|B^{-1}h\|_{p}\lesssim \|h\|_{p}.
$$

We turn to $q<2$. 

\begin{lem}\label{lem:r2} Let $\phi\in \mR^{2}(S_{\mu})$. Then for all $2<p<\infty$
$$
\|tD \phi ({tBD}) h\|_{T^p_{2}} \lesssim \|\IP h\|_{L^p}, \quad \forall h\in L^2,
$$
and for all $0\le \alpha<1$, 
$$
\|tD \phi ({tBD}) h\|_{T^\infty_{2,\alpha}} \lesssim \|\IP h\|_{\dot \Lambda^\alpha}, \quad \forall h\in L^2.
$$
\end{lem}

\begin{proof} It suffices to do it for $\phi(tBD)=(I+itBD)^{-2}$. The proof is roughly the same as for 
Lemma \ref{lem:r1} (playing with the projection and constants) as soon as we establish the following: Let $h\in L^2$. Fix a ball  $B_{r}\subset \R^n$, with radius $r$, set $\Omega=(0,r)\times     B_{r}$, then 
\begin{equation}\label{varphi3}
\left(
\frac1{|B_{r}|}\,\iint_{\Omega} |tD\phi(tBD) h(x)|^2\,\frac{dtdx}{t}
\right)^{1/2}
\lesssim \sum_{j=1}^\infty 2^{-j\,(N-\frac{n}{2})}\,\left(\barint_{2^jB_{r}}|h|^2\right)^{1/2}.
\end{equation}
Indeed, one can always write
$$
tD\phi(tBD) h= tD\phi(tBD) (\IP h) =  tD\phi(tBD) (\IP h-c)
$$
for any constant $c$, noting  that $tD\phi(tBD)(c)=tD(\phi(0) c)=0$ and apply this inequality as needed.
Again the proof of \eqref{varphi3} follows by  decomposing $h=h_{0}+h_{1}+\ldots$. The terms $h_{j}$ for $j\ge 1$  are localized in annuli  away from the ball $B_{r}$. One can use Lemma  \ref{lem:localcoerc} to control  integrals $\int_{B_{r}} |tD\phi(tBD) h_{j}|^2 $ by the sum of $\int_{\tilde B_{r}} |tBD\phi(tBD) h_{j}|^2 $ and 
$\int_{\tilde B_{r}} |\phi(tBD) h_{j}|^2 $ with slightly larger balls $\tilde B_{r}$. Now, one uses the $L^2$ off-diagonal decay of  combinations and iterates of resolvents. It remains to look at the term with $h_{0}=h1_{2B}$. One has
$$\int_{B_{r}} |tD\phi(tBD) h_{0}|^2  \le 
\int_{\R^n} |tD\phi(tBD) h_{0}|^2   \lesssim \int_{\R^n} |tBD\phi(tBD) h_{0}|^2 
$$
 using the accretivity  of $B$ on $\ran_{2}(D)$.  We conclude by plugging this in the $dt$ integral and using the square function bounds for $tBD\phi(tBD)$. 
\end{proof}

Armed with this lemma, we begin as in the proof of the upper bounds for Theorem \ref{thm:goodhpbd} by observing that our assumption implies
  for $h\in \clos{\ran_{2}(BD)}$,  $\|\IP \chi^\pm(BD)h\|_{p}\lesssim \|\IP h\|_{p}$ for $p=q'$  if $q>1$ or $\|\IP \chi^\pm(BD)h\|_{\dot \Lambda^\alpha}\lesssim \|\IP h\|_{\dot \Lambda^\alpha}$ for $\alpha=n(\frac{1}{q}-1)$ if $q\le 1$. 

Now let $ \phi\in\mR^2_{\sigma}(S_{\mu})\cap \Psi^\tau_{0}(S_{\mu})$ with $\tau>1$. 
Pick $\phi_{\pm}\in \mR^{2}(S_{\mu})$ such that 
\begin{equation*}
|\phi(z)-\phi_{\pm}(z)|=O(|z|^\sigma), \quad \forall z\in S_{\mu\pm}.
\end{equation*}
The key point is the following observation: the functions $\psi_{\pm}(z):= z\tilde \psi_{\pm}(z) $ with $\tilde \psi_{\pm}(z):=(\phi-\phi_{\pm})(z)\chi^\pm(z)$ satisfy $\tilde\psi_{\pm}\in \Psi_{\sigma}^{\tau}(S_{\mu})$ and
$\psi_{\pm} \in \Psi_{\sigma+1}^{\tau-1}(S_{\mu})$. Hence, for  $h\in \clos{\ran_{2}(BD)}$, using $h=\chi^{+}(BD)h+ \chi^{-}(BD)h=h^++h^-$, we have the decomposition
$$
tD\phi(tBD) h = tD\tilde\psi_{+}(tBD) h + tD\tilde\psi_{-}(tBD)h+ tD\phi_{+}(BD)h^+ +  tD\phi_{-}(BD)h^-.
$$
  In the case $q=p'$ we deduce from Lemma \ref{lem:r2}  
 $$
\|tD\phi_{+}(BD)h^+ \|_{T^p_{2}} \lesssim \|\IP h^+ \|_{p}  \lesssim  \|\IP h\|_{p}
$$
and similarly for the term with $h^-$.
Now using the local coercivity assumption \eqref{eq:localcoerc}, up to opening the cones in the definition of the square function, we have
$$
\|tD\tilde\psi_{\pm}(tBD) h\|_{T^p_{2}} \lesssim \|\Qpsi {\psi_{\pm}} {BD} h \|_{T^p_{2}} + \|\Qpsi {\tilde\psi_{\pm}} {BD} h \|_{T^p_{2}}.
$$
But  $\psi_{\pm}$ and $\tilde \psi_{\pm}$ are allowable for $\IH^p_{BD}$ as we assumed $\sigma>\gamma(q)=\gamma(p)$, thus 
 $$ 
\|tD\tilde\psi_{\pm}(tBD) h\|_{T^p_{2}} \lesssim \|\IP h \|_{p}.
$$
The argument when $q\le 1$ and $\alpha= n(\frac{1}{q}-1)$ is similar. This completes the proof of 
the upper bounds in Theorem \ref{thm:supergoodhpbd}.

\section{Completions}\label{sec:completions}

As said, completions of the pre-Hardy spaces may lead to abstract spaces. 
The results above will give us favorable situations in appropriate ranges. This is in spirit with the results in \cite{HMMc}   obtained for second order operators in divergence form.

Strictly speaking, we could proceed this article without including such completions except in Section \ref{sec:perturbation}. This can be skipped in a first reading.

Here $T=DB$ or $BD$ on $L^2$ but the theory could be more generally defined.

 For $0<p<\infty$, define $H^p_{T}$ to be the completion of $\IH^p_{T}$  with respect to   $\|\Qpsi \psi {T} h\|_{T^p_{2}}$ for any allowable $\psi$.  For $p<1$, this is a quasi-Banach space. 
  
 For $p=\infty$,  
 we have two options. One is  $h^\infty_{T}=\dot\lambda^0_{T}$ be the completion $\IH^{\infty}_{T}$  with respect to $\|\Qpsi \psi {T} h\|_{T^\infty_{2}}$ for any allowable $\psi$. We do not see any crucial use of it but we mention it for completeness. The other one is $H^\infty_{T}=\dot \Lambda^0_{T}$ be the dual space of $H^1_{T^*}$. For $\alpha>0$, let  $\dot \lambda^\alpha_{T}$ be the completion of $\IL^{{\alpha}}_{T}$ with respect to any of the allowable norms  $\|\Qpsi \psi {T} h\|_{T^\infty_{2,\alpha}}$. Alternately, let  $\dot \Lambda^\alpha_{T}$ be the dual space of $H^p_{T^*}$ with $\alpha=n(\frac{1}{p}-1)$.

 The following properties hold:
 
 1) For $1<p<\infty$, $H^p_{T}$ and $H^{p'}_{T^*}$ are dual spaces for a duality extending the $L^2$ sesquilinear inner product. In particular, $H^p_{T}$ is reflexive. 
 
 2) $\dot \lambda^\alpha_{T}$ is a closed subspace of $\dot \Lambda^\alpha_{T}$ when $\alpha\ge 0$. 
 
 3) On each $H^p_{T}$, $1< p<\infty$, there is a unique  bisectorial operator $U=U_{H^p_{T}}$ with $H^\infty$-calculus such that for all $b\in H^\infty(S_{\mu})$, $b(U)h= b(T)h$ for all $h\in \IH^p_{T}$. In particular there is a continuous, bounded and analytic semigroup $(e^{-t|U|})_{t>0}$ which extends the semigroup $(e^{-t|T|})_{t>0}$ on $\clos{\ran_{2}(T)}$.  Moreover, $U$ is injective. Finally,  $(U_{H^p_{T}})^*= U_{H^{p'}_{T^*}}$. 
  
 4) If $p\le1$, the $H^\infty$-calculus originally defined on $\clos{\ran_{2}(T)}$ extends to $H^p_{T}$. In particular, we have bounded extension of the operators $e^{-t|T|}$, $t\ge0$. They form a  semigroup and we have shown the strong continuity at 0  on a dense subspace in Proposition \ref{prop:stronglimit}. Thus strong continuity at 0 remains on the completion.  Similarly, we can define the spectral spaces $H^{p,\pm}_{T}$ as the completion of $\IH^{p,\pm}_{T}$ (within $H^{p}_{T}$)  or, equivalently, as the image of the extension to $H^p_{T}$ of $\chi^\pm(T)$.  Similarly, by taking adjoints (in the duality extending the $L^2$ sesquilinear inner product), we can extend the $H^\infty$-calculus originally defined on $\clos{\ran_{2}(T)}$ to $\dot \Lambda^\alpha_{T}$ when $\alpha\ge 0$ and then the semigroup is weakly-$*$ continuous. Moreover, $\dot \Lambda^{\alpha,\pm}_{T}$ is the dual space to $H^{p,\pm}_{T^*}$.

5) The spaces $H^p_{T}$ can be defined in such a way they form a complex interpolation family for $0<p\le \infty$. 

See \cite{AMcR} for 1) and 5). Assertions 2) and 4) are easy.  We give a proof of 3) together with the construction.

\begin{proof}
[Proof of 3)] Fix $1<p<\infty$. Define $H^{p,\pm}_{T}$ as the completion of $\IH^{p,\pm}_{T}$ (within $H^{p}_{T}$). Clearly,  the splitting of the pre-Hardy spectral subspaces passes to completion. Also  $(e^{\mp tT}\chi^\pm(T))_{t>0}$ extends to  an analytic semigroup on $ H^{p, \pm}_{T}$ in the open sector $S_{(\pi/2-\omega)+}$. As $H^{p,\pm}_{T}$ is a Banach space, this semigroup has  a generator $-U_{\pm}$ which is $\omega$-sectorial and densely defined (see \cite{pazy}).  
On $H^p_{T}=  H^{p,+}_{T} \oplus H^{p,-}_{T}$, define
$$Uh= U_{+}h^+- U_{-}h^-, \quad \dom(U)=\{h\in H^p_{T}\, ;\, h^\pm \in \dom(U_{\pm})\}.
$$
Then, $U$ is clearly $\omega$-bisectorial and densely defined on $H^p_{T}$.
As $e^{\mp zU_{\pm}}$ coincides with $e^{\mp zT}\chi^\pm(T)$ on $\IH^{p,\pm}_{T}$ when $z\in S_{(\pi/2-\omega)+}$, and $\chi^\pm(T)$ is the identity on $\IH^{p,\pm}_{T}$,    the resolvents $(I+isU)^{-1}$ and $(I+isT)^{-1}$ coincide on both $\IH^{p,\pm}_{T}$, thus on their direct sum $\IH^p_{T}$,  when $s\in S_{\nu}$ where $0\le \nu<\pi/2 - \omega$. 
As a consequence, $\psi(T)$ and $\psi(U)$ coincide on $\IH^p_{T}$ for any $\psi\in \Psi(S_{\mu})$ by the Cauchy formula. As $\psi(T)$ has a bounded extension to $H^p_{T}$ with norm controlled by $\|\psi\|_{\infty}$, this implies that $U$ has a $H^\infty$-calculus on $H^p_{T}$ and that  $b(T)$ and $b(U)$ coincide on $\IH^p_{T}$ for any $b\in H^\infty(S_{\mu})$. 

The uniqueness of $-U$ follows from   that of $-U_{\pm}$ as generators of semigroups. 

The operator $|U|$ may now be defined as $|U|=\sgn(U)   U$,  or alternately as
 $|U|h= U_{+}h^+ + U_{-}h^-$ with $\dom(|U|)=\{h\in H^p_{T}\, ;\, h^\pm \in \dom(U_{\pm}\}=\dom(U).$ The semigroup generated by $-|U|$ thus coincides with the one generated by ${-|T|}$ on $\IH^p_{T}$. 
 
 The injectivity is a little trickier.  We have seen in Proposition \ref{prop:stronglimitinfty} that for any $h\in \IH^p_{T}$, $\lim_{s\to \infty} \|e^{-s|T|}h\|_{\IH^p_{T}}=0.$ By density, we have $ \lim_{s\to \infty} \|e^{-s|U|}h\|_{H^p_{T}}=0$ for any $h\in H^p_{T}$. If $h\in \nul(U)$, then $h\in \nul(|U|)$ and thus $e^{-s|U|}h=h$ for all $s>0$. Taking the limit at $\infty$ yields  $h=0$. 
 
 Finally, calling $U=U_{H^p_{T}}$ and using the duality between $H^p_{T}$ and $H^{p'}_{T^*}$, it it is easy to conclude that $(U_{H^p_{T}})^*= U_{H^{p'}_{T^*}}$.
\end{proof}

\begin{rem} Except for the last duality formula, the proof works for $H^1_{T}$, which is a Banach space, as reflexivity is not used.
\end{rem}

\

Let us come back to our concrete situation.

 \begin{prop}\label{cor:hpdb} Let $\frac{n}{n+1}<p<\infty$. If $\IH^p_{DB}=\IH^p_{D}$ with equivalence of norms, then they have same completions $H^p_{DB}=  H^p_{D}$ with equivalence of norms.  In particular, $H^p_{DB}$ is a  complemented subspace of $H^p$ where  $H^p=L^p$ if $p>1$. Moreover, the extended semigroup of $(e^{-t|DB|})_{t>0}$ is strongly continuous in $H^p_{D}$.  
\end{prop}

\begin{proof}That $\IH^p_{DB}=\IH^p_{D}$ with equivalence of norms implies they have  same completion is an exercise in functional analysis. We have seen in Theorem \ref{thm:hpd} that $\IH^p_{D}=\IP(H^p\cap L^2)$. As $H^p\cap L^2$ is dense in $H^p$ and $\IP$  has a bounded extension to $H^p$, we have $H^p_{D}=\IP(H^p)$, hence $H^p_{DB}=H^p_{D}$ is a complemented subspace of $H^p$.  We have seen that the semi-group is strongly continuous on $H^p_{DB}$. This passes to $H^p_{D}$. 
\end{proof}

The following result is in spirit of \cite{HMMc} and \cite{AMcMo}.

 \begin{prop}\label{prop:inclusion} Let $\frac{n}{n+1}<p<\infty$. If $H^p_{DB}=H^p_{D}$ with equivalence of norms, then  $H^p_{DB}\cap L^2=  \IH^p_{DB}$ and $\IH^p_{DB}=\IH^p_{D}$ with equivalence of norms.
\end{prop}

\begin{proof} If $p=2$, there is nothing to prove. In the other case, it suffices to show the first set equality as the second one, with the equivalence of norms, follows from it. 

If $p>2$, then 
$$
\IH^p_{DB}\subset H^p_{DB}\cap L^2= H^p_{D}\cap L^2 = \IH^p_{D} \subset \IH^p_{DB}
$$
using Theorem \ref{thm:hpd} and Proposition \ref{cor:upperbounddbp>2}.

Assume now that $p<2$. 
It is enough to show $H^p_{DB}\cap L^2 \subset  \IH^p_{DB}$ as the other inclusion is by construction.  Let $h\in H^p_{DB}\cap L^2$. Take an allowable $\psi$ for $\IH^p_{DB}$. We have to show that 
$\Qpsi \psi {DB} h\in T^p_{2}$. By definition, there exists $h_{k}\in \IH^p_{DB}$ such that 
$h_{k}$ converges to $h$ in $H^p_{DB}$. Thus, $(\Qpsi \psi {DB} h_{k})$ is a Cauchy sequence in  $T^p_{2}$ and has a limit $H$. Also, by the assumption, $(h_{k})$ converges to $h$ for the 
$H^p$ topology. It remains to show that $H=\Qpsi \psi {DB} h$, for example in the sense of distributions in $\reu$. Let $F\in C^\infty_{0}(\reu)$, then we can write
$$
(H-\Qpsi \psi {DB} h, F)= (H-\Qpsi \psi {DB} h_{k}, F)+ \pair {h_{k}-h} {\Tpsi {\psi^*} {B^*D} F},
$$
the computation being justified by the $H^2_{DB}$ theory. 
The first term of the right hand side converges to 0, since $F\in (T^p_{2})^*$ as easily checked. 
For the second term, we remark that it  equals  $\pair {h_{k}-h} {\IP\Tpsi {\psi^*} {B^*D} F}$ and we claim that $\IP\Tpsi {\psi^*} {B^*D} F \in (H^p)^*$. Thus convergence to 0  follows and finishes the argument. 

To prove the claim, let $[a,b]\times \R$ contain the support of $F$. Then 
$$
{\IP\Tpsi {\psi^*} {B^*D} F} = \int_{a}^b \IP\psi^*(tB^*D) F(t,\, \cdot\,)\, \frac{dt}{t} =\int_{a}^b \IP\psi^*(tB^*D)\IP F(t,\, \cdot\,)\, \frac{dt}{t}.
$$
Remark that for each $t$, $F(t,\, \cdot\,) \in (H^p)^*\cap L^2$ with uniform bound for $t\in [a,b]$. Thus $\IP F(t,\, \cdot\,) \in \IP((H^p)^*\cap L^2)= \IH^{p'}_{D}$ or $\IL^\alpha_{D}$ depending on the value of $p$.    We now verify the assumption of Remark  \ref{rem:weaker} : From $p<2$, we know $\IH^p_{DB}\subset \IH^p_{D}$ (Proposition \ref{cor:lowerbounddbp<2}). Next, from  $H^p_{DB}=H^p_{D}$ with equivalence of norms,  we see that 
$\|h\|_{\IH^p_{D}}\sim \|h\|_{\IH^p_{DB}}$ for all $h\in \IH^p_{DB}$. Finally, the density of $\IH^p_{DB}$ in $H^p_{DB}$ guarantees that  the above set inclusion is dense for the $\IH^p_{D}$ topology.  Thus, the conclusion of  Corollary \ref{cor:hpbd0}  applies: 
$\IP\psi^*(tB^*D)$ bounded on $\IH^{p'}_{D}$ or $\IL^\alpha_{D}$ uniformly in $t$. This implies that ${\IP\Tpsi {\psi^*} {B^*D} F} \in (H^p)^*$ as desired. 
\end{proof}

\begin{prop}
The set of exponents $q\in (\frac{n}{n+1},\infty)$ for which $\IH^q_{DB}=\IH^q_{D}$ with equivalence of norms  is equal to the set of those $q$ for which    $H^q_{DB}=H^q_{D}$ with equivalence of norms.  Moreover, it is an  interval which  contains $((p_{-}(DB))_{*}, p_{+}(DB))$. \end{prop} 

\begin{proof} The first statement follows from  the previous two propositions.
We know from $H^\infty$-calculus in $L^2$ that the identity map $I: H^2_{D}=\clos{\ran_{2}(D)} \to H^2_{DB}$ is an isomorphism.  Let $q$ be such that $H^q_{DB}=H^q_{D}$ with equivalence of norms.   It means that $I$ is an isomorphism from $H^q_{D}$ onto  $H^q_{DB}$. As  $H^p_{D}$ and $H^p_{DB}$ are complex interpolation families for $0<p<\infty$, 
this shows that the set of  $q$ for which $H^q_{DB}=H^q_{D}$ is an interval which contains $2$.  
That it contains $((p_{-}(DB))_{*}, p_{+}(DB))$ has been proved  in Theorem \ref{thm:hpdb} for the pre-Hardy spaces, hence for their completions.  
\end{proof}

\begin{prop}
The interval  of exponents $q\in (\frac{n}{n+1},\infty)$ for which $\IH^q_{DB}=\IH^q_{D}$ with equivalence of norms is open. 
\end{prop}

\begin{proof} We begin with openness about an exponent $q<2$. Take $\psi\in \Psi^{\frac{n}{2}+1}(S_{\mu})$ and $\varphi\in \Psi_{\frac{n}{2}+1}(S_{\mu})$ for which the Calder\'on formula \eqref{eq:CalderonT} holds. We have the bounded maps $\Qpsi \psi {DB}: \IH^p_{DB}\to T^p_{2}\cap T^2_{2}$ and $\Tpsi \varphi {DB}: T^p_{2}\cap T^2_{2} \to \IH^p_{D}$ for all $p\in (\frac{n}{n+1}, 2]$ by  Proposition \ref{prop:table} and Proposition \ref{cor:lowerbounddbp<2}. The composition is the identity map. Consider bounded extensions $H^p_{DB}\to T^p_{2}$  and $T^p_{2} \to H^p_{D}$ that are consistent for this range of $p$.  The composition is assumed to be the identity at $p=q$. By the result in \cite{Snei, KalMit}, it remains invertible for $p$ in a neighborhood of $q$. It readily follows that $H^p_{DB}$ and $H^p_{D}$ are isomorphic for those $p$. 
Since we already have the inclusion $\IH^p_{DB}\subset \IH^p_{D}$, it is easy to conclude the isomorphism is the identity.   

In the case $p>2$, we know from Proposition \ref{cor:upperbounddbp>2} that  $\IH^p_{D}\subset \IH^p_{DB}$. So we revert the roles  of $DB$ and $D$ and consider $\Qpsi \psi {D}$ and $\Tpsi \varphi {D}$ for appropriate $\psi,\varphi$. We skip details. 
\end{proof}

\section{Openness}

We would like to prove that for any $p$ such that   $\IH^p_{DB}=\IH^p_{D}$ with equivalence of norms, then the same holds for small $L^\infty$ perturbations of $B$. We do not know this in the abstract. However, we can do this in the range found in Theorem \ref{thm:hpdb}.

\begin{prop}\label{prop:open} Fix $p\in ((p_{-}(DB))_{*},p_{+}(DB))$. Then for any $B'$ with $\|B-B'\|_{\infty}$ small enough (depending on $p$), $\IH^p_{DB'}=\IH^p_{D}$ with equivalence of norms. Furthermore, for any $b\in H^\infty(S_{\mu})$ with  $\omega_{B}<\mu<\pi/2$, we have 
\begin{equation}
\label{eq:analytic}
\|b(DB)-b(DB')\|_{\mL(\IH^p_{D})}\lesssim  \|b\|_{\infty}\|B-B'\|_{\infty}.
\end{equation}
\end{prop}

\begin{proof} We shall use analyticity: let $B_{z}=B-zM$ for $M(x)$ normalized with $L^\infty$ norm 1 and $z\in \C$ so that $B_{0}=B$. We shall show the conclusion for $B'=B_{z}$ with $z$ in a small enough disk. First there is $r>0$ such that for $|z|<r$, $B_{z}$ is accretive on $\clos{\ran_{2}(D)}$ with constant half the one for $B$ and bounded with $L^\infty$ bound twice that of $B$. 
Using a Neumann series expansion, for $\lambda\notin S_{\mu}$, where $\mu>\omega_{B}$ (the accretivity angle of $B$)
\begin{align*}
   (\lambda-DB_{z})^{-1} &= \sum_{k=0}^\infty z^k \big((\lambda-DB)^{-1}DM\big)^k(\lambda-DB)^{-1}\\
   &=
\sum_{k=0}^\infty z^k \big((\lambda-DB)^{-1}DBB^{-1}M\big)^k(\lambda-DB)^{-1}.  
\end{align*}
Thus if $|z|<\varepsilon_{2}$, the series converges in $\mL(L^2)$ and this shows that  $DB_{z}$ is $\omega_{B}$ bisectorial  on $L^2$  for all $|z|<\varepsilon_{2}$. As $B_{z}$ has the same form as $B$, it follows that $DB_{z}$ has $H^\infty$-calculus on bisectors $S_{\mu}$ with uniform bounds with respect to $|z|<\varepsilon_{2}$. Furthermore $z\mapsto b(DB_{z})$ is an analytic $\mL(L^2)$-valued function for any $b\in H^\infty(S_{\mu})$. This is shown in \cite[Section 6]{AKMc}  together with \eqref{eq:analytic} in $\mL(L^2)$. 

Now, the same Neumann series shows that $DB_{z}$ is also bisectorial on $L^p$ if
$p_{-}(DB)<p<p_{+}(DB)$ and $|z|<\varepsilon_{p}$ small enough. Thus such operators also have $H^\infty$-calculus on $L^p$ by the theory recalled in Section \ref{sec:lpresults}. Furthermore, analyticity of $z\mapsto b(DB_{z})$ in $\mL(\IH^p_{D})$ on  $|z|<\varepsilon_{p}$  for any $b\in H^\infty(S_{\mu})$ can be proved as follows. First, for any $\lambda\notin S_{\mu}$, the Neumann series, show that $z\mapsto (\lambda-DB_{z})^{-1}$ is analytic in $\mL(L^p)$ on $|z|<\varepsilon_{p}$.  Next, for $\psi\in \Psi(S_{\mu})$, using the Cauchy formula, one has  
analyticity of  $z\mapsto \psi(DB_{z})$ in $\mL(L^p)$ on $|z|<\varepsilon_{p}$. Finally, $b$ can be approximated for the topology of the uniform convergence on compact subsets of $S_{\mu}$ by a sequence $(\psi_{k})$ with $\psi_{k}\in \Psi(S_{\mu})$ for each $k$. This implies strong convergence of $\psi_{k}(DB_{z})$ to $b(DB_{z})$ in $\mL(\IH^p_{D})$ and examination shows it is  uniform on compact subsets of   $|z|<\varepsilon_{p}$. Analyticity follows and also \eqref{eq:analytic}  in  $\mL(\IH^p_{D})$. 

We next turn to values  $(p_{-}(DB))_{*}<p\le p_{-}(DB)$. For those,  the method of proof of  Theorem \ref{thm:hpdb} (in particular Lemma \ref{lem:extrapolation})
shows that for a suitable $\varepsilon_{p}$ (which can be taken equal to $\varepsilon_{q}$ for some $q\in (p_{-}(DB),p^*)$) and a suitable allowable $\psi$,  $\|\Qpsi \psi {DB_{z}}h\|_{T^p_{2}}\lesssim \|h\|_{\IH^p_{D}}$ when $h\in \clos{\ran_{2}(D)}$ uniformly on compact subsets of $|z|<\varepsilon_{p}$. Hence, $\IH^p_{DB_{z}}=\IH^p_{D}$ with equivalence of norms, uniformly on compact subsets of  $|z|<\varepsilon_{p}$. This implies that 
$b(DB_{z})$ are uniformly bounded operators on $\IH^p_{D}$  when $|z|<\varepsilon_{p}$.

If $1<p$, this gives analyticity  as follows: for $h\in \IH^p_{D}, g\in \IH^{p'}_{D}$, the map $ z\mapsto \pair {b(DB_{z})h}g$ is uniformly bounded, and analytic because of the $L^2$ case. Then \eqref{eq:analytic} follows from Cauchy estimates. 

If $p\le 1$, it is likely that the abstract results developed in Kalton \cite{Ka} apply. We follow a different route taking advantage of the atomic-molecular theory. 

Let us admit the following lemma for the moment.

\begin{lem}\label{lem:atom-mol} Let $(p_{-}(DB))_{*}<p\le 1$ and $b\in H^\infty(S_{\mu})$.   For some $\varepsilon>0$ depending only on $p$ and $ n$,   then  for all   $(\IH^p_{D},1)$-atoms $a$, with associated cube $Q$ and all  $j\ge 0$, 
$$
||b(DB)a||_{L^{2}\left(S_{j}\left(Q\right)\right)}\lesssim \|b\|_{\infty} \left(2^{j}\ell\left(Q\right)\right)^{\frac{n}{2}-\frac{n}{p}}2^{-j\varepsilon} 
$$
and moreover $\int b(DB)a=0$. In all, $b(DB)a$ is a classical $H^p$ molecule. 
\end{lem}

  Now the strategy is to prove analyticity is as follows.  The same estimate applies to $b(DB_{z})$ for $|z|<\varepsilon_{p}/2$, uniformly  in $z$. We fix the  $(\IH^p_{D},1)$-atom $a$. It follows from 
  the molecular estimate that $m_{z}=b(DB_{z})a$ belongs to the Hilbert space $H$ of $L^2(w_{Q}) $ functions $m$ with   $\int_{\R^n}m=0$, where  $w_{Q}(x) =|Q|^{\frac{2}{p}-1} \big( 1+ \frac{d(x,4Q}{\ell(Q)}\big)^{2s}$, $\frac{n}{p}-\frac{n}{2}<s<\frac{n}{p}-\frac{n}{2}+\varepsilon$ and $Q$ is the cube associated to $a$ in the definition.  Note that $H\subset L^1$. The bounded compactly supported functions with mean value 0  form a  dense subspace of $H$. For $f$ such a function, $fw_{Q}\in L^2$ as well as $b(DB_{z})a$ since $a\in L^2$. Thus, by the analyticity on $L^2$,  $z\mapsto \pair {m_{z}} {fw_{Q}} $ is analytic. Next, by Cauchy estimates using the uniform bounds in the space $H$, we have  for $|z|$ small enough,
  $$
 | \pair {m_{z}} {fw_{Q}}-\pair {m_{0}} {fw_{Q}}|\lesssim \|b\|_{\infty} |z| \|f\|_{H},$$
 hence
  $$
  \|b(DB_{z})a- b(DB)a\|_{H}\lesssim \|b\|_{\infty} |z|.
  $$
  Since $s>\frac{n}{p}-\frac{n}{2}$, this implies the $H^p$ estimate 
  $$
  \|b(DB_{z})a-b(DB)a\|_{H^p}\lesssim  \|b\|_{\infty} |z|.
  $$
  Note that $ b(DB_{z})a-b(DB)a \in \IH^2_{D}$, hence this is also an estimate in the space $\IH^p_{D}$.  Since we know already boundedness of $b(DB_{z})-b(DB)$ on $\IH^p_{D}$ (but it can be obtained by extension), we conclude for   \eqref{eq:analytic} by density of linear combinations of $(\IH^p_{D},1)$-atoms.
   \end{proof}
   
   \begin{proof}[Proof Lemma \ref{lem:atom-mol}.]  This is basically the same strategy as for proving the square function estimate. Assume $\|b\|_{\infty}=1$ to simplify matters. Fix a $(\IH^p_{D},1)$-atom $a$. Choose $\psi\in \Psi_{\sigma}^\tau(S_{\mu})$ with $\sigma,\tau$ large and so that $\int_{0}^\infty
\psi(tz) \, \frac{dt}{t}=1$ for $z\in S_{\mu}$. Thus, $m=b(DB)a= \int_{0}^\infty
 (b\psi_{t})(DB)a\, \frac{dt}{t}$ with $\psi_{t}(z) =\psi(tz)$.  
 Now write $a=Du$ as in the definition of  $(\IH^p_{D},1)$-atoms. We show estimates on $m$. Let $Q$ be the cube associated to $a$.  On $4Q$, by $H^\infty$-calculus
$$
\bigg(\int_{4Q}|m|^2\bigg)^{1/2}  \lesssim  \bigg(\int |a|^2\bigg)^{1/2} \leq  |Q|^{\frac{1}{2}-\frac{1}{p}}.
$$
On $S_{j}(Q)$, $j\ge 2$,  we write $(b\psi_{t})(DB)a= D(b\psi_{t})(BD)u$ and 
$$
\bigg(\int_{S_{j}(Q)}|m|^2\bigg)^{1/2} \lesssim  \int_{0}^\infty  \bigg(\int_{S_{j}(Q)}|D(b\psi_{t})(BD)u|^2\bigg)^{1/2}\, \frac{dt}{t}.
$$
We use once more Lemma \ref{lem:localcoerc} and the fact that $\sigma,\tau>0$ are large enough in the $L^q-L^2$ estimates of Section \ref{sec:lplq} applied to $(b\psi_{t})(BD)$  to obtain 
$$
\bigg(\int_{S_{j}(Q)}|D(b\psi_{t})(BD)u|^2\bigg)^{1/2} \lesssim t^{-1} t^{\frac{n}{2}-\frac{n}{q}}\brac{2^j\ell(Q)/t}^{-K}\|u\|_{q}
$$
with $K$ large and  $q$ chosen so that $p_{-}<q<p^*$ and $q\le 2$.
Plugging this estimate into  the $t$-integral we have
$$
\bigg(\int_{S_{j}(Q)}|m|^2\bigg)^{1/2} \lesssim (2^j\ell(Q))^{-1} (2^{j}\ell(Q))^{\frac{n}{2}-\frac{n}{q}}\|u\|_{q} \lesssim (2^{j}\ell(Q))^{\frac{n}{2}-\frac{n}{p}} 2^{-j\varepsilon}$$
with $\varepsilon= 1+\frac{n}{q}-\frac{n}{p}>0$. 
It remains to prove $c=\int {m}=0$.  Indeed, $m-c1_{Q}\in H^p$ as it is a classical molecule for $H^p$ using the estimates on $m$. Now, we know that  $m\in \IH^p_{D}\subset H^p$,  thus $c1_{Q}\in H^p$ and it is classical (for example, using the characterization by maximal function) that $1_{Q}\notin H^p$ since its mean value is not 0 and  $c$ must be 0. 
\end{proof}

\begin{rem}
It is unclear to us whether $m$ is itself an $\left(\IH^{p}_{D},\epsilon,1\right)$-molecule in the sense of our definition. One can indeed write $m=Dv$ with $v=\int_{0}^\infty
 (b\psi_{t})(BD)u\, \frac{dt}{t} $ and obtain by the same method
 $$
||v||_{L^{2}\left(S_{j}\left(Q\right)\right)}\lesssim \|b\|_{\infty} \left(2^{j}\ell\left(Q\right)\right)^{\frac{n}{2}-\frac{n}{p}}2^{-j\varepsilon}  2^j.
$$
There is an extra factor $2^j$. However, this is sufficient to prove a uniform $L^{p^*}$ bound on $v$ if one needs it. 
\end{rem}

\section{Regularization via semigroups}\label{sec:reg}

This section will be used in Section 14 below. 

It is well known that classical semigroups have regularization properties: for example, the usual Poisson semigroup on $\R^n$ maps $L^1$ into $L^\infty$, as easily seen using the Poisson kernel. Here, there is no kernel information. Nevertheless, such regularization holds abstractly in the Hardy spaces. 

\begin{thm} Let $T=BD$ or $DB$. 
Let $0<p\le q\le \infty$ and $0\le \alpha\le \beta <\infty$. Fix $t>0$. Then the operator $e^{-t|T|}$ has extensions with the following mapping properties and bounds
\begin{align*}
\label{eq:semigroupbounds}
  H^p_{T} \to H^q_{T}  &\quad  \mathrm{with\ bound\  }  Ct^{-(\frac{n}{p}-\frac{n}{q})}.   \\
  \dot \Lambda^\alpha_{T}  \to \dot \Lambda^\beta_{T} & \quad  \mathrm{with\ bound\  }   Ct^{\, \alpha-\beta}. \\
   H^p_{T} \to   \dot\Lambda^\alpha_{T}  & \quad  \mathrm{with\ bound\  }  Ct^{- \frac{n}{p}-\alpha} .  
\end{align*}
Moreover, the mapping properties hold with the same bounds when the  spaces are replaced by the corresponding pre-Hardy spaces $\IH^\mT_{T}$ with the possible exception of the first line when $p<q\le 1$. 
\end{thm}

\begin{rem}
The proof will show this result is not limited to $BD$ or $DB$. It holds for any operator $T$ on $\R^n$ having a
Hardy space theory (for example, bisectorial with $H^\infty$-calculus plus $L^2$ off-diagonal bounds). The bounds are valid for an operator having the scaling of a first order operator. For an operator of order $m$, then raise the bounds to power $\frac{1}{m}$.
\end{rem}

\begin{proof} 

\textbf{Step 1:}   $p\le 1$ and $q=2$ in the first line. 

We pick an $\left(\IH^{p}_{T},\epsilon,M\right)$-molecule $a$ with $M> \frac{n}{p}-\frac{n}{2}$. Let $\ell$ be the sidelength of the associated cube. Observe that $a\in \ran_{2}(T)\subset \IH^2_{T}$, thus $e^{-t|T|}a \in \IH^2_{T}$ and $\|e^{-t|T|}a\|_{2} \sim \|e^{-t|T|}a\|_{\IH^2_{T}}$. 
Since   $\|a\|_{2} \lesssim \ell^{-(\frac{n}{p}-\frac{n}{2})} $ and $e^{-t|T|}$ is uniformly bounded on $L^2$ we have $\|e^{-t|T|}a\|_{2} \lesssim \ell^{-(\frac{n}{p}-\frac{n}{2})} $. Now we have $a=T^Mb$ with $b\in \dom_{2}(T^M)$ and $\|b\|_{2}\lesssim \ell^{M}\ell^{-(\frac{n}{p}-\frac{n}{2})}$. As $(tT)^Me^{-t|T|}$  is uniformly bounded on $L^2$, we have $\|e^{-t|T|}a\|_{2} \lesssim t^{-M}\ell^{M-(\frac{n}{p}-\frac{n}{2})}$. 
Thus
$$
\|e^{-t|T|}a\|_{2} \lesssim \inf (\ell^{-(\frac{n}{p}-\frac{n}{2})}, t^{-M}\ell^{M-(\frac{n}{p}-\frac{n}{2})}) \le 
t^{-(\frac{n}{p}-\frac{n}{2})}.
$$
Next, let $f\in \IH^p_{T}$. Pick a molecular $\left(\IH^{p}_{T},\epsilon,M\right)$-representation
$f=\sum \lambda_{j} a_{j}$ which converges in $L^2$ and also with $\sum |\lambda_{j}|^p \le 2^p \|f\|_{\IH^p_{T}}^p$. From $L^2$ continuity of the semigroup we have
$e^{-t|T|} f=\sum \lambda_{j} e^{-t|T|} a_{j}$, hence
$$
\|e^{-t|T|}f\|_{2} \lesssim \sum |\lambda_{j}| t^{-(\frac{n}{p}-\frac{n}{2})}  \le   \|(\lambda_{j})\|_{\ell^p} \,
t^{-(\frac{n}{p}-\frac{n}{2})} \le 2  \|f\|_{\IH^p_{T}}  t^{-(\frac{n}{p}-\frac{n}{2})}.
$$
as $p\le 1$. Finally, taking completion we have proved step 1. 

\

\textbf{Step 2: }  $p=2$ and $q=\infty$ in the first line. 

This an easy consequence of Proposition \ref{prop:duality}. Let $g\in H^2_{T}=\IH^2_{T}$. Let $f\in \IH^1_{T^*}$. Using the first step with $T^*$ (which is of the same type as $T$), 
$$| \pair  f {e^{-t|T|}g} | =  | \pair  {e^{-t|T^*|}f} {g} | \le \|g\|_{\IH^2_{T}}\|e^{-t|T^*|}f\|_{\IH^2_{T^*}} \lesssim 
 \|g\|_{\IH^2_{T}} \|f\|_{\IH^1_{T^*}} t^{-({n}-\frac{n}{2})}.
 $$
 Thus,  $e^{-t|T|}g \in (\IH^1_{T^*})^*= H^\infty_{T}$ by definition of $H^\infty_{T}$ and density of $\IH^1_{T^*}$ in $H^1_{T^*}$, and $\|e^{-t|T|}g\|_{H^\infty_{T}}\lesssim  \|g\|_{\IH^2_{T}}  t^{-({n}-\frac{n}{2})}$.
 
 \
 
 \textbf{Step 3: } All cases in the first line. Using the semigroup property and combining the first two steps, we have the first line for $(p,\infty)$ for any $0<p<\infty$ and we also know the first line for all pairs $(p,p)$ for $0<p\le \infty$ from the discussion in Section~\ref{sec:completions}.  We conclude this line by complex interpolation.
 
 \
 
 \textbf{Step 4: } The second line. This is the dual of the first line $H^p_{T^*} \to H^q_{T^*} $, where 
 $\alpha= n(\frac{1}{q}-1)$ and $\beta= n (\frac{1}{p}-1)$.
 
  \
 
 \textbf{Step 5: } The third line. Combine $H^p_{T}\to H^\infty_{T}= \dot\Lambda^0_{T}$ with $\dot\Lambda^0_{T} \to \dot\Lambda^\alpha_{T}$ using the semigroup property.
 
  \
 
 \textbf{Step 6: } The first line with the pre-Hardy spaces. Before we begin  recall that this is not immediate from the results above as we do not know whether $\IH^p_{T}= H^p_{T}\cap H^2_{T}$ in general. We come back to the definition. Let $f\in \IH^p_{T}$. As $e^{-t|T|}f\in \IH^2_{T}$, we want to show that $\Qpsi \psi T (e^{-t|T|}f) \in T^q_{2}$ with the desired bound for some allowable $\psi$ for $\IH^q_{T}$. We can only do this when $q>1$. We choose $\psi$ matching the  conditions of the third and fourth columns for  the exponent $q$  in the table before Proposition \ref{prop:table}. By duality in tent spaces and density, it is enough to bound $(\Qpsi \psi T (e^{-t|T|}f), G)$ by $\|G\|_{T^{q'}_{2}}$ for any $G\in T^{q'}_{2}\cap T^2_{2}$. By the choice of $G$, we have
 $$
(\Qpsi \psi T (e^{-t|T|}f), G)= \pair {f}{ e^{-t|T^*|}(\Tpsi {\psi^*} {T^*} G)}.
$$
Now,   the choice for $\psi$  implies $\Tpsi {\psi^*} {T^*} G \in \IH^{q'}_{T^*}$ and using the just proved first  or  third lines  and duality, $e^{-t|T^*|}(\Tpsi {\psi^*} {T^*} G)\in H^{p'}_{T^*}=(H^p_{T})^*$. We obtain 
$$
|\pair {f}{ e^{-t|T^*|}(\Tpsi {\psi^*} {T^*} G)}| \lesssim t^{-(\frac{n}{p}-\frac{n}{q})}\|f\|_{\IH^p_{T}}\|G\|_{T^{q'}_{2}}.
$$

\
 
 \textbf{Step 7: } The third line with the pre-Hardy spaces, that is $\IH^p_{T} \to \IL^\alpha_{T}$. Let  $f\in \IH^p_{T}$. As $e^{-t|T|}f\in \IH^2_{T}$,  we have to show that   $\Qpsi \psi T (e^{-t|T|}f) \in T^\infty_{2,\alpha}$ with the desired bound for some allowable $\psi$ for $\IL^\alpha_{T}$. We let $\alpha=n(\frac{1}{q}-1)$ for some  $q<1$ and choose   $\psi$ matching the conditions of the third and fourth columns for the exponent $\alpha$ in the table before Proposition \ref{prop:table}.
 By duality in tent spaces and density, it is enough to bound $(\Qpsi \psi T (e^{-t|T|}f), G)$ by $\|G\|_{T^{q}_{2}}$ for any $G\in T^{q}_{2}\cap T^2_{2}$. By the choice of $G$, we have
 $$
(\Qpsi \psi T (e^{-t|T|}f), G)= \pair {f}{ e^{-t|T^*|}(\Tpsi {\psi^*} {T^*} G)}.
$$
Now,    the choice of $\psi$  implies $\Tpsi {\psi^*} {T^*} G \in \IH^{q}_{T^*}$, and using the just proved  first or third line and duality, $e^{-t|T^*|}(\Tpsi {\psi^*} {T^*} G)\in (H^p_{T})^*$. We obtain 
$$
|\pair {f}{ e^{-t|T^*|}(\Tpsi {\psi^*} {T^*} G)}| \lesssim t^{-\frac{n}{p}-\alpha}\|f\|_{\IH^p_{T}}\|G\|_{T^{q}_{2}}.
$$

\
 
 \textbf{Step 8: } The second line with the pre-Hardy spaces, that is $\IL^\alpha_{T} \to \IL^\beta_{T}$. The argument is similar to the previous ones and we leave details to the reader. 
\end{proof}

\begin{cor} Let  $p\le q$ with $p\le 2$. If   both $p,q$ belong to the interval  of exponents  in $ (\frac{n}{n+1},\infty)$ for which $\IH^q_{DB}=\IH^q_{D}$, then the semigroup $e^{-t|DB|}$ has a bounded extension from $H^p_{DB}=H^p_{D}$ to $\IH^q_{DB}=\IH^q_{D}$ with bound $Ct^{-(\frac{n}{p}-\frac{n}{q})}$. 
\end{cor}

\begin{proof} 
From the previous theorem, the semigroup $e^{-t|DB|}$  extends to a bounded operator from   $H^p_{DB}$ to $H^q_{DB}$ with the desired bound as $p\le q$ and it also maps $H^p_{DB}$ to $H^2_{DB}\subset L^2$  as $p\le 2$.   By Proposition \ref{prop:inclusion}, we have that $H^q_{DB}\cap L^2=  \IH^q_{DB}$ for $q$ in the prescribed interval and the result follows. 
\end{proof}

\section{Non-tangential maximal estimates}\label{sec:ntmax}

In this section, we establish the following results for $\tN$ defined in \eqref{eq:KP}. 

\begin{thm}\label{thm:ntdb} Let $(a,p_{+}(DB))$ be an interval with $a>\frac{n}{n+1}$ on which 
$\IH^p_{DB}=\IH^p_{D}$ with equivalence of norms. Then for $p\in (a, (p_{+})^*)$, we have 
$ \|\tN(e^{-t|DB|}h) \|_{p} \sim  \|h\|_{p}$ for all $ h\in \clos{\ran_{2}(D)}$  if $p>1$.  If $p\le 1$, we have $ \|\tN(e^{-t|DB|}h) \|_{p} \sim  \|h\|_{H^p}$ for all $ h\in \clos{\ran_{2}(D)}$ and $B$ pointwise accretive, or  for all $ h\in \IH^{2,\pm}_{DB}$. This applies for $a=(p_{-}(DB))_{*}$. 
\end{thm}

\begin{rem}
We think that the hypothesis of pointwise accretivity is not necessary but we are unable to remove it at this time: this is the only result of this memoir where this hypothesis is used.  Nevertheless, the validity of the equivalence for  $ h\in \IH^{2,\pm}_{DB}$ suffices  for   applications to BVPs. 
\end{rem}

\begin{thm}\label{thm:ntbd} Let $(a,p_{+}(DB^*))$ be an interval with $a\ge 1$ on which 
$\IH^q_{DB^*}=\IH^q_{D}$ with equivalence of norms.
Then for  $1<p<a'$,   we have
$ \|\tN(e^{-t|BD|}\IP h) \|_{p} \sim  \|\IP h\|_{p}$ for all $ h\in \clos{\ran_{2}(BD)}$  if $p \ge 2$ and
$ \|\tN(e^{-t|BD|} h) \|_{p} \sim \|h\|_{p}\sim  \|\IP h\|_{p} $ for all $ h\in \clos{\ran_{2}(BD)}$ if $p_{-}(BD) <p<2$. 
This applies with $a=\max((p_{-}(DB^*))_{*},1)$. 
\end{thm}

\begin{rem}  The inequality  $ \|\tN(e^{-t|T|} h) \|_{p} \lesssim \|h\|_{\IH^p_{T}} $ holds for
$0<p\le 2$ when $ h\in \clos{\ran_{2}(T)}$ for $T=BD$ or $DB$ thanks to Lemma \ref{lem:upperp<2} and the equivalence at $p=2$ (which we prove next).
\end{rem}

\begin{rem}\label{rem:ntbd}
We shall also prove  $ \|\tN(e^{-t|BD|}h) \|_{p} \lesssim    \|h\|_{p}$ for $2<p<(p_{+}(BD))^*$ and
$h\in L^2$, hence in particular $h\in \clos{\ran_{2}(BD)}$. But if $p\ge p_{+}(BD)$ the right hand side is not equivalent to the $\IH^p_{BD}$ norm, while $\|\IP h\|_{p}$ is. This is why we have to insert $\IP$ in   Theorem \ref{thm:ntbd}.
\end{rem}

\begin{rem}  Note that the result  in Theorem \ref{thm:ntbd} for $p<2$ sounds different.   Let the  $r$ variant of $\tN$   be defined as 
 $$
  \tN^r(g)(x):= \sup_{t>0}  \bigg(\bariint_{W(t,x)} |g|^r\bigg)^{1/r}, \qquad x\in \R^n,
$$
so that $\tN^2=\tN$. In fact, one can only prove $\|\tN^r(e^{-t|BD|}\IP h) \|_{p}
\sim  \|\IP h\|_{p}$ for all $ h\in \clos{\ran_{2}(BD)}$ with  $r<p$ if $p<2$. And this is sharp since, as 
$e^{-t|BD|}\IP h- e^{-t|BD|} h= \IP h - h$ for all $t>0$,  $\tN^r(e^{-t|BD|}\IP h-e^{-t|BD|} h) \sim \MM_{r}(\IP h -h)$ and $\MM_{r}$ is not bounded on $L^p$ if $p\le r$. 
\end{rem}

\begin{rem}
We thank M. Mourgoglou for pointing out to us that the results in this section hold with the non-tangential maximal function on Whitney regions  replaced by the non-tangential maximal function on slices $$
 \sup_{t>0}  \bigg(\barint_{B(x,c_{1}t)} |e^{-t|T|}h|^2\bigg)^{1/2}, \qquad x\in \R^n.
$$ For the lower bounds, this is trivial as there is a pointwise domination of $\tN$ by the latter. For the upper bounds, the arguments need some adjustements. The main one is to go from
the  integral on slices $\barint_{B(x,c_{1}t)} |\psi(tT)h|^2$  to a solid integral on a Whitney region in order to use square function estimates. This can be done using the method of  proof of Proposition 2.1 in \cite{AAAHK}, up to using 2 different $\psi$, which is not a problem. We skip details. 
\end{rem}

\begin{rem} All the results of this section concerning $T=DB$ are valid with 
 $e^{-s|T|}$ replaced $\varphi(sT)$ where $\varphi\in H^\infty (S_{\mu})$ with $|\varphi(z)| \lesssim |z|^{-\alpha}$, $|\varphi(z)-\varphi(0)| \lesssim  |z|^\alpha$ for some $\alpha>0$. It suffices to write 
 $\varphi(z)=\varphi(0) e^{-\modz} + \psi(z)$. Concerning $T=BD$, all results hold in the range 
 $p_{-}(BD)<p<(p_{+}(BD))^*$ for such $\varphi$. For $p\ge (p_{+}(BD))^*$, we also impose 
 $\varphi\in \mR^2_{\sigma}$  for $\sigma$ large enough. 
\end{rem}

\subsection{$L^2$ estimates and Fatou type results}

\begin{thm}\label{thm:NTmaxandaeCV} Let $T=DB$ or $BD$. 
Then one has the estimate
\begin{equation}
\label{eq:Ntmax}
 \|\tN(e^{-t|T|}h) \|_{2} \sim \|h\|_{2} , \ \forall h\in \clos{\ran_{2}(T)}.
\end{equation}
Furthermore, for any $h\in L^2$ (not just $\clos{\ran_{2}(T)}$), we have that the Whitney averages of $e^{-t|T|}h$ converge to $h$ in $L^2$ sense, that is for  almost every $x_{0}\in \R^n$,
\begin{equation}
\label{eq:CVae}
\lim_{t\to 0}\ \bariint_{W(t,x_{0})} |e^{-s|T|}h-h(x_{0})|^2 =0.
\end{equation}
In particular, this implies the  almost everywhere convergence of Whitney averages
\begin{equation}
\label{eq:CVaew}
\lim_{t\to 0}\ \bariint_{W(t,x_{0})} e^{-s|T|}h =h(x_{0}).
\end{equation}
\end{thm}

\begin{proof} Let us begin with the non-tangential maximal estimate. 
 The bound from below is easy:
$$
\|h\|_{2}^2 = \lim_{t\to 0}  \frac 1t\int_t^{2t} \| e^{-s|T|}h \|_{2}^2\,  ds \lesssim \| \tN(e^{-s|T|}h) \|_{2}^2.
$$

Next, the bound from above for $T=DB$ is due to \cite[Theorem 5.1]{RosenJEE} (When $D$ has a special form it appeared first in  disguise in \cite{AAH}).   We provide a different proof in the spirit of the decompositions above. It is easy to check that $e^{-\modz} \in \mR^2_{2}(S_{\mu})$: there exist $\phi_{\pm}\in \mR^2(S_{\mu})$ such that $\psi_{\pm}(z):=(e^{-\modz} - \phi_{\pm}(z))\chi^\pm(z) \in \Psi_{2}^2(S_{\mu})$. Thus, $\tN(\psi_{\pm}(tDB) h) \lesssim \SF(\psi_{\pm}(tBD) h)$ and the $L^2$ bound follows from the square function bounds for $DB$. It remains to 
check the $L^2$ bounds for $\tN(\phi_{\pm}(tBD) h^\pm)$ where $h^\pm=\chi^\pm(DB)h$. It suffices to do it for $h\in \ran_{2}(D)$ by density. Thus $h^\pm\in \ran_{2}
(D)$ and there exist $v^\pm\in \dom_{2}(D)\cap \clos{\ran_{2}(D)}$ such that $h^\pm=Dv^\pm$, and  we can write 
$$\phi_{\pm}(tDB) h^\pm= D\phi_{\pm}(tBD) (v^\pm-c^\pm),$$ where $c^\pm$ is any constant. Fix a Whitney region $W(t,x)=(c_{0}^{-1}t, c_{0}t)\times B(x,c_{1}t)$, choose $c^\pm$ as the average of $v^\pm$ on the ball  $B(x,c_{1}t)$.  Using the local coercivity estimate \eqref{eq:localcoerc}, we have, with a slightly  enlarged Whitney region $\wt W(t,x)$ in the right hand side, 
\begin{align*}
\bariint_{W(t,x) }|\phi_{\pm}(tDB) h^\pm|^2 & \lesssim  \bariint_{\wt W(t,x) }|BD\phi_{\pm}(tBD) (v^\pm-c^\pm)|^2 \\ 
   & \qquad + t^{-2} \bariint_{\wt W(t,x) }|\phi_{\pm}(tDB) (v^\pm-c^\pm)|^2. 
   \end{align*}
As $\phi_{\pm}\in \mR^2(S_{\mu})$,   $z\phi^\pm$ and $\phi^\pm$ have $L^2$ off-diagonal decay  with decay as large as one wants, using the usual analysis in annuli and  Poincar\'e inequality  for $\frac{2n}{n+2}\le p< 2$ and $p\ge 1$, we obtain
$$
\bigg(\bariint_{W(t,x) }|\phi_{\pm}(tDB) h^\pm|^2\bigg)^{1/2} \lesssim \MM_{p}(\nabla v^\pm)(x).
$$
Thus, the $L^2$ norm of $ \tN(\phi_{\pm}(tDB) h^\pm)$ is controlled by $\|\nabla v^\pm\|_{2}$ and we use the coercivity of $D$ on $ \dom_{2}(D)\cap \clos{\ran_{2}(D)}$ to get a bound $\|Dv^\pm\|_{2}=\|h^\pm\|_{2}\lesssim \|h\|_{2}$.  The proof for $DB$ is complete.

The proof for $T=BD$ follows from the result for $DB$:  If $g\in \clos{\ran_{2}(BD)}$, then $B^{-1}g=h\in \clos{\ran_{2}(DB)}$  with $\|h\|_{2} \sim \|g\|_{2}$ and  $e^{-t|BD|}g=B e^{-t|DB|}h$. Thus
$$
 \| \tN(e^{-t|BD|}g) \|_{2} =  \| \tN(B e^{-t|DB|}h) \|_{2} \leq \|B\|_{\infty}  \| \tN( e^{-t|DB|}h) \|_{2}  \sim \|h\|_{2}\sim \|g\|_{2}. 
$$

It remains to show the  almost everywhere convergence result.   We begin with $BD$. Let $h\in L^2$. 
Pick $x_{0}$ a Lebesgue point for the condition
\begin{equation}
\lim_{t\to 0}\ \barint_{B(x_{0},t)} |h-h(x_{0})|^2=0. \end{equation}
Write as above, $e^{-s|BD|}h=\psi(sBD)h + (I+isBD)^{-1}h$ with $\psi(z)=e^{-\modz}- (1+iz)^{-1}$. The quadratic estimate  \eqref{eq:psiT} implies that 
$$
\lim_{t\to 0}\ \bariint_{W(t,x_{0})} | \psi(sBD)h|^2  =0
$$
for almost every $x_{0}\in \R^n$. Now the key point is that $Dc= 0$ if $c$ is a constant, thus $(I+isBD)^{-1} [h(x_{0})]= h(x_{0}).$  It follows that 
$$
 (I+isBD)^{-1}h -h(x_{0})= (I+isBD)^{-1}(h -h(x_{0})) $$
so that  for arbitrarily large $N$, 
\begin{equation}
\label{eq:NTbound}
 \bariint_{W(t,x_{0})}  | (I+isBD)^{-1} (h -h(x_{0}))|^2
\lesssim  \sum_{j\ge 1} 2^{-jN}  t^{-n} \int_{B(x_{0}, 2^j t)}  | h -h(x_{0}) |^2 .
\end{equation}
Breaking the sum at $j_{0}$ with $2^{-j_{0}}\sim \sqrt t$ and choosing $N\ge n+1$,  we obtain a bound
$$
 \sup_{\tau\le \sqrt t}\ \barint_{B(x_{0},\tau)} |h-h(x_{0})|^2 + \sqrt t  \, \MM(|h-h(x_{0})|^2)(x_{0}),
 $$
where $\MM$ is the Hardy-Littlewood maximal function. Using  the weak type (1,1) of $\MM$, almost every $x_{0}\in \R^n$ satisfy $\MM(|h|^2)(x_{0})<\infty$.  Hence,  the latter expression goes to 0  as $t\to 0$ at those $x_{0}$ meeting all the requirements.

We turn to the proof for $T=DB$. Let $g\in L^2$. If $g\in \nul_{2}(DB)$, this is a consequence of the Lebesgue differentiation theorem  on $\R^n$ as $e^{-s|DB|}g=g$ is independent of $s$. 
We assume next that $g\in \clos{\ran_{2}(DB)}$. As 
$$
\lim_{t\to 0}\ \bariint_{W(t,x_{0})} |g-g(x_{0})|^2= \lim_{t\to 0}\ \barint_{B(x_{0},c_{1}t)} |g-g(x_{0})|^2 = 0$$
for almost every $x_{0}\in \R^n$,
 it is enough to show the almost everywhere limit
$$
\lim_{t\to 0}\ \bariint_{W(t,x_{0})} |e^{-s|DB|}g-g|^2 =0.
$$
We also choose $x_{0}$ so that 
$$
 \lim_{t\to 0}\ \barint_{B(x_{0},c_{1}t)} |Bg-(Bg)(x_{0})|^2 = 0.$$
Write again
$e^{-s|DB|}g-g= \psi(sDB)g + (I+isDB)^{-1}g-g$. The quadratic estimate  \eqref{eq:psiT} implies that 
$$
\lim_{t\to 0}\ \bariint_{W(t,x_{0})} | \psi(sDB)g|^2  =0
$$
for almost every $x_{0}\in \R^n$. Now $(I+isDB)^{-1}g-g= -isDh_{s}$ with $ h_{s}=B(I+isDB)^{-1}g= (I+isBD)^{-1}(Bg)$ and $Bg\in L^2$. Let
 $$\tilde h_{s}:= (I+isBD)^{-1}(Bg)-       (Bg)(x_{0})  .  
      $$
      Applying Lemma \ref{lem:localcoerc} to $u=\tilde h_{s}$ using $Dh_{s}=D\tilde h_{s}$ and integrating with respect to $s$ implies
\begin{align*}
\bariint_{W(t,x_{0})} |isDh_{s}|^2   & \lesssim  \bariint_{\widetilde W(t,x_{0})} |isBDh_{s}|^2  +  \bariint_{\widetilde W(t,x_{0})} |\tilde h_{s}|^2  \\
    & \lesssim  \bariint_{\widetilde W(t,x_{0})} |(I+isBD)^{-1}(Bg) -Bg|^2 
    \\
  &  \qquad +  \bariint_{\widetilde W(t,x_{0})} \bigg|Bg- 
      (Bg)(x_{0})    |^2,
\end{align*}
where $\widetilde W(t,x_{0})$ is a slightly expanded version of $W(t,x_{0})$ and, in the last inequality, we have written $\tilde h_{s}= (I+isBD)^{-1}(Bg) -Bg + Bg-       (Bg)(x_{0}). $
The last  two integrals have been shown to converge to $0$ for almost every $x_{0}\in \R^n$ in the argument for $BD$.  This concludes the proof.
\end{proof}

\subsection{Lower bounds for $p\ne2$} 

A first argument follows from the almost everywhere bounds. 

\begin{prop}\label{prop:lowerp>1} Let $T=DB$ or $BD$ and $1<p<\infty$.
Then one has the estimate
\begin{equation}
\label{eq:lowerp>1}
 \|h\|_{p}  \lesssim  \|\tN(e^{-t|T|}h) \|_{p}, \ \forall h\in L^2.
 \end{equation}
\end{prop}

\begin{proof} It follows from the almost everywhere limit \eqref{eq:CVaew} that
\begin{equation}
\label{eq:controlae}
|h| \le  \tN(e^{-t|T|}h)
\end{equation}
almost everywhere. It suffices to integrate. 
\end{proof}

Our second result, inspired by an argument found in \cite{HMiMo} in the case of second order divergence form operators, yields the following improvement under a supplementary hypothesis.

\begin{prop}\label{prop:lowerp<1} Assume $B$ is pointwise accretive.  Let $T=DB$  and $\frac{n}{n+1}<p< \infty$. 
Then one has the estimate
\begin{equation}
\label{eq:lowerp<1}
 \|h\|_{H^p}  \lesssim  \|\tN(e^{-t|DB|}h) \|_{p}, \ \forall h\in \clos{\ran_{2}(D)}.
 \end{equation}
\end{prop}

There is no corresponding statement for $BD$ for $p\le 1$. It has to do with the cancellations.  Note that we assume \textit{a priori} knowledge for $h$ to make sense of the action of the semigroup. As we shall see, if we only have $B$ accretive on the range on $D$, our argument provides us with the weaker bound
$$
 \|h\|_{H^p}  \lesssim  \|\tN(e^{-t|DB|}h) \|_{p} + \|\tN(e^{-t|DB|}(\sgn (DB) h)) \|_{p}.
 $$
 
We begin with the following  Caccioppoli inequality.

\begin{lem}\label{lem:caccio} Assume $B$ is pointwise accretive. Assume $F,\partial_{t}F, DBF\in L_{loc}^2(\reu; \C^N)$,  and $F$  is  a solution of
\begin{equation}
\label{eq:weak}
\iint \pair {\partial_{t}F}{\partial_{t}G} + \pair {BDBF}{DG} =0, 
\end{equation}
for all compactly supported $G\in L_{loc}^2(\reu; \C^N)$ with  $\partial_{t}G, DG\in L_{loc}^2(\reu; \C^N)$, the inner product being the one of $\C^N$. Then 
\begin{equation}
\label{eq:Caccio}
\bariint_{W(t,x_{0})} |\partial_{t}F|^2 +  \bariint_{W(t,x_{0})} |DBF|^2 \leq \frac {C}{t^2} \bariint_{\wt W(t,x_{0})} |F|^2,
 \end{equation}
 where $W(t,x_{0})$ is a Whitney box and $\wt W(t,x_{0})$ a slightly enlarged Whitney box. The constant $C$ depends on the ratio of enlargements, dimension and accretivity bounds for $B$. 
 In particular this holds for  $F(t,x)=e^{-t|DB|}h(x)$ with $h\in \clos{\ran_{2}(D)}$.

\end{lem}

\begin{proof} Let us begin with the end of the statement. If $h\in \clos{\ran_{2}(D)}$, then by semigroup theory, for fixed $t$, $F$ and $\partial_{t}F$ are in $L^2$, as well as $DBF= -\sgn(DB) \partial_{t}F$ using the $H^\infty$-calculus. Now we remark that $F$ satisfies the equation $\partial_{t}^2F =DBDBF$ because $|DB|^2=DBDB$. Thus using the self-adjointness of $D$ and the skew-adjointness of $\partial_{t}$, we obtain \eqref{eq:weak}. 

Let us prove \eqref{eq:Caccio} assuming \eqref{eq:weak}. 
Let $\chi(s,y)$ be a  real-valued smooth function with support in $\wt W(t,x_{0})$, value 1 on $W(t,x_{0})$ and $|\nabla \chi| \lesssim \frac{1}{t}$. It is enough to prove
\begin{equation}\label{eq:caccio+}
\frac{\kappa}{2} \iint |\chi DBF|^2 +  \frac{\kappa}{2}\iint |\chi \partial_{t}F|^2 \lesssim   t^{-2} \iint_{\wt W(t,x_{0})} |F|^2, 
\end{equation}
where $\kappa $ is the accretivity constant for $B$. 
  Let $D_{\chi}=[D, \chi]$. As in the proof of  \eqref{eq:localcoerc}, $D_{\chi}$ is multiplication by a matrix supported on $\wt W(t,x_{0})$ and bounded by $Ct^{-1}$.  First, 
\begin{equation*}
\kappa \iint  |\chi DBF|^2   
\leq  \kappa \iint |D(\chi BF)|^2 
+ Ct^{-2} \iint_{\wt W(t,x_{0})} |F|^2. 
\end{equation*}
Then, the accretivity of $B$ on the range of $D$ yields 
\begin{equation*}
\kappa \iint |D(\chi BF)|^2  \leq   \re \iint   \pair{ BD(\chi BF)} {D(\chi BF)}
\end{equation*}
and the right hand side can be computed using
\begin{align*}
 \iint   \pair{ BD(\chi BF)} {D(\chi BF)}    &=  \iint   \pair{ BD_{\chi} BF} {D(\chi BF)} + \iint   \pair{ \chi BDBF} {D(\chi BF)}   \\
    &= \iint    \pair{ BD_{\chi} BF} {D_{\chi} BF}  -  \iint   \pair{ BD_{\chi} BF)} {\chi DBF} 
    \\
    & \qquad -  \iint   \pair{  BDBF} {D_{\chi}(\chi BF)}  + \iint   \pair{  BDBF} {D(\chi^2 BF)}. 
\end{align*}
In the last four integrals, the first is on the right order and  the second and third are controlled by 
absorption inequalities isolating $\chi DBF$ and we  arrive at 
\begin{equation}\label{eq:caccio'}
\frac{\kappa}{2} \iint |\chi DBF|^2  \lesssim \re \iint   \pair{ BDBF} {D(\chi^2 BF)} +  Ct^{-2} \iint_{\wt W(t,x_{0})} |F|^2.
\end{equation}
Similarly, using the pointwise accretivity of $B$, 
\begin{align*}
\kappa \iint |\chi \partial_{t}F|^2   & \le \re  \iint \pair  {\chi\partial_{t}F}  {B\chi\partial_{t}F}
  \\
    &= \re  \iint \pair{  \partial_{t} F} {\partial_{t}(\chi^2 BF)}   + 2\re   \iint \pair{  \chi\partial_{t} F} {\partial_{t}\chi BF}.
\end{align*}
Again, by absorption inequalities, we obtain
\begin{equation}\label{eq:caccio''}
\frac{\kappa}{2}\iint |\chi \partial_{t}F|^2 \le  \re  \iint \pair{  \partial_{t} F} {\partial_{t}(\chi^2 F)} +  Ct^{-2} \iint_{\wt W(t,x_{0})} |F|^2.
\end{equation}
Combining the two estimates \eqref{eq:caccio'} and \eqref{eq:caccio''}, and using  \eqref{eq:weak},  prove \eqref{eq:caccio+}, hence the lemma.
\end{proof}

\begin{rem}
If we only assume the accretivity of $B$ on $\clos{\ran_{2}(D)}$ then it is not clear how to dominate  
$\int\!\!\!\!\int|\chi \partial_{t}F|^2$ by an expression involving $F$. If $F=e^{-t|DB|}h$ then  observing that $\partial_{t}F= 
-DB(\sgn(DB) F)$ and one can  repeat the proof of \eqref{eq:caccio'} which we have done on purpose using only the accretivity of $B$  on the range.  But this brings  an average of $t^{-2}|\sgn(DB) F|^2$ in the right hand side, which means replacing $h$ by $\sgn(DB) h$.  
\end{rem}

\begin{proof}[Proof of Proposition \ref{prop:lowerp<1}]  We use auxiliary functions. Let  $a,b$ be the constants such that the function $\rho=a1_{[1,2)}+ b1_{[2,3)}$ satisfies  $\int \rho(s)\, ds=1$ and $\int \rho(s)s\, ds=0$. Define the bounded holomorphic function 
$$m(z)=\int_{1}^3 \rho(s) e^{- s\modz}\, ds
$$
in the half-planes $\re z>0$ and $\re z<0$ and at $z=0$ with $m(0)=1$. So one has $m(tDB)$ is well defined by the $H^\infty$-calculus.  Let $\tilde \rho(t)=- \int_{1}^t \rho(s) s\, ds=  \int_{t}^\infty  \rho(s) s \, ds. $ Thus $\tilde \rho$ has support in $[1,3]$ as well. Integrating by parts, we have
\begin{align*}
m'(z)    &=-\sgn(z) \int_{1}^3 \rho(s)s e^{- s\modz}\, ds   \\
    &  = \sgn(z)  \int_{1}^3 \tilde \rho(s)  \modz e^{- s\modz}\, ds \\
    & =   \int_{1}^3 \tilde \rho(s) z e^{- s\modz}\, ds.
\end{align*}
Now, set $F_{t}=e^{-t|DB|}h$,   $G_{t}=m(tDB)h$ and $\wt G_{t}=m'(tDB)h$. We have 
$$  G_{t} 
= \int_{1}^3 \rho(s) F_{st}\, ds,  \quad 
   \wt G_{t} 
     = \int_{1}^3  \frac
   {\tilde\rho(s)}
   {s} 
   (stDB
   F_{st}
   )
   \, ds,
$$ 
and it follows from the support in [1,3] of $\rho$ and $\tilde \rho$  that 
$$
\tN(G) + \tN(\wt G) \lesssim \tN(F) + \tN(tDBF)
.
$$
Thus, using Lemma \ref{lem:caccio} and 
adjusting the parameters in Whitney boxes, it suffices to prove
$$
\|h\|_{H^p}  \lesssim  \|\tN(G) \|_{p} +  \|\tN(\tilde G) \|_{p}
.
$$
Using the formula for $G_{t}$, and $F_{t}\to h$  in $\mH$ when $t\to 0$ and  $h\in \clos{\ran_{2}(D)}$, we have $G_{t}\to h$ in $L^2$ (convergence in the Schwartz distributions suffices for this argument) as $t\to 0$.
To evaluate the $H^p$ norm, we use the maximal characterisation of Fefferman and Stein: Let $\varphi(y)= r^{-n}\phi(\frac{x-y}{r})=\phi_{r}(x-y)$ for some fixed function $\phi$ assumed to be  $C^\infty$, real-valued, compactly supported in $B(0,c_{1})$ with  $\int \phi =1$.  It is enough to  prove   
\begin{equation}
\label{eq:Hp}
\bigg |\int_{\R^n} h\varphi \bigg| \lesssim  \tN (G)(x)+ \MM_{\frac{n}{n+1}}( \tN (\wt G))(x),
\end{equation}
since this shows that $\sup_{r>0} |h * \phi_{r} |$ is controlled by an $L^p$ function as desired. The argument works for $\frac{n}{n+1}<p\le 1$ by the Fefferman-Stein's theorem, but also for $1<p<\infty$ by Lebesgue's theorem. 

To prove \eqref{eq:Hp}, let $\chi(t)$ be an $L^\infty$-normalized, scalar, bump function on $[0,\infty)$: it is $C^1$, supported in  $[0, c_{0}r)$ with value 1 on  
$[0, c_{0}^{-1}r]$ and $\|\chi\|_{\infty}+r\|\chi'\|_{\infty}\lesssim 1$.
The function $\Phi(s,y)= \varphi(y)\chi(s)$ is an extension of $\varphi$ to $\R^{1+n}_{+}$. Thus 
$$\int_{\R^n} h\varphi = -\iint_{\reu} \partial_{s}(G\Phi)= -\iint_{\reu} G\partial_{s}\Phi - \iint_{\reu} \partial_{s} G \Phi= I + II.
$$
Note that the integrand of $I$  is supported in the Whitney box $W(r, x)$, so this integral is dominated by $\tN G(x)$. For $II$, observe that $\partial_{s}G=DBm'(sDB)h= DB\wt G_{s}$. Integrating $D$ by parts, and using the boundedness of $B$, we obtain 
$$
\bigg| \iint_{\reu} \partial_{s} G \Phi \bigg| \lesssim  \iint_{T} |\wt G| \|\nabla_{y} \Phi\|_{\infty} \lesssim r^{-n-1} \iint_{T} |\wt G| ,
$$
where $T:=(0, c_{0}r)\times B(x,c_{1}r)$. Then,  using the inequality 
$$
\iint_{\reu} |u| \lesssim \| \tN u\|_{{\frac{n}{n+1}}}
$$
found in \cite{HMiMo} for $u=|\wt G|1_{T}$ and support considerations, we obtain 
$$
r^{-n-1} \iint_{T} |\wt G| \lesssim \left( r^{-n} \int_{(1+c_{0})B(x,c_{1}r)} (\tN (\wt G))^{\frac{n}{n+1}} \right)
^{\frac{n+1}{n}}$$
and \eqref{eq:Hp} is proved. 
\end{proof}

\begin{rem}  An examination of the argument above shows that one can take the $q$-variant $\tN^q$ with any $q\in [1,2]$.
\end{rem}

\begin{prop}\label{prop:lowerp<1a} Let $T=DB$  and $\frac{n}{n+1}<p< \infty$.  
Then one has the estimate
\begin{equation}
\label{eq:lowerp<1a}
 \|h\|_{H^p}  \lesssim  \|\tN(e^{-t|DB|}h) \|_{p}, \ \forall h\in \IH^{2,\pm}_{DB}.
 \end{equation}
\end{prop}

Here the difference is that we  restrict $h$ in one of the spectral spaces. 

\begin{proof} If $h\in \IH^{2,+}_{DB}$, then $F=e^{-t|DB|}h=e^{-tDB}\chi^+(DB)h$ and $\partial_{t}F= -DBF$. Thus we can run the previous argument with $F$ replacing $G$ and get the inequality \eqref{eq:Hp} with $F$ replacing both $G$ and $\wt G$. 

When  $h\in \IH^{2,-}_{DB}$, then $F=e^{-t|DB|}h=e^{tDB}\chi^-(DB)h$ and $\partial_{t}F= DBF$, so that  we conclude as above.
\end{proof}

\subsection{Some upper bounds for $p\ne 2$} 

\begin{prop}\label{prop:upperp>2} Let $T=DB$ or $BD$ and $2<p<(p_{+}(T))^*$.
Then one has the estimate
\begin{equation}
\label{eq:lowerp>2}
  \|\tN(e^{-t|T|}h) \|_{p} \lesssim  \|h\|_{p}, \ \forall h\in L^2.
 \end{equation}
\end{prop}

\begin{proof} Write $e^{-t|T|}h=\psi(tT)h + (I+itT)^{-1}h$ where $\psi(z)=e^{-\modz}- (1+iz)^{-1}\in \Psi_{1}^1(S_{\mu})$. By geometric considerations, 
$$
 \|\tN(\psi(tT)h) \|_{p} \lesssim \|\psi(tT)h\|_{T^p_{2}} 
 $$
 and we may apply Corollary \ref{cor:psi1} to obtain
 $$
 \|\psi(tT)h\|_{T^p_{2}}  \lesssim \|h\|_{p}
 $$
 in the given range of $p$. 
 Next, the $L^2$ off-diagonal estimates \eqref{lem:odd}  for the resolvent $(I+itT)^{-1}$ yields the pointwise estimate
 $$
  \tN((I+itT)^{-1}h) \lesssim \MM_{2}(|h|)
  $$
which gives an $L^p$ estimate for all $2<p\le \infty$. 
\end{proof}

Note that the argument for $BD$ provides a proof of the assertion in Remark \ref{rem:ntbd}.

We continue  with some upper bounds when $p<2$.

\begin{prop}\label{cor:ntp<2} \begin{enumerate}
  \item For $p_{-}(BD)<p<2,$   we have 
$ \|\tN(e^{-t|BD|}h) \|_{p} \lesssim    \|h\|_{p}$ for all $ h\in \clos{\ran_{2}(BD)}$.
  \item For $(p_{-}(DB))_{*}<p<2$,   we have 
$ \|\tN(e^{-t|DB|}h) \|_{p} \lesssim  \|h\|_{H^p}$ for all $ h\in \clos{\ran_{2}(D)}$ where $H^p=L^p$ if $p>1$.
  \end{enumerate}\end{prop}

\begin{proof} The first item follows from   Lemma \ref{lem:upperp<2} and Theorem  \ref{thm:hpbd}: for 
 $ h\in \clos{\ran_{2}(BD)}$ and $p_{-}(BD)<p<2$
 $$
\|\tN(e^{-t|BD|}h) \|_{p} \lesssim  \|h\|_{\IH^p_{BD}} \sim  \|\IP h\|_{p}\sim \|h\|_{p}.
$$
The equivalence $\|h\|_{p}\sim \|\IP h\|_{p}$ for all $ h\in \clos{\ran_{2}(BD)}$ in this range of $p$ was obtained in Proposition \ref{prop:projectionlp}.

The second item   follows from Lemma \ref{lem:upperp<2} and Theorem \ref{thm:hpdb}: for $ h\in \clos{\ran_{2}(D)}$,
 $$
\|\tN(e^{-t|DB|}h) \|_{p} \lesssim  \|h\|_{\IH^p_{DB}} \sim  \|h\|_{H^p}. 
$$
\end{proof}

\subsection{End of proof of Theorem \ref{thm:ntdb}}

For the lower bounds, combine Propositions \ref{prop:lowerp>1} and \ref{prop:lowerp<1} when $B$ is pointwise accretive and Proposition \ref{prop:lowerp<1a} in general. We note that we do not use the assumption on equality of Hardy spaces in the statement. 

We turn to upper bounds. So far we have completed the theorem  when $p> p_{-}(DB)$ on applying Propositions  \ref{prop:upperp>2} and \ref{cor:ntp<2}, (2).  But by Theorem \ref{thm:hpdb}, the argument of Proposition  \ref{cor:ntp<2}, (2), applies when $p<2$ is such that $\IH^p_{DB}=\IH^p_{D}$ with equivalence of norms. This concludes the proof.

\subsection{End of proof of Theorem \ref{thm:ntbd}}

  Combining Propositions \ref{prop:lowerp>1}, \ref{prop:upperp>2} and  \ref{cor:ntp<2} gives all the lower bounds for any $p>1$ and also the upper bounds  in the range $p_{-}(BD)<p<(p_{+}(DB))^*$ by specializing to $\IP h$ for  $ h\in \clos{\ran_{2}(BD)}$.  
  
  It remains to provide an argument for upper bounds when $p\ge (p_{+}(DB))^*$ and  $p=q'$ where $q>1$ is assumed such that $\IH^q_{DB^*}=\IH^q_{D}$ with equivalence of norms.  We do this  now. 
  
  As in  the proof of Theorem \ref{thm:NTmaxandaeCV}, observe that our assumption implies
  for $h\in \clos{\ran_{2}(BD)}$,  $\|\IP \chi^\pm(BD)h\|_{p}\lesssim \|\IP h\|_{p}$. Now  $\phi(z)=e^{-\modz} \in\mR^2_{\sigma}(S_{\mu})\cap \Psi^\tau_{0}(S_{\mu})$ for any $\sigma>0$ and $\tau>0$. 
Pick $\phi_{\pm}\in \mR^{2}(S_{\mu})$ such that 
\begin{equation*}
|\phi(z)-\phi_{\pm}(z)|=O(|z|^\sigma), \quad \forall z\in S_{\mu\pm}.
\end{equation*}
Then  $\psi_{\pm}(z):=(\phi-\phi_{\pm})(z)\chi^\pm(z)$ satisfy $\psi_{\pm}\in \Psi_{\sigma}^{2}(S_{\mu})$. Hence, for  $h\in \clos{\ran_{2}(BD)}$, using $h=\chi^{+}(BD)h+ \chi^{-}(BD)h=h^++h^-$, we have the decomposition
$$
\phi(tBD) \IP h = \psi_{+}(tBD) \IP h + \psi_{-}(tBD)\IP h+ \phi_{+}(tBD)\IP h^+ +  \phi_{-}(tBD)\IP h^-.
$$
  From geometric considerations,  we deduce from Lemma \ref{lem:sfpinfty}   if  $\sigma$ is large enough
 $$
\|\tN(\psi_{+}({tBD})\IP  h) \|_{p} \lesssim \|\psi_{+}(tBD)\IP h \|_{T^p_{2}}  \lesssim  \|\IP h\|_{p}
$$
and similarly for the term with $\psi_{-}$.
Next, the $L^2$ off-diagonal estimates of Lemma \ref{lem:odd}  for the combinations of iterates of resolvents $(I+itT)^{-2}$ yields the pointwise estimate
 $$
  \tN(\phi_{+}(tBD)\IP h^+) \lesssim \MM_{2}(|\IP h^+|)
  $$
  Thus, as $p>2$ and using the assumption on $p$, 
  $$
\|\tN(\phi_{+}(tBD)\IP h^+) \|_{p}   \lesssim \|\IP h^+ \|_{p}  \lesssim  \|\IP h\|_{p}.
$$
We argue similarly for $\phi_{-}(tBD)\IP h^-$.  This finishes the proof.

\section{Non-tangential sharp functions for $BD$}\label{sec:sharp}

As we saw, the  non-tangential maximal inequality that involves the pre-Hardy space 
$\IH^p_{BD}$ is with $e^{-t|BD|}\IP $, that is taking the semigroup after having projected on $\clos{\ran_{2}(D)}$. The problem with $\IP$ is one cannot use kernel estimates in such a  context as it is a singular integral operator. 

Also when  for some reason (for example $p_{+}>n$), we want to reach $\BMO$ or $\dot \Lambda^\alpha$ spaces, the non-tangential maximal function is inappropriate.  

We observe that for all $h\in L^2$ and all $t>0$ we have the following relation
\begin{equation}
\label{eq:sharp}
e^{-t|BD|}\IP h -\IP h= e^{-t|BD|}h -h.
\end{equation}
Indeed, $g=\IP h - h \in \nul_{2}(D)= \nul_{2}(BD)$, so that $e^{-t|BD|} g= g$ for all $t>0$. 

We are therefore led to consider
$$
\tNs(e^{-t|BD|}h):= \tN (e^{-t|BD|}h -h),$$which we name  \emph{non-tangential sharp function (of $e^{-t|BD|}h$) associated to $BD$}.
Thanks to \eqref{eq:sharp}, 
we have
$$
|\tNs (e^{-t|BD|}h) - \tN (e^{-t|BD|}\IP h)|\le  \MM_{2}(|\IP h|).
$$
Thus, if $2<p$, $\tNs (e^{-t|BD|}h)$ and $ \tN (e^{-t|BD|}\IP h)$ have same $L^p$ behavior. 
In particular, 
$$
\|\tNs (e^{-t|BD|}h)\|_{p} \lesssim \|\IP h\|_{p}
$$
holds in the range of $p>2$ where the same upper bound holds for $ \tN (e^{-t|BD|}\IP h)$.
If this range is all $(2,\infty)$ we may wonder what happens at $p=\infty$.

It is also convenient to introduce the $\alpha\ge 0$ variant of $\tNs$:
$${\tNsa}(e^{-t|BD|}h)(x) = \sup_{t>0}   t^{-\alpha} \bigg(\bariint_{W(t,x)} |e^{-s|BD|}h-h|^2\bigg)^{1/2}.
$$
Note that  for $\alpha=0$, this is $\tNs$.

\begin{thm} Assume that for some $q$ with $\frac{n}{n+1}<q< 2$,    we have $\IH^q_{DB^*}= \IH^q_{D}$ with 
equivalent norms.
 If $q> 1$ and  $p=q'$, we have
$$ \|\tNs(e^{-t|BD|} h) \|_{p} \sim  \|\IP h\|_{p},\quad \forall h\in \clos{\ran_{2}(BD)},$$ 
 and if $q\le 1$ and $\alpha=n(\frac{1}{q}-1)$, 
$$ \|\tNsa(e^{-t|BD|} h) \|_{\infty}\sim  \|\IP h\|_{\dot \Lambda^\alpha}, \quad \forall h\in \clos{\ran_{2}(BD)}.$$

\end{thm}

This result rests on two lemmata. 

\begin{lem}\label{lem:sharp1}For $2<p\le \infty$, we have
$$
\|h\|_{\IH^p_{BD}} \lesssim \|\tNs(e^{-t|BD|}h)\|_{p} ,\quad \forall h\in \clos{\ran_{2}(BD)},
$$
and for $0\le \alpha < 1$,
$$
\|h\|_{\IL^\alpha_{BD}} \lesssim \|\tNsa(e^{-t|BD|}h)\|_{\infty},\quad \forall h\in \clos{\ran_{2}(BD)}.
$$

\end{lem}

\begin{lem}\label{lem:sharp2} For $2<p\le \infty$, we have
$$
 \|\tNs(e^{-t|BD|}h)\|_{p} \lesssim \|\IP h^+\|_{p}+ \|\IP h^-\|_{p} +  \|h\|_{\IH^p_{BD}},\quad \forall h\in \clos{\ran_{2}(BD)},
$$
and for $0\le \alpha < 1$,
$$
 \|\tNsa(e^{-t|BD|}h)\|_{\infty} \lesssim \|\IP h^+\|_{\dot \Lambda^\alpha} + \|\IP h^-\|_{\dot \Lambda^\alpha}  + \| h\|_{\IL^\alpha_{BD}}  ,\quad \forall h\in \clos{\ran_{2}(BD)},
 $$
 where $h^\pm= \chi^\pm(BD)h$. 
\end{lem}

Let us admit the lemmata and prove the theorem. As seen many times, if $q>1$ and $p=q'$, the hypothesis
implies that $\|\IP h^+\|_{p}+ \|\IP h^-\|_{p} \sim \|\IP h\|_{p} \sim  \|h\|_{\IH^p_{BD}}$. If $q\le 1$ and $\alpha=n(\frac{1}{q}-1)$ then
$\|\IP h^+\|_{\dot \Lambda^\alpha} + \|\IP h^-\|_{\dot \Lambda^\alpha}  \sim \|\IP h\|_{\dot \Lambda^\alpha}\sim  \| h\|_{\IL^\alpha_{BD}}$. The conclusion follows right away. 

\begin{proof}[Proof of Lemma \ref{lem:sharp1}] To prove this result, we introduce
the Carleson function $C_{\alpha}F$ by
$$
C_{\alpha}F(x):= \sup \bigg(\frac {1}{r^{2\alpha}  |B(y,r)|}\iint_{T_{y,r}}   |F(t,z)|^2\, \frac{dtdz}{t}\bigg)^{1/2},
$$
the supremum being taken over all  open balls $B(y,r) \ni x$  in $\R^n$ and $T_{y,r}=(0,r)\times B(y,r)$. 
For $0\le \alpha<1$ and a suitable allowable $\psi$ for both $\IH^p_{BD}$ and $\IL^\alpha_{BD}$,
we shall show the pointwise bound
\begin{equation}
\label{eq:Carlalpha}
C_{\alpha}(\psi(tBD)h) \lesssim \MM_{2}\big(\tNsa(e^{-t|BD|}h)\big), 
\quad \forall h\in \clos{\ran_{2}(BD)}. 
\end{equation}
Admitting this inequality, we have
$$
 \|h\|_{\IH^p_{BD}} \lesssim \|\psi(tBD)h\|_{T^p_{2}}\lesssim \|C_{0}(\psi(tBD)h)\|_{p}\lesssim  \|\tNs(e^{-t|BD|}h)\|_{p}.
 $$
 The first inequality is the lower bound valid for any $\psi\in \Psi(S_{\mu})$, the second inequality is from   \cite[Theorem 3(a)]{CMS} and the last one uses  \eqref{eq:Carlalpha},  the maximal theorem and $p>2$.  Similarly
 $$
 \|h\|_{\IL^\alpha_{BD}} \lesssim \|\psi(tBD)h\|_{T^\infty_{2,\alpha}}= \|C_{\alpha}(\psi(tBD)h)\|_{\infty }\lesssim  \|\tNsa(e^{-t|BD|}h)\|_{\infty}.
 $$
 
 We turn to the proof of \eqref{eq:Carlalpha}.  We adapt an argument in \cite{DY}, Theorem 2.14, to our situation. We  choose $\tilde\psi(z)=z^N e^{-[z]}$ and $ \psi(z)=\tilde\psi(z)(e^{-\modz }-1)$ so that $\tilde\psi\in \Psi_{N}^\tau(S_{\mu})$ and $\psi\in \Psi_{N+1}^\tau(S_{\mu})$ for all $\tau>0$. The integer $N$ will be chosen large. It will be convenient to set $P_{t}=e^{-t|BD|}$, so that $\tilde\psi(tBD)= (tBD)^NP_{t}$ and $\psi(tBD)= (tBD)^NP_{t}(P_{t}-I)$.
 
 We fix $h\in \clos{\ran_{2}(BD)}$ and $x\in \R^n$. Consider $T_{y,r}=(0,r)\times B(y,r)$ such that $x\in B(y,r)$. Recall that $W(t,z):= (c_0^{-1}t,c_0t)\times B(z,c_1t)$, for some fixed constants $c_0>1$, $c_{1}>0$. We set $I_{t}= (c_0^{-1}t,c_0t)$.
 
Set $g=h- \barint_{\hspace{2pt}I_{r}} P_{\tau}h\, d\tau$ and   consider $I(y,r)=\int\!\!\!\!\int_{T_{y,r}}   |\psi(sBD)g(z)|^2\, \frac{dsdz}{s}$. Pick $a>0$ such that the balls $B_{k}=B(x+akr, \frac{c_{1}}{2}r)$, $k\in \Z^n$,  cover $\R^n$ with bounded overlap.  We set $g_{k}=g1_{B_{k}}$. If $B_{k}\cap 2B(y,r)\ne \emptyset$, which occurs for boundedly  (with respect to $x, y, r$) many $k$  then we use  the square function estimate  and definition of $g_{k}$ to  obtain
$$
\iint_{T_{y,r}}   |\psi(sBD)g_{k}(z)|^2\, \frac{dsdz}{s} \lesssim \|g_{k}\|_{2}^2 \le |B_{k}| \bariint_{I_{r}\times B_{k}}  |h - P_{\tau}h|^2.
$$
If $B_{k}\cap 2B(y,r)= \emptyset$,  which occurs when $|k| \ge K$ for some integer $K\ne 0$, then we can use the $L^2$ off-diagonal decay  \eqref{eq:odnpsipq} for each $s$ to obtain
  $$
\iint_{T_{y,r}}   |\psi(sBD)g_{k}(z)|^2\, \frac{dsdz}{s} \lesssim  |k|^{-2(N+1)} \|g_{k}\|_{2}^2
\le |k|^{-2(N+1)}|B_{k}| \bariint_{I_{r}\times B_{k}}  |h - P_{\tau}h|^2.
$$
For $N+1>n$, we obtain (using Minkowski  inequality for the integral followed by Cauchy-Schwarz inequality for the sum) 
$$
   I(y,r) \lesssim \sum_{k\in \Z^n} (1+|k|)^{-N-1} |B_{k}| \bariint_{I_{r}\times B_{k}}  |h - P_{\tau}h|^2.
  $$
 Now observe that $|B_{k}|=2^{-n}|B(z,c_{1}r)|$ and if $z\in B_{k}$, then $B_{k}\subset B(z, c_{1}r)$. Hence 
\begin{align*}
 |B_{k}| \bariint_{I_{r}\times B_{k}}  |h - P_{\tau}h|^2 & \le 2^n |B_{k}| \inf_{z\in B_{k}}\bariint_{W(r,z)}  |h - P_{\tau}h|^2\\
 &
 \le 2^n  r^{2\alpha}|B_{k}| \inf_{z\in B_{k}}  \tNsa(e^{-t|BD|}h)^2(z) 
 \\
 &\le 2^n r^{2\alpha}\int_{B_{k}} \tNsa(e^{-t|BD|}h)^2(z)  \, dz  
\end{align*}
and this implies
$$ I(y,r) \lesssim r^{2\alpha}\sum_{k\in \Z^n} (1+|k|)^{-N-1} \int_{B_{k}} \tNsa(e^{-t|BD|}h)^2(z)  \, dz \lesssim \MM_{2}\big(\tNsa(e^{-t|BD|}h)\big)^2(x) r^{n+2\alpha},
 $$
 where  the last inequality uses the bounded overlap of the balls $B_{k}$ and requires $N+1>n$.

Next, we bound $J(y,r)=\int\!\!\!\!\int_{T_{y,r}}   |\psi(sBD)\big(\barint_{\hspace{2pt}I_{r}} P_{\tau}h\, d\tau\big)(z)|^2\, \frac{dsdz}{s}$. We compute
\begin{align*}
   \psi(sBD) P_{\tau}&= (sBD)^N P_{s+\frac{\tau}{2}}(P_{s+\frac{\tau}{2}}-P_{\frac{\tau}{2}})\\&=(sBD)^N P_{s+\frac{\tau}{2}}(P_{s+\frac{\tau}{2}}-I)+ (sBD)^N P_{s+\frac{\tau}{2}}(I-P_{\frac{\tau}{2}}).    
    \end{align*}
Let us call $J_{1}(y,r)$ and $J_{2}(y,r)$ the integrals corresponding to the first term and  second term respectively. We first handle $J_{2}$. Use $s\le s+\frac{\tau}{2}$,  change variable $s\mapsto s+\frac{\tau}{2}$,   and observe that as $\tau\in I_{r}$ and $0<s<r$, we have $s+\frac{\tau}{2} \in [\frac{c_{0}^{-1}}{2}r, r+ \frac{c_{0}}{2}r]=J_{r}$. Thus, 
\begin{align*}
J_{2}(y,r)& \lesssim \barint_{\hspace{-2pt}I_{r}}\int_{B(y,r)}\int_{J_{r}} | \tilde\psi(sBD)(h-P_{\frac{\tau}{2}}h)(z)|^2\, \frac{ds}{s}dzd\tau\\
 & = \barint_{\hspace{-2pt}I_{\frac{r}{2}}}\int_{B(y,r)}\int_{J_{r}} | \tilde\psi(sBD)(h-P_{\tau}h)(z)|^2\, \frac{ds}{s}dzd\tau.     
\end{align*}
We use the $L^2$ off-diagonal  estimates for $\tilde\psi(sBD)$ with $N>n$, which are uniform in $s\in J_{r}$, and obtain the desired bound on $J_{2}(y,r)$ with the same analysis (change $r$ to $\frac{r}{2}$ in the definition of the balls $B_{k}$) as above. 

For $J_{1}$, we operate the same change of variable to get
\begin{align*}
  J_{1}(y,r)&\lesssim \barint_{\hspace{-2pt}I_{r}}\int_{B(y,r)}\int_{J_{r}} | \tilde\psi(sBD)(P_{s}h-h)(z)|^2\, \frac{ds}{s}dzd\tau \\
 & = \int_{B(y,r)}\int_{J_{r}} | \tilde\psi(sBD)(P_{s}h-h)(z)|^2\, \frac{ds}{s}dz .  
      \end{align*}
Now, we observe that $J_{r}$ can be covered by a bounded (with respect to $r$) number of interval $I_{c_{0}^{2i} \frac{r}{2}} $ 
We proceed   a similar analysis as before  for each integral $\int_{B(y,r)}\int_{I_{c_{0}^{2i} \frac{r}{2}} }$  with the appropriate $B_{k}$ type balls, use the $L^2$ off-diagonal  estimates for $\tilde\psi(sBD)$ with $N>n$. This leads to  the same bound for $J_{1}(y,r)$ as for $I(y,r)$. 
We leave details to the reader. \end{proof}

\begin{proof}[Proof of Lemma \ref{lem:sharp2}]  We begin with the $L^p$ estimates and  proceed exactly as in the proof of Theorem \ref{thm:ntbd}.  We have $\phi(z)=e^{-\modz} \in\mR^2_{\sigma}(S_{\mu})\cap \Psi^\tau_{0}(S_{\mu})$ for any $\sigma>0$ and $\tau>0$. 
Pick $\phi_{\pm}\in \mR^{2}(S_{\mu})$ such that 
\begin{equation*}
|\phi(z)-\phi_{\pm}(z)|=O(|z|^\sigma), \quad \forall z\in S_{\mu\pm}.
\end{equation*}
Then  $\psi_{\pm}(z):=(\phi-\phi_{\pm})(z)\chi^\pm(z)$ satisfy $\psi_{\pm}\in \Psi_{\sigma}^{2}(S_{\mu})$. Hence, for  $h\in \clos{\ran_{2}(BD)}$, using $h=\chi^{+}(BD)h+ \chi^{-}(BD)h=h^++h^-$, we have the decomposition
$$
\phi(tBD) \IP h -\IP h = \psi_{+}(tBD) \IP h + \psi_{-}(tBD)\IP h+ \phi_{+}(tBD)\IP h^+  -\IP h^+ +  \phi_{-}(tBD)\IP h^-  -\IP h^-.
$$
  From geometric considerations,  we deduce from Lemma \ref{lem:sfpinfty}   if  $\sigma$ is large enough
 $$
\|\tN(\psi_{+}({tBD})\IP  h) \|_{p} \lesssim \|\psi_{+}(tBD)\IP h \|_{T^p_{2}}  \lesssim  \|h\|_{\IH^p_{BD}}
$$
and similarly for the term with $\psi_{-}$.
Next, the $L^2$ off-diagonal estimates of Lemma  \ref{lem:odd}  for the combinations of iterates of resolvent $(I+itT)^{-2}$ yield the pointwise estimate
 $$
  \tN(\phi_{+}(tBD)\IP h^+- \IP h^+ ) \lesssim \MM_{2}(|\IP h^+|).
  $$
  Thus, as $p>2$, 
  $$
\|\tN(\phi_{+}(tBD)\IP h^+- \IP h^+ ) \|_{p}   \lesssim \|\IP h^+ \|_{p}.
$$
We argue similarly for $\phi_{-}(tBD)\IP h^-$.  This proves the first estimate since  $\phi(tBD) \IP h -\IP h= \phi(tBD)  h -h$. 

For the H\"older estimates, we use the same decomposition and observe that $\tNsa(g)\lesssim C_{\alpha}g$ pointwise. Hence, for $\sigma$ large enough, 
$$
\|\tNsa(\psi_{+}({tBD})\IP  h) \|_{\infty} \lesssim \|\psi_{+}(tBD)\IP h \|_{T^\infty_{2,\alpha}}  \lesssim  \| h\|_{\IL^\alpha_{BD}}
$$ 
and similarly for the term with $\psi_{-}$. 
Next, we fix a Whitney box $W(t,x)$ and let $c^{\pm}$ be the average of $\IP h^\pm$ on the ball $B(x,c_{1}t)$. Then we write 
$$
\phi_{+}(sBD)\IP h^+- \IP h^+= \phi_{+}(sBD)(\IP h^+-c^+)- (\IP h^+-c^+).
$$
The $L^2$ off-diagonal estimates of Lemma \ref{lem:odd}  for the combinations of iterates of resolvent $(I+itT)^{-2}$ yield the pointwise estimate
$$
\tNsa(\phi_{+}(sBD)\IP h^+ )^2(x) \lesssim \sup_{t>0} t^{-\alpha} \barint_{B(x,c_{1}t)} |\IP h^+-c^+|^2
$$
which leads to the estimate
$$
\|\tNsa(\phi_{+}(sBD)\IP h^+ ) \|_{\infty}   \lesssim \|\IP h^+ \|_{\dot\Lambda^\alpha}.
$$
The argument for $\phi_{-}(tBD)\IP h^-$ is similar. 
\end{proof}

\section{Sobolev  spaces for $DB$ and $BD$}\label{sec:Sobolev}

So far, we have privileged the  $L^2$ theory: we considered estimates with \textit{a priori} knowledge for $h$ in the closure of the $L^2$ range. But this is only for convenience.
As mentioned in the introduction, we can consider a Sobolev theory as well and relax this \textit{a priori} information on $h$. This is required for use of energy spaces. For any bisectorial operator with a $H^\infty$-calculus on the closure of its range, there is a Sobolev space theory associated to this operator as developed by means of quadratic estimates in this context in \cite{AMcN1}, extending many earlier works for self-adjoint operators, positive operators... (see the references there). But here, we want a theory that leads to concrete spaces. 

For  the operator $DB$, the relevant  Sobolev theory is for regularity indices  $s\in [-1,0]$. For $s=0$, this is already done. We shall do it for $s<0$ in this section. This has been considered in some special cases  for $D$ in relation with the boundary value problems \cite{R2, AMcM}.  For $BD$, things are more complicated.  There are two options for regularity indices  $0\le s\le 1$: the  Sobolev spaces associated to $BD$ or the  Sobolev spaces   associated to  the operators $\IP BD$ after projecting by $\IP$. The first theory leads to abstract spaces and the second to concrete spaces. They are both useful. 

\subsection{Definitions and properties}

For convenience, we denote by 
 $\mH^0_{D}=\clos{\ran_{2}(\dirac)}$ and $\mH= L^2(\R^n;\C^N)$.  
Let $S= \dirac|_{\mH^0_{D}}$ with domain $\dom_{2}(\dirac)\cap \mH^0_{D}$. Then $S$ is an injective, self-adjoint operator. Recall that $\IP$ is the orthogonal projection from $\mH$ onto $\mH^0_{D}$. Let $\bet$ be the operator on $\mH^0_{D}$ defined by
$\bet h =\IP Bh = \IP B\IP h$ for  $h\in \mH^0_{D}$.  Recall that as $B$ is a strictly accretive  operator on $\mH^0_{D}$, the restriction of   $\IP$ on $B\mH^0_{D}$ is an isomorphism onto $ \mH^0_{D}$  and  $\bet$ is a strictly accretive operator on $\mH^0_{D}$. 

Define 
    $$
T: \mH^0_{D} \to \mH^0_{D}, \quad T= \bet S= \IP  B\dirac_{|\mH^0_{D}} \ \mathrm{with}\  \dom_{2}(T)=\dom_{2}(S)
$$
and 
$$
\uT: \mH^0_{D} \to \mH^0_{D}, \quad \uT= S \bet = \dirac \IP B_{|\mH^0_{D}}= \dirac B_{|\mH^0_{D}}\ \ \mathrm{with}\  \dom_{2}(\uT)=\bet^{-1}\dom_{2}(S).
$$
Using  Proposition \ref{prop:typeomega} and the comment that follows it, $T$ and $\uT$ are $\omega$-bisectorial operators on $\mH^0_{D}$. Moreover, they are injective. Observe also that 
$$
V: \clos{\ran_{2}(B\dirac)} \to \clos{\ran_{2}(B\dirac)}, \quad V=BD|_{\clos{\ran_{2}(B\dirac)}}\ \ \mathrm{with}\  \dom_{2}(V)={\clos{\ran_{2}(B\dirac)}} \cap \dom_{2}(D)
$$
is also an injective $\omega$-bisectorial operator with $H^\infty$-calculus on $\clos{\ran_{2}(B\dirac)}$. 

We remark that if  $\psi \in \Psi(S_{\mu})$, we have the intertwining relation
 \begin{equation}
\label{eq:intert}
\psi(T)\IP h= \IP \psi(BD) h= \IP \psi(V) h, \quad h\in \clos{\ran_{2}(B\dirac)}, 
\end{equation}
and
\begin{equation}
\label{eq:tdb}
\psi (\uT) h= \psi(DB) h,  \quad h\in \mH^0_{D}.
\end{equation}
These relations are easily verified for the resolvent  and then one uses \eqref{eq:cauchyformula}.
It follows that  the operator norms of $\psi(T)$ and $ \psi(\uT)$ are bounded by $C_{\mu}\|\psi\|_{\infty}$, which guarantees that $T$ and $\uT$ have $H^\infty$-calculus on $\mH^0_{D}$ and the two formul\ae\ above extend to all $b\in H^\infty(S_{\mu})$.

We define the Sobolev spaces next. We use the curly style $\mH$ to distinguish them from pre-Hardy  and Hardy spaces where we use  the mathbb style  $\IH$ or roman style  $H$.

For $s\in \R$, define the inhomogeneous  Sobolev space associated with $S$, $\mH^s_{S}$, as the subspace of $\mH^0_{D}$ for which 
$$
\|h\|_{S,s}= \left\{ \int_{0}^\infty t^{-2s}\|\psi_{t}(S)h\|_{2}^2\, \frac{dt}{t}\right\}^{1/2} <\infty
$$
for a suitable $\psi\in \Psi(S_{\mu})$, for example $\psi(z)=z^k e^{-\modz}$ and $k$ an integer with $ k>\max (s,0)$. We define the homogeneous  Sobolev space associated with $S$, $\dot \mH^s_{S}$, as the completion of $\mH^s_{S}$ for $\|h\|_{S,s}$. 

Remark that from the  spectral theorem $\dot \mH^0_{S}=\mH^0_{S}=\mH^0_{D}$. Next, it can be checked that $\|h\|_{S,s}=c_{\psi, s}\||S|^sh\|_{2}$ where $|S| = (S^2)^{1/2}$. As $S=D|_{\mH^0_{D}}$, $ \mH^s_{S}$ is the closed subspace of  the usual inhomogeneous Sobolev  space $\mH^s$, equal to the  image of $\mH^s$ under the projection $\IP$, and similarly 
$\dot \mH^s_{S}$ is the image of  the usual homogeneous Sobolev  space $\dot \mH^s$ under (the extension of)  $\IP$ (which extends boundedly to $\dot \mH^s$ as it is a smooth singular integral convolution operator). It is not hard to check that $\dot \mH^s_{S} \cap \dot \mH^0_{S}=  \mH^s_{S}$.

Note that the $H^\infty$- and self-adjoint calculi of $S$ on $\dot \mH^0_{S}$ extend to $\dot \mH^s_{S}$ and that  $S$ extends to an isomorphism between $\dot \mH^s_{S}$ and $\dot \mH^{s-1}_{S}$. 
Classically, the intersection of $\dot \mH^s_{S}$  is dense in each of them.  Here is a precise statement whose proof is left to the reader.  Alternately, one can do  this using the usual Sobolev spaces $\dot \mH^s$ and project under $\IP$. 

\begin{lem}\label{lem:approxk} Let $\theta(z)=ce^{-\modz - \modz^{-1}} \in  \cap_{\sigma>0,\tau>0} \Psi_{\sigma}^\tau(S_{\mu})$ with $c^{-1}=\int_{0}^\infty \theta(t)\, \frac{dt}{t}$. For any $s\in \R$ and $h\in \dot \mH^s_{S}$, $h_{k}= \int_{1/k}^k \theta(tS)h \, \frac{dt}{t}\in \bigcap_{s'\in \R}  \mH^{s'}_{S}$ and converges to $h$ in $\dot \mH^s_{S}$ as $k\to \infty$. 
\end{lem}

Having defined $S$ and the associated Sobolev spaces, we use the more concrete notation 
$\dot \mH^s_{D}=\dot \mH^s_{S}$ and similarly for the inhomogeneous spaces. 

\textbf{We also use the notation $DB$ for $\uT$, $BD$ for $V$, $\IP BD$ for $T$.} 

We come back to the formal notation when needed for clarity in the proofs.

We define similarly the inhomogeneous Sobolev spaces $ \mH^s_{DB}$, $\mH^s_{BD}$ and $\mH^s_{\IP BD}$ replacing $S$ by  $\uT$, $V$ and $T$ respectively. 

\begin{prop}\label{prop:Sobolevinhom} Let $s\in \R$. 
\begin{enumerate}
\item The quadratic norms are equivalent under changes of suitable non-degenerate  $\psi$. 
  \item The bounded holomorphic functional calculus extends : for any $b\in H^\infty(S_{\mu})$, $b(X)$  is bounded on $\mH^s_{X}$ if $X=DB, BD$ or $\IP BD$. 
   \item $\IP: \mH^{s}_{BD} \to \mH^{s}_{\IP BD}$ is an isomorphism.
   \item  $\mH^s_{DB}$ and $\mH^{-s}_{B^*D}$ are in duality for the $L^2$ inner product. 
   \item  $\mH^s_{DB}$ and $\mH^{-s}_{\IP B^*D}$ are in duality for the $L^2$ inner product. 
  \end{enumerate}
\end{prop}

\begin{proof}
(1) is standard and we skip it. (2) is a straightforward consequence of the definitions of the spaces and of the norms. For (3), using the intertwining property  \eqref{eq:intert}, and the isomorphism $\IP:\mH^0_{BD}= \clos{\ran_{2}(BD)} \to \clos{\ran_{2}(D)}=\mH^0_{\IP BD}$, we obtain 
$$\|\psi(\IP BD)\IP h\|_{2}= \|\IP  \psi(BD)h\|_{2} \sim \|\psi(BD)h\|_{2}$$
for all $h\in \mH^0_{BD}$ and $\psi\in \Psi(S_{\mu})$. We  conclude easily for the isomorphism  using the defining norms of the Sobolev spaces. 
The proof of (4) is a simple consequence of the Calder\'on reproducing formula so that for suitable $\psi,\varphi$ we have 
$$
\pair f g = (\Qpsi \psi {DB}f, \Qpsi \varphi {B^*D} g)
$$
for all $f\in \mH^0_{DB}$ and $g\in \mH^0_{B^*D}$. We skip details. 
For (5), we use the intertwining property: for all $f\in \mH^0_{DB}$ and $h\in \mH^0_{\IP B^*D}$, writing $h=\IP g$ with $g\in \mH^0_{B^*D}$
$$
\pair f {h} =  \pair f {g}= (\Qpsi \psi {DB}f,  \Qpsi \varphi {B^*D} g)=(\Qpsi \psi {DB}f, \IP \Qpsi \varphi {B^*D} g)=  (\Qpsi \psi {DB}f,  \Qpsi \varphi {\IP B^*D} h)
$$
and the conclusion follows easily. 
\end{proof}

Now define their completions $\dot \mH^s_{DB}$, $\dot \mH^s_{BD}$ and $\dot\mH^s_{\IP BD}$  respectively.   So far, these completions are abstract spaces.

\begin{prop} \label{prop:sobolev} 
\begin{enumerate}
\item For $s\in \R$, for all bounded holomorphic functions $b\in H^\infty(S_{\mu})$,  $ b(\IP BD)$ extends to a bounded operator on $\dot \mH^s_{\IP BD}$.   In particular, this holds for $ \sgn(\IP BD)$ which is a bounded self-inverse operator on $\dot\mH^s_{\IP BD}$. Also, $\IP BD$ and $|\IP BD|= \sgn(\IP BD ) \IP BD$ extend to isomorphisms between $\dot \mH^s_{\IP BD}$ and $\dot \mH^{s-1}_{\IP BD}$.  The operator $|\IP BD|$ extends to a sectorial operator on $\dot \mH^s_{\IP BD}$ and fractional powers $|\IP BD|^\alpha$ are isomorphisms from $\dot \mH^s_{\IP BD}$ onto $\dot \mH^{s-\alpha}_{\IP BD}$.
\item   $\dot \mH^s_{\IP BD}$ topologically splits as the sum of the two spectral closed subspaces  $\dot\mH^{s,+}_{\IP BD}=\nul(\sgn (\IP BD)-I)=\ran(\chi^+(\IP BD))$ and $\dot\mH^{s,-}_{\IP BD}=\nul(\sgn(\IP BD) +I)=\ran(\chi^-(\IP BD))$. 
  \item The same two items hold with $\IP BD$ replaced by  $DB$ or $BD$. 
  \item For $0\le s\le 1$, $\dot \mH^s_{\IP BD}= \dot \mH^{s}_{D}$  and for $-1\le  s\le 0$, 
$\dot \mH^s_{DB}= \dot \mH^{s}_{D}$ with equivalence of norms. 
  \item  Furthermore, for $-1\le s<0$, we have for  
  $\| h\|_{D,s} \approx \left\{ \int_{0}^\infty t^{-2s}\|e^{-t|DB|}h\|_{2}^2\, \frac{dt}{t}\right\}^{1/2}$.
  \item For all $s\in \R$, $\IP$ extends to an isomorphism from  $\dot \mH^s_{BD}$   onto $\dot \mH^s_{\IP BD}$. 
  \item For all $s\in \R$, $\dot \mH^s_{DB}$ and $\dot \mH^{-s}_{B^*D}$ are dual spaces for a duality extending the $L^2$ inner product.
   \item For all $s\in \R$, $\dot \mH^s_{DB}$ and $\dot \mH^{-s}_{\IP B^*D}$ are dual spaces for a duality extending the $L^2$ inner product.
   \end{enumerate}
   \end{prop}

\begin{proof} For (1)-(5), this is the theory of \cite{AMcN1}, except   for the cases $s=-1$ and $s=1$ of (4), proved in  \cite[Proposition 4.4]{AMcM} using the holomorphic functional calculus on $L^2$ for $DB$ and $BD$.  
 
The items (6)-(8) are easy consequences of the previous proposition and density.  
\end{proof}

\begin{cor}\label{cor:isom}
Let  $-1\le s\le 0$. Then  $D:\dot \mH^{s+1}_{\IP BD}=  \dot \mH^{s+1}_{D}\to \dot \mH^{s}_{D}= \dot \mH^{s}_{DB}$ is an isomorphism. In particular, for $t>0$ and $h\in \dot \mH^{s+1}_{\IP BD}$, we have
$$
De^{-t|\IP BD|}h=e^{-t|DB|} Dh.
$$
Similarly $D$ extends to an isomorphism  $\dot \mH^{s+1}_{BD} \to \dot \mH^{s}_{DB}$. In particular, for $t>0$ and $h\in \dot \mH^s_{BD}$, we have
$$
De^{-t|BD|}h=e^{-t|DB|} Dh.
$$
\end{cor}

\begin{proof} Let us consider the first assertion. 
Take a suitable $\psi\in \Psi(S_{\mu})$ and $h\in \dom_{2}(S)$. Then $Dh=Sh$ and
$$ \psi(\uT)Sh=\psi(DB) Dh=D\psi(BD)h = S\psi(T) h.
$$
Then change $\psi(z)$ to $\psi(tz)$ and use  the isomorphism property of $S$, the 
property (4) in the proposition above and also the density of $\dom_{2}(S)=\mH^1_{D}$ in $\dot \mH^{s+1}_{D}$. 
For the second part, the extension is defined as $D\circ \IP$, where $\IP$ is the extension given in 
item (6) of the previous proposition and  $D$ is the isomorphism just described. 
\end{proof}

\begin{prop}\label{prop:sharpsobolev} Let $0<s\le 1$.   
\begin{enumerate}
  \item For any $h\in \dot \mH^s_{BD}$, $e^{-t|BD|}h -h$ can be defined in $L^2$ with 
$\|e^{-t|BD|}h -h\|_{2} \le C t^s$. 
  \item For any $h\in \dot \mH^s_{\IP BD}$, $e^{-t|\IP BD|}h -h$ can be defined in $L^2$  with 
$\|e^{-t|\IP BD|}h -h\|_{2} \le C t^s$.
  \item For any $h\in \dot \mH^s_{BD}$, with the above definition $\IP  (e^{-t|BD|}h -h)= e^{-t|\IP BD|}\IP h - \IP h$. 
\end{enumerate}
\end{prop}

\begin{proof} For (1), observe that $\phi(z)=\frac{e^{-t\modz}-1}{\modz^s}\in H^\infty(S_{\mu})$ with 
bound $Ct^s$ and  that $\||BD|^s h\|_{2}\sim  \|h\|_{BD,s}$ when $h\in \dot \mH^s_{BD}$. The  relation 
$
e^{-t|BD|}h -h = \phi(BD) |BD|^s h$ valid for $h\in \mH^s_{BD}$ thus  extends to $h\in\dot \mH^s_{BD}$. The proof for the second item is the same. The third item is the intertwining property of the $H^\infty$-calculi, extended to $\dot \mH^s_{BD}$ and $\dot \mH^s_{\IP BD}$. 
\end{proof}

\subsection{A priori estimates}\label{sec:apriori}

The following lemma tells us that we can use different norms, more suitable to extensions.

\begin{lem}\label{lem:D} We have
$$\|Dh\|_{\dot W^{-1,p}} \sim \|h\|_{p},\quad  \forall p\in (1,\infty) \  \forall h \in \clos{\ran_{2}(D)},$$
and
$$\|Dh\|_{\dot \Lambda^{\alpha-1}} \sim \|h\|_{\dot \Lambda^\alpha},\quad  \forall \alpha \in [0,1) \  \forall h \in \clos{\ran_{2}(D)}.$$
\end{lem}

\begin{proof} 
First, assume $h\in L^{p}$.  Then  $Dh \in \dot W^{-1,p}$ and  if $g\in \dot W^{1,p'} $, 
$$
|\pair {Dh} g | =  |\pair {h} {Dg} |\le \|h\|_{p}\|Dg\|_{p'} \lesssim \|h\|_{p}\|g\|_{\dot W^{1,p'}}.
$$
We conclude $\|Dh\|_{\dot W^{-1,p}} \lesssim \|h\|_{p}$. For the converse, recall that  $\mS_{0}$ is the space of Schwartz functions with  compactly supported Fourier transforms away from the origin.  By density,  we have $\|h\|_{p}= \sup\{ |\pair {h} g|; g\in  \mS_{0} , \|g\|_{p'}=1\}$ and for   $g\in  \mS_{0}$, we have $\IP g\in \mS_{0}$  as well, so 
$$ 
|\pair {h} g | =  |\pair {h} {\IP g} | =|\pair {Dh} {D^{-1}\IP g} |  \lesssim \|Dh\|_{\dot W^{-1,p}} \|D^{-1}\IP g\|_{\dot W^{1,p'}}.
$$
Here, we observe that   $D^{-1}\IP g\in \mS_{0}$ is a Schwartz distribution (using a Fourier transform argument) and as $\nabla D^{-1} \IP$ is bounded on $L^{p'}$, we obtain 
$ \|D^{-1}\IP g\|_{\dot W^{1,p'}} \lesssim \|g\|_{p'}=1$. 

Consider now the second statement. Clearly, $h\in \dot \Lambda^\alpha$ implies $Dh\in \dot \Lambda^{\alpha-1}$. For the converse, 
note that if $g\in \IP \mS_{0}$, then $D^{-1}g\in \dot H^{1,q}$. Indeed,  $D^{-1}g \in \mS'$, $\nabla D^{-1}g= \nabla D^{-1}\IP g \in H^q$. Thus,  
$$
|\pair h g| = |\pair {Dh} {D^{-1} g}| \le \|Dh\|_{\dot \Lambda^{\alpha-1}}\|D^{-1}g\|_{\dot H^{1,q}} \lesssim  \|Dh\|_{\dot \Lambda^{\alpha-1}}\|g\|_{H^{q}}.
$$
By density of $\IP(\mS_{0})$ in $H^q_{D}$, this implies that $h\in\dot \Lambda^\alpha $ with the desired estimate. 
\end{proof}

We continue with the extension of the functional calculus of $DB$ to negative Sobolev spaces of the type $\dot W^{-1,p}$ or negative H\"older spaces $\dot \Lambda^{\alpha-1}$ under the appropriate assumption. 

\begin{prop}\label{prop:fcnegsob} Let $q \in (\frac{n}{n+1},p_{+}(DB^*))$ be such that $\IH^q_{DB^*}=\IH^q_{D}$ with equivalence of norms. Let $\mT= T^{q'}_{2}, Y=L^{q'}, \dot Y^{-1}= \dot W^{-1,q'}$ if $q>1$ and $\mT=T^{\infty}_{2,\alpha}, Y=\dot \Lambda^\alpha, \dot Y^{-1}=\dot \Lambda^{\alpha-1} $ with $\alpha=n(\frac{1}{q}-1)$ if $q\le 1$. Let $b\in H^\infty(S_{\mu})$. Then
$$
\|b(DB)h\|_{\dot Y^{-1}} \lesssim \|b\|_{\infty}  \|h\|_{\dot Y^{-1}}, \quad
\forall h\in \bigcup_{-1\le s\le 0} \dot \mH^s_{D},
$$
and 
$$
\|D b(BD)\tilde h\|_{\dot Y^{-1}} \lesssim \|b\|_{\infty}  \|D\tilde h\|_{\dot Y^{-1}}, \quad
\forall \tilde h\in \bigcup_{-1\le s\le 0} \dot \mH^{s+1}_{BD}.
$$
\end{prop}

\begin{proof}
Let us begin with $h\in \dot \mH^{-1}_{D}$. By  Corollary \ref{cor:isom}, there exists a unique $ g \in \dot \mH^0_{BD}$ with $h=Dg=D\IP g$. By similarity, $b(DB)h = Db(BD)g=D\IP b(BD)g$. Thus, using Lemma \ref{lem:D} twice, since $\IP g, \IP b(BD)g\in \dot \mH^0_{D}= \clos{\ran_{2}(D)}$, 
$$
\|b(DB)h\|_{\dot Y^{-1}} \sim \|\IP b( BD)g\|_{Y}\lesssim \|b\|_{\infty} \|\IP g\|_{Y}\sim \|b\|_{\infty}\|h\|_{\dot Y^{-1}}.  
$$
Next, we assume $h\in  \dot \mH^s_{D}$ with $-1<s\le 0$. Consider the approximations $h_{k}$ of Lemma \ref{lem:approxk}. They belong in particular to $\dot \mH^{-1}_{D}$. Thus
$\|b(DB)h_{k}\|_{\dot Y^{-1}} \lesssim \|h_{k}\|_{\dot Y^{-1}}$ uniformly in $k$. Now,  using  Fourier transform and the Mikhlin theorem, $h\mapsto h_{k}$ is bounded on $\dot Y^{-1}$, uniformly in $k$. Hence $(b(DB)h_{k})$ is a bounded sequence in $\dot Y^{-1}$, thus has a weak-$*$ converging  subsequence in $\dot Y^{-1}$, and in particular in the Schwartz distributions. But, 
by Proposition \ref{prop:sobolev}, $b(DB)$ is bounded on $\dot \mH^s_{D}$, hence  $b(DB)h_{k} \to b(DB)h$ in $\dot \mH^s_{D}$ so also in the Schwartz distributions. Thus, the limit of the above subsequence is $b(DB)h$ which, therefore, belongs to $\dot Y^{-1}$ with the desired estimate. 

Let us turn to the second point. If $\tilde h\in \dot \mH^{s+1}_{BD}$, then $h=D\tilde h\in \dot \mH^{s}_{D}$ and $Db(BD)\tilde h= b(DB)h$ by the isomorphism property in Corollary \ref{cor:isom}.
Thus,
$$
\|D b(BD)\tilde h\|_{\dot Y^{-1}} =
\| b(DB) h\|_{\dot Y^{-1}} \lesssim   \|b\|_{\infty}\|h\|_{\dot Y^{-1}} =  \|b\|_{\infty}\|D\tilde h\|_{\dot Y^{-1}}. 
$$
\end{proof}

The following result is an extension of earlier results with \textit{a priori} Sobolev initial elements instead of just $L^2$ so far.  This result will be especially useful for $s=- \frac{1}{2}$ later. 

\begin{thm}\label{thm:summary} \begin{enumerate}
  \item Let $I$ be the subinterval in $(\frac{n}{n+1},p_{+}(DB))$   on which we have $\IH^q_{DB}= \IH^q_{D}$ with 
equivalent norms. Then the following holds. For $DB$ we have,  for all $h\in \bigcup_{-1\le s\le 0} \dot \mH^s_{D}$,
$$
\|\SF(tDB{e^{-t|DB|}}h)\|_{q} \sim \|\SF(t\partial_{t}{e^{-t|DB|}}h)\|_{q}\sim  \|h\|_{H^q}
$$
and
$$  \|\tN(e^{-t|DB|}h)\|_{q} \sim  \|h\|_{H^q}
$$
when $q>1$, or $q\le 1$ and $B$ pointwise accretive, or $q\le 1$ and $h\in \bigcup_{-1\le s\le 0} \dot \mH^{s,\pm}_{DB}$. 
  \item If $I^*$ designates the same interval but for $DB^*$ and $q\in I^*$, let  $\mT= T^{q'}_{2}, Y=L^{q'}, \dot Y^{-1}= \dot W^{-1,q'}$ if $q>1$ and $\mT=T^{\infty}_{2,\alpha}, Y=\dot \Lambda^\alpha, \dot Y^{-1}=\dot \Lambda^{\alpha-1} $ with $\alpha=n(\frac{1}{q}-1)$ if $q\le 1$. Then, we obtain the following equivalences: 
  \subitem {\rm{(2a)}} Tent space estimate for $BD$ in disguise:
$$
\|t{e^{-t|DB|}}h\|_{\mT}\sim  \| h\|_{\dot Y^{-1}}, \quad \forall h\in \bigcup_{-1\le s\le 0} \dot \mH^s_{D}. 
$$
\subitem{\rm{(2b)}} Tent space estimate for $BD$:   $  \forall \tilde h\in \bigcup_{-1 \le s\le 0} \dot \mH^{s+1}_{BD},$
$$
\|tD{e^{-t|BD|}}\tilde h\|_{\mT} \sim \|tBD{e^{-t|BD|}}\tilde h\|_{\mT}\sim \|t\partial_{t}{e^{-t|BD|}}\tilde h\|_{\mT}\sim  \| D\tilde h\|_{\dot Y^{-1}}.$$  \subitem{\rm{(2c)}} Sharp function for $BD$: Finally,  if $1<q\le 2$ we have
$$
 \|\tNs (e^{-t|BD|} \tilde h)\|_{p}\sim \| D\tilde h\|_{\dot Y^{-1}}
 ,  \quad \forall \tilde h\in\bigcup_{-1 \le s\le 0} \dot \mH^{s+1}_{BD},$$
and in the case $q\le 1$, we have
$$
\|\tNsa (e^{-t|BD|} \tilde h)\|_{\infty} \sim \| D\tilde h\|_{\dot Y^{-1}}
, \quad \forall \tilde h\in \bigcup_{-1 \le s\le 0} \dot \mH^{s+1}_{BD}.$$
\end{enumerate}
\end{thm}

\begin{proof}  So far, and thanks to Lemma \ref{lem:D}, all statements have been proved when $h\in \dot \mH^0_{D}$ for those involving $DB$ and  when $\tilde h\in \dot \mH^0_{BD}$ for  those involving  $BD$.  Our goal is thus to extend this to more general $h$ or $\tilde h$. The argument consists in  tedious verifications with adequate approximation procedures.

\

\paragraph{Proof of  (1)}

We begin with the quadratic estimates. We fix $q$ in the prescribed interval. Let $\psi\in \Psi(S_{\mu})$ for which we have
$\|\Qpsi \psi {DB} h\|_{T^q_{2}} \lesssim \|h\|_{H^q}$ for all $h\in \dot \mH^0_{DB}=\dot \mH^0_{D}$.  We want to extend it to $h\in \dot \mH^s_{D}$ for some $s\in [-1, 0)$.
Let $h$ be such. Assume also $\|h\|_{H^q}<\infty$ otherwise there is nothing to prove. 
Consider the  functions $h_{k}\in \dot \mH^0_{D}$ as in Lemma \ref{lem:approxk}: they converge in $\dot \mH^s_{D}$ to $h$. Classical Hardy space theory  also shows convergence  in $H^q$.  Now, the estimates apply to $h_{k}$.  Thus $(\Qpsi \psi {DB} h_{k})$ is a Cauchy sequence in $T^q_{2}$, hence converges to some $F$ in $T^q_{2}$. This enforces the convergence in $L^2_{loc}(\reu)$. As $\dot \mH^s_{D}=\dot \mH^s_{DB}$, it is easy to see that the sequence $(\Qpsi \psi {DB} h_{k})$ converges also in $L^2_{t,loc}(L^2_{x})$ to  $\Qpsi \psi {DB} h$. Thus $ \Qpsi \psi {DB} h=F\in T^q_{2}$ and  this concludes the extension.

Conversely, assume that $ \|h\|_{H^q} \lesssim \|\Qpsi \psi {DB} h\|_{T^q_{2}} $ for all $h\in \dot \mH^0_{DB}=\dot \mH^0_{D}$ and some $\psi\in \Psi(S_{\mu})$. Again, we have to extend it to $h\in \dot \mH^s_{D}$ for some $s\in [-1, 0)$. We assume $\|\Qpsi \psi {DB} h\|_{T^q_{2}}<\infty$, otherwise there is nothing to prove. Take  $\varphi\in \Psi(S_{\mu})$ for which we have the Calder\'on reproducing formula 
\eqref{eq:Calderon} and also that $\Tpsi \varphi {DB}$ maps $T^q_{2}\cap T^2_{2}$ into $\IH^q_{DB}$. Let $\chi_{k}$ be the indicator function of $[1/k,k]\times B(0,k)$. Then   $ h_{k}:= \Tpsi \varphi {DB}(\chi_{k}\Qpsi \psi {DB} h) \in \IH^q_{DB}=\IH^q_{D}$.  By taking the limit as $k\to \infty$, $h_{k}$ converges to some $\tilde h\in H^q_{DB}=H^q_{D}$. 
Next, by testing against a Schwartz function $g$,  
$$\pair {h_{k}}g= \big( \chi_{k}\Qpsi \psi {DB} h, \Qpsi {\varphi^*}{B^*D}g\big)= \int_{0}^\infty \pair{\chi_{k}\psi(tDB)h}{ \IP\varphi^*(tB^*D)g} \frac{ dt}{t}.$$
If $\varphi(z)=z\tilde\varphi(z)$ for some $\tilde \varphi\in \Psi(S_{\mu})$, then $\varphi^*(tB^*D)g= t \tilde \varphi^*(tB^*D)(B^*Dg)$.
 It easily  follows using $-1\le s\le 0$ and treating differently the integral for $t<1$ or $t>1$,  that $$\int_{0}^\infty t^{2s}\|\IP{\varphi^*(tB^*D)g}\|_{2}^2\,   \frac{ dt}{t} <\infty,$$
 (for $s=-1$ use the square functions estimates) while
$$\int_{0}^\infty t^{-2s}\|\psi(tDB)h\|_{2}^2\,  \frac{ dt}{t}\lesssim \|h\|^2_{DB, s}\sim \|h\|^2_{D,s}.$$
Thus dominated convergence theorem applies to yield that 
$$\pair {h_{k}}g\to  \int_{0}^\infty \pair{\psi(tDB)h} {\IP\varphi^*(tB^*D)g}\, \frac{ dt}{t} = \pair {h}g.$$
This shows that $h=\tilde h $ in the sense of Schwartz distributions, so that $h\in H^q_{D}$   with the desired estimate. 

Let us look at the extension for non-tangential maximal estimates. The extension of $ \|\tN(e^{-t|DB|}h)\|_{q} \lesssim  \|h\|_{H^q}$ to all $h\in \dot \mH^s_{D}$ for some $s\in [-1, 0)$ can be handled as for square functions.  Conversely, an inspection of the proofs  of Propositions \ref{prop:lowerp<1} and \ref{prop:lowerp<1a} shows the converse in the different cases of the statement.

\

\paragraph{Proof of (2a) and (2b)}   We fix $-1\le s \le 0$. 
The extension for the upper bound $
\|t{e^{-t|DB|}}h\|_{\mT}\lesssim  \| h\|_{\dot Y^{-1}}$, when $ h\in \dot \mH^s_{D} $, can be done as for (1) when $s<0$. Consider the  functions $h_{k}\in \dot \mH^0_{D}$ as in Lemma \ref{lem:approxk}: they converge in $\dot \mH^s_{D}$ to $h$. It is easy to check that  $(h_{k})$  is uniformly bounded in $\dot Y^{-1}$ with $\|h_{k}\|_{\dot Y^{-1}}\lesssim \|h\|_{\dot Y^{-1}}$. Thus it remains to go to the limit for $\|t{e^{-t|DB|}}h_{k}\|_{\mT}$.  Convergence in $\dot \mH^s_{D}$ implies that   $(t{e^{-t|DB|}}h_{k})$ converges to $t{e^{-t|DB|}}h$ in $L^2_{\loc}(\reu)$ and, at the same time, as it is a bounded sequence in $\mT$, which is a dual space, it has a weakly-$*$ convergent subsequence. Testing against bounded function with compact support in $\reu$, we conclude that the limit must also be
$t{e^{-t|DB|}}h$ and the desired estimate follows.

Now, for (2b), let $\tilde h\in \dot \mH^{s+1}_{BD}$. Then we know from Corollary \ref{cor:isom} that $h= D\tilde h \in \dot \mH^{s}_{DB}=\dot \mH^{s}_{D}$ and $De^{-t|BD|}\tilde h=e^{-t|DB|} D\tilde h= e^{-t|DB|} h$. Using what we just did 
$$\|tD{e^{-t|BD|}}\tilde h\|_{\mT} \lesssim  \|D\tilde h\|_{\dot Y^{-1}}
.
$$
Using the boundedness of $B$ we also have the upper bound 
$$
\|tBD{e^{-t|BD|}}\tilde h\|_{\mT} \lesssim  \|tD{e^{-t|BD|}}\tilde h\|_{\mT} \lesssim  \|D\tilde h\|_{\dot Y^{-1}}
.
$$
Finally, $t\partial_{t}{e^{-t|BD|}}\tilde h= tBD{e^{-t|BD|}} \sgn(BD)\tilde h$, so that  
$$
\|t\partial_{t}{e^{-t|BD|}}\tilde h\|_{\mT} \lesssim \| D  \sgn(BD)\tilde  h\|_{\dot Y^{-1}} \lesssim \|  D \tilde  h\|_{\dot Y^{-1}},$$
where the last inequality follows from Proposition \ref{prop:fcnegsob}.

For the converse inequalities in (2a) and (2b), a moment's reflection tells us that   it is enough to show, when  $ \tilde h\in \dot \mH^{s+1}_{BD}, $ that  
 $\|D\tilde h\|_{\dot Y^{-1}} \lesssim \|tBD{e^{-t|BD|}}\tilde h\|_{\mT}$. 
 Set $\psi(z)=z e^{-\modz}$ as the other inequalities follow from this one. Consider $\varphi$ allowable for $\IH^q_{DB^*}$ such that the Calder\'on formula \eqref{eq:Calderon} holds. Let $g\in \IH^q_{D}\cap \dot \mH^{-s-1}_{D}=\IH^q_{DB^*}\cap \dot \mH^{-s-1}_{DB^*}$. Hence, for the inner product in tent spaces
 $$
 |(\Qpsi \psi {BD} \tilde h, \Qpsi {\varphi^*} {DB^*}g)| \lesssim \|\Qpsi \psi {BD} \tilde h\|_{\mT}\| g\|_{\IH^q_{DB^*}}.
 $$
 Using the approximations with the functions $\chi_{k}$ above, let  $\tilde h_{k}= \Tpsi \varphi {BD} (\chi_{k}\Qpsi \psi {BD} \tilde h) \in \IH^\mT_{BD}$. Then, using Lemma \ref{lem:D} and $\tilde h_{k}\in \dot \mH^0_{BD}$,  
 $$\|D\tilde h_{k}\|_{\dot Y^{-1}}\sim \|\tilde h_{k}\|_{Y} \lesssim \|\chi_{k}tBD{e^{-t|BD|}}\tilde h\|_{\mT} \lesssim \|tBD{e^{-t|BD|}}\tilde h\|_{\mT}.$$
 It remains to show that $D\tilde h_{k}$ converges to $D\tilde h$ in the sense of distributions as this will imply $\|D\tilde h\|_{\dot Y^{-1}} \le \liminf \|D\tilde h_{k}\|_{\dot Y^{-1}}$.
 Let $g$ be a Schwartz function. Then 
 $$\pair {D\tilde h_{k}} g= \pair {\tilde h_{k}} {Dg}=  \big( \chi_{k}\Qpsi \psi {BD} \tilde h, \Qpsi {\varphi^*}{DB^*}(Dg)\big).
 $$
Then, as $-1\le s\le 0$ and $Dg\in \dot\mH^{-s-1}_{D}=\dot\mH^{-s-1}_{DB^*}$, 
 $$\int_{0}^\infty t^{2(s+1)}\|{\varphi^*(tDB^*)(Dg)}\|_{2}^2\,   \frac{ dt}{t} <\infty,$$ while
$$\int_{0}^\infty t^{-2(s+1)}\|\psi(tBD)\tilde h\|_{2}^2\,  \frac{ dt}{t}\lesssim \|\tilde h\|^2_{BD, s+1}\sim \|h\|^2_{D,s}.$$
Thus dominated convergence theorem applies to yield that 
$$
\pair {\tilde h_{k}} {Dg} \to  \big( \Qpsi \psi {BD} \tilde h, \Qpsi {\varphi^*}{DB^*}(Dg)\big).$$
If $\varphi\psi$ has enough decay at 0 and $\infty$ then 
$$(\Qpsi \psi {BD} \tilde h, \Qpsi {\varphi^*} {DB^*}(Dg)) = \pair {\Tpsi \varphi {BD} \Qpsi \psi {BD}\tilde h} {Dg}= \pair {\tilde h} {Dg}=\pair {D\tilde h} {g}.$$

\

\paragraph{Proof of (2c)} As in  (2b),  $\|D\tilde h\|_{\dot Y^{-1}} \sim \|\psi(tBD)\tilde h\|_{\mT}$ for any allowable $\psi$ for $\IH^\mT_{BD}$ and  $ \tilde h\in \dot \mH^{s+1}_{BD}$. 
As observed in Proposition \ref{prop:sharpsobolev},   $e^{-t|BD|}\tilde h -\tilde h\in L^2$ when $\tilde h\in \dot \mH^{s+1}_{BD}$,  so that the proof of Lemma \ref{lem:sharp1} goes through without change. This proves the lower bounds for $\tNs (e^{-t|BD|} \tilde h)$ and $\tNsa (e^{-t|BD|} \tilde h)$. 

As for the upper bounds, let $\tilde h_{\varepsilon}=e^{-\varepsilon|BD|}\tilde h - e^{-(1/\varepsilon)|BD|}\tilde h$, $\varepsilon>0$. It follows from Proposition \ref{prop:sharpsobolev}  that $\tilde h_{\varepsilon} \in \clos{\ran_{2}(BD)}$, thus we obtain from Theorem \ref{thm:ntbd} the uniform upper bounds,  
$$
\|\tNs (e^{-t|BD|} \tilde h_{\varepsilon})\|_{q'} \lesssim \|\IP \tilde h_{\varepsilon}\|_{Y}
$$
in the case $Y=L^{q'}$ and 
$$
\|\tNsa (e^{-t|BD|} \tilde h_{\varepsilon})\|_{\infty} \lesssim \|\IP \tilde h_{\varepsilon}\|_{Y}
$$
in the case $Y=\dot \Lambda^\alpha$. 
 Remark that  $D\tilde h_{\varepsilon}= e^{-\varepsilon|DB|}D\tilde h - e^{-(1/\varepsilon)|DB|}D\tilde h$, so that    by Lemma \ref{lem:D}  and Proposition \ref{prop:fcnegsob},
$$\|\IP \tilde h_{\varepsilon}\|_{Y}\sim \|D\tilde h_{\varepsilon}\|_{\dot Y^{-1}}  \lesssim \|D\tilde h\|_{\dot Y^{-1}}. 
$$ 
As
$$
e^{-t|BD|}\tilde h_{\varepsilon} - \tilde h_{\varepsilon}=e^{-\varepsilon|BD|} (e^{-t|BD|}\tilde h -\tilde h) -  e^{-(1/\varepsilon)|BD|}(e^{-t|BD|}\tilde h -\tilde h),$$
$e^{-t|BD|}\tilde h_{\varepsilon} - \tilde h_{\varepsilon}$  converges in $L^2_{loc}(\reu)$ to $e^{-t|BD|}\tilde h -\tilde h $. A linearisation of the non-tangential sharp function, together with Fatou's lemma in the case where $Y=L^p, p<\infty$,  yields the conclusion. We skip easy details. 
  \end{proof}

\section{Applications to elliptic PDE's}

In this section, we are given $L=-\divv A \nabla$ as in the introduction ($t$-independent, bounded and accretive on $\mH^0=\mH^0_{D}$, coefficients). We first discuss representations of solutions in the class $\mE$. Then, we prove here Theorem \ref{thm:main1} and Theorem \ref{thm:main2} with some further estimates.

\subsection{A priori results for conormal gradients of solutions in $\mE$}

 We recall that 
$ \mE=\cup_{-1\le s\le 0}\ \mE_{s}$ where 
$$\mE_{s}= \begin{cases}
 \{u;\|\tN(\nabla u )\|_{2}<\infty\},     & \text{if\ } s=0, \\
 \{u; \|\SF(t^{-s}\nabla u)\|_{2}<\infty\} ,    & \text{otherwise}.
\end{cases}
$$

Recall  from \cite{AA1,R2}  that   conormal gradients 
$$
F(t,x)= \nabla_A u(t,x)=  \begin{bmatrix} \pd_{\nu_A}u(t,x)\\ \nabla_x u(t,x) \end{bmatrix} \in L^2_{loc}(\reu)
$$
(we omit the target space of $F$ in the notation) of  weak solutions $u\in \mE_{s}$ of $Lu=0$ on $\reu$ satisfy the equation (in distributional sense at first, and eventually in strong semigroup  sense)  
\begin{equation}
\label{eq:rep}
\pd_{t}F+DB F=0,
\end{equation}
and  have a trace on $\R^n$  and  semigroup representation
\begin{align}
    \nabla_A u|_{t=0}& \in \dot \mH^{s,+}_{DB} \subset \dot \mH^{s}_{D}, \nonumber  \\
   \nabla_{A}u(t,\,.\,)    = e^{-t|DB|}\nabla_A u|_{t=0}&=e^{-t|DB|}\chi^+(DB)\nabla_A u|_{t=0}= e^{-tDB}\nabla_A u|_{t=0} \label{eq:sg}, 
\end{align}
where
$$
  \dirac:= 
    \begin{bmatrix} 0 & \divv_{x} \\ 
     -\nabla_{x} & 0 \end{bmatrix},\qquad  \dom(\dirac) = \begin{bmatrix}  \dom(\nabla)\\ \dom (\divv)\end{bmatrix} \subset L^2(\R^n, \C^N), \ N=m(1+n),
$$
and
 \begin{equation}
\label{eq:hat}
B= \hat A:=  \begin{bmatrix} 1 & 0  \\ 
    c & d \end{bmatrix}\begin{bmatrix} a & b  \\ 
    0 & 1 \end{bmatrix}^{-1}=  \begin{bmatrix} a^{-1} & -a^{-1} b  \\ 
    ca^{-1} & d-ca^{-1}b \end{bmatrix}
\end{equation}
whenever we write 
$$
A= \begin{bmatrix} a & b  \\ 
    c & d \end{bmatrix}
    $$
    and $L$  in the form
    $$
    L=-\begin{bmatrix}   \pd_{t}& \nabla_x  \end{bmatrix}
  \begin{bmatrix} a & b  \\ 
    c & d \end{bmatrix} \begin{bmatrix}   \pd_{t}\\ \nabla_x  \end{bmatrix}.
    $$
Here, $D$ and $B$  satisfy the necessary requirements and the semigroup $e^{-t|DB|}$ is appropriately interpreted as in Section  \ref{sec:Sobolev}. 

Conversely,  for any $h\in \dot \mH^{s,+}_{DB}$, the $L^2_{\loc}(\reu)$ function 
$$F(t,x)= e^{-t|DB|}h(x)= e^{-tDB}\chi^+(DB)h(x)$$
 is the conormal gradient of a weak solution $u\in \mE_{s}$ of $Lu=0$ on $\reu$ and $h=\nabla_A u|_{t=0}$. Note that $u$ is unique modulo constants. Note also that  $u$ is a continuous function of $t\ge0$ valued in $L^2_{loc}(\R^n)$. See \cite{AA1} for $s=-1$ and \cite[Remark 8.9]{AM} for all $s\in [-1,0]$.

 It is convenient to use  the notation $v= \begin{bmatrix} v_{\no}\\ v_{\ta}\end{bmatrix}$ for vectors  in $\C^{m(1+n)}$, where $v_{\no}\in \C^m$ is called the scalar part and $v_{\ta}\in \C^{mn}=(\C^m)^n$ the tangential part of $v$.  With this notation, for any $s\in \R$,  
 \begin{equation}
\label{eq:Hsd}
 \dot \mH^{s}_{D}= \begin{bmatrix} \dot \mH^{s}_{\no}\\ \dot \mH^{s}_{\ta}\end{bmatrix}.
\end{equation}
 Given the definition of $D$, we have
$$
\IP= \begin{bmatrix} I & 0  \\ 
    0 & RR^* \end{bmatrix},
    $$
    where $R$ is the array of Riesz transforms on $\R^n$ acting componentwise on $\C^m$-valued functions and  $R^*$ is its adjoint. It follows that
    $\dot \mH^{s}_{\no}= \dot \mH^{s}(\R^n;\C^m)$ and $\dot \mH^{s}_{\ta}= R\dot \mH^{s}_{\no}$, which is also denoted by $\dot \mH^{s}_{\nabla}(\R^n;\C^{mn})$ in \cite{AM}.

Let $u\in \mE_{s}$ be a solution to $Lu=0$ in $\reu$. Using that $D: \dot \mH^{s+1,+}_{\IP BD} \to \dot \mH^{s,+}_{DB}$ is an isomorphism, there exists a unique $U(0,\,.\,) \in  \dot \mH^{s+1,+}_{\IP BD}\subset \dot \mH^{s+1}_{D}$ such that
$$
DU(0,\,.\,):= - \nabla_{A}u|_{t=0} \in  \dot \mH^{s,+}_{DB}.$$ 
Then, define 
$$
U(t,\,.\,)=  e^{-t|\IP BD|} U(0,\,.\,)= e^{-t\IP BD} \chi^+(\IP BD)U(0,\,.\,) , \quad t\ge 0,
$$
 accordingly to Proposition \ref{prop:sharpsobolev} with $U(t,\, .\,)-U(0,\,.\,) \in L^2$. 
 Using that $\IP$ extends to an isomorphism $\dot \mH^{s+1,+}_{BD} \to \dot \mH^{s+1,+}_{\IP BD}$,  there exists a unique $v(0,\,.\,)\in \dot \mH^{s+1,+}_{BD}$ such that 
 \begin{equation}
\label{eq:U=Pv}
U(0,\,.\,)= \IP v(0,\,.\,)
\end{equation}
and this $v$ satisfies  \begin{equation*}
Dv(0,\,.\,)=DU(0,\,.\,)= -\nabla_{A}u|_{t=0},
\end{equation*} 
where   $Dv(0,\,.\,)$ is taken in the appropriate  sense.  One defines,  in $\dot \mH^{s+1,+}_{BD}$, 
$$
v(t,\,.\,)= e^{-t|BD|} v(0,\,.\,)= e^{-tBD}\chi^+(BD) v(0,\,.\,),  \quad t\ge 0,
$$
accordingly to Proposition \ref{prop:sharpsobolev}, so that $v(t,\,. \,)-v(0,\,.\,) \in L^2$, and one has
$$
U(t,\,.\,)= \IP v(t,\,.\,)
$$
in $\dot \mH^{s+1}_{D}$ and
\begin{equation}
\label{eq:du=dv}
Dv(t,\,.\,)=DU(t,\,.\,)= -\nabla_{A}u(t,\, .\, )
\end{equation} 
  in $L^2_{loc}(\reu)\cap C([0,\infty); \dot \mH^{s}_{D})$.

In fact, $U$ and $v$ share the same first component as $\IP$ is the identity on scalar parts and
their tangential parts satisfy for all $t\ge 0$, 
$$
(U(t,\,.\,))_{\ta}=(\IP v)_{\ta}(t,\,.\,)= ((RR^*v_{\ta})(t,\,.\,)),  \ \mathrm{in}\  \dot \mH^{s+1}_{\ta}
$$
or, equivalently, 
$$
(R^*U_{\ta})(t,\,.\,)=(R^*v_{\ta})(t,\,.\,),   \ \mathrm{in}\  \dot \mH^{s+1}_{\no}.
$$
Here,  $RR^*v_{\ta}$  is meant as the appropriate extension of the tangential part of $\IP$ acting on $v$, so $R^*v_{\ta}$ is to be interpreted in this way.   It tells us  that   any estimate on $U_{\ta}$ is thus an estimate on $R^*v_{\ta}$.

We finish this discussion with the pointwise relation between $u$, $U_{\no}$ and $v_{\no}$. Recall that $u\in \mE_{s}$ and is  continuous as a function of $t$ valued in $L^2_{loc}(\R^n;\C^m)$.  
Also $U_{\no}=v_{\no}\in \dot \mH^{1/2}_{\no}$ at $t=0$. They can be regarded as $L^2_{loc}$ functions and they agree up to a constant. We decide to set the constant to be 0.   Moreover,  $U(t,\, .\,)-U(0,\,.\,)=\IP (v(t,\,.\,)-v(0,\,.\,))$ belongs to $ L^2$ and is continuous as a function of $t$.   
As $\IP$ is the identity on scalar parts, we have the equality $U_{\no}=v_{\no}$ in $C([0,\infty); L^2_{loc}(\R^n))$.  Following the proof in  \cite{AA1} where the case 
 $s=-1$ is treated 
  (we changed signs compared to \cite{AA1}), there  exists   a constant $c\in \C^m$ such that for  all $t\ge 0$
  $$
  u(t,\,.\,)=(U(t,\,.\,))_{\perp}+c=(v(t,\,.\,))_{\perp}+c   \ \mathrm{in}\  L^2_{loc}(\R^n),
  $$ 
  (it is no longer modulo constants)
  so that we have the following representations for $u$ in $C([0,\infty); L^2_{loc}(\R^n))$ with $h=v(0,\,.\,)\in \dot\mH^{s+1,+}_{BD}$, 
  $$
   u(t,\,.\,)-c=(e^{-t|\IP BD|}\IP h)_{\perp} = ( e^{-t|BD|}h)_{\perp}= (\IP e^{-t|BD|}h)_{\perp}.
   $$
Thus $U$ and $v$ are potential vectors  for the solution $u$. Both are useful.

 If, furthermore, $s=-1$, \textit{i.e.} $h=v(0,\,.\,) \in \clos{\ran_{2}(BD)}$, then $e^{-t|\IP BD|}\IP h = \IP e^{-t| BD|}\IP h$, so we also have $u(t,\,.\,)-c= (\IP e^{-t| BD|}\IP h)_{\perp}= (e^{-t| BD|}\IP h)_{\perp}$.  

Let us mention a consequence of this discussion.

\begin{lem}\label{lem:y-1} Assume $u\in \mE_{s}, -1\le s\le 0$, is a weak solution of $Lu=0$.  Assume $q$ is such that $\IH^q_{DB^*}=\IH^q_{D}$ with equivalence of norms. Let $p=q'$ if $q>1$. Then 
$ \|\nabla_{A} u|_{t=0}\|_{\dot W^{-1,p}}<\infty$ if, and only if, there exists $h\in \dot\mH^{s+1,+}_{BD}\cap H^{p,+}_{BD}$ with $Dh=\nabla_{A} u|_{t=0}$,
 and we have 
$$
 \|\nabla_{A} u|_{t=0}\|_{\dot W^{-1,p}} \sim  \|\IP h \|_{p}.
  $$
Let $\alpha=n(\frac{1}{q}-1)$ if $q\le 1$. Then $\|\nabla_{A} u|_{t=0}\|_{\dot \Lambda^{\alpha-1}}<\infty$ if, and only if, there exists $h\in \dot\mH^{s+1,+}_{BD}\cap \dot \Lambda^{\alpha,+}_{BD}$ with $Dh=\nabla_{A} u|_{t=0}$,
 and we have 
$$
 \|\nabla_{A} u|_{t=0}\|_{\dot \Lambda^{\alpha-1}} \sim \|\IP h\|_{\dot \Lambda^\alpha}.$$
  \end{lem}

\begin{proof} Let us consider the case $q>1$. 
Remark that $\dot\mH^{s+1}_{BD}$ is the dual of 
$\dot\mH^{-s-1}_{DB^*}=\dot\mH^{-s-1}_{D}$ and $H^p_{BD}$ is the dual of $H^q_{DB^*}=H^q_{D}$ with identical dualities when restricted to dense subspaces. The intersection $\dot\mH^{-s-1}_{D}\cap H^q_{D}$ is  well-defined within the Schwartz distributions, dense  in each factor,  and the intersection of duals $\dot\mH^{s+1}_{BD}\cap H^p_{BD}$ makes sense (as a subspace of the sum). 
If $h\in \dot\mH^{s+1,+}_{BD}\cap H^{p,+}_{BD}$ with $\nabla_{A} u|_{t=0}=Dh=D\IP h$, then 
$
 \|\nabla_{A} u|_{t=0}\|_{\dot W^{-1,p}}= \|D\IP h\|_{\dot W^{-1,p}} \lesssim \|\IP h\|_{p} $
 by an argument similar to that of Lemma \ref{lem:D}.
 Conversely, let $g\in \dot\mH^{-s-1}_{D}\cap H^q_{D}$. Then $D^{-1}g\in \dot\mH^{-s}_{D}\cap \dot W^{1,q}$. Indeed, if $g\in \dot\mH^{-s-1}_{D}\cap H^q_{D}$, then $\nabla D^{-1}g=\nabla D^{-1}\IP g \in \dot\mH^{-s-1}_{D}\cap H^q_{D}$. Thus, the map $g\mapsto \pair { \nabla_{A} u|_{t=0}} {D^{-1}g}$ is defined on $\dot\mH^{-s-1}_{DB^*}\cap H^q_{DB^*}$ and defines  $h\in \dot\mH^{s+1}_{BD}\cap H^p_{BD}$ with  $Dh=\nabla_{A} u|_{t=0}$ and one has $\IP h \in \dot\mH^{s+1}_{D}\cap H^p_{D}$ so that $\| \IP h  \|_{p}\lesssim  \|\nabla_{A} u|_{t=0}\|_{\dot W^{-1,p}}$.  Applying the projector $\chi^+(DB)$ leaves  $\nabla_{A} u|_{t=0}$ unchanged, thus it follows that $h=\chi^+(BD)h$ (in both spaces).

In the case $q\le 1$, we argue as above and replace $\dot W^{1,q}$ by $\dot H^{1,q}$. 
\end{proof}

\begin{rem}
This proof reveals that one can make the Sobolev and Hardy space theories consistent in the appropriate ranges of exponents. 
\end{rem}

\subsection{A priori comparisons of various norms}\label{sec:comp}

 We may now  translate Theorem \ref{thm:summary} in the context of solutions of $Lu=0$ in $\reu$. 
We remark that if $L$ is associated to $B$, then  the operator $L^*$, with coefficients $A^*$, is associated to $\wt B=\widehat {A^*}= NB^*N$, with $N= \begin{bmatrix} I & 0  \\ 
    0 & -I \end{bmatrix}$. As $DN+ND=0$ and  $N$ preserves $\clos{\ran_{2}(D)}$, we see  $D\wt B=-N(DB^*)N= N^{-1}(-DB^*)N$,  as $N=N^{-1}$ For the functional calculi of $D\wt B$ and $DB^*$, we obtain $b(D\wt B) = Nb(-DB^*)N$ for all $b\in H^\infty(S_{\mu})$. Therefore, we see that  $h\in \IH^{q,\pm}_{D\wt B}$ if and only if $Nh\in \IH^{q,\mp}_{DB^*}$. Also, $\IH^q_{DB^*}=\IH^q_{D}$ if and only if $\IH^q_{D\wt B}= \IH^q_{D}$. More directly, Proposition \ref{prop:duality} applies to the pair of spaces $(\IH^\mT_{D\wt B},\IH^{\mT^*}_{BD})$  for the pairing $\pair  {Nf} g$ on $\clos{\ran_{2}(D)}\times \clos{\ran_{2}(BD)}$. Similarly, $h\in \dot\mH^{s,\pm}_{D\wt B}$ if and only if $Nh\in \dot\mH^{s,\mp}_{DB^*}$ and $\dot\mH^s_{D\wt B}, \dot \mH^{-s}_{BD}$ are dual spaces for this pairing (or, rather, its extension). Hence all statements proved  before adapt to this new pairing. 
    
  \begin{thm}

We  set $p_{\pm}(L)=p_{\pm}(DB)=p_{\pm}(BD)$ and $I_{L}$ be the subinterval  of $ (\frac{n}{n+1}, p_{+}(L))$  for which 
$\IH^q_{DB}=\IH^q_{D}$ with equivalence of norms. 

\

For  any $q\in I_{L}$, we have that all weak solutions of $Lu=0$ with $u\in \mE$, satisfy
\begin{equation}
\label{eq:NeuReg1}
\|\tN(\nabla u )\|_{q}\sim \|\nabla_A u|_{t=0}\|_{H^q}\sim \|S(t\partial_{t}\nabla u)\|_{q},
\end{equation}
where $H^q=L^q$ if $q>1$. 

\

For  any $q\in I_{L}$, we have that all weak solutions of $L^*u=0$ with $u\in \mE$, satisfy with $p=q'$ if $q>1$ and $\alpha=n(\frac{1}{q}-1)$ if $q\le 1$,
\begin{equation}
\label{eq:Dir1}
 \|S(t\nabla u)\|_{p}\sim \|\nabla_{A^*} u|_{t=0}\|_{\dot W^{-1,p}},  \end{equation}
\begin{equation}
\label{eq:Dir1'}
 \|t\nabla u\|_{T^\infty_{2,\alpha}}\sim \|\nabla_{A^*} u|_{t=0}\|_{\dot \Lambda^{-1,\alpha}}\sim \|\tNsa(v)\|_{\infty}.
\end{equation}
For those $p$ with $p>2$, we also have
\begin{equation}
\label{eq:Dir1''}
 \|\nabla_{A^*} u|_{t=0}\|_{\dot W^{-1,p}} \sim  \|\tNs(v)\|_{p}.
\end{equation}
Finally, we note the \textit{a priori}  ``$N<S$'' inequality. For $p$ as above, up to an additive normalizing constant $c$, we have
\begin{equation}
\label{eq:N<S}
\|\tN(u-c)\|_{p} \lesssim \|S(t\nabla u)\|_{p}. 
\end{equation}
\end{thm}

\begin{proof} The only thing to prove is \eqref{eq:N<S}.
Assume $\|S(t\nabla u)\|_{p}<\infty$, otherwise there is nothing to prove. 
Since $u\in \mE$, we know that 
$u(t,\,.\,)-c =(e^{-t|\wt BD|} h)_{\perp}=v_{\perp}$ for some $h\in \dot\mH^{s+1,+}_{\wt BD}$, which by Lemma \ref{lem:y-1} can also be chosen in $H^{p,+}_{\wt BD}$ for $p$ in the specified range,  and some $c\in \C^m$, and we have $\|S(t\nabla u)\|_{p}\sim \|\nabla_{A^*} u|_{t=0}\|_{\dot W^{-1,p}}\sim \|\IP h\|_{p}$.  
Approximate $h$ by $h_{k}\in \IH^{p,+}_{\wt BD}$ (one first approximates $h$ in $\IH^{p}_{\wt BD}$ and then, apply $\chi^+(\wt BD)$), then this gives a solution  $u_{k}$ by  $u_{k}(t,\,.\,)-c =(e^{-t|\wt BD|} h_{k})_{\perp}= (\IP e^{-t|\wt BD|} h_{k})_{\perp}$ and  Theorem \ref{thm:ntbd} implies
$$
\|\tN(u_{k}(t,\,.\,)-c)\|_{p} \lesssim \|\IP h_{k}\|_{p}.
$$
By the isomophism property of $\IP$,  $\IP  h_{k}$ converges to $\IP h$ in $L^p$ and also $u_{k}$ converges to $u$ in $L^2_{loc}(\reu)$. It is then easy to conclude using Lemma \ref{lem:y-1}.
\end{proof}

\begin{rem} The comparison  \eqref{eq:Dir1} and the first comparison in \eqref{eq:Dir1'} were used in \cite{AM}. Note that for $\alpha=0$, this is a Carleson measure/$\BMO$ comparison.

\end{rem}

\begin{rem}
Let us mention that under the De Giorgi condition on $L^*_{\ta}$ in Section \ref{sec:DGN},  we have a range  $(1-\varepsilon',2+\varepsilon)$ for \eqref{eq:NeuReg1}, a range $(2-\varepsilon, \infty)$ for $p$ in \eqref{eq:Dir1},  \eqref{eq:Dir1''} and \eqref{eq:N<S}, and a range $[0,\varepsilon)$ for  \eqref{eq:Dir1'}. Again, this is \textit{a priori} for weak solutions $u\in \mE$. 
\end{rem}

\subsection{Boundary layer potentials}\label{sec:boundarylayers}

Following \cite{R1}, the boundary layer operators are identified as follows:
for $t\ne 0$,  $\nabla_{A}\mS_{t}$ and $\mD_{t}$ are defined as $L^2$ bounded operators   by, for $f\in L^2(\R^n;\C^m)$, 
\begin{equation} 
\label{eq:gradst}
\nabla_{A}\mS_{t}f:= \begin{cases}
 +  e^{-t\dirac B}\chi^{+}(\dirac B) \begin{bmatrix} f \\ 0\end{bmatrix}    & \text{if } t>0, \\
 - e^{+t\dirac B}\chi^{-}(\dirac B) \begin{bmatrix} f \\ 0\end{bmatrix}      &  \text{if } t<0,
\end{cases}
\end{equation}
and 
\begin{equation}
\label{eq:dt}
\mD_{t}f:=
\begin{cases}
   - \bigg(e^{-tB\dirac} \chi^{+}(B\dirac) \begin{bmatrix} f \\ 0\end{bmatrix}\bigg)_{\no}  & \text{if } t>0, \\
   + \bigg(e^{+tB\dirac} \chi^{-}(B\dirac) \begin{bmatrix} f \\ 0\end{bmatrix}\bigg)_{\no}   & \text{if } t<0.
\end{cases}
\end{equation}
We recall that for any $h\in L^2$, $(\IP h)_{\no}=(h)_{\no}$, hence
\begin{equation}
\label{eq:dtp}
\mD_{t}f:=
\begin{cases}
   - \bigg(\IP e^{-tB\dirac} \chi^{+}(B\dirac) \begin{bmatrix} f \\ 0\end{bmatrix}\bigg)_{\no}  & \text{if } t>0, \\
   + \bigg(\IP e^{+tB\dirac} \chi^{-}(B\dirac) \begin{bmatrix} f \\ 0\end{bmatrix}\bigg)_{\no}   & \text{if } t<0.
\end{cases}\end{equation}

Now that we have the Sobolev space $\dot \mH^s_{D}$, \eqref{eq:gradst} makes sense for $f\in \dot \mH^s(\R^n;\C^m)=\dot \mH^s_{\no}$  for $-1\le s\le 0$ and we can even define $\mS_{t}$ consistently  from $\dot \mH^s(\R^n;\C^m)$ to $\dot \mH^{s+1}(\R^n;\C^m)$   by
\begin{equation} 
\label{eq:st}
\mS_{t}f:= \begin{cases}
  -\bigg( D^{-1} e^{-t\dirac B}\chi^{+}(\dirac B) \begin{bmatrix} f \\ 0\end{bmatrix} \bigg)_{\no}   & \text{if } t>0, \\
\bigg(D^{-1} e^{+t\dirac B}\chi^{-}(\dirac B) \begin{bmatrix} f \\ 0\end{bmatrix}\bigg)_{\no}      &  \text{if } t<0.
\end{cases}
\end{equation}
We remark that $D^{-1}$ can be indifferently  thought as a  $\dot \mH^{s}_{D}\to \dot \mH^{s+1}_{D}$ or  $\dot \mH^{s}_{DB} \to \dot \mH^{s+1}_{BD}$ map. As we take scalar components the conclusion is the same. 
\

Similarly the right hand side of  \eqref{eq:dtp} makes sense for $f\in \dot \mH^{s}(\R^n;\C^m)$ for $0\le s \le 1$ by the results of Section \ref{sec:Sobolev}. Indeed,  $  \begin{bmatrix} f \\ 0\end{bmatrix} \in \dot\mH^{s}_{D}= \IP\dot\mH^{s}_{BD}$ and   $\IP$ is the identity on the scalar part. Hence $\begin{bmatrix} f \\ 0\end{bmatrix} \in \dot\mH^{s}_{BD}$.  We define $\mD_{t}f$ by 
\eqref{eq:dtp} for such $f$. 

\

Note that we may let $t\to0$ from above or below using the strong continuity of the semigroups (In Sobolev spaces, this follows from the sectoriality of their generators as observed in Proposition \ref{prop:sobolev}) to obtain the  jump relations. Those were proved in \cite{AAAHK} under De Giorgi-Nash assumptions on $L$ and $L^*$. Let us see that. From \eqref{eq:gradst} we have for all $f\in \dot \mH^s(\R^n;\C^m)$, $-1\le s\le 0$, 
\begin{equation}
\label{eq:jumpst}
\nabla_{A}\mS_{0+}f -\nabla_{A}\mS_{0-}f= (\chi^+(DB)+ \chi^-(DB))  \begin{bmatrix} f \\ 0\end{bmatrix} =  \begin{bmatrix} f \\ 0\end{bmatrix}
\end{equation}
which encodes the jump relation of the conormal derivative of $\mS_{t}$ across the boundary and the continuity of the tangential gradient of  $\mS_{t}$ across the boundary. 
We used that $\chi^+(DB)+ \chi^-(DB)=I$ on $\dot \mH^s_{D} \ni  \begin{bmatrix} f \\ 0\end{bmatrix} $.
For the double layer, we have for $f\in \dot \mH^{s}(\R^n;\C^m)$, $0\le s\le 1$,
\begin{equation}
\label{eq:jumpdt}
\mD_{0+}f -\mD_{0-}f= -\bigg(\IP (\chi^+(BD)+ \chi^-(BD))  \begin{bmatrix}  f \\ 0\end{bmatrix}\bigg)_{\no} = -\bigg(  \begin{bmatrix} f \\ 0\end{bmatrix}\bigg)_{\no }=  -f .
\end{equation}
We used  that $\chi^+(BD)+ \chi^-(BD)=I$ on $\dot\mH^{s}_{BD}\ni  \begin{bmatrix} f \\ 0\end{bmatrix}$, by the results of Section \ref{sec:Sobolev}.  

\

Finally, we have   the usual duality relations of single layer potentials and double layer potentials. Denote for a moment $\mS_{t}=\mS_{t}^A$. Then, in the   $L^2(\R^n;\C^m)$ sesquilinear  duality, for $f\in \dot \mH^s(\R^n;\C^m)$ and $g \in \dot \mH^{-s-1}(\R^n;\C^m)$, $-1\le s\le 0$, 
\begin{equation}
\label{eq:dualst}
\pair {g} {\mS_{t}^A f}= \pair { \mS_{-t}^{A^*}g} f.
\end{equation}
We provide the proof for convenience using the duality $\pair {N\wt h} h$ for vectors and the relation between $A^*$ and $\wt B$. We may assume $t>0$. We have
\begin{align*}
 \pair {g} {\mS_{t}^A f}&= \Bpair {\begin{bmatrix}  {g} \\ 0\end{bmatrix}} {- D^{-1} e^{-t\dirac B}\chi^{+}(\dirac B) \begin{bmatrix} f \\ 0\end{bmatrix}}    \\
    &  = - \Bpair {N \begin{bmatrix}  {g} \\ 0\end{bmatrix}} {D^{-1} e^{-t\dirac B}\chi^{+}(\dirac B) \begin{bmatrix} f \\ 0\end{bmatrix}}
    \\
    & = + \Bpair {ND^{-1}\begin{bmatrix}  {g} \\ 0\end{bmatrix}} { e^{-t\dirac B}\chi^{+}(\dirac B) \begin{bmatrix} f \\ 0\end{bmatrix}}
    \\
    &
    = + \Bpair {N e^{t \wt B D}\chi^{-}(\wt BD) D^{-1}\begin{bmatrix}  {g} \\ 0\end{bmatrix}} {  \begin{bmatrix} f \\ 0\end{bmatrix}}
    \\
    & = + \Bpair {N D^{-1} e^{t D\wt B }\chi^{-}(D\wt B) \begin{bmatrix}  {g} \\ 0\end{bmatrix}} {  \begin{bmatrix} f \\ 0\end{bmatrix}}
\\
& =  \pair  {\mS_{-t}^{A^*}{g}} f. 
\end{align*}
Similarly, one has that, writing $\mD_{t}^A=\mD_{t}$ for a moment, for $f\in \dot \mH^s(\R^n;\C^m)$ and $g \in \dot \mH^{-s}(\R^n;\C^m)$, $0\le s\le 1$,  
\begin{equation}
\label{eq:dualdt}
\pair {g} {\mD_{t}^A f}= \pair {\partial_{{\nu}_{A^*}} \mS_{-t}^{A^*}g} f.
\end{equation}
The proof is similar to the above one.  Assume again $t>0$. We have
\begin{align*}
 \pair {g} {\mD_{t}^A f}&= \Bpair {N \begin{bmatrix}  {g} \\ 0\end{bmatrix}} {-e^{-t B\dirac }\chi^{+}(B\dirac ) \begin{bmatrix} f \\ 0\end{bmatrix}}    \\
    &  = - \Bpair {N  e^{t\dirac \wt B}\chi^{-}(\dirac \wt B) \begin{bmatrix}  {g} \\ 0\end{bmatrix}} { \begin{bmatrix} f \\ 0\end{bmatrix}}
    \\
    &
    = + \Bpair {N \nabla_{A^*}\mS_{-t}^{A^*} \begin{bmatrix}  {g} \\ 0\end{bmatrix}} {  \begin{bmatrix} f \\ 0\end{bmatrix}}
\\
& =  \pair  {\partial_{\nu_{A^*}}\mS_{-t}^{A^*}{g}} f. 
\end{align*}
The proof with $t<0$ is left to the reader.

 The extension of the semigroups to Hardy spaces $H^p_{DB}$ and $H^p_{BD}$ and identification with usual spaces made in Section \ref{sec:completions}  yield the following result.

\begin{thm}\label{thm:bl} Let $I_{L}$ be the interval  in $(\frac{n}{n+1}, p_{+}(L))$  on which $\IH^q_{DB}=\IH^q_{D}$ with equivalence of norms  and $I_{L^*}$ be the interval in  $(\frac{n}{n+1}, p_{+}(L^*))$  on which $\IH^q_{D\wt B}=\IH^q_{D}$ with equivalence of norms. 
\begin{enumerate}
  \item For $\in I_{L}$, we have the estimate
$$
\sup_{t>0}\|\nabla_{A}\mS_{t}f\|_{H^q} \lesssim \|f\|_{H^q}, \quad \forall f\in \bigcup_{-1\le s \le 0} \dot \mH^s(\R^n;\C^m),
$$
 where $H^q=L^q$ if $q>1$, and $\nabla_{A}\mS_{t}f$ converges strongly in $H^q$ as $t\to 0^+$. In particular, $\mS_{t}$, $\partial_{\nu_{A}}\mS_{t}$ and $\partial_{t}\mS_{t}$ extend to  uniformly bounded operators 
$$
 \mS_{t}: H^q\to \dot H^{1,q},  \quad  \partial_{\nu_{A}}\mS_{t}:  H^q\to  H^{q} $$
 and
 $$
\partial_{t}\mS_{t}: L^q\to L^q,  \quad \text{if, moreover, }  q>1, $$
with strong limit as $t\to 0+$.

 \item  For $q\in I_{L}$, we have the estimate
$$
\sup_{t>0}\|\nabla_{A}\mD_{t}f\|_{H^q} \lesssim \|\nabla f\|_{H^{q}}= \| f\|_{\dot H^{1,q}}, \quad \forall f\in  \bigcup_{0\le s \le 1} \dot \mH^s(\R^n;\C^m),
$$
 where $H^q=L^q$ if $q>1$, and $\nabla_{A}\mD_{t}f$ converges strongly in $H^{q}$ as $t\to 0^+$. In particular, $\mD_{t}$ extends to  uniformly bounded operators 
$$
 \mD_{t}: \dot H^{1,q} \to \dot H^{1,q},
 $$
 with strong limit as $t\to 0+$. 
 
  \item  For $q\in I_{L^*}$, 
we have the estimate
$$
\sup_{t>0}\|\mS_{t}f\|_{L^p} \lesssim \|f\|_{\dot W^{-1,p}}, \quad \forall f\in \bigcup_{-1\le s \le 0} \dot \mH^s(\R^n;\C^m),
$$
 where $p=q'$ if $q>1$, and $\mS_{t}f$  converges strongly in $\dot W^{-1,p}$ as $t\to 0^+$, and
 $$
\sup_{t>0}\|\mS_{t}f\|_{\dot \Lambda^\alpha} \lesssim \|f\|_{\dot \Lambda^{\alpha-1}}, \quad \forall f\in \bigcup_{-1\le s \le 0} \dot \mH^s(\R^n;\C^m),
$$
if $q\le 1$ and $\alpha=n(\frac{1}{q}-1)$ and $\mS_{t}f$ converges for the weak-$*$ topology of $ \dot \Lambda^\alpha$ if $t\to 0^+$.
 In particular, for those specified $p$ and $\alpha$, $\mS_{t}$ extends by density to  uniformly bounded operators 
$$
 \mS_{t}: \dot W^{-1,p} \to L^p 
 $$
with strong limit as $t\to 0+$  and by duality to bounded operators 
 $$
 \mS_{t}: \dot \Lambda^{\alpha-1} \to \dot \Lambda^\alpha, 
 $$
 with weak-$*$ limit  as $t\to0+$.
 
  \item  For $q\in I_{L^*}$, 
we have the estimate
$$
\sup_{t>0}\|\mD_{t}f\|_{L^p} \lesssim \|f\|_{L^p}, \quad \forall f\in \bigcup_{0\le s \le 1} \dot \mH^s(\R^n;\C^m),
$$
 where $p=q'$ if $q>1$, and $\mD_{t}f$  converges strongly  in $L^p$ as $t\to 0^+$, and
 $$
\sup_{t>0}\|\mD_{t}f\|_{\dot \Lambda^\alpha} \lesssim \|f\|_{\dot \Lambda^\alpha}, \quad \forall f\in \bigcup_{0\le s \le 1} \dot \mH^s(\R^n;\C^m),
$$
if $q\le 1$ and $\alpha=n(\frac{1}{q}-1)$ and $\mD_{t}f$ converges for the weak-$*$ topology of $\dot \Lambda^\alpha$ if $t\to 0^+$.
 In particular, for those specified $p$ and $\alpha$, $\mD_{t}$ extends by density to  uniformly bounded operators 
$$
 \mD_{t}: L^p \to L^p 
 $$
 with strong limit as $t\to 0+$ and by duality to bounded operators 
 $$
 \mD_{t}: \dot \Lambda^\alpha \to \dot \Lambda^\alpha
 $$
 with weak-$*$ limit as $t\to 0+$. 
 
 \item For any integer $k\ge 0$, the same estimates than for $\mS_{t}$ hold for $(t\partial_{t})^k\mS_{t}$     in the specified ranges  of the above items.
 The same estimates than for $\mD_{t}$ hold for $(t\partial_{t})^k\mD_{t}$ in the specified ranges of the above items. 
 \item The above items holds changing $t$ to $-t$. 
 \item The jump relations \eqref{eq:jumpst} and \eqref{eq:jumpdt} hold in all the  topologies  above where $\mS_{t}$ and $\mD_{t}$ are bounded respectively.  
 \end{enumerate}  
\end{thm}

According to Corollary \ref{cor:DGN}, this    improves the known results obtained in \cite{HMiMo} for operators with De Giorgi-Nash conditions as far as convergence at the boundary is concerned (strong convergence is obtained: it was known only for $p=2$ combining \cite{AA1} and \cite{R1}) and also with a weaker hypothesis (only an assumption on $L^*_{\ta}$ or $L_{\ta}$). Also these boundedness results are new  without De Giorgi-Nash conditions.  Let us now isolate the results concerning square functions and non-tangential maximal estimates for boundary layers.

\begin{thm}\label{thm:bl2} 
Let $I_{L}$ be the interval  in $(\frac{n}{n+1}, p_{+}(L))$  on which $\IH^q_{DB}=\IH^q_{D}$ with equivalence of norms  and $I_{L^*}$ be the interval in  $(\frac{n}{n+1}, p_{+}(L^*))$  on which $\IH^q_{D\wt B}=\IH^q_{D}$ with equivalence of norms. 
\begin{enumerate}
\item For $q\in I_{L}$, we have the estimate
\begin{align*}
   \|\tN(\nabla \mS_{\pm t}f) \|_{q}\sim \|t\partial_{t}\nabla \mS_{\pm t }f\|_{T^q_{2}} & \lesssim \|f\|_{H^q},  \\
     \|\tN(\nabla \mD_{\pm t}f) \|_{q}\sim \|t\partial_{t}\nabla \mD_{\pm t }f\|_{T^q_{2}} & \lesssim \|\nabla_{x} f\|_{H^q}\sim \|f\|_{\dot H^{1,q}},
\end{align*}
where $H^q=L^q$ if $q>1$. 
\item For $q\in I_{L^*}$, $q>1$ and $p=q'$ then
\begin{align*}
   \|\tN( \mS_{\pm t}f) \|_{p} \lesssim \|t\nabla \mS_{\pm t }f\|_{T^p_{2}} & \lesssim \|f\|_{\dot W^{-1,p}},  \\
     \|\tN(\mD_{\pm t}f) \|_{p} \lesssim \|t\nabla \mD_{\pm t }f\|_{T^p_{2}} & \lesssim \| f\|_{L^p}, 
\end{align*}
\item For $q\in I_{L^*}$, $q\le 1$  and $\alpha=n(\frac{1}{q}-1)$, then 
\begin{align*}
   \|\tNsa( \mS_{\pm t}f) \|_{\infty} \lesssim  \|t\nabla \mS_{\pm t }f\|_{T^\infty_{2,\alpha}} & \lesssim \|f\|_{\dot \Lambda^{\alpha-1}},  \\
     \|\tNsa(\mD_{\pm t}f )\|_{\infty} \lesssim \|t\nabla \mD_{\pm t }f\|_{T^\infty_{2,\alpha}} & \lesssim \| f\|_{\dot \Lambda^{\alpha}}, 
\end{align*}
\end{enumerate}
For statements concerning $\mS_{\pm t}$ we \textit{a priori} assume $f\in \bigcup_{-1\le s \le 0} \dot \mH^s(\R^n;\C^m)$, and for statements concerning $\mD_{\pm t}$, $f\in \bigcup_{0\le s \le 1} \dot \mH^s(\R^n;\C^m)$. Here, $\nabla$ is the full gradient  $(\partial_{t}, \nabla_{x})$. Alternately, it can be replaced by the conormal gradient $(\partial_{\nu_{A}}, \nabla_{x})$.  
The  non-tangential sharp functions are meant as the corresponding non-tangential maximal functions for $\mS_{\pm t}f- \mS_{\pm 0}f$ or $\mD_{\pm t}f- \mD_{\pm 0}f$. Also in (2), if $p>2$, the corresponding quantities $\|\tNs(.)\|_{p}$ are equivalent to the $T^p_{2}$ terms in the middle.
\end{thm}

As proved in \cite{AM}, there is a generalized boundary layer representation for the conormal gradients 
of solutions in $\mE$. This  can be integrated to give the ``usual'' boundary layer representation for the solution itself. It improves the results found in \cite{AM} and \cite{HKMP2}.  Theorem 8.1 in \cite{BM} proved under De Giorgi-Nash assumptions on $L$ and $L^*$ is of the same spirit.

\begin{cor}\label{cor:BL} Let  $I_{L^*}$ be the interval in  $(\frac{n}{n+1}, p_{+}(L^*))$  on which $\IH^q_{D\wt B}=\IH^q_{D}$ with equivalence of norms.
Let $u\in \mE_{s}$, $-1\le s\le 0$,  be a solution of $Lu=-\divv A \nabla u=0$ in $\reu$. Let $p\in (1,\infty)$ with $q=p'\in I_{L^*}$ such that   $u|_{t=0} \in L^p(\R^n; \C^m)$ and $\pd_{\nu_{A}}u|_{t=0}\in \dot W^{-1,p}(\R^n; \C^m)$. Then the abstract boundary layer representation
$$
u(t,x)=\mS_{t}(\pd_{\nu_{A}}u|_{t=0})(x) -  \mD_{t}(u|_{t=0})(x)
$$
holds for all $t\ge 0$ in $L^1_{loc}(\R^n;\C^m)$.  In particular, $\sup_{t\ge 0}\|u(t,\cdot)\|_{L^p(\R^n;\C^m)}<\infty$.
\end{cor}

\begin{proof}  Let  $s\in [-1,0]$ for which $u\in \mE_{s}$. By Corollary 8.4 in \cite{AM}, we have
$$
 \nabla_{A}u(t,\cdot)=  \nabla_{A}\mS_{t}(\pd_{\nu_{A}}u|_{t=0}) -  \nabla_{A}\mD_{t}(u|_{t=0}).
 $$
 The equality holds in  $\mE_{s} \cap C([0,\infty); \dot\mH^{s, +}_{\dirac B})$.
Thus, we have
\begin{equation}
\label{eq:green}
u(t,x)=\mS_{t}(\pd_{\nu_{A}}u|_{t=0})(x) -  \mD_{t}(u|_{t=0})(x) + c, \quad t>0,  
\end{equation}
in $L^2_{\loc}(\reu; \C^m)$,  but also in $ L^1_{loc}(\R^n; \C^m)$ for each $t>0$ as  the right hand side belongs to $L^p(\R^n; \C^m) + \C^m$  by the boundedness properties of the boundary layers established in Theorem \ref{thm:bl}  and the left hand side is in $L^2_{\loc}(\R^n;\C^m)$ as $u\in \mE_{s}$. We also point out that $c$ is independent of $t$ because both sides are weak solutions with the same conormal gradient at the boundary.  One can pass to the limit in $t\to 0$, after testing against a $C^\infty_{0}(\R^n;\C^m)$ function. For the right hand side, we use the strong limits in the theorem above and for the left hand side, this is because $t\mapsto u(t,\cdot)$ is continuous at $0$ in $L^2_{\loc}(\R^n;\C^m)$ as $u\in \mE_{s}$ (this observation is Remark 8.9 in \cite{AM}). One  obtains $u|_{t=0}(x)=\mS_{0}(\pd_{\nu_{A}}u|_{t=0})(x) -  \mD_{0^+}(u|_{t=0})(x) + c$.  As all the functions belong to $L^p(\R^n;\C^m)$, we conclude that $c=0$. 
\end{proof}

\begin{rem}
Note that \eqref{eq:green} holds under the sole assumption that $u\in \mE_{s}$. So for H\"older  or $\BMO$ spaces, the equality holds in those spaces. 
\end{rem}
\subsection{The block case}
Consider $$
A= \begin{bmatrix} a & 0 \\ 0 & d \end{bmatrix},$$
that is,  $A$ is block diagonal. In this case, $B$ is also block diagonal with
$$
B= \begin{bmatrix} a^{-1}& 0 \\ 0 & d \end{bmatrix}.$$

\subsubsection{The case $a=1$}  We  assume $a=1$. The Hardy space theory for $1<p<\infty$ was explicitly developed in \cite{HNP}.  The limitation to $p>1$ is due to the fact that these authors work with UMD-valued functions. Remark that 
$$
DB=\begin{bmatrix} 0& \divv d \\ -\nabla & 0 \end{bmatrix}, \qquad  (DB)^2 =\begin{bmatrix} -\divv d \nabla & 0 \\ 0 & -\nabla  \divv d\end{bmatrix}.
$$
In particular, $(DB)^2$ is sectorial with angle $\omega$ (instead of $2\omega$ if $B$ is an arbitrary matrix with  angle of accretivity $\omega$). Also $(DB)^2$ has an $H^\infty$-calculus on $L^2(\R^n; \C^N)$. 
Set $L=-\divv d \nabla$ and $M=-\nabla\divv d $, both defined as $\omega$-sectorial operators on $L^2(\R^n;\C^m)$ and $L^2(\R^n;\C^{nm})$ with $H^\infty$-calculus. Note that $M=0$ on $\nul(\divv d)$ and that the Hodge decomposition
$$
L^2(\R^n;\C^{nm})=   \clos{\ran_{2}(\nabla)}  \oplus \nul(\divv d)
$$
 is consistent with the splitting
$$
L^2(\R^n;\C^{n(1+m})= \clos{\ran_{2}(DB)}\oplus \nul(DB)= \begin{bmatrix} L^2(\R^n;\C^m) \\ \clos{\ran_{2}(\nabla)}   \end{bmatrix} \oplus \begin{bmatrix} 0 \\ \nul(\divv d)  \end{bmatrix}.
$$
It was shown in \cite{AS} that the interval $(p_{-}(DB), p_{+}(DB))$ is the largest interval of $p$ such that  one has the corresponding Hodge decomposition for $L^p$, which is also  $(q_{+}(L^*)', q_{+}(L))$ where $q_{+}(L)$ was introduced in \cite{Auscher}. 

Since $DB$ admits $L^2$ off-diagonal estimates to any order, so does $(DB)^2$ and, as $(DB)^2$ is diagonal, so do $L$ and $M$.  So both $L$ and $M$ enjoy a Hardy space theory. Only the decay of the allowable $\psi$ changes because of the second order nature of $L$ and $M$. Explicit conditions on $\psi$ can be found \cite{HNP} (see also \cite{HMMc}). Using even (with respect to $z\mapsto -z$) allowable $\psi$ for all these Hardy spaces $\IH^p$ below, we obtain that 
$$
f=\begin{bmatrix} f_{\perp} \\ f_{\ta}   \end{bmatrix}  \in \IH^p_{DB} \Longleftrightarrow f_{\perp} \in \IH^p_{L} \ \mathrm{and}\  f_{\ta}\in \IH^p_{M}, \mathrm{with\ } \|f\|_{\IH^p_{DB}}\sim \|f_{\perp}\|_{\IH^p_{L}}+\|f_{\ta}\|_{\IH^p_{M}}.$$

Using the $\IH^p_{DB}$ theory for $0<p<\infty$, we have that $\sgn(DB)$ is bounded on  $\IH^p_{DB}$. We note that this is equivalent to 
$$
\| L^{1/2} u \|_{\IH^p_{L}} \sim \|\nabla u\|_{\IH^p_{M}}, \quad \forall u \in \dot W^{1,2}(\R^n;\C^m).
$$
Indeed, pick $f\in \IH^2_{DB}=\IH^2_{D}$ so that $f_{\perp}\in L^2(\R^n;\C^m)$  and $f_{\ta}=\nabla g_{\perp}$ for $g_{\perp}\in  \dot W^{1,2}(\R^n;\C^m)$. Also, one can write $f_{\perp}=L^{1/2} h_{\perp}$  with $h_{\perp}\in  \dot W^{1,2}(\R^n;\C^m)$ by the solution of the Kato problem for operators and systems \cite{AHLMcT, AHMcT}. Then as $|DB|$ is the diagonal operator with entries $ L^{1/2}$, $M^{1/2}$, we have
$$
\sgn(DB)f= \begin{bmatrix}  L^{-1/2} \divv d \nabla g_{\perp} \\  -M^{-1/2} \nabla f_{\perp} \end{bmatrix} = \begin{bmatrix} - L^{1/2}  g_{\perp} \\  - \nabla L^{-1/2}f_{\perp} \end{bmatrix}= 
\begin{bmatrix} - L^{1/2}  g_{\perp} \\  - \nabla h_{\perp} \end{bmatrix}.
$$
For the last line, we used the equality  $(I+t^2M)^{-1}\nabla f= \nabla (I+t^2L)^{-1} f$ for all $f\in W^{1,2}$, extended to $f\in L^2$ (by extending the resolvents), and
\begin{align*}
   M^{-1/2}\nabla f&= \frac{2}{\pi}\int_{0}^\infty (I+t^2M)^{-1}tM^{1/2}M^{-1/2}\nabla f\,  \frac{dt}{t}
   \\
   & =  \frac{2}{\pi}\int_{0}^\infty \nabla L^{-1/2}tL^{1/2} (I+t^2L)^{-1} f\,  \frac{dt }{t}= \nabla L^{-1/2} f ,  
\end{align*}
where, classically,  the integrals  converge strongly in $L^2$  by the $H^\infty$-calculus for $L$ and $M$ and since both operators are bounded on $L^2$ (for the one on the left, one can see that by duality). Thus we may apply  the equality to $f_{\perp}\in L^2$. 
Thus 
$$
\|\sgn(DB)f\|_{\IH^p_{DB}}\sim \|L^{1/2}  g_{\perp} \|_{\IH^p_{L}}+ \|\nabla h_{\perp}\|_{\IH^p_{M}}
$$
while
$$
\|f\|_{\IH^p_{DB}}\sim \|L^{1/2} h_{\perp} \|_{\IH^p_{L}}+ \|\nabla g_{\perp}\|_{\IH^p_{M}}.
$$
As $h_{\perp}$ and $g_{\perp}$ are arbitrary and unrelated in $ \dot W^{1,2}(\R^n;\C^m)$, this shows the announced equivalence.\footnote{The direction from  boundedness of $\sgn(DB) $ to 
the statement for $L^{1/2}$ has been known for long: it is for example in \cite{AMcN1}. It is explicitly in \cite{HNP} in this context. The converse was pointed out to us by A. McIntosh.}

\begin{prop}
Let $p\in (\frac{n}{n+1},\infty)$. If $\IH^p_{DB}=\IH^p_{D}$ with equivalence of norms then $\IH^p_{L}=H^p\cap L^2$ and $\IH^p_{M}=H^p \cap \nabla \dot W^{1,2}$ and $\|L^{1/2} u \|_{H^p} \sim \|\nabla u\|_{H^p}$ for all  $u \in \dot W^{1,2}(\R^n;\C^m)$, 
where $H^p$ is the classical Hardy space if $p\le 1$ and $L^p$ is $p>1$.  \end{prop}

\begin{proof} Recall that $\IH^p_{D}= H^p\cap \IP(L^2)$ and $\IP(L^2)= L^2(\R^n;\C^m) \oplus \nabla \dot W^{1,2}(\R^n;\C^m)$.
Thus,  $\IH^p_{DB}=\IH^p_{D}$ if and only if $\IH^p_{L}=H^p\cap L^2$ and $\IH^p_{M}=H^p \cap \nabla \dot W^{1,2}$ so that they are both subspaces of $H^p$. The conclusion for the Riesz transform $\nabla L^{-1/2}$ follows right away. 
\end{proof}

The interval of $L^p$ boundedness of the Riesz transform $\nabla L^{-1/2}$ is characterized in \cite{Auscher} as the interval $(q_{-}(L), q_{+}(L))$, which is the largest open interval on which 
$\sqrt t \nabla e^{-tL}$ is bounded on $L^p$, uniformly in $t>0$. And it is also known that $q_{-}(L)=p_{-}(L)$ where $(p_{-}(L),p_{+}(L))$ is the largest open interval on which $e^{-tL}$ is  bounded on $L^p$, uniformly in $t>0$.  
It was shown in \cite{HMMc}  (in the case of equations: $m=1$)  that for $1<p<\infty$, $H^p_{L}=L^p$ if and only if  $p\in (p_{-}(L),p_{+}(L))$.  When $0<p\le 1$,  \cite{HMMc} proves that $\|f\|_{H^p} \lesssim \|f\|_{\IH^p_{L}}$ and, when $(p_{-}(L))_{*}<p\le 1$,  that 
$\|L^{1/2} u\|_{\IH^p_{L}} \sim \|\nabla u\|_{H^p}$ when $u \in \dot W^{1,2}(\R^n)$. But  $H^p_{L}$ is not identified when $p\le 1$.  The possibility of identifying $H^p_{L}$ for $p\le 1$ seems new.  It turns out that the number $p_{-}(L)$ may not be the relevant critical exponent for this.  We isolate a number of interesting facts in this corollary.

\begin{cor} Let $I$ be the interval in $(\frac{n}{n+1}, \infty)$ on which $\IH^p_{DB}=\IH^p_{D}$ with equivalence of norms. Then,
$I\cap (1,\infty) \subset (q_{-}(L), q_{+}(L))$. As $q_{+}(L)=p_{+}(DB)$, we also conclude that $\sup I= p_{+}(DB)$.  Also, if $p\in I\cap (1,\infty)$, then $e^{-tL}$ is bounded on $L^p$ uniformly in $t>0$. 
Finally, if  $\inf I<p\le 1$, $H^p_{L}=H^p$.
\end{cor}

A large part of  \cite{HMMc} is concerned with developing  the $H^p_{L}$ theory, for the full range $0<p<\infty$ together with variants involving regularity indices.  See also \cite{JY} for $0<p \le 1$.  See also non-tangential maximal estimates in \cite{May} towards solving the associated second order PDE $\partial_{t}^2u+\divv d \nabla u=0$, which can be seen as a special case of \eqref{eq:divform}. Some larger ranges of exponents are obtained there, probably due to the ``diagonal'' structure of the PDE (no cross terms in $t$ and $x$).

\subsubsection{The case $a\ne 1$} The full block diagonal case with $a\ne 1$ can  be treated similarly. 
 In this situation,  $L=-\divv d \nabla a^{-1}$ and $M=- \nabla a^{-1} \divv d$, which are $2\omega$-sectorial operators on $L^2(\R^n;\C^N)$ with  $H^\infty$-calculus on $L^2(\R^n;\C^N)$ as diagonal components of $(DB)^2$. The same discussion applies concerning the links between $\IH^p_{DB}$, $\IH^p_{L}$ and $\IH^p_{M}$ and that $\|L^{1/2}u\|_{\IH^p_{L}}\sim \|\nabla (a^{-1}u)\|_{\IH^p_{M}}$. Thus if $\IH^p_{DB}=\IH^p_{D}$, then $\IH^p_{L}=H^p\cap L^2$ and $H^p_{L}=H^p$ (again, this is by convention $L^p$ if $p>1$). Remark also that if $\IH^p_{DB}=\IH^p_{D}$ and $p>1$, then the resolvent of $L$ and semigroup generated by $L^{1/2}$ are bounded on $L^p$ (There may be no semigroup generated by $-L$ if $2\omega\ge \pi/2$).

 If $\IH^p_{DB}=\IH^p_{D}$, by similarity, we obtain a characterization of the Hardy space associated to $-a^{-1}\divv d \nabla$ as $a^{-1}H^p$. 
 
 In boundary dimension $n=1$, $M$ and $L$ are of the same type because $\divv$ and $\nabla$ both become $\frac{d}{dx}$. Although not formulated in the language of the current article, it was shown in \cite{AT} that $H^p_{L}=H^p$ for all $p\in (\frac{1}{2},\infty)$ (in the case of equations, that is when $m=1$). The same thus holds for $M$ replacing $L$ and therefore $H^p_{DB}=H^p$ for those $p$. The proof there extends to arbitrary systems with $m>1$. Nevertheless, this follows directly on applying Proposition \ref{prop:n=1} for any $m$ as the symbol of $D$ is invertible on $\R\setminus \{0\}$.

\section{Systems with Giorgi type conditions}\label{sec:DGN}

  We are given  $B=\widehat A$, $A= \begin{bmatrix} a & b \\ c & d \end{bmatrix}$ and $D= \begin{bmatrix} 0 & \divv \\ -\nabla & 0 \end{bmatrix}$ as before, which corresponds to the second order system $L=-\divv A \nabla$.

Let $L_{\ta}= -\divv d \nabla $ where $d$ is the lower right coefficient in $A$. This operator acts on the boundary $\R^n$ of $\reu$. Classical elliptic theory implies  there exists $\lambda(L_{\ta}) \in (0, n]$ such that the following holds: 

 For any $\lambda\in [0, \lambda(L_{\ta}))$, there exists a constant
$C \ge 0$  such that for any ball  $B(x_0,R)$, for any $v \in
W^{1,2}(B(x_0,R))$ weak solution in
$B(x_0,R)$ of $L_{\ta}v=0$  and for all $0
<
\rho
\le R$
\begin{equation}\label{DG}
\int_{B(x_0,\rho)} |\nabla v|^2 \le C\left( \frac{\rho}{R}\right)^{\lambda}\int_{B(x_0,R)} |\nabla v|^2. 
\end{equation}
 
 The constant $C$  depends on  $L^\infty$ bounds and accretivity of $d$ on $\ran_{2}(\nabla) $, $\lambda$ and $\lambda(L_{\ta})$.
 
 \begin{defn} (from \cite{A}) 
$L_{\ta}$ satisfies the De Giorgi condition if $\lambda(L_{\ta})>n-2$.
\end{defn}

 It is equivalent to the fact that weak solutions of $L_{\ta}$ are locally bounded and H\"older continuous with exponent less than $\alpha(L_{\ta})= \frac{\lambda(L_{\ta})-n+2}{2}$.  See \cite{HK} for explicit proofs. 

 This condition holds for any $L_{\ta}$ as above if $n\le 2$, for real $d$  and their $L^\infty$ perturbations when $m=1, n\ge 3$. It also holds if $d$ is constant for any $n,m$ (with $\lambda_{+}(L_{\ta}^*)=n$) and  if $d$ is an $L^\infty$ perturbation of a constant (with any $\lambda(L_{\ta})<n$).

 \begin{thm} \label{th:main}  
Assume that $L^*_{\ta}$ satisfies the De Giorgi condition.  For  $p_{\ta}< p \le 1$, with $p_{\ta}= \frac{n}{n+\alpha(L_{\ta}^*)}$,   any $\left(\IH^{p}_{D},1\right)$-atom $\alpha$  and  integer $M\ge M(n)$, we have
$$ \| tDB (I+itDB)^{-M} \alpha\|_{T^p_{2}} \lesssim 1
$$
with implicit constant depending only on $n, m, M$, the $L^\infty$ and accretivity bounds of $B$,  and the constants in the De Giorgi condition for $L^*_{\ta}$.
\end{thm}

It is quite striking that no regularity is imposed on the weak solutions of $L_{\ta}$, nor any condition on the other coefficients $a,b,c$ of $L$.

\begin{cor}\label{cor:DGN}
Assume that $L^*_{\ta}$ satisfies the De Giorgi condition. Then we have 
$\IH^p_{DB}=\IH^p_{D}$ with equivalence of norms for  $p_{\ta}< p <p_{+}(DB)$.
\end{cor}

We remark that this identification is obtained without knowing kernel bounds. 

\begin{proof} The case $2<p<p_{+}(DB)$ is from the general theory and there is nothing new. We consider $p<2$.

Remark that $\psi(z)=z(1+iz)^{-M}\in \Psi_{1}^{M-1}(S_{\mu})$ is allowable for $\IH^p_{DB}$ for any $p\in (\frac{n}{n+1},2)$ if $M-1>\frac{n}{2}+1$. The theorem above  tells that  for $p_{\ta}<p\le 1$
and  $\left(\IH^{p}_{D},1\right)$-atoms $\alpha$, $\alpha\in \IH^p_{DB}$ and $\|\alpha\|_{\IH^p_{DB}}
\lesssim 1$. A density argument provides that $\IH^p_{D}\subset \IH^p_{DB}$ with continuous inclusion. By complex interpolation (arguing as in the proof of Corollary \ref{lem:lowerp<2}) this holds for $1<p<2$. Now the converse inclusion and continuity bound  were known from Corollary \ref {cor:lowerbounddbp<2}  for $\frac{n}{n+1}<p<2$. 
\end{proof}

Thus, by duality, all the \textit{a priori} estimates for weak solutions of $Lu=0$  with $u\in \mE$ apply to this situation assuming  $L_{\ta}$ satisfies the De Giorgi condition with exponent $\lambda(L_{\ta})>n-2$. For example, we have, normalizing $u$ by an additive constant in the first inequality, 
$$ 
\|\tN(u)\|_{p} \lesssim \|S(t\nabla u)\|_{p},\quad \forall p\in (2-\varepsilon,\infty),
$$
$$
\|\tNs(u)\|_{p} \lesssim \|S(t\nabla u)\|_{p},\quad \forall p\in (2,\infty),
$$
$$
\|\tNsa(u)\|_{\infty} \lesssim \|t\nabla u\|_{T^\infty_{2,\alpha}}, \quad \forall \alpha\in [0,\alpha({L})),  \ \alpha(L)= \frac{\lambda(L_{\ta}) -n+2}{2}.
$$

The first inequality was known if $L$ is a real and scalar operator \cite{HKMP1}. However there is a subtle difference. In that work, the \textit{a priori } assumption  $u\in \mE$ is not required and $p$ can be any positive number: the proof  in this specific situation  uses the $p=2$ case in \cite{AA1} and good lambda arguments requiring the converse inequality  $\|S(t\nabla u)\|_{p} \lesssim \|\tN(u)\|_{p}$ (which \cite{HKMP1}  proves using changes of variables, so it is not clear at all whether this can extend to complex situations) and finiteness of $ \|\tN(u)\|_{p}$ (which can even be replaced by the usual non-tangential function by interior regularity estimates).  So, in fact, \cite{HKMP1} proves that on the class of weak solutions with $ \|\tN(u)\|_{p}<\infty$,  one has $\|S(t\nabla u)\|_{p} \sim \|\tN(u)\|_{p}$ for any $0<p<\infty$.   Here, we show that when $u\in \mE$, then $\|\tN(u)\|_{p} \lesssim \|S(t\nabla u)\|_{p}$ when $2-\varepsilon<p<\infty$ (Note that the \textit{a priori} information $u\in \mE$ will be removed in \cite{AM2}: the combination of all this shows that  the two classes of weak solutions (one with square function control and the other with non-tangential maximal control)  are identical (up to additive normalisation) in this range of $p$ and this class of operators (real and scalar)).

The latter two inequalities seem new even when $L$ is  a real and scalar operator.

\subsection{Preliminary computations}

We begin with some computation. As before, we write  $f \in L^2(\R^n;\C^{(1+n)m})$ as $f = \begin{bmatrix}
      f_{\perp}    \\
      f_{\ta}
\end{bmatrix}$ with $f_{\perp} \in L^2(\R^n;\C^{m})$ and $f_{\ta}\in L^2(\R^n;\C^{nm})$. We also write $L^2$ from now on without precision. 

For $t\in \R$ set   $R_{t}=(I+itDB)^{-1}$ and 
$$
L_{t}= \begin{bmatrix} 1  & it\divv_x \end{bmatrix} \begin{bmatrix} a(x) & b(x) \\ c(x) & d(x) \end{bmatrix}  \begin{bmatrix} 1 \\ it\nabla_x  \end{bmatrix}. 
$$

\begin{lem} Let $f \in L^2$ and $t\in \R$. Then the equation $R_{t}f=F$ is equivalent to the system
\begin{align}
\label{systemf1}
 F_{\perp}   &= au_{t}+ bF_{\ta}   \\
 \label{systemf2}
  F_{\ta}  &= it\nabla_{x} u_{t}  + f_{\ta}
\end{align}
with 
$$u_{t}= L_{t}^{-1}  \begin{bmatrix} 1  & it\divv_x \end{bmatrix} \begin{bmatrix} f_{\perp} - bf_{\ta} \\ -d f_{\ta} \end{bmatrix}.
$$
\end{lem}

\begin{proof}
Let $g, G$ defined by $f=\oA g$ and $F=\oA G$ with $\oA(x)=  \begin{bmatrix} a(x) & b(x) \\ 0 & 1 \end{bmatrix}$. Then, by \cite[Lemma 2.53]{AAH} (see \cite[Lemma 9.3]{AA1} for a direct proof in this context), $R_{t}f=F$ is equivalent to 
\begin{align}
\label{systemg}
 G_{\perp}   &= u_{t}   \\
  G_{\ta}  &= it\nabla_{x} u_{t}  + g_{\ta}
\end{align}
It suffices to note that $F_{\perp}= aG_{\perp}+ bG_{\ta}$ and $F_{\ta}=G_{\ta}$.
\end{proof}

\begin{lem} Assume $f \in L^2$ has the form $f = \begin{bmatrix}
      f_{\perp}    \\
      it\nabla h
\end{bmatrix}$ with $f_{\perp} \in L^2$ and $h\in W^{1,2}$. Then the equation $R_{t}f=F$ is equivalent to  $F= \begin{bmatrix}
      F_{\perp}    \\
      it\nabla H
\end{bmatrix}$ with $F_{\perp} \in L^2$ and $H\in W^{1,2}$ given by 
$$
\begin{bmatrix}
      F_{\perp}    \\
      H
\end{bmatrix}= \mR_{t} \begin{bmatrix}
      f_{\perp}    \\
    h
\end{bmatrix}
$$
with $\mR_{t}$ being the $2\times 2$ matrix of operators
$$\mR_{t}= \begin{bmatrix} aL_{t}^{-1} & T_{t} \\ L_{t}^{-1} & U_{t} \end{bmatrix} ,
$$
where
 $$U_{t}=L_{t}^{-1} (a +it\divv c),$$
and
$$
T_{t}= -a + (a+itb\nabla) L_{t}^{-1} (a +it\divv c).
$$
Here, $a,b,c,d$ mean multiplication by the corresponding functions $a(x), b(x), c(x), d(x)$. 
\end{lem}

\begin{proof} Write 
$$
u_{t}= L_{t}^{-1}  ( f_{\perp} - itb\nabla h-it \divv d \nabla h)
$$
and using the definition of $L_{t}h$ we obtain
$$
u_{t}= -h + L_{t}^{-1}  ( f_{\perp} + ah + it\divv c h).
$$
Thus \eqref{systemf2} is equivalent to 
$$
F_{\ta}= it \nabla L_{t}^{-1}  ( f_{\perp} + ah + it\divv c h) =it \nabla (L_{t}^{-1}   f_{\perp}  + U_{t}h)
$$
because $-it \nabla h +f_{\ta}=0$, and   \eqref{systemf1} is equivalent to
$$
F_{\perp}=aL_{t}^{-1} f_{\perp} + T_{t}h.
$$
\end{proof}

\subsection{Proof of Theorem \ref{th:main}}

We start the proof of the theorem. Let $\alpha=D\beta$ be an $\left(\IH^{p}_{D},1\right)$-atom. This means that
$\alpha,\beta$ are both supported in a ball $Q$, with $\|\alpha\|_{2}\le |Q|^{\frac{1}{2}-\frac{1}{p}}$ and $\|\beta\|_{2}\le r(Q)|Q|^{\frac{1}{2}-\frac{1}{p}}$, with $r(Q)$ the radius of $Q$. Note that $\alpha_{\perp}$ is the divergence of $\beta_{\ta}$. In particular, $\alpha_{\perp}$ is a classical $L^2$-atom (valued in $\C^m$) for the Hardy space $H^p$ and each component has mean value 0. Also $\alpha_{\ta}$ is a gradient field. 

Call $C_{k}(T_{Q})$ the following regions in $\R^{1+n}_+$. For $k\ge0$, $R_{k}(T_{Q})=(0, 2^{k}r(Q)] \times 2^{k}Q$,  $C_{0}(T_{Q})=R_{1}(T_{Q})$ and  $C_{k}(T_{Q})=R_{k+1}(T_{Q})\setminus R_{k}(T_{Q})$ for $k>0$. It is enough to show
\begin{equation}
\label{tentk}
\iint_{C_{k}(T_{Q})} |tDBR_{t}^M \alpha|^2\,  \frac{dtdx}{t} \lesssim |2^kQ|^{1-\frac{2}{p}} 2^{-k\varepsilon}
\end{equation}
for some $\varepsilon>0$ and $M$ large enough.   

For simplicity we assume that $Q$ is the unit ball centered at 0. All estimates are affine invariant because all assumptions in the theorem are stable under   affine changes of variables so this is no loss of generality. 

 First for $k=0$, \eqref{tentk} holds as a consequence of the 
square function estimate \eqref{eq:psiT} for  $DB$ and the size of $\|\alpha\|_{2}$. 

For $k>0$, we note that $itDBR_{t}^M\alpha = R_{t}^{M-1}\alpha
- R_{t}^M\alpha$ and it is enough to treat each term separately. Hence we have show
\begin{equation}
\label{tentk1}
\iint_{C_{k}} |R_{t}^M \alpha|^2\,  \frac{dtdx}{t} \lesssim 2^{k(n-\frac{2n}{p}-\varepsilon)}
\end{equation}
for large enough $M$, where we set $C_{k}= C_{k}(T_{Q})$. 

 The part of the integral in \eqref{tentk} where $t\le 1$ can be treated using  
the $L^2$ off-diagonal decay of $R_{t}^M$ \eqref{odn}
\begin{equation}
\label{tentk2}
\int_{2^{k+1}Q\setminus 2^kQ}  |R_{t}^M \alpha|^2\,  dx \lesssim (2^k/t)^{-N} \| \alpha\|_{2}^2
\end{equation}
for all $N$. 
Thus integrating this estimate in $t \in (0,1]$ yields a bound $2^{-kN}$.

For the remaining part, when $t>1$, we claim  assuming $M$ large enough and all $N$, we have that  for $1\le t < 2^k$, we have
\begin{equation}
\label{tent2}
\int_{2^{k+1}Q\setminus 2^kQ}  |R_{t}^M \alpha|^2\,  dx \lesssim (2^k/t)^{-N} t^{n- \frac{2n}{p}- \varepsilon}
\end{equation}
and for $2^k \le t \le 2^{k+1}$,
\begin{equation}
\label{tent3}
\int_{2^{k+1}Q}  |R_{t}^M \alpha|^2 \, dx \lesssim   t^{n-\frac{2n}{p} -\varepsilon}.
\end{equation}
Then, integrating in the corresponding $t$ intervals the above estimates concludes the proof of 
\eqref{tentk2}. 

To end the proof of the theorem, it remains to prove the claim. This is where we use fully that $\alpha$ is an $\left(\IH^{p}_{D},1\right)$-atom  and the above calculations. Write $\alpha=f$, $f^{(k)}=R_{t}^kf$. Since   
$f = \begin{bmatrix}
      f_{\perp}    \\
      it\nabla h
\end{bmatrix}$ with $h=-(it)^{-1} \beta_{\perp}$, we have $f^{(k)} = \begin{bmatrix}
      f_{\perp}^{(k)}   \\
      it\nabla h^{(k)}
\end{bmatrix}$ and $\begin{bmatrix}
      f_{\perp}^{(k)}   \\
       h^{(k)}
\end{bmatrix} = \mR_{t}^k  \begin{bmatrix}
      f_{\perp}   \\
       h
\end{bmatrix}$.
Fix $t>0$. 
Since $L_{t}h^{(k+1)} =f_{\perp}^{(k)}+ ah^{(k)} + it\divv c h^{(k)}$, the usual Caccioppoli argument  
for the (non homogeneous) operator $L_{t}$ yields
$$
\int_{B_{t}} | it\nabla h^{(k+1)}|^2 dx \le C \int_{cB_{t}} (|h^{(k+1)}|^2 + |f_{\perp}^{(k)}|^2 + |h^{(k)}|^2) dx
$$
for any $c>1$ and some $C>0$ independent of the ball $B_{t}$ of radius within $t/2$ and $2t$, $k$ and depending only on the $L^\infty
$ and accretivity bounds of $A$. From $|f^{(k+1)}|^2= |f_{\perp}^{(k+1)}|^2+|it\nabla h^{(k+1)}|^2$ and using a bounded covering by balls of radius $\sim t$, we see that it is enough to prove
\eqref{tent2} and \eqref{tent3} by replacing $R_{t}^M \alpha$ by $\mR_{t}^M \begin{bmatrix}
      f_{\perp}    \\
       h
\end{bmatrix}$ (up to fattening  slightly $C_{k}$ to a similar type of region, which we ignore in the sequel as this is only a cosmetic change in the estimates). Hence, it suffices  to prove assuming $M$ large enough that, for  all $N$  and $1\le t < 2^k$, we have
\begin{equation}
\label{tent2'}
\int_{2^{k+1}Q\setminus 2^kQ}  | \mR_{t}^M \begin{bmatrix}
      f_{\perp}    \\
       h
\end{bmatrix} |^2  dx \lesssim (2^k/t)^{-N} t^{n- \frac{2n}{p}-\varepsilon}
\end{equation}
and for $2^k \le t \le 2^{k+1}$,
\begin{equation}
\label{tent3'}
\int_{2^{k+1}Q}  | \mR_{t}^M \begin{bmatrix}
      f_{\perp}    \\
       h
\end{bmatrix} |^2  dx \lesssim t^{n-\frac{2n}{p} -\varepsilon}.
\end{equation}

To do this, we proceed to an analysis of the iterates of the adjoint  of $\mR_{t}$, starting from $L^2$ using the scales of  Morrey spaces  and Campanato 
spaces  (here for functions  defined on $\R^n$ and valued in $\C^m$) following \cite{A}.
For $0 \le \lambda \le n$, define the Morrey
space
$L^{2,\lambda}(\R^n;\C^m)=L^{2,\lambda}_{0}\subset L^2_{loc}$ by the condition
$$
\|f\|_{L^{2,\lambda}_{0}} \equiv   \sup_{x \in \R^n,\,  0 < R \le 1} \left(
R^{-\lambda}
\int_{B(x,R)} |f|^2 \right)^{1/2} < \infty,
$$
where  $B(x, r)$ denotes the Euclidean ball of center $x$
and radius
$r>0$.  For $0 \le \lambda \le n+2$, one
defines the Campanato space 
$ L^{2,\lambda}_1(\R^n;\C^m)=L^{2,\lambda}_1 \subset L^2_{loc}$ by
$$
\|f\|_{ L^{2,\lambda}_1}\equiv   \sup_{x \in \R^n,\,  0 < R \le 1} \left(
R^{-\lambda}
\int_{B(x,R)} |f- (f)_{x,R}|^2 \right)^{1/2} < \infty. 
$$
The notation $(u)_{x,R}$ stands for the mean value of $u$ over the ball
$B(x,R)$. The  space $L^2 \cap 
 L^{2,\lambda}_i$ is equipped with
 the norm $\|f\|_2+\|f\|_{ L^{2,\gamma}_i}$.
 We also denote by $\mL^{2,\lambda}_{i}$ the corresponding homogeneous spaces when dropping the constraint that $R\le 1$.

 Here are a few facts for the appropriate ranges of $\lambda$.

(a)  $L^{2,\lambda_{1}} _{i}\subset L^{2,\lambda_{2}}_{i}$ if $\lambda_{1} >\lambda_{2}$.

 (b) $L^2 \cap L^{2,\lambda}_1 \equiv L^2 \cap L^{2,\lambda}_{0}$ if $\lambda <n$.
 
 (c) $L^2 \cap L^{2,\lambda}_{i} \equiv L^2 \cap \mL^{2,\lambda}_{i}$.
 
 (d) $L^{2,\lambda}_{0}$ is preserved by multiplication by bounded functions.

In particular the higher the $\lambda$, the better the regularity in these scales. We have the following lemma.

\begin{lem}\label{lem:MC} For $M$ large enough (depending only on dimension) and $0\le \lambda< \lambda(L_{\ta}^*)$ ($\le n$), we have that $\mR_{t}^{*M}$ maps $L^2 \times L^2$ into $\mL^{2,\lambda+2}_{1} \times \mL^{2,\lambda}_{0}$ for all $t\ne 0$. Furthermore, the operator norm of 
$\begin{bmatrix} 1 & 0 \\ 0 & 0 \end{bmatrix}\mR_{t}^{*M}$ from $L^2\times L^2$ into $\mL^{2,\lambda+2}_{1} $  is bounded by $C|t|^{-\lambda/2 - 1}$ and 
the operator norm of 
$\begin{bmatrix} 0 & 0 \\ 0 & 1 \end{bmatrix}\mR_{t}^{*M}$ from $L^2\times L^2$ into $\mL^{2,\lambda}_{0} $  is bounded by $C|t|^{-\lambda/2}$. 
\end{lem}

Assuming this lemma, we argue as follows to prove \eqref{tent2'} and \eqref{tent3'}.  First, the De Giorgi condition and $p_{\ta}<p$ means that we can take $\lambda=n-2+2\alpha$ for some $\alpha>n(\frac{1}{p}-1)$  in the previous lemma and the sought $\varepsilon$ will be $2\alpha-2n(\frac{1}{p}-1)$.  
Next,  we prove \eqref{tent3'} by dualizing against $g \in L^2 \times L^2$, supported in $2^{k+1}Q$, with norm 1. Then
$$ \Bpair { \mR_{t}^M \begin{bmatrix}
      f_{\perp}    \\
       h
\end{bmatrix}}  g    =  \Bpair
     { f_{\perp}}    {  \begin{bmatrix} 1 & 0 \\ 0 & 0 \end{bmatrix}\mR_{t}^{*M}g } +
      \Bpair
      {h }{  \begin{bmatrix} 0 & 0 \\ 0 & 1 \end{bmatrix}\mR_{t}^{*M}g}.
      $$
     For the first term,  since $f_{\perp}$ has mean value 0 on $Q$, we can subtract the mean value on $Q$ of $\begin{bmatrix} 1 & 0 \\ 0 & 0 \end{bmatrix}\mR_{t}^{*M}g$ then use
     Cauchy-Schwarz inequality and the $\mL^{2,\lambda+2}_{1} $ estimate which leads to a bound $\|f_{\perp}\|_{2} Ct^{-\lambda/2 - 1}\|g\|_{2} \le Ct^{-\lambda/2 - 1}$.
     For the second term, we merely use Cauchy-Schwarz inequality and the $\mL^{2,\lambda}_{0} $ estimate which leads to a bound $\|h\|_{2}Ct^{-\lambda/2}\|g\|_{2} \le Ct^{-\lambda/2-1}$ using that $\|h\|_{2}\le t^{-1}$. This proves  \eqref{tent3'}. 
     
     To prove  \eqref{tent2'}  we need to incorporate some decay in the bounds of the above lemma. This is done using the standard exponential perturbation argument. Let $\varphi$ be a real-valued,    Lipschitz  function. We also assume $\varphi$ bounded but do not use its bound. Let $\mR_{t,\varphi}= \exp(-\varphi/t) \mR_{t}\exp(\varphi/t)$.  A simple computation shows that this operator has the same form  and properties as $\mR_{t}$ with $d$ unchanged and  $a,b,c$ modified by an additive $O(\|\nabla \varphi\|_{\infty})$ term. Also since the higher order coefficient of $\exp(-\varphi/t) L_{t}\exp(\varphi/t)$ is the same as the one of $L_{t}$, we also have the De Giorgi condition on the adjoint of the higher order term. 
     Thus,  we have the same bounds for $\mR_{t,\varphi}^M$ uniformly for $\|\nabla \varphi\|_{\infty}\le \delta $ for some $\delta >0$ depending solely on $L^\infty$ and accretivity bounds,  and on the De Giorgi condition of $L_{\ta}^*$. 
     Having fixed $Q$ (the unit ball) and $k\ge 1$, we choose  $\varphi(x)=\delta \inf (d(x,Q),  N)$ for a fixed $N\ge 2^{k+1}$. Hence, $\|\nabla \varphi\|_{\infty}\le \delta $ and $\inf \varphi = \delta (2^k - 1)$ on $ 2^{k+1}Q\setminus 2^kQ$. Using the support condition of $f_{\perp}, h$ and the definition of $\varphi$, we obtain that
     $$
    \mR_{t}^M  \begin{bmatrix}
      f_{\perp}    \\
       h
\end{bmatrix} =  \mR_{t}^M  \begin{bmatrix}
   \exp(\varphi/t)   f_{\perp}    \\
    \exp(\varphi/t)   h
\end{bmatrix} = \exp(\varphi/t) \mR_{t,\varphi}^M  \begin{bmatrix}
      f_{\perp}    \\
       h
\end{bmatrix}. $$
Using the bounds for $\mR_{t,\varphi}^M$, we obtain  powers of $t$  as above, multiplied by the supremum on $ 2^{k+1}Q\setminus 2^kQ$ of $ \exp(-\varphi/t)$, that is $ \exp (-\delta (2^k-1)/t)$. This proves \eqref{tent2'}. The proof of the theorem is complete modulo that of the last lemma. 

For later use, we record the following estimate that comes from a modification of the above arguments.

\begin{cor} Assume  $ \lambda(L_{\ta}^*)>n-2$ and let $p_{\ta} <p\le 1$. If $\alpha$ is a $(\IH^p_{D},1)$-atom associated to the ball $Q$, then for any other ball $Q'$, we have for large enough $M$ (depending only on dimension and $\lambda(L_{\ta}^*)$) 
\begin{equation}
\label{tent3''}
\int_{Q'}  |R_{t}^M \alpha|^2\,  dx \lesssim e^{-\delta \frac{\dist(Q',Q)}{t}} t^{n- \frac{2n}{p}- \varepsilon}
\end{equation}
for all $t>0$ and some $\delta >0$ and $\varepsilon>0$. 

\end{cor}

\subsection{Proof of  Lemma \ref{lem:MC}} First by scaling it suffices to assume $t=1$. Since the Morrey and Campanato spaces of the statement are the homogeneous ones, the powers of $t$ follow automatically by a rescaling argument (which yields operators with the same hypotheses). 
We thus drop the index $t$ in the notation. From fact (c), it suffices to work in the inhomogeneous spaces. 
It follows from \cite[ Theorem 3.10] {A} (this is done for real  equations but the proof applies \textit{mutatis mutandi} to complex systems with G\aa rding inequality) that for $\lambda\ge 0$ we have the boundedness properties
\begin{align*}
  L^{*-1}  & :   L^2 \cap L^{2,\lambda}_{1} \to L^2 \cap L^{2,\lambda' }_{1}, \quad 0\le \lambda' \le \lambda+4, \quad \lambda'< \lambda(L_{\ta}^*),\\
  \nabla L^{*-1}  & :   L^2 \cap L^{2,\lambda}_{1} \to L^2 \cap L^{2,\lambda' }_{1}, \quad 0\le \lambda' \le \lambda+2, \quad \lambda' < \lambda(L_{\ta}^*),\\
     L^{*-1} \divv  & :   L^2 \cap L^{2,\lambda}_{1} \to L^2 \cap L^{2,\lambda' }_{1}, \quad 0\le \lambda' \le \lambda+2, \quad \lambda' < \lambda(L_{\ta}^*),     \\  
     \nabla L^{*-1} \divv & :   L^2 \cap L^{2,\lambda}_{1} \to L^2 \cap L^{2,\lambda' }_{1}, \quad 0\le \lambda'  \le \lambda, \quad \qquad \lambda' < \lambda(L_{\ta}^*).
\end{align*}
Note that $U^*$ is a combination of the first two lines, so there is a gain of $2$ at most. However for $T^*$, we must use the fourth line so there is no gain. 
Since
$$\mR^*= \begin{bmatrix} L^{*-1}a^* & L^{*-1}  \\ T^* & U^* \end{bmatrix},$$
 starting from  $g^{(0)} \in L^2\times L^2$ and letting $g^{(k+1)}= \mR^*g^{(k)}$ for $k\ge 0$, we argue as follows using facts (b) and (d). 
As $g^{(0)} \in ( L^2 \cap L^{2,0}_{1}) \times  (L^2 \cap L^{2,0}_{1} )$, we see that  $g^{(1)} \in ( L^2 \cap L^{2,2}_{1}) \times  (L^2 \cap L^{2,0}_{1}) $. Next, we see $g^{(2)} \in ( L^2 \cap L^{2,4}_{1}) \times  (L^2 \cap L^{2,2}_{1}) $  unless $\lambda(L_{\ta}^*)\le 2$ in which case we stop and have obtained 
$g^{(2)} \in ( L^2 \cap L^{2,\lambda+2}_{1}) \times  (L^2 \cap L^{2,\lambda}_{1}) $   for all $\lambda <\lambda(L_{\ta}^*)$  (because of (a)). In the case $\lambda(L_{\ta}^*)>2$, we see that 
$g^{(3)} \in ( L^2 \cap L^{2,6}_{1}) \times  (L^2 \cap L^{2,4}_{1}) $  unless $\lambda(L_{\ta}^*)\le 4$ in which case we stop and have obtained 
$g^{(3)} \in ( L^2 \cap L^{2,\lambda+2}_{1}) \times  (L^2 \cap L^{2,\lambda}_{1}) $   for all $\lambda <\lambda(L_{\ta}^*)$. Since $\lambda(L_{\ta}^*)\le n$, we must stop in a finite number of steps.

\subsection{Openness}

We want to prove the analog statement to Proposition \ref{prop:open}, in the range found  in 
Corollary \ref{cor:DGN}, namely

\begin{prop}\label{prop:open2} Fix $p\in (p_{\ta},p_{+}(DB))$. Then for any $B'$ with $\|B-B'\|_{\infty}$ small enough (depending on $p$), $\IH^p_{DB'}=\IH^p_{D}$ with equivalence of norms. Furthermore, for any $b\in H^\infty(S_{\mu})$ with  $\omega_{B}<\mu<\pi/2$, we have 
\begin{equation}
\label{eq:analytic2}
\|b(DB)-b(DB')\|_{\mL(H^p_{D})}\lesssim  \|b\|_{\infty}\|B-B'\|_{\infty}.
\end{equation}
\end{prop}

The proof is the same as for Proposition \ref{prop:open}. Indeed, from \cite{A}, we know that the De Giorgi condition is an open condition of the coefficients of $L^*_{\ta}$. Thus  Corollary \ref{cor:DGN} applies to any perturbation of the corresponding $DB$. Then $H^\infty(S_{\mu})$-functions of $DB'$ are bounded on $H^p_{D}$ uniformly for $\|B-B'\|_{\infty}$ small enough. Thus, the estimate \eqref{eq:analytic2} holds directly for $1<p$ by the theory of analytic functions valued in Banach spaces. For $p\le 1$, it suffices to prove the 
atom to molecule estimate as in Lemma \ref{lem:atom-mol}. This is the only point requiring a specific argument. 

For some $\varepsilon>0$ depending only on $p$ and $ n$,   then  for all   $(\IH^p_{D},1)$-atoms $\alpha$, with associated cube $Q$ and all  $j\ge 0$, 
$$
||b(DB)\alpha||_{L^{2}\left(S_{j}\left(Q\right)\right)}\lesssim \|b\|_{\infty} \left(2^{j}\ell\left(Q\right)\right)^{\frac{n}{2}-\frac{n}{p}}2^{-j\varepsilon} 
$$
and moreover $\int b(DB)\alpha=0$. 

To show this we argue as follows.  For each  integer $M$, there are constants $c_{M\pm}$ such that $\psi(z)= c_{M\pm} (iz)^M(1+iz)^{-2M} (iz)(1+iz)^{-M}$ if $z\in S_{\mu\pm}$ satisfies $\int_{0}^\infty \psi(tz) \, \frac{dt}{t}=1$ for all $z\in S_{\mu}$.  Thus we can resolve $b(DB)$ as $\int_{0}^\infty (b\psi_{t})(DB)\, \frac{dt}{t}$.  As before, it is no loss of generality to assume that the ball associated to $\alpha$ is the unit ball. For $M$  large enough, for all $t>0$ and arbitrary integer $N$ and $j\ge 2$, 
$$
\|(itDB)(I+itDB)^{-M}\alpha||^2_{L^{2}\left(S_{j}\left(Q\right)\right)}\lesssim  \brac{2^j/t}^{-N} t^{n- \frac{2n}{p}- \varepsilon}.
$$ 
This is also valid for $S_{j}(Q)$ replaced by $4Q$. 
This is the estimate \eqref{tent3''}.  Next, the $L^2$ off-diagonal estimates  \eqref{eq:odnpsipq} apply to $b(DB)(itDB)^M(1+itDB)^{-2M}$ to give 
$$
\|1_{E}b(DB)(itDB)^M(I+itDB)^{-2M}1_{F} u\|_{2} \lesssim \|b\|_{\infty} \brac{\dist (E,F)/t}^{-M}\|u\|_{2}
$$
for all $t>0$, Borel sets $E,F \subset \R^n$ and $u\in L^2$ with support in $F$. It is an easy computation to obtain
$$
\|(b\psi_{t})(DB)\alpha||_{L^{2}\left(S_{j}\left(Q\right)\right)}\lesssim  \brac{2^j/t}^{-M} t^{\frac{n}{2}- \frac{n}{p}- \frac{\varepsilon'}{2}}
$$
for large enough $M$ and $0<\varepsilon'<\varepsilon$.  
With this in hand, one can estimate the $t$-integral upon taking $M$ large enough and get the desired bound for $\int_{S_{j}(Q)}|b(DB)\alpha|^2$  when $j\ge 2$. The integral of $\int_{4Q}|b(DB)\alpha|^2$  is controlled as usual using the $H^\infty$-calculus. We skip further details.

\section{Application to perturbation of solvability for the  boundary value problems}\label{sec:perturbation}

Here, we continue some developments started in \cite{AM}.  Some words are necessary. At the time \cite{AM} was written, Theorems \ref{thm:main1} and \ref{thm:main2} of this memoir were known from the present authors.  Part of Theorem  \ref{thm:main1} was reproved in \cite{AM} under some De Giorgi conditions allowing a more direct argument bypassing Hardy space estimates (parts of this proof was due to other authors as mentioned in the introduction) and Theorem \ref{thm:main2} was quoted in \cite{AM} as well as the development on  boundary layers from \cite{HMiMo}. While writing the present article, we have improved the development on boundary layers as presented in  Section \ref{sec:boundarylayers}.

In \cite{AM} the goal was to prove  extrapolation of solvability results for boundary value problems using a method ``\`a la Calder\'on-Zygmund''. For example, it was shown that the solvability of the
regularity (resp. Neumann) problem  in $L^p$, $1<p\le 2$,\footnote{The limitation  $p\le 2$ is inherent to the method used there but can be lifted to $p<p_{+}(DB)$ once we have the needed boundedness
.} 
 with energy solutions can be pushed down to
obtain solvability  in $L^q$ with $1<q<p$ and also $H^q$ with $q_{0}<q\le 1$ where $q_{0}$ is derived from the De Giorgi-Nash conditions used there, which involved interior and boundary regularity for the system \eqref{eq:divform} and its adjoint.  
Also extrapolation for the Dirichlet problems and Neumann problems in negative Sobolev spaces (going up the scale of exponents this time) was deduced   thanks to  Regularity/Dirichlet  and Neumann/Neumann duality principles (see \cite{AM} for explanations).

It is not clear at this time what could be  the similar results as in \cite{AM} in our general framework. First, we do not use here interior regularity.  Secondly, those results require some kind of boundary regularity.

Instead,  we can prove an extrapolation result ``\`a la {\v{S}}ne{\u\i}berg'', namely Theorem \ref{thm:extra}, which  does not require any boundary regularity. Also we establish a stability result in the coefficients.

\subsection{Proof of Theorem \ref{thm:extra}}

We begin with the regularity problem. 

For $\frac{n}{n+1}<q<\infty$ and $X=H^q$, one can formulate two notions of solvability as follows.
First,  $ (R)^L_{X}$ is solvable for the energy class  if  there exists $C_{X}<\infty$ such that
for any $f\in H^q_{\ta}\cap \dot \mH^{-1/2}_{\ta}$ 
the energy solution $u$ of $\divv A \nabla u=0$ on $\reu$ with regularity data $\nabla_{x}u|_{t=0}=f $ satisfies 
$$
\|\tN(\nabla_{A}u)\|_{q}\le C_{X} \|f\|_{H^q_{\ta}}.
$$

We say that $ (R)^L_{X}$ is solvable if there exists a constant $C_{X}<\infty$ such that  for any $f\in H^q_{\ta}$  there exists a weak solution 
 $u$ of $\divv A \nabla u=0$ in $\reu$ with regularity data $\nabla_{x}u|_{t=0}=f $ (in the prescribed sense below) and 
$$
\|\tN(\nabla_{A}u)\|_{q}\le C_{X} \|f\|_{H^q_{\ta}}.
$$
This means that solvability is existence of a solution with prescribed boundary trace and interior estimate.

Although one can formulate these problems for all $q$, they take meaningful sense in the restricted range $I_{L}$. We recall that  $I_{L}$ is the interval  in $(\frac{n}{n+1}, p_{+}(L))$  on which $\IH^q_{DB}=\IH^q_{D}$ with equivalence of norms. For $ q$ in this range, the map
$$
N_{\ta}: H^q_{DB}\to H^q_{\ta}, h\mapsto h_{\ta}
$$
is well-defined and bounded.  

To prove  Theorem \ref{thm:extra}, the first lemma tells us that we can build solutions from our semigroup approach. This is a feature of this method.

\begin{lem}\label{lem:sol} Assume $q\in I_{L}$. Let $S_{q}^+(t)$ be the extension of the semigroup $e^{-t|DB|}$, $t\ge 0$, to $H^{q,+}_{DB}$. Let $h\in H^{q,+}_{DB}$. Then, the function $(t,x)\mapsto S_{q}^+(t)h(x)$ is the conormal gradient of a weak solution $u$ (uniquely determined up to a constant)  of $\divv A \nabla u=0$ on $\reu$ with $$
\|\tN(\nabla_{A}u)\|_{q}\le C_{q} \|h\|_{H^q}.
$$
Moreover, this solution is such that $(\nabla_{A}u)(t,\cdot)$ converges to $h$ in strong $H^q$ topology as $t\to 0$.
\end{lem}

\begin{proof} For $q$ in this range, we know that $H^{q,+}_{DB}$ is a closed subspace of $H^q_{DB}=H^q_{D}$ with $H^q$ topology. When $h$ belongs to the dense class $\IH^{q,+}_{DB}$, we know that $F=e^{-t|DB|}h$ satisfies the non-tangential maximal estimates. 
  Passing to completion for $h\in H^{q,+}_{DB}$, we have  $$
\|\tN(S_{q}^+(t)h)\|_{q}\le C_{q} \|h\|_{H^q},
$$
and in particular, $S_{q}^+(t)h(x)  \in L^2_{loc}$. Also for $h\in \IH^{q,+}_{DB}$, we knew that $F$ was an $L^2_{loc}$ and a solution to  $\pd_{t}F+DBF=0$ in the weak sense, so it is preserved by taking limit in $L^2_{\loc}$.  
Thus there exists a weak solution $u$  (uniquely determined up to a constant)  of $\divv A \nabla u=0$ on $\reu$
such that $\nabla_{A}u(t,x)= S_{q}^+(t)h(x)$ in $L^2_{loc}$ sense. Finally, we have seen the strong convergence of $S_{q}^+(t)$ on $H^{q,+}_{DB}$ (this is easy from the one of the extended semigroup $S_{q}(t)$ on $H^q_{DB}$). So the strong limit as $t\to 0$ is granted.  \end{proof}

\begin{lem}\label{lem:claim} Let $q\in I_{L}$ and $X=H^q$. 
If $ (R)^L_{X}$ is  solvable for the energy class then $N_{\ta}: H^{q,+}_{DB}\to H^q_{\ta}$ is an isomorphism. If $N_{\ta}$ is surjective onto $H^q_{\ta}$ then $ (R)^L_{X}$ is  solvable with strong limit as $t\to 0$ for $\nabla_{\ta}u(t,.)$ in $H^q$ topology.
\end{lem}

   Admitting this lemma, we can finish the proof of Theorem \ref{thm:extra} by applying  the result  of  {\v{S}}ne{\u\i}berg \cite{Snei} in the Banach case and Kalton-Mitrea \cite{KalMit} in the quasi-Banach case. Indeed,  the spaces $H^{q,+}_{DB}$ are complex interpolation spaces: we know this for $H^{q}_{DB}$ and the spectral spaces $H^{q,+}_{DB}$ are the images of $H^{q}_{DB}$ under the bounded extension of the  projection $\chi^+(DB)$. Thus $N_{\ta}: H^{p,+}_{DB}\to H^p_{\ta}$ is invertible  for $p$ in a neighborhood of $q$. This implies that $ (R)^L_{H^p}$ is solvable for $p$ in this neighborhood, applying the second part of the previous Lemma. 
   
   \begin{proof}[Proof of Lemma \ref{lem:claim}]  Let us prove the second statement first. By the open mapping theorem (see \cite{KalMit} for the quasi-Banach version of it), there exists a constant $C>0$ such that  for all $f\in H^q_{\ta}$, one can find  $h\in H^{q,+}_{DB}$, with $N_{\ta}h=f$ and   $\|h\|_{H^q} \lesssim C\|f\|_{H^q}$. Applying Lemma \ref{lem:sol} with $h$ yields a solution.  
   
We now prove the first part.    On the energy class, we know there is a Dirichlet to Neumann map $\Gamma_{DN}: \dot \mH^{-1/2}_{\ta} \to \dot \mH^{-1/2}_{\no}$ that is bounded and invertible by existence and uniqueness of energy solutions with prescribed Dirichlet or Neumann data. See \cite{AMcM} for a proof in this context. 
   Also, we have 
$$
N_{\ta}\circ (\Gamma_{DN}, I)= I_{\dot \mH^{-1/2}_{\ta}}
$$
and 
$$
 (\Gamma_{DN}, I) \circ N_{\ta}= I_{\dot \mH^{-1/2, +}_{DB}}.
$$
Here we use the same name for the map $N_{\ta}:\dot \mH^{-1/2, +}_{DB} \to \dot \mH^{-1/2}_{\ta}$.
We know  $ (R)^L_{X}$  is solvable for the energy class if and only if there exists $C>0$ such that 
$\|\Gamma_{DN} f\|_{H^q}\lesssim \|f\|_{H^q}$ for all $f\in H^q_{\ta}\cap \dot \mH^{-1/2}_{\ta}$ by \cite{AM}, Lemma 10.4. As $H^q_{\ta}\cap \dot \mH^{-1/2}_{\ta}$ is dense in $H^q_{\ta}$, this means that $\Gamma_{DN}$ extends to a bounded operator from $H^q_{\ta}$ into $H^q_{\no}$. 
As $ H^{q,+}_{DB} \cap \dot \mH^{-1/2,+}_{DB}$ is also dense in $H^{q,+}_{DB} $ (see the argument below for convenience), this means that the operator  $ (\Gamma_{DN}, I)$ extends to a bounded operator from $H^q_{\ta}$ into $H^{q,+}_{DB} $. Extending the above operator identities shows that this extension is  the inverse of $N_{\ta}: H^{q,+}_{DB}\to H^q_{\ta}$.

To conclude, we show that $H^{q,+}_{DB} \cap \dot \mH^{-1/2,+}_{DB}$ is dense in $H^{q,+}_{DB}$ in the  $H^q_{DB}$ topology as this topology is equivalent to the $H^q$ topology. As $H^{q,+}_{DB} \cap \dot \mH^{-1/2,+}_{DB}=\chi^+(DB)(H^{q}_{DB} \cap \dot \mH^{-1/2}_{DB})$, it suffices to show that  $\IH^{q}_{DB} \cap \dot \mH^{-1/2}_{DB}$ is dense in $\IH^{q}_{DB}$ (which is dense in $H^q_{DB}$).  Let   $h\in \IH^q_{DB}$.
 Pick a Calder\'on reproducing formula $h=\int_{0}^\infty \psi(tDB)h\, \frac{dt}{t}$ which converges in $\IH^{q}_{DB}$ by construction of these spaces for an appropriate $\psi$. Observe that for fixed $t>0$,  $ \psi(tDB)h \in \dot \mH^{0}_{DB}$ and if $\psi(z)=z\tilde\psi(z)$, we have $ \psi(tDB)h \in \dot \mH^{-1}_{DB}$. Thus, $ \psi(tDB)h \in \dot \mH^{-1/2}_{DB}$. This concludes the argument for the density.  
\end{proof}

\

Let us turn to the Neumann problem.  We say that   $ (N)^L_{X}$ is solvable for the energy class  if  there exists $C_{X}<\infty$ such that
for any $f\in H^q\cap \dot \mH^{-1/2}$ 
the energy solution $u$ of $\divv A \nabla u=0$ on $\reu$ with regularity data $\pd_{\nu_{A}}u|_{t=0}=f $ satisfies 
$$
\|\tN(\nabla_{A}u)\|_{q}\le C_{X} \|f\|_{H^q}.
$$

We say that $ (N)^L_{X}$ is solvable if there exists a constant $C_{X}<\infty$ such that  for any $f\in H^q$  there exists a weak solution 
 $u$ of $\divv A \nabla u=0$ in $\reu$ with regularity data $\pd_{\nu_{A}}u|_{t=0}=f $ (in the prescribed sense below) and 
$$
\|\tN(\nabla_{A}u)\|_{q}\le C_{X} \|f\|_{H^q}.
$$
This means that solvability is existence of a solution with prescribed boundary trace and interior estimate.

   The proof of Theorem \ref{thm:extra} for the Neumann problem on $X=H^q$ is the same with same range for $q$, changing  $N_{\ta}$ to $N_{\no}$ where
$N_{\no}h=h_{\no}$,  and using the following lemma, the proof of which is entirely analogous to the previous one with the Neumann to Dirichlet map $\Gamma_{ND}: \dot \mH^{-1/2}_{\no}\to \dot \mH^{-1/2}_{\ta}$  replacing the Dirichlet to Neumann map $\Gamma_{DN}$ (one being the inverse of the other). 

\begin{lem}\label{lem:claim2}
If $ (N)^L_{X}$ is  solvable for the energy class then $N_{\no}: H^{q,+}_{DB}\to H^q_{\no}=H^q$ is an isomorphism. If this map  is surjective  then $ (N)^L_{X}$ is  solvable with strong limit at $t=0$  for $\pd_{\nu_{A}}u(t,.)$ in $H^q$ topology.
\end{lem}

Let us turn to the Dirichlet problem (formulated with square functions as in the introduction).  We  argue in the dual range  of the interval  in $(\frac{n}{n+1}, p_{+}(L))$  on which $\IH^q_{DB}=\IH^q_{D}$ with equivalence of norms.  By the results in Section \ref{sec:apriori}, it is convenient to introduce new spaces. For $1<p<\infty$, we let $\dot W^{-1,p}_{D}$ be the image of $\dot W^{-1,p}$ under (the bounded extension of) $\IP$. Thanks to  Lemma \ref{lem:D}, it can be identified to the image of $\clos{\ran_{p}(D)}=H^p_{D}$ under $D$, which becomes an isomorphism.  We now assume $p=q'$ with $q$ as above. Thanks to proposition \ref{prop:fcnegsob},  $H^\infty$ functions  of $D\wt B$ act boundedly on $\dot W^{-1,p}_{D}$. Also, by Theorem  \ref{thm:duality} and Corollary \ref{cor:hpbd0}, we can see that $D$ extends to an isomorphism from $H^p_{\wt BD}$ onto $\dot W^{-1,p}_{D}$ and the relation $Db(\wt BD)= b(D\wt B)D$ valid on an appropriate dense subspace  of $H^{p}_{\wt BD}$ for $b\in H^\infty(S_{\mu})$ extends by density to  $H^p_{\wt BD}$. Thus we can define $\dot W^{-1,p, \pm}_{D\wt B}=D H^{p,\pm}_{\wt BD}= D\chi^+(\wt B D)H^{p}_{\wt BD}$  and  a strongly continuous semigroup on $\dot W^{-1,p, \pm}_{D\wt B}$,  which extends $(e^{-t|D \wt B|})_{t\ge 0}$. All this is consistent as long as $p=q'$ because we  work in the ambient space of Schwartz distributions. 

If $q\le 1$ and $\sigma= n(\frac{n}{q}-1)$, we can define $\dot \Lambda^{\alpha-1}_{D}$ and 
$\dot \Lambda^{\alpha-1, \pm}_{D\wt B}$ as images of $\Lambda^{\alpha}_{\wt BD}$ and 
$\dot \Lambda^{\alpha, \pm}_{\wt BD}$ under the extension of $D$, which is an isomorphism (this uses again Theorem  \ref{thm:duality} and Corollary \ref{cor:hpbd0}). 
Again, by similarity, the boundedness and regularity of semigroups carry to this setting. So the semigroup  extending $e^{-t|D\wt B|}$ by this construction is weakly-$*$ continuous. The natural predual  in the duality defined in Section \ref{sec:comp}  is $\dot H^{1,q}_{BD}$ which is defined via
completion of the space $\IH^2_{BD}$ for the norm $\|t^{-1} \psi(tBD)h\|_{T^q_{2}}$ for appropriate $\psi$. This is routinely done as for the Hardy spaces we have developed with much details and we skip those here. But, as $q\in I_{L}$, this space identifies to $\dot H^{1,q}_{D}$  under the projection $\IP$. So the weak-$*$ continuity is against any distribution in $\dot H^{1,q}_{D}$ or even in $\dot H^{1,q}$ (because the $D$ null distributions in $\dot H^{1,q}$ are annihilated by 
$\dot \Lambda^{\alpha-1}_{D}$ elements).  

We mention, that in the range of $p$ and $\alpha$ specified above ($p=q'$ or $\alpha=n(\frac{1}{q}-1)$,  the scalar  parts of $H^p_{\wt BD}$ elements  are in fact $L^p$ functions. Similarly the scalar parts of $\dot \Lambda^{\alpha}_{BD}$ elements are $\dot \Lambda^\alpha$ functions. 

For $Y=L^{p}$ or $\dot \Lambda^{\alpha}$ with $1<p<\infty$ or $0\ge \alpha<1$, and $\mT=T^p_{2}$ or $T^\infty_{2,\alpha}$, 
one can formulate two notions of solvability for the Dirichlet problem as follows.
First,  $ (D)^{L^*}_{Y}$ is solvable for the energy class  if  there exists $C_{Y}<\infty$ such that
for any $f\in Y\cap \dot \mH^{1/2}_{\no}$ 
the energy solution $u$ of $\divv A^* \nabla u=0$ on $\reu$ with Dirichet data $u|_{t=0}=f $ satisfies 
$$
\|t\nabla_{A^*}u\|_{\mT}\le C_{Y} \|f\|_{Y}\sim C_{Y}\|\nabla f\|_{\dot Y^{-1}}.
$$

We say that $ (D)^{L^*}_{Y}$ is solvable if for any $f\in Y$  there exists a solution 
 $u$ of $\divv A^* \nabla u=0$ in $\reu$ with regularity data $u|_{t=0}=f $ (in the prescribed sense below) and 
$$
\|t\nabla_{A^*}u\|_{\mT}\le C_{Y} \|f\|_{Y}\sim C_{Y}\|\nabla f\|_{\dot Y^{-1}}.
$$
This means that solvability is existence of a solution with prescribed boundary trace and interior estimate.  

Although one can formulate these problems for all $p$ or $\alpha$, they take meaningful sense in the restricted  dual range of $I_{L}$. We recall that  $I_{L}$ is the interval  in $(\frac{n}{n+1}, p_{+}(L))$  on which $\IH^q_{DB}=\IH^q_{D}$ with equivalence of norms. For $ q$ in this range, the map
$$
N_{\ta}: \dot Y^{-1,+}_{D\wt B}\to \dot Y^{-1}_{\ta}, h\mapsto h_{\ta}
$$
is well-defined and bounded. It is convenient to set  $\dot Y^{-1}_{\ta}$, the space of distributions of the form $\nabla f$ in $\dot Y^{-1}$. 

To prove  Theorem \ref{thm:extra} for the Dirichlet problem, the first lemma tells us that we can build solutions from our semigroup approach.

\begin{lem}\label{lem:soldir} Assume $q\in I_{L}$. Let $\wt S_{Y^{-1}}^+(t)$ be the extension of the semigroup $e^{-t|D\wt B|}$, $t\ge 0$, to $\dot Y^{-1,+}_{D\wt B}$ described above. Let $h\in \dot Y^{-1,+}_{D\wt B}$. Then, the function $(t,x)\mapsto \wt S_{Y^{-1}}^+(t)h(x)$ is the conormal gradient of a weak solution $u$ (uniquely determined up to a constant)  of $\divv A^* \nabla u=0$ on $\reu$ with $$
\|t\nabla_{A^*}u\|_{\mT}\le C_{Y}\|h\|_{\dot Y^{-1}}.
$$
Moreover this solution, is such that $(\nabla_{A^*}u)(t,\cdot)$ converges to $h$ at $t\to 0$  in  the strong  topology  of $\dot Y^{-1}$   if $q>1$ and  in the  weak-$*$ topology of $\dot Y^{-1}$ if $q\le 1$. 

Moreover, in the case $q>1$ and $p=q'$,
 $t\mapsto u(t,\cdot) \in C_{0}([0,\infty); L^p(\R^n;\C^m)) +\C^m$. If  one normalizes the constant to be 0 (by either imposing $u|_{t=0}\in L^p(\R^n;\C^m)$ or by imposing that the solution 
converges to 0 at $\infty$ is some weak sense), then  this solution satisfies the layer potential representation as in Corollary \ref{cor:BL} taking the bounded extensions of the layer potentials for $L^*$,
$\mS_{t}^{A^*}$ from $ \dot W^{-1,p}(\R^n; \C^m)$ to $L^p(\R^n;\C^m)$ and $\mD_{t}^{A^*}$ on $L^p(\R^n;\C^m)$ proved in Theorem \ref{thm:bl} (3) and (4). Finally, one has
the non-tangential maximal estimate $\|\tN{u}\|_{p}\lesssim \|t\nabla u\|_{T^p_{2}}$ (again the  constant is imposed to be 0). 
\end{lem}

\begin{proof}
The first part of the proof is again is consequence of the construction and the estimates, once we see that $(t,x)\mapsto \wt S_{Y^{-1}}^+(t)h(x)$ is an $L_{loc}^2$ function on $\reu$. We see this and skip other details. By construction it is a tempered distribution on $\reu$. If $q>1$, then the semigroup extends by density from  $\IH^{2,+}_{D\wt B}\cap \dot W^{-1,p,+}_{D\wt B}$ and  on such a dense space we have seen that  $(t,x) \to t\wt S_{Y^{-1}}^+(t)h(x)$ belongs to $T^p_{2}$. The density argument yields convergence in $T^p_{2}$, thus in $L^2_{loc}$. 
For $q\le 1$ and $\alpha=n(\frac{1}{q}-1)$, $(t,x) \to t\wt S_{Y^{-1}}^+(t)h(x)$ is build as a weak-$*$ limit in $T^\infty_{2,\alpha}$, hence it also has the $L^2_{loc}$ property. 

Let us turn to the second part of the proof. By assumption, $h_{\ta}=\nabla f$ for some $f\in L^p(\R^n;\C^m)$. Also $h_{\no}\in  \dot W^{-1,p}(\R^n; \C^m)$.  Then we have  
$$
 \nabla_{A^*}u(t,\cdot)=  \nabla_{A^*}\mS_{t}^{A^*}h_{\no} -  \nabla_{A^*}\mD_{t}^{A^*}f,
 $$
 where $\mS_{t}^{A^*}$ and $\mD_{t}^{A^*}$ are understood as the appropriate extensions.  
To see this, we proceed exactly as in the proof of Corollary  8.4 in \cite{AM}, starting from the fact that $ \nabla_{A^*}u(t,\cdot)$ is defined by the semigroup representation  using the abstract definitions of the layer potentials and density arguments. Once this is established,  the rest of the proof is similar to   that  of  Corollary \ref{cor:BL} for the convergence issues. We skip details.
The non-tangential maximal estimate follows from a similar approximation argument as for the proof of \eqref{eq:N<S}. 
\end{proof}

Then the result concerning the solvability of Dirichlet problems is the following one.

\begin{lem}\label{lem:claimdir} Let $q\in I_{L}$ and $Y$ be as above. 
If $ (D)^{L^*}_{Y}$ is  solvable for the energy class then $N_{\ta}: \dot Y^{-1,+}_{D\wt B}\to \dot Y^{-1}_{\ta}$ is an isomorphism. If $N_{\ta}$ is surjective onto $\dot Y^{-1}_{\ta}$ then $ (D)^{L^*}_{Y}$ is  solvable with  limit as $t\to 0$ for $u(t,.)$ in $L^p$ topology if $q>1$ and $p=q'$ or with  limit as $t\to 0$ for $u(t,.)$ in $\dot \Lambda^\alpha$ weak-$*$ topology if $q\le1$ and $\alpha=n(\frac{1}{q}-1)$.
\end{lem}

\begin{proof} The first part of the proof proceeds as the one of Lemma \ref{lem:claim} with the  Dirichlet to Neumann map $\Gamma_{DN}: \dot \mH^{-1/2}_{\ta} \to \dot \mH^{-1/2}_{\no}$.  We have that   $ (D)^{L^*}_{Y}$  is solvable for the energy class if and only if there exists $C>0$ such that 
$\|\Gamma_{DN} g\|_{\dot Y^{-1}}\lesssim \|g\|_{\dot Y_{\ta}^{-1}}$ for all $g\in \dot Y^{-1}_{\ta}\cap \dot \mH^{-1/2}_{\ta}$. This is a reformulation of  \cite{AM}, Corollary 11.3. Then similar density arguments show that the extension of the map $(\Gamma_{DN},I)$ is the desired inverse of $N_{\ta}$.  The second part is an application of the open mapping theorem again. 
\end{proof}

The proof of Theorem \ref{thm:extra} is now  done as the one for the regularity problem.

\

We finish with the Neumann problem on negative Sobolev/H\"older spaces. Again $\dot Y^{-1}=\dot W^{-1,p}$ or $\dot \Lambda^{\alpha-1}$. First,  $(N)_{Y^{-1}}^{L^*}$ is solvable for the energy class  if  there exists $C_{Y}<\infty$ such that
for any $f\in \dot Y^{-1}\cap \dot \mH^{-1/2}$ 
the energy solution $u$ of $\divv A^* \nabla u=0$ on $\reu$ with Neumann data $\pd_{\nu_{A^*}}u|_{t=0}=f $ satisfies 
$$
\|t\nabla_{A^*}u\|_{\mT}\le C_{Y}\| f\|_{\dot Y^{-1}}.
$$

We say that $(N)_{Y^{-1}}^{L^*}$ is solvable if for any $f\in \dot Y^{-1}$  there exists a weak solution 
 $u$ of $\divv A^* \nabla u=0$ in $\reu$ with Neumann data $\pd_{\nu_{A^*}}u|_{t=0}=f $ (in the prescribed sense below) and 
$$
\|t\nabla_{A^*}u\|_{\mT}\le  C_{Y}\| f\|_{\dot Y^{-1}}.
$$
This means that solvability is existence of a solution with prescribed boundary trace and interior estimate.

The  proof of Theorem \ref{thm:extra} for the Neumann problem on  a negative Sobolev/H\"older space is the same as for the  Dirichlet problem with same range for $q$:  we know how to construct solutions by Lemma \ref{lem:soldir}. Next, changing  $N_{\ta}$ to  $N_{\no}: \dot Y^{-1,+}_{D\wt B}\to \dot Y^{-1}_{\no}=\dot Y^{-1},  N_{\no}h=h_{\no}$,    we use the following lemma, the proof of which is entirely analogous to the previous one with the Neumann to Dirichlet $\Gamma_{ND}: \dot \mH^{-1/2}_{\no}\to \dot \mH^{-1/2}_{\ta}$ map replacing the Dirichlet to Neumann map $\Gamma_{DN}$ (one being the inverse of the other). 

\begin{lem}\label{lem:claimneu} Let $q\in I_{L}$ and $Y$ be as above. 
If $(N)_{Y^{-1}}^{L^*}$ is  solvable for the energy class then $N_{\no}: \dot Y^{-1,+}_{D\wt B}\to \dot Y^{-1}$ is an isomorphism. If $N_{\no}$ is surjective onto $\dot Y^{-1}$ then $(N)_{Y^{-1}}^{L^*}$ is  solvable with  limit as $t\to 0$ for $\pd_{\nu_{A^*}}u(t,.)$ in  strong  topology of $\dot W^{-1,p}$ if $q>1$ and $p=q'$ or with  limit as $t\to 0$ for $\pd_{\nu_{A^*}}u(t,.)$ in  weak-$*$ topology on $\dot \Lambda^{\alpha-1}$  if $q\le1$ and $\alpha=n(\frac{1}{q}-1)$.
\end{lem}

\begin{rem} Concerning the Dirichlet problem under De Giorgi type condition on $L_{\ta}$, this theorem covers the case of $\BMO$  data. In this case, this shows that if the Dirichlet problem  for $L$ is solvable for the energy class with $\BMO$ data, then it is solvable (may be not for the energy class) with $L^p$ data for unspecified large $p$'s.  This result for real, non-necessarily $t$-independent equations, is in \cite{DKP}  and we extend it here to more general systems  when the coefficients are $t$-independent. In case of real equations, solvability for the energy class is reached due to the harmonic measure techniques used.  
\end{rem}

\subsection{Stability in the coefficients}
We now establish stability under perturbation of the coefficients in the $t$-independent coefficients class. We do this for the regularity problem. For each of the other 3 boundary value problems, there will be  similar statement and proof which we shall not include and leave to the reader. This can be compared to  prior results established in the literature  for systems in the upper half-space with $t$-independent coefficients (\cite{DaK, Br, KP, KalMit, B, HKMP2}, etc)  or bi-lipschitz diffeomorphic images of this situation. The only point is that we do not know how to obtain solvability in the energy class in the conclusion but only prove solvability. 

\begin{thm}\label{thm:stability}
Let $I_{L}$ be the interval   $(p_{-}(DB))_{*}, p_{+}(DB))$ of Theorem \ref{thm:hpdb} or  $(p_{\ta}, p_{+}(DB))$ of Corollary \ref{cor:DGN}   on which $\IH^q_{DB}=\IH^q_{D}$ with equivalence of norms and set $X=H^q$. If $ (R)^L_{X}$ is solvable for the energy class then  
$(R)^{L'}_{X}$ is solvable where $L'=-\divv A'\nabla $ has $t$-independent coefficients with $\|A-A'\|_{\infty}$ small enough depending on $X$. 
\end{thm}

\begin{proof} The assumption allows us to apply Proposition \ref{prop:open} or Proposition \ref{prop:open2}.  In both situations, if $q\in I_{L}$ we have, 
$\|\chi^+(DB')-\chi^+(DB)\|_{\mL(H^q_{D})}\le C \|A-A'\|_{\infty}$ for small enough $  \|A-A'\|_{\infty}$, where $B=\widehat A, B'=\widehat A'$. As $\chi^+(DB')$ is a  projector, it implies that it is an isomorphism from $H^{q,+}_{DB}$ onto $H^{q,+}_{DB'}$ with uniform bounds for small enough $  \|A-A'\|_{\infty}$.  
Next, $N_{\ta}\chi^+(DB'): H^{q,+}_{DB} \to H^q_{\ta}$ is a perturbation of $N_{\ta}=N_{\ta}\chi^+(DB): H^{q,+}_{DB} \to H^q_{\ta}$ in operator norm. As solvability of $ (R)^L_{X}$  for the energy class implies $N_{\ta}: H^{q,+}_{DB} \to H^q_{\ta}$ is invertible by Lemma \ref{lem:claim}, it follows that $N_{\ta}\chi^+(DB'): H^{q,+}_{DB} \to H^q_{\ta}$ is invertible for small enough $  \|A-A'\|_{\infty}$. Combining these two informations, we obtain that $N_{\ta}: H^{q,+}_{DB'} \to H^q_{\ta}$ is invertible  uniformly for $\|A-A'\|_{\infty}$ small enough. This implies solvability of $ (R)^{L'}_{X}$ by Lemma \ref{lem:claim}. 
\end{proof}

\begin{rem}  Although it seems natural to expect it, we are not able to remove the assumption on $I_{L}$ at this time. 
\end{rem}

\bibliographystyle{acm}

\end{document}